\documentclass[a4paper,10pt]{report}
\usepackage{amssymb,latexsym,amsmath,mathrsfs,amsthm,verbatim,soul}
\usepackage[all]{xy}
\usepackage{appendix}

\newtheorem{theorem}{Theorem}

\newtheorem{thm}{Theorem}
\newtheorem{lem}[thm]{Lemma}
\newtheorem{prop}[thm]{Proposition}
\newtheorem{cor}[thm]{Corollary}

\theoremstyle{definition}
\newtheorem{defn}[thm]{Definition}

\theoremstyle{definition}

\theoremstyle{definition}
\newtheorem{examples}[thm]{Examples}

\theoremstyle{definition}

\theoremstyle{definition}
\newtheorem{rem}[thm]{Remark}

\newcommand{\Hide}[1]{}

\newenvironment{pf}[1][Proof Visible]{\begin{proof}}{\end{proof}}

\numberwithin{thm}{chapter}

\newdir{ >}{{}*!/-5pt/@{>}}

 \newcommand{\To}{\longrightarrow}

 \newcommand{\ds}{\displaystyle}
 \newcommand{\Fl}{\mathbb{F}_{l}}
 \newcommand{\Fq}{\mathbb{F}_{q}}
 \newcommand{\fl}[1]{\mathbb{F}_{l^{#1}}}
 \newcommand{\fq}[1]{\mathbb{F}_{q^{#1}}}
 \newcommand{\Zp}{\mathbb{Z}_p}
 
 \newcommand{\Flbar}{\overline{\mathbb{F}}_{l}}
 \newcommand{\tensor}{\otimes}
 
 \newcommand{\Hom}{\text{Hom}}
 \newcommand{\End}{\text{End}}
 \newcommand{\Rep}{\text{Rep}}
 \newcommand{\Irr}{\text{Irr}}
 \newcommand{\Gal}{\text{Gal}}
 \newcommand{\Syl}{\text{Syl}}
 \newcommand{\Frob}{F}
 \newcommand{\Map}{\text{Map}}
 \newcommand{\Aut}{\text{Aut}}
 \newcommand{\im}{\text{image}}
 \newcommand{\id}{\text{id}}
 \newcommand{\orb}{\text{orb}}
 \newcommand{\res}{\text{res}}
 \newcommand{\tr}{\text{tr}}
 \newcommand{\rank}{\text{rank}}
 \newcommand{\stab}{\text{stab}}
 \newcommand{\nil}{\text{nil}}
 \newcommand{\ann}{\text{ann}}
 \newcommand{\coker}{\text{coker}}
 \newcommand{\soc}{\text{soc}}
 \newcommand{\spf}{\text{spf}}
 \newcommand{\skel}{\text{skel}}
 \newcommand{\transfer}{\text{tr}}
 \newcommand{\colim}{\underset{\to}{\text{lim }}}
 \newcommand{\limit}{\underset{\leftarrow}{\text{lim }}}
 \newcommand{\thstar}{\Theta^*}
 \newcommand{\Gflbar}{GL_d(\Flbar)}
 
 \newcommand{\conj}{\text{conj}}
 \newcommand{\euler}{\text{euler}}
 \newcommand{\nin}{\not \in}
 
 \newcommand{\iso}{\overset{\sim}{\longrightarrow}}

 \newcommand{\lpow}{[\hspace{-1.5pt}[}
 \newcommand{\rpow}{]\hspace{-1.5pt}]}
 
\begin{document}

\begin{titlepage}
\begin{center}
\phantom{blank}~\\
\vspace{6cm}
{\bf \Large The Morava $\mathbf{E}$-theories of finite general linear groups}\\
\vspace{1cm}
{\bf \large Samuel John Marsh}\\
\vspace{1cm} {\small A thesis submitted for the degree of}\\
{\small Doctor of Philosophy}\\
\vspace{4cm}
{\small Department of Pure Mathematics\\ School of Mathematics and Statistics\\
The University of Sheffield}\\
\bigskip
{\small April 2009} \vspace{1cm}
\end{center}
\end{titlepage}


\begin{abstract}
By studying the representation theory of a certain infinite $p$-group and using the generalised characters of
Hopkins, Kuhn and Ravenel we find useful ways of understanding the rational Morava $E$-theory of the classifying
spaces of general linear groups over finite fields. Making use of the well understood theory of formal group
laws we establish more subtle results integrally, building on relevant work of Tanabe. In particular, we study
in detail the cases where the group has dimension less than or equal to the prime $p$ at which the $E$-theory is
localised.\\
\vspace{6cm} ~
\end{abstract}

\tableofcontents


\chapter{Introduction}
\label{ch:intro}

\section{Project overview}

\subsection{Background}

For any finite group $G$ there is an associated topological space, $BG$, known as the classifying space for $G$.
This space is closely related in structure to the group itself and lends itself well to being studied by
cohomological methods (see \cite{Benson}). We will look at groups of the form $GL_d(K)$, where $K$ is a finite
field, and their associated classifying spaces.

We study the spaces $BGL_d(K)$ using a family of generalised cohomology theories known as the Morava
$E$-theories. For each integer $n\geq 1$ and each prime $p$ there is an even-periodic cohomology theory $E$ with
coefficient ring $E^*=\mathbb{Z}_p\lpow u_1,\ldots,u_{n-1}\rpow[u,u^{-1}]$, where $\mathbb{Z}_p$ denotes the
$p$-adic integers, $u_1,\ldots,u_{n-1}$ all lie in degree zero and $u$ is an invertible element in degree $-2$.
These theories turn out to be computable yet, taken together, give a great deal of information (see
\cite{RavenelNil}). As covered in Chapter \ref{ch:FGLs}, the Morava $E$-theories are complex oriented and have
close relations with the theory of formal group laws. Under mild hypotheses, reduction modulo the ideal
$(p,u_1,\ldots,u_{n-1})$ gives a related theory, $K$, with which is associated the theory of formal group laws
over finite fields.

The starting point for the calculation of $E^*(BGL_d(K))$ is the work of Friedlander and Mislin
(\cite{FriedlanderMislin}, \cite{Friedlander}) who showed that, whenever $l$ is a prime different to $p$, the
mod $p$ cohomology of $BGL_d(\Flbar)$ coincides with that of $BGL_d(\mathbb{C})$, where $\Flbar$ denotes an
algebraic closure of the finite field with $l$ elements. Borel (\cite{Borel}) had already shown that, letting
$T$ denote the maximal torus in $GL_d(\mathbb{C})$, the latter could be described in terms of the invariant
elements of the cohomology of $BT$ under the permutation action of the relevant symmetric group.

In \cite{Tanabe}, Tanabe used the above ideas to show that, for a theory $K(n)$ closely related to $K$ above,
the $K(n)$-cohomology of $BGL_d(\fl{r})$ can be recovered from that of $BGL_d(\Flbar)$ as the coinvariants under
the action of the Galois group $\Gal(\Flbar/\fl{r})$, and that $K(n)^*(BGL_d(\Flbar))$ is just a power series
ring over $\mathbb{F}_p$.

There are some general techniques due to Hovey and Strickland (\cite{HoveyStrickland}) that enable results from
the theory $K(n)$ to be carried over into $E$-theory. In this vein, an extensive study of the $E$-theory of
$B\Sigma_d$ has been carried out by the latter author in \cite{StricklandSymmetricGroups}. Also relevant is
\cite{StricklandK(n)duality} where it was shown that, for finite $G$, $E^*(BG)$ has duality over its coefficient
ring.

The most useful tool for our understanding turns out to be the generalised character theory of Hopkins, Kuhn and
Ravenel (\cite{HKR}). There they show that, rationally, studying the $E$-theory of $BG$ for a finite group $G$
reduces to understanding commuting $n$-tuples of $p$-elements of $G$. We find that when $G=GL_d(K)$ for some
finite field $K$ of characteristic not equal to $p$ this reduces to understanding the $K$-representation theory
of the group $\mathbb{Z}_p^n$.

\subsection{Thesis outline}

In Chapter 2 we outline the basic material and conventions used in the thesis looking, in particular, at finite
fields, local rings and the $p$-adic integers. We also explore the notion of duality algebras and make some
preliminary calculations, establishing some basic results on the $p$-divisibility of integers of the form
$k^s-1$.

Fixing a prime $p$, in Chapter 3 we study the $p$-local structure of the finite general linear groups $GL_d(K)$
for finite fields $K$ of characteristic different from $p$ and find that it relates closely to that of the
symmetric group $\Sigma_d$. We give particular attention to the groups of dimension $p$ over fields for which
$p$ divides $|K^\times|$ and find that they have only two different conjugacy classes of abelian $p$-subgroups,
one being represented by the maximal torus and the other by a cyclic group. We also look at the normalizer of
the maximal torus in this case, finding that it has a finer $p$-local structure.

Chapter 4 details the relevant theory of formal group laws and defines the standard $p$-typical formal group law
that is fundamental to the development of the Morava $E$-theories, which we later go on to define. We also
outline the relevant known results in the Morava $E$-theory of classifying spaces, using the relationship
between mod $p$ cohomology and the Morava $K$- and $E$-theories. We show that, for all of the groups $G$ we
consider, the Morava $E$-theory of $BG$ is free and lies in even degrees. It follows that the Morava $K$-theory
of $BG$ is recoverable from the $E$-theory in simple algebraic terms. We introduce variants of the standard
chern and euler classes which prove to be more convenient in our setting.

In Chapter 5 we look at the generalised character theory of Hopkins, Kuhn and Ravenel and apply it to the
general linear group $GL_d(\Fq)$, where $q=l^r$ is a power of a prime different to $p$. We introduce the groups
$\thstar=\mathbb{Z}_p^n$, $\Phi=(\mathbb{Z}/p^\infty)^n$ (where $\mathbb{Z}/p^\infty=\colim (\mathbb{Z}/p^k)$)
and $\Lambda=\langle q\rangle\leqslant \mathbb{Z}_p^\times$. We let $\Lambda$ act on $\Phi$ and find that the
set of $d$-dimensional $\Fq$-representations of $\thstar$ bijects with $(\Phi^d/\Sigma_d)^\Lambda$ (see Theorem
\ref{Rep(thstar,G)}). We also give thought to the cases where $d$ is less than or equal to $p$ finding that,
under the hypothesis that $p$ divides $q-1$, we can understand the latter set well. The generalised character
theory then gives us a complete description of $L\tensor_{E^*} E^*(BGL_d(\Fq))$, where $L$ is some extension of
$\mathbb{Q}\tensor E^*$.

The aim in Chapter 6 is to get a good description of $E^*(BGL_d(\Fq))$ for the cases where $d$ is at most $p$
and $p$ divides $q-1$. We show that Tanabe's results on the Morava $K$-theory of the relevant spaces lifts to
$E$-theory; that is, we have $E^*(BGL_d(\Fq))\simeq E^*(BGL_d(\Flbar))_{\Gamma}$ where $\Gamma=\Gal(\Flbar/\Fq)$
acts on $GL_d(\Flbar)$ component-wise and hence also on its cohomology. Letting $T_d\simeq (\Fq^\times)^d$
denote the maximal torus of $GL_d(\Fq)$ and letting $\Sigma_d$ act by permuting the coordinates we show that the
restriction map $\beta:E^*(BGL_d(\Fq))\to E^*(BT)$ has image $E^*(BT)^{\Sigma_d}$. Further, when $d<p$ this map
is an isomorphism onto its image (Theorem \ref{d<p}). For the case $d=p$, we choose a basis for $\fq{p}$ over
$\Fq$ to get an embedding $\fq{p}^\times\rightarrowtail GL_p(\Fq)$ and hence a map in $E$-theory
$E^*(BGL_p(\Fq))\to E^*(B\fq{p}^\times)$. There is a quotient ring $D$ of $E^*(B\fq{p}^\times)$ and we let
$\alpha$ be the composition $E^*(BGL_p(\Fq))\to D$. Since $\Gamma=\Gal(\Flbar/\Fq)$ acts on $\fq{p}$ it also
acts on $D$ and we show that $\alpha$ has image $D^\Gamma$. We find that $\alpha$ and $\beta$ are jointly
injective and that they induce an isomorphism $\mathbb{Q}\tensor E^*(BGL_p(\Fq))\simeq \mathbb{Q}\tensor
E^*(BT)^{\Sigma_p}\times \mathbb{Q}\tensor D^\Gamma$. We also show that, in this situation, the kernel of
$\beta$ is principal and we are able to give an explicit basis for $E^*(BGL_p(\Fq))$ over $E^*$ (see Theorem
\ref{Theorem, d=p}).

\section{Acknowledgements}

I am very grateful to my supervisor, Neil Strickland, who has been so very patient in giving me detailed
accounts of some of the complex background theory and has, I am sure, bit his lip at some of the lower-level
questions I have asked over the past four years. I also would like to thank MJ Strong, Dave Barnes, Ian Young
and David Gepner who have all helped me to get through relatively unscathed, whether it be with academic support
or by simply lending an ear. There are many others in Sheffield to whom I should say thanks, but I reserve the
right not to do so here. Finally, I would like to say a big thank you to my fianc\'ee, Lucy, who has been
extremely supportive and has helped me to live a full and varied life outside of the strange and sometimes
isolating world of research mathematics.

\section{Notational conventions}

\begin{itemize}
\item For a positive integer $k$ we define $C_k=\{z\in S^1\mid z^k=1\}$ to be the cyclic subgroup of $S^1$ of order
$k$.
\item Let $H$ be a subgroup of $G$. Then we write $N_G(H)$ for the normalizer of $H$ in $G$.
\item Given a group $G$ and a prime $p$ we will write $\Syl_p(G)$ to denote a Sylow $p$-subgroup of $G$. Note
that $\Syl_p(G)$ is determined up to non-canonical isomorphism.
\item For a group $G$ and an element $g\in G$ we write $\conj_g$ for the conjugation map $G\to G$, $h\mapsto ghg^{-1}$.
\item Given a group $G$ acting on a set $S$ we let $S^G$ denote the $G$-invariant elements of $S$. Similarly, if
$G$ acts on a ring $R$ we let $R_G$ denote the coinvariants of the action; that is, $R_G=R/(r-g.r\mid r\in R,
g\in G)$.
\item Given a ring $R$ we denote by $R[x]$ the ring of polynomials in $x$ with coefficients in $R$ and, likewise, by $R\lpow
x\rpow$ the ring of formal power-series.
\item We will write $R^\times$ for the group of units of a ring $R$ under multiplication. If $a,b\in R$ then we
write $a\sim b$ to denote that $a$ is a unit multiple of $b$ in $R$.
\item We write $\text{nil}(R)$ for the nilradical and $\text{rad}(R)$ for the Jacobson radical of a ring $R$,
the former being the set of nilpotents and the latter the intersection of the maximal ideals.
\item Unless otherwise indicated, the symbol $\tensor$ will denote the tensor product over $\mathbb{Z}$.
\item We will write $\Hom(-,-)$ for the set of homomorphisms between two objects, where the structure should be clear from the
context, and $\Aut(-)$ for the set of automorphisms of an object. We will write $\Map(-,-)$ for the set of
functions between two sets.
\item We will use $H$ to denote singular homology and cohomology.
\item For a (generalised) cohomology theory $h$ we will write $h^*=h^*(pt)$ to denote the ring of coefficients.
\item Given a subspace $Y\subseteq X$ we write $\res_{Y}^{X}$ for the map in cohomology $h^*(X)\to h^*(Y)$.
\end{itemize}



\chapter{Preliminaries}\label{ch:prelims}

\section{Definitions, conventions and preliminary results}

We outline some of the basic definitions and results needed. Unless otherwise stated, all rings and algebras are
commutative and unital with homomorphisms respecting the units. All of the material presented here is well known
and good reference texts include \cite{LangAlgebra}, \cite{Benson} and \cite{Matsumura}.

\subsection{Local rings}\label{sec:local rings}

A ring $R$ is known as a \emph{local ring} if it has precisely one maximal ideal. We write $(R,\mathfrak{m})$ to
denote the local ring $R$ with maximal ideal $\mathfrak{m}$. It is easy to show that $R$ is local if and only if
$R\setminus R^\times$ is an ideal in $R$, since every element of a proper ideal is a non-unit and every non-unit
generates a proper ideal. This ideal will necessarily be the unique maximal ideal of $R$. If $I$ is any ideal in
$R$ then the ring $R/I$ is again local with maximal ideal $\mathfrak{m}/I$. We will need the following result.

\begin{lem} If $(R,\mathfrak{m})$ is a local ring then so is $R\lpow x\rpow$ with maximal ideal
$\mathfrak{m}\lpow x\rpow+(x)$.\end{lem}
\begin{pf} Any power series $f(x)\in R\lpow x\rpow$ is invertible if and only if $f(0)\in R^\times$ (see, for example, \cite[Proposition 1]{Frohlich}). Since
$\mathfrak{m}=R\setminus R^\times$ it follows that $R\lpow x\rpow\setminus (R\lpow x\rpow)^\times$ is the ideal
$\mathfrak{m}\lpow x\rpow +(x)$.\end{pf}

Given any ring $R$ and an ideal $I$ in $R$ we define the {\em $I$-adic topology} on $R$ to be the topology
generated by the open sets $x+I^n$ ($x \in R$, $n\in \mathbb{N}$). If $\bigcap_n I^n=0$ then this topology
coincides with the one given by the metric
$$d(a,b)=\left\{\begin{array}{ll} 2^{-n} & \text{if $a-b\in I^n$ but $a-b\nin I^{n+1}$}\\
0 & \text{if $a-b\in I^n$ for all $n$}\end{array}\right.$$ where we use the convention that $I^0=R$. Here the
number 2 occurring is arbitrary; we get an equivalent metric choosing any real number greater than 1.

If not otherwise stated we assume that a local ring $(R,\mathfrak{m})$ carries the $\mathfrak{m}$-adic topology.
In particular, by a {\em complete local ring} we will mean a local ring that it is complete with respect to the
topology generated by its maximal ideal.

If $(R,\mathfrak{m})$ is a local ring we define the \textit{socle} of $R$, denoted $\soc(R)$, to be the
annihilator of the maximal ideal $\mathfrak{m}$. That is, $\soc(R)=\ann_R(\mathfrak{m})= \{r\in R\mid
r\mathfrak{m}=0\}$. Since $\mathfrak{m}.\soc(R)=0$ it follows that $\soc(R)$ is a vector space over the field
$R/\mathfrak{m}$.

\subsection{The $p$-adic numbers}\label{sec:p-adics}

Let $p$ be a prime. We define the {\em $p$-adic integers}, denoted $\mathbb{Z}_p$, to be the completion of
$\mathbb{Z}$ with respect to the $(p)$-adic topology. Note that, since $\bigcap_n (p^n)=0$, $\mathbb{Z}_p$ is in
fact a metric space. The ideal $p\mathbb{Z}_p$ is the unique maximal ideal of $\mathbb{Z}_p$ and hence
$\mathbb{Z}_p$ is a complete local ring.

An alternative characterisation of the $p$-adic integers is obtained by considering the inverse system
$$\mathbb{Z}/p\leftarrow \mathbb{Z}/p^2 \leftarrow
\mathbb{Z}/p^3\leftarrow\ldots$$ and defining $\mathbb{Z}_p=\limit \mathbb{Z}/p^n$. In this way, it is clear
that $\mathbb{Z}_p$ carries the structure of a commutative ring and we topologise it using the norm
$|a|_p=p^{-v_p(a)}$, where $v_p$ is the {\em $p$-adic valuation} given by
$$v_p(a)=\left\{\begin{array}{ll} n & \text{if $a=0$ in  $\mathbb{Z}/p^n$ but $a\neq 0$ in $\mathbb{Z}/p^{n+1}$}\\
\infty & \text{if $a=0$.}\end{array}\right.$$

With the latter definition in mind, it is perhaps easiest to think of $\mathbb{Z}_p$ as the set of sequences of
integers $(a_0,a_1,\ldots)$ such that $a_{k+1} = a_k$ mod $p^k$ with componentwise multiplication and addition.
Another useful observation is that every $p$-adic integer $a$ can be given a unique expansion
$a=\sum_{i=0}^\infty a_i p^i$ where $a_i\in\{0,\ldots,p-1\}$ for each $i$.

We can equally well apply the norm $|-|_p$ to $\mathbb{Q}$, defining $v_p(p^n \frac{s}{t})=n$ whenever both $s$
and $t$ are coprime to $p$. The completion of $\mathbb{Q}$ with respect to $|-|_p$ then gives the field of
$p$-adic numbers, denoted $\mathbb{Q}_p$. There is a presentation
$$\mathbb{Q}_p = \left\{\frac{a}{p^n}\mid a\in\mathbb{Z}_p\text{ and }n\geq 0\right\},$$
and, similarly to above, every $p$-adic number $b$ has a unique expansion $b=\sum_{k}^\infty b_i p^i$ for some
$k\leq 0$, where $b_i\in\{0,\ldots,p-1\}$ for each $i\geq k$. Note that $\mathbb{Z}_p$ is a subring, and hence
an additive subgroup, of $\mathbb{Q}_p$.

A related construction is that of the {\em Pr\"ufer group}. There is a direct system of embeddings
$$\mathbb{Z}/p\to \mathbb{Z}/p^2\to \mathbb{Z}/p^3\to \ldots$$
with each map corresponding to multiplication by $p$ and we define
$$\mathbb{Z}/p^{\infty}=\colim \mathbb{Z}/p^n=\bigcup_n \mathbb{Z}/p^n.$$
There is a canonical identification $\mathbb{Z}/p^\infty\simeq\{z\in S^1\mid z^{p^n}=1 \text{ for some } n\}$,
although $\mathbb{Z}/p^\infty$ usually carries the discrete topology rather than that of the subspace of $S^1$.

\begin{lem} There are isomorphisms $\Hom(\mathbb{Z}/p^\infty,S^1)\simeq \mathbb{Z}_p$ and $\mathbb{Z}/p^\infty\simeq \mathbb{Q}_p/\mathbb{Z}_p$.\end{lem}
\begin{pf} View $\mathbb{Z}/p^\infty$ as a subgroup of $S^1$ and let $\phi:\mathbb{Z}/p^\infty\to S^1$ be any homomorphism. Then,
for each $n\geq 0$, $\phi(e^{2\pi i/p^n})=e^{2 k_n \pi i/p^n}$ for some $k_n\in \mathbb{Z}/p^n$. Since
$k_{n+1}=k_n$ mod $p^n$ for each $n$ we have a well defined element $(k_n)\in\mathbb{Z}_p$. Conversely, any
$a\in\mathbb{Z}_p$ gives a unique $\phi_a\in\Hom(\mathbb{Z}/p^\infty,S^1)$ defined by $\phi_a(e^{2\pi
i/p^n})=e^{2 a_n \pi i/p^n}$. It is clear that this construction gives an isomorphism of the groups.

For the second statement, the homomorphism $\mathbb{Z}/p^\infty\to \mathbb{Q}_p/\mathbb{Z}_p$ given by $e^{2k\pi
i/p^n}\mapsto k/p^n+\mathbb{Z}_p$ is well defined and injective. It is surjective since, for $a\in\mathbb{Z}_p$,
we have $a/p^n+\mathbb{Z}_p=\left(\sum_{i=0}^{n-1} a_i p^i\right)/p^n + \mathbb{Z}_p$ with the latter visibly
lying in the image.\end{pf}

\subsection{The Teichm\"uller lift map}\label{sec:Theichmuller lift map}

The ring of $p$-adic integers, $\mathbb{Z}_p$, contains precisely $p-1$ roots of $x^{p-1}-1=0$. Further, these
are all distinct (and necessarily non-zero) mod $p$. We define the {\em Teichm\"uller lift map} to be the
monomorphism of groups $\omega:(\mathbb{Z}/p)^\times\to \mathbb{Z}_p^\times$ sending $a$ to the unique
$(p-1)^\text{th}$ root of unity congruent to $a$ mod $p$. We often write $\hat{a}$ for
$\omega(a)\in\mathbb{Z}_p^\times$. We will need the following result.

\begin{lem}\label{Wilson's theorem} With the notation above we have $\ds \prod_{a\in(\mathbb{Z}/p)^\times} \hat{a}=-1\in \mathbb{Z}_p^\times.$\end{lem}

\begin{pf}This is an immediate corollary to the analogous result for $(\mathbb{Z}/p)^\times$; we briefly outline the details. Suppose $a\in\mathbb{Z}/p$ with $a^2=1$. Then $a^2-1=0$ so that $(a+1)(a-1)=0$. Hence, since $\mathbb{Z}/p$ is a
field, we have $a=\pm 1$, both of which are indeed square roots of $1$. Thus, the only self-inverse elements in
$\mathbb{Z}/p$ are $\pm 1$. Hence
$$\prod_{a\in(\mathbb{Z}/p)^\times} a = 1\times(-1)\times\prod_{\underset{a\neq\pm 1}{a\in(\mathbb{Z}/p)^\times}} a = -1$$
since the latter product comprises pairs of inverse elements. Applying the homomorphism $\omega$ then gives
the result.
\end{pf}

\subsection{Finite fields and their algebraic closures}\label{sec:finite fields}

For each prime $p$ we define $\mathbb{F}_p$ to be the field $\mathbb{Z}/p$. It is well known that we can choose
an algebraic closure $\overline{\mathbb{F}}_p$ for $\mathbb{F}_p$ and from here on we will assume that we have
done so for each prime. For a natural number $r$ we then define
$$\mathbb{F}_{p^r}=\left\{a\in \overline{\mathbb{F}}_p\mid a^{p^r}=1\right\}$$ which is a field containing $p^r$ elements. The two definitions coincide for the case $r=1$. It is a classical result that, for each $r$,
$\mathbb{F}_{p^r}$ is the unique field containing $p^r$ elements up to a non-canonical isomorphism (see, for
example, \cite{LangAlgebra}). Further, every finite field is isomorphic to $\mathbb{F}_{p^r}$ for some $p$ and
$r$. We refer to $p$ as the {\em characteristic} of the field. Note that the characteristic of a finite field
$K$ is given by $\text{char}(K)=\text{min}\{n\in \mathbb{N}\mid n.1=0\text{ in $K$}\}$.

If $K$ is a field of characteristic $p$ and $\overline{K}$ is an algebraic closure for $K$ then the map $\Frob:
\overline{K}\to \overline{K}$ sending $a\mapsto a^p$ satisfies $\Frob(1)=1$, $\Frob(ab)=\Frob(a)\Frob(b)$ and
$$\Frob(a+b)=(a+b)^p=a^p+b^p=\Frob(a)+\Frob(b)$$ since $p=0$ in $K$ and $p|{p \choose i}$ for $i\neq 0,p$. It
follows that $\Frob$ is a homomorphism of fields, and we refer to $\Frob$ as the {\em Frobenius homomorphism}.
For any prime $p$ the Galois group $\Gal(\overline{\mathbb{F}}_p/\mathbb{F}_p)$ is isomorphic to the profinite
integers (see \cite[p207]{Weibel}), topologically generated by $\Frob$. We will write
$\Gamma=\Gamma_p=\langle\Frob\rangle\simeq \mathbb{Z}$ for the subgroup of
$\Gal(\overline{\mathbb{F}}_p/\mathbb{F}_p)$ generated by $\Frob$.

Another important property of finite fields is that their multiplicative group of units is always cyclic (again,
see \cite{LangAlgebra}) and, for each $p$, there is a (non-canonical) embedding
$\overline{\mathbb{F}}_p^\times\to S^1$. If $l$ and $p$ are distinct primes then the embedding
$\overline{\mathbb{F}}_l^\times\to S^1$ induces a group isomorphism
$$\{a\in\Flbar^\times\mid a^{p^n}=1\text{ for some $n$}\}\simeq \mathbb{Z}/p^\infty$$
where $\mathbb{Z}/p^\infty$ is the Pr\"ufer group of Section \ref{sec:p-adics}. We will assume, from here on,
that we have chosen such embeddings for each prime. Note, however, that $\Flbar$ and $\Flbar^\times$ still both
carry the discrete topology.

\subsection{The symmetric and general linear groups}\label{sec:Symmetric and General Linear Groups}

For each $d\geq 1$ the \emph{symmetric group on $d$ symbols}, denoted $\Sigma_d$, is the group of permutations
of the finite set $\{1,\ldots,d\}$. For any $s$ and $t$ there is an obvious embedding
$\Sigma_s\times\Sigma_t\rightarrowtail \Sigma_{s+t}$. In particular, we can view $\Sigma_{d-1}$ as the subgroup
of $\Sigma_d$ fixing $d$. We will refer to the permutation $(1\ldots d)\in \Sigma_d$ as the {\em standard
$d$-cycle} and denote it by $\gamma_d$.

Let $K$ be a field. Then the {\em general linear group over $K$ of dimension $d$} is the group of invertible
$d\times d$ matrices with entries in $K$ and is denoted $GL_d(K)$. Equivalently, it is the group of
automorphisms of the $d$-dimensional vector space $K^d$. Similarly to above, there is an obvious embedding
$GL_s(K)\times GL_t(K)\rightarrowtail GL_{s+t}(K)$ for any $s$ and $t$. We can also view $\Sigma_d$ as a
subgroup of $GL_d(K)$ via the map $\sigma\mapsto \left(\sigma_{ij}\right)$ where
$$\sigma_{ij}=\left\{
\begin{array}{ll}
1 & \mbox{if $\sigma(j)=i$}\\
0 & \mbox{otherwise}\end{array}\right.
$$
and our convention is that $\sigma_{ij}$ denotes the entry in the $i^\text{th}$ row and $j^\text{th}$ column.
Another important subgroup of $GL_d(K)$ is the embedding of $(K^\times)^d$ along the diagonal, which we refer to
as the {\em maximal torus} and denote by $T_d$.

\subsection{Semidirect and wreath products}

Let $G,H$ be groups and let $G$ act on $H$ via group automorphisms. Then we define the \emph{semidirect product}
of $G$ and $H$, written $G\ltimes H$, to be the group with underlying set $G\times H$ but multiplication given
by
$$(g_1;h_1).(g_2;h_2)=(g_1g_2;(g_2^{-1}.h_1)h_2).$$

Note that there is an exact sequence $1\to H\rightarrowtail G\ltimes H\twoheadrightarrow G\to 1$. One of the
main sources of semidirect products is the following.

\begin{prop}\label{GH as semidirect product} Let $G$ and $H$ be subgroups $K$ with $G \cap H=1$ and $G\leqslant N_K(H)$. Then $G$
acts on $H$ by $g.h=ghg^{-1}$ and $GH$ is a subgroup of $K$ isomorphic to $G\ltimes H$.\end{prop}

\begin{pf} Since $G$ is contained in the normaliser of $H$ in $K$ the action of $G$ on $H$ is well defined and it is straightforward to check that $GH$ is a
subgroup of $K$. Now, define a map $\phi:G\ltimes H\to GH$ by $\alpha(g;h)=gh$. Then $\phi$ is clearly
surjective. To see that it is a group homomorphism, we have
\begin{eqnarray*}
\phi((g_1;h_1)(g_2;h_2)) &=& \phi(g_1g_2;(g_2^{-1}.h_1)h_2)\\
& = & g_1 g_2 (g_2^{-1}h_1g_2)h_2\\
& = & g_1 h_1 g_2 h_2\\
& = & \phi(g_1;h_1)\phi(g_2;h_2).
\end{eqnarray*}
Finally, for injectivity note that if $gh=1$ then $g=h^{-1}\in G\cap H=1$ whereby $g=h=1$.\end{pf}

Now, let $S\leqslant \Sigma_d$ for some $d$ and let $G$ be any group. Then $S$ acts on $G^d$ by
$\sigma.(g_1,\ldots,g_d)=(g_{\sigma^{-1}(1)},\ldots,g_{\sigma^{-1}(d)})$ and we define the \emph{wreath product}
of $S$ and $G$ by $S\wr G = S\ltimes G^d$. One important property of the wreath product is that it is
associative in the following sense.

\begin{lem}\label{A wr B wr G} Let $A\leqslant \Sigma_s$ and $B\leqslant \Sigma_t$ and $G$ be any group. Then there is a canonical
embedding $A\wr B\rightarrowtail \Sigma_{st}$ and a canonical isomorphism $(A\wr B)\wr G\simeq A\wr(B\wr
G)$.\end{lem}

\begin{pf} For $k=1,\ldots,s$ write $S_k=\{(k-1)t+1,\ldots,kt\}$. Then $\{1,\ldots,st\}=S_1\sqcup\ldots\sqcup
S_s$ and we get embeddings $B^s\rightarrowtail \Sigma_{st}$ (where the $k^\text{th}$ factor permutes $S_k$) and
$A\rightarrowtail\Sigma_{st}$ (where $A$ permutes $\{S_1\ldots,S_s\}$ in the obvious way). Then, since any
element of $A\cap B^s$ must map $S_k\iso S_k$ for each $k$, it is clear that $A\cap B^s=1$. Further, if
$\sigma\in A$ and $\tau\in B^s$ then, taking $i\in S_k$, we have $\sigma\tau\sigma^{-1}(i)\in
S_{\sigma\sigma^{-1}(k)}=S_k$, so that $\sigma\tau\sigma^{-1}\in B^s$. Hence $A\leqslant N_{\Sigma_{st}}(B^s)$
and an application of Proposition \ref{GH as semidirect product} gives us $AB^s=A\ltimes B^s=A\wr B$ as a
subgroup of $\Sigma_{st}$.

The proof that $(A\wr B)\wr G\simeq A\wr(B\wr G)$ follows on careful checking that the map
$$((a;b_1,\ldots,b_s);g_1,\ldots,g_{st})\mapsto
(a;(b_1;g_1,\ldots,g_t),\ldots,(b_s;g_{(s-1)t+1},\ldots,g_{st}))$$ is an isomorphism.\end{pf}

Another useful feature is that the wreath product distributes over the cross product.

\begin{lem}\label{Wreath products distribute over times} Let $C\leqslant \Sigma_s$ and $D\leqslant\Sigma_t$. Then, viewing $C\times D$ as a subgroup of
$\Sigma_{s+t}$ the map $(C\wr G)\times (D\wr G)\to (C\times D)\wr G$ given by
$$((\sigma;g_1,\ldots,g_s),(\tau;h_1,\ldots,h_t))\mapsto ((\sigma,\tau);g_1,\ldots,g_s,h_1,\ldots,h_t)$$ is an
isomorphism.\end{lem}
\begin{pf} The map is visibly a bijection. Checking that it is a group homomorphism is straightforward, if a little fiddly.\end{pf}

\subsection{Classifying spaces}

A {\em topological group} is a group equipped with a Hausdorff topology for which the multiplication and inverse
maps are continuous. Given a topological group $G$ (with a CW structure) there is a CW-complex known as the
\textit{classifying space of $G$}, denoted $BG$, which is formed as the geometric realisation of the nerve of
the category $\mathbb{G}$ in which there is just one object with morphisms indexed by elements of $G$.

The assignment $G\mapsto BG$ is functorial and, for a large class of groups (in particular, all countable
groups), we have a homeomorphism $B(G\times H)\simeq BG\times BH$.\footnote{The problem that can arise here is
that the topology on $B(G\times H)$ does not, in general, coincide with the product topology on $BG\times BH$.
Instead the right hand-side must be given the compactly generated topology (see \cite{Segal}). They do coincide,
however, if $BG$ and $BH$ have countably many cells (see \cite[Appendix]{Hatcher}).} If $G$ carries the discrete
topology, such as when $G$ is finite, then $BG$ is a $K(G,1)$ Eilenberg-MacLane space, that is $\pi_1(BG)=G$ and
$\pi_n(BG)=0$ for all $n\neq 1$. Of fundamental importance is the space $BS^1$ which turns out to be
$\mathbb{C}P^\infty$.

We have the following useful result.

\begin{prop}\label{conj_g:BG to BG is homotopic to identity} Let $G$ be a topological group. Then the map $\conj_g:G\to G, h\mapsto ghg^{-1}$ induces a map
$BG\to BG$ which is homotopic to the identity.\end{prop}

\begin{pf} This is covered in \cite[Section 3]{Segal}. It is a corollary of the fact that for any topological
categories $\mathcal{C}$ and $\mathcal{C}'$ and continuous functors $F_1,F_2:\mathcal{C}\to\mathcal{C}'$, if
there is a natural transformation $F:F_0\to F_1$ then $BF_0,BF_1:B\mathcal{C}\to B\mathcal{C}'$ are homotopic.
Putting $\mathcal{C},\mathcal{C}'=G$, $F_0=\conj_g$ and $F_1=\text{id}_G$ then for any $h\in G$ we have
$F_1(h)=ghg^{-1}$ and $F_0(h)=h$ and hence a commutative diagram
$$
\xymatrix{ F_1(*) \ar[r]^{g^{-1}} \ar[d]_{F_1(h)} & F_0(*) \ar[d]^{F_0(h)}\\
F_1(*) \ar[r]^{g^{-1}} & F_0(*).}
$$
Thus we have a natural transformation given by $F_0(*) \overset{g^{-1}}{\To} F_1(*)$ and the result
follows.\end{pf}

\subsection{The elementary symmetric functions}\label{sec:Elementary symmetric functions}

Let $R$ be a ring. Then $\Sigma_d$ acts on the power series ring $R\lpow x_1,\ldots,x_d\rpow$ by
$\tau.x_i=x_{\tau(i)}$ and the ring of invariants is given by $R\lpow x_i,\ldots,x_d\rpow ^{\Sigma_d}=R\lpow
\sigma_1\ldots,\sigma_d\rpow$, where $\sigma_k$ is known as the {\em $k^\text{th}$ elementary symmetric
function} and is defined by
$$\sigma_k=\sum_{1\leq i_1<\ldots<i_k\leq d}
x_{i_1}\ldots x_{i_k}$$ (so that
$\sigma_1=x_1+\ldots+x_d,~\sigma_2=x_1x_2+x_1x_3+\ldots+x_{d-1}x_d,\ldots,~\sigma_d=x_1\ldots x_d$).

Letting $N\in\mathbb{N}$, write $q:R\lpow x_1,\ldots,x_d\rpow \to R\lpow
x_1,\ldots,x_d\rpow/(x_1^N,\ldots,x_d^N)$ for the quotient map and identify $\sigma_i$ with $q(\sigma_i)$ for
each $i$. We have the following lemma.

\begin{lem}\label{linear independence of symmetric functions} Let $N\in\mathbb{N}$. Then the elements $\sigma_1^{\beta_1}\ldots \sigma_d^{\beta_d}\in R\lpow x_1,\ldots,x_d\rpow/(x_1^N,\ldots,x_d^N)$
for $\beta_1,\ldots,\beta_d\in\mathbb{N}$ with $0\leq \beta_1+\ldots +\beta_d < N$ are linearly
independent.\end{lem}
\begin{pf} Let $B$ be the set $\{(\beta_1,\ldots,\beta_d)\in\mathbb{N}^d\mid 0\leq \beta_1+\ldots +\beta_d < N\}$ and
suppose that $\sum_{{\bf \beta}\in B} r_\beta \sigma_1^{\beta_1}\ldots \sigma_d^{\beta_d}=0$ for some
$r_\beta\in R$. Then, for each $\beta$ and each $1\leq i\leq d$, the highest power of $x_i$ occurring in the
expression $\sigma_1^{\beta_1}\ldots \sigma_d^{\beta_d}$ is no more than $\beta_1+\ldots+\beta_d<N$. Hence the
relation lifts to $R\lpow x_1,\ldots,x_d\rpow$ whereby $r_\beta=0$ for all $\beta$.\end{pf}

\begin{prop}\label{Basis for Sigma_d-invariants} Let $N\in\mathbb{N}$. Then the free $R$-module $(R\lpow
x_1,\ldots,x_d\rpow/(x_1^N,\ldots,x_d^N))^{\Sigma_d}$ has basis $B=\{\sigma_1^{\beta_1}\ldots
\sigma_d^{\beta_d}\mid 0\leq \beta_1+\ldots +\beta_d < {N}\}$.\end{prop}
\begin{pf} Take a non-zero element $y\in (R\lpow x_1,\ldots,x_d\rpow/(x_1^N,\ldots,x_d^N))^{\Sigma_d}$. Let $A$ denote the set $\{\alpha\in
\mathbb{N}^d\mid 0\leq \alpha_i<N\}$. Then, for $\alpha\in A$, we write ${\bf x}^\alpha$ for
$x_1^{\alpha_1}\ldots x_d^{\alpha_d}$ and, using the standard basis for $R\lpow
x_1,\ldots,x_d\rpow/(x_1^N,\ldots,x_d^N)$, we write $y=\sum_{\alpha\in A} r_\alpha {\bf x}^\alpha$ for some
$r_\alpha\in R$.

Note that we can define an action of $\Sigma_d$ on $A$ by
$\tau.(\alpha_1,\ldots,\alpha_d)=(\alpha_{\tau^{-1}(1)},\ldots,\alpha_{\tau^{-1}(d)})$ and, with this action,
$\tau.{\bf x}^\alpha={\bf x}^{\tau.\alpha}$. Letting $\tau\in\Sigma_d$ then, since $\tau^{-1}.y=y$, we have
$$
\sum_{\alpha\in A} r_\alpha {\bf x}^\alpha  = \tau^{-1}.\sum_{\alpha\in A}
r_\alpha {\bf x}^\alpha = \sum_{\alpha\in A} r_\alpha {\bf x}^{\tau^{-1}.\alpha}=\sum_{\alpha\in A} r_{\tau.\alpha} {\bf x}^{\alpha}.\\
$$
Hence we see that $r_{\alpha}=r_{\tau.\alpha}$ for all $\alpha\in A$ and all $\tau\in\Sigma_d$. Next, introduce
an ordering $\prec$ on the monomials in $(R[[x_1,\ldots,x_d]]/(x_1^{N},\ldots,x_d^{N}))^{\Sigma_d}$ by
\begin{eqnarray*}
x_1^{\alpha_1}\ldots x_d^{\alpha_d} \succ x_1^{\beta_1}\ldots
x_d^{\beta_d} & \Longleftrightarrow & \alpha_1 > \beta_1\\
& & \text{or } \alpha_1=\beta_1 \text{ and } \alpha_2>\beta_2\\
& & \text{or } \alpha_1=\beta_1, \alpha_2=\beta_2 \text{ and } \alpha_3>\beta_3, \text{ etc.}
\end{eqnarray*}
This is a total ordering on the monomials in $(R[[x_1,\ldots,x_d]]/(x_1^{N},\ldots,x_d^{N}))^{\Sigma_d}$ (the
{\em lexicographical ordering}). Now, let $B=\{\sigma_1^{\beta_1}\ldots \sigma_d^{\beta_d}\mid 0\leq
\beta_1+\ldots +\beta_d < N \}$. Let $r_m m=r_m \mathbf{x}^\alpha$ be the largest monomial appearing as a
summand of $y$. Note that $m$ is of the form $x_1^{\alpha_1}\ldots x_d^{\alpha_d}$ with $\alpha_1\geq \ldots
\geq \alpha_d$ since otherwise we could find some $\tau\in\Sigma_d$ such that $\tau.m\succ m$ and $\tau.m$
necessarily appears as a summand of $y$. Now,
\begin{eqnarray*}
\sigma_1^{\alpha_1-\alpha_2}\sigma_2^{\alpha_2-\alpha_3}\ldots \sigma_d^{\alpha_d} & = &
x_1^{\alpha_1-\alpha_2}(x_1x_2)^{\alpha_2-\alpha_3}\ldots(x_1\ldots
x_d)^{\alpha_d} + \text{ lower terms}\\
& = & x_1^{\alpha_1}x_2^{\alpha_2}\ldots x_d^{\alpha_d} + \text{
lower terms}\\
& = & m + \text{ lower terms.}
\end{eqnarray*}
Hence $y-r_m \sigma_1^{\alpha_1-\alpha_2}\sigma_2^{\alpha_2-\alpha_3}\ldots \sigma_d^{\alpha_d}$ consists of
monomials strictly smaller than $m$. Since $y$ has a finite number of monomial summands we can continue in this
way to get $y$ expressed as a linear sum of elements of $B$ in a finite number of steps. Thus $B$ spans
$(R[[x_1,\ldots,x_d]]/(x_1^{N},\ldots,x_d^{N}))^{\Sigma_d}$ and hence, using Lemma \ref{linear independence of
symmetric functions}, is a basis.\end{pf}

\subsection{Nakayama's lemma and related results}

In this section we include a few useful results from commutative algebra. We begin with a version of Nakayama's
lemma.

\begin{prop} [Nakayama's lemma] Let $R$ be a local ring and $M$ a finitely generated $R$-module.
If $I$ is a proper ideal of $R$ and $M=IM$ then $M=0$.\end{prop}
\begin{pf} This is covered in \cite{Matsumura}.\end{pf}

We will usually apply a corollary of this result, but first need the following lemma.

\begin{lem}\label{R/I tensor M = M/IM} Let $R$ be a ring, $I$ an ideal in $R$ and $M$ an $R$-module. Then
$R/I\tensor_RM\simeq M/IM$.\end{lem}
\begin{pf} Define a map $f:M\to (R/I)
\tensor_R M$ by $f(m)=1\tensor m$. Then it is easy to show that $IM\subseteq \ker(f)$ so that $f$ factors
through a map $\overline{f}:M/IM\to (R/I)\tensor_R M$. It is then not difficult to check that the map
$(R/I)\tensor_R M\to M/IM,~\overline{a}\tensor m\mapsto \overline{a}.\overline{m}$ is inverse to $\overline{f}$
which gives the result.\end{pf}

\begin{prop}\label{M/IM=N/IN implies M=N} Let $R$ be a local ring, $I$ a proper ideal in $R$ and $\alpha:M\to N$ be a map of finitely generated $R$-modules. If the induced map $M/IM\to N/IN$ is an
isomorphism then $\alpha$ is surjective. Hence, if $\alpha$ is just the inclusion of $M$ in $N$ then
$M=N$.\end{prop}
\begin{pf} The exact sequence $M\overset{\alpha}{\to} N\to N/\alpha(M)\to 0$ induces an exact sequence $$(R/I)\tensor M\to (R/I)\tensor
N\to (R/I)\tensor (N/\alpha(M))\to 0.$$ Using Lemma \ref{R/I tensor M = M/IM} and the fact that $M/IM\to N/IN$
is an isomorphism we see that $(N/\alpha(M))/I(N/\alpha(M))=(R/I)\tensor (N/\alpha(M))=0$. Hence,
$N/\alpha(M)=I(N/\alpha(M))$ and, by Nakayama's lemma, $N/\alpha(M)=0$ so that $N=\alpha(M)$.\end{pf}

Let $R$ be a ring and $M$ an $R$-module. Then an element $x\in R$ is {\em regular on $M$} if $x.m=0$ implies
$m=0$ ($m\in M$). The ordered sequence $x_1,\ldots,x_n$ of elements of $R$ is a {\em regular sequence on $M$} if
$x_1$ is regular on $M$, $x_2$ is regular on $M/x_1M$,$\ldots$, $x_n$ is regular on $M/(x_1,\ldots,x_{n-1})M$
and $M/(x_1,\ldots,x_n)M\neq 0$.

\begin{lem} Let $(R,\mathfrak{m})$ be a local Noetherian ring and $\alpha:M\to N$ a map of finitely generated $R$-modules. Suppose that
$\mathfrak{m}=(x_1,\ldots,x_n)$ and $x_1,\ldots,x_n$ is a regular sequence on both $M$ and $N$. If the induced
map $M/\mathfrak{m}M\to N/\mathfrak{m}N$ is an isomorphism then so is $\alpha$.\end{lem}
\begin{pf} By Lemma \ref{M/IM=N/IN implies M=N} we know that $\alpha$ is surjective, so it remains to show
injectivity. Let $K=\ker(\alpha)$. Then, since $x_1$ is regular on $N$, $M$ and $K$, we have a diagram of exact
sequences
$$
\xymatrix{ x_1K \ar@{ >->}[r] \ar@{ >->}[d] & x_1M \ar@{ >->}[d] \ar@{->>}[r]^{\alpha} & x_1N \ar@{ >->}[d] \\
 K \ar@{ >->}[r] \ar@{->>}[d] & M \ar@{->>}[d] \ar@{->>}[r]^\alpha & N \ar@{->>}[d] \\
  K/x_1K \ar[r] & M/x_1M \ar@{->>}[r] & N/x_1N}
$$
and a diagram chase shows that the map $K/x_1K\to M/x_1M$ is injective. Hence we can repeat the process to end
up with an exact sequence
$$K/(x_1,\ldots,x_n)K \rightarrowtail M/(x_1,\ldots,x_n)M \twoheadrightarrow N/(x_1,\ldots,x_n)N.$$
But, by our hypothesis, $(x_1,\ldots,x_n)=\mathfrak{m}$ and $M/\mathfrak{m}M\to N/\mathfrak{m}N$ is an
isomorphism, so $K/\mathfrak{m}K=0$. Thus an application of Nakayama's lemma gives $K=0$ and $\alpha$ is
injective.\end{pf}

\begin{cor}\label{If m is regular on M then M is free over R} Let $(R,\mathfrak{m})$ be a local Noetherian ring and $M$ a finitely generated $R$-module. Suppose that
$\mathfrak{m}=(x_1,\ldots,x_n)$ and $x_1,\ldots,x_n$ is a regular sequence on $M$. Then $M$ is free over
$R$.\end{cor}
\begin{pf} Reduce modulo $\mathfrak{m}$ and choose a basis of the finite dimensional $R/\mathfrak{m}$-vector
space $M/\mathfrak{m}M$. Lift this basis to get a map $R^d\to M$ for some $d$ which gives a mod-$\mathfrak{m}$
isomorphism. Now applying the previous lemma we find that the map $R^d\to M$ is an isomorphism and $M$ is free
over $R$.\end{pf}

\subsection{Regular local rings and related algebra}

Given a ring $R$ we define the {\em Krull dimension} of $R$ to be the supremum over the lengths $r$ of all
strictly decreasing chains of prime ideals
$\mathfrak{p}_0\supset\mathfrak{p}_1\supset\ldots\supset\mathfrak{p}_r$. Of particular interest will be power
series rings $R\lpow x_1,\ldots,x_k\rpow$ which, if $R$ is Noetherian of Krull dimension $n$, have Krull
dimension $n+k$. Note that if $(R,\mathfrak{m})$ is a local ring then the Krull dimension of $R$ is zero if and
only if $\mathfrak{m}$ is the only prime ideal of $R$.

\begin{lem}\label{dimension zero local K-algebras are finite dimensional} Let $(R,\mathfrak{m})$ be a local Noetherian $K$-algebra for some field $K$. Suppose that $R$ has Krull dimension
$0$ and that $R/\mathfrak{m}$ is finite dimensional over $K$. Then $R$ is finite dimensional over $K$.\end{lem}
\begin{pf} Since $R$ is Noetherian it follows that $\mathfrak{m}$ is a finitely generated ideal. Hence, each of
the vector spaces $\mathfrak{m}^i/\mathfrak{m}^{i+1}$ are finite dimensional over $R/\mathfrak{m}$ and hence
also over $K$. Further, since $R$ has Krull dimension $0$, $\mathfrak{m}$ is the unique prime ideal of $R$. Thus
$\nil(R)= \mathfrak{m}$ so that, in particular, all the generators of $\mathfrak{m}$ are nilpotent. It follows
that there is $N\in\mathbb{N}$ such that $\mathfrak{m}^{N+1}=0$, whereby
$\mathfrak{m}^N=\mathfrak{m}^N/\mathfrak{m}^{N+1}$ is also finite dimensional. Thus we find that $R\simeq
R/\mathfrak{m}\oplus
\mathfrak{m}/\mathfrak{m}^2\oplus\ldots\oplus\mathfrak{m}^{N-1}/\mathfrak{m}^N\oplus\mathfrak{m}^N$ is finite
dimensional over $K$.\end{pf}

Let $(R,\mathfrak{m})$ be a Noetherian local ring of Krull dimension $n$. Then the dimension of
$\mathfrak{m}/\mathfrak{m}^2$ is called the {\em embedding dimension} of $R$. This is equal to the smallest
number of elements needed to generate $\mathfrak{m}$ over $R$ and hence $\text{embdim}(R)\geq n$. If
$\text{embdim}(R)=n$ then $R$ is called a {\em regular local ring} and a minimal generating set for
$\mathfrak{m}$ is called a {\em regular system of parameters}. Such a generating set is automatically a regular
sequence on $R$ (see \cite[Chapter 5]{Matsumura}).

\begin{lem}\label{M/mM finitely generated implies M finitely generated} Let $(R,\mathfrak{m})$ be a complete local Noetherian
ring. If $M$ is an $R$-module such that $M/\mathfrak{m}M$ is generated over $R/\mathfrak{m}$ by
$\mu_1,\ldots,\mu_r$ and $m_i\in M$ lifts $\mu_i$ then $M$ is generated over $R$ by $m_1,\ldots,m_r$. Hence, if
$M/\mathfrak{m}M$ is finitely generated over $R/\mathfrak{m}$ then $M$ is finitely generated over $R$.\end{lem}
\begin{pf} This is Theorem 8.4 in \cite{Matsumura}.\end{pf}

We include the following elementary lemmas for reference later on.

\begin{lem}\label{Right exactness for red mod ideals} Let $R$ be a ring, $I$ an ideal in $R$, $A$
an $R$-algebra and $J$ an ideal in $A$. Then the exact sequence
$$0\to J\rightarrowtail A\twoheadrightarrow A/J\to 0$$
induces a (right) exact sequence
$$(R/I)\tensor_R J\to (R/I)\tensor_R A\twoheadrightarrow (R/I)\tensor_R (A/J)\to 0.$$
In particular, the map $(R/I)\tensor_R J\to\ker\big((R/I)\tensor_{R} A\twoheadrightarrow (R/I)\tensor_R
(A/J)\big)$ is surjective.\end{lem}
\begin{pf} This is an immediate consequence of the right-exactness of the tensor product.\end{pf}

\begin{cor}\label{Mod I reduction of principal ideal surjects} Let $R$ be a ring, $I$ an ideal in $R$, $A$ an $R$-algebra and $a \in A$.
Then reduction modulo $I$ induces a surjective map $Aa/I(Aa)\twoheadrightarrow (A/IA)\overline{a}$, where
$\overline{a}$ is the image of $a$ in $A/IA$.\end{cor}
\begin{pf} This is a special case of Lemma \ref{Right exactness for red mod ideals}. Let $K=\ker\left(A/IA\twoheadrightarrow (A/Aa)/I(A/Aa)\right)$.
Then it is clear that $(A/IA)\overline{a}\subseteq K$ and we must show that $K=(A/IA)\overline{a}$. Take $x\in
K\leqslant A/IA$. Then $x$ lifts to $\tilde{x}\in A$ and, writing $q$ for the quotient map $A\to A/Aa$, we have
$q(\tilde{x})\in I(A/Aa)$, say $q(\tilde{x})=s.y$ for some $s\in I, y\in A/Aa$. Then $y$ lifts to $\tilde{y}\in
A$ and $\tilde{x}-s\tilde{y}\in\ker{q}=Aa$. Reducing mod $I$ we get $x\in (A/IA)\overline{a}$, as
required.\end{pf}

\begin{rem} It is important to note that, in general, $Aa/I(Aa)\twoheadrightarrow(A/IA)\overline{a}$ is not an isomorphism. As an
example, let $R=A=\mathbb{Z}$, $a=p$ and $I=p\mathbb{Z}$. Then $Aa=I$ so that $Aa/I(Aa)=I/I^2\simeq
\mathbb{Z}/p$ whereas $(A/IA)\overline{a}=(I/I^2)p=(\mathbb{Z}/p)p=0$. In particular note that the composite
$Aa/I(Aa)\twoheadrightarrow (A/IA)\overline{a}\rightarrowtail A/IA$ is not injective.\end{rem}

We use the following results from commutative algebra.

\begin{lem}\label{CRT} (Chinese Remainder Theorem) Let $R$ be a commutative ring and $I$ and $J$ ideals in $R$ with $1\in I+J$. Then the map
$R/IJ\to R/I\times R/J,~a\mapsto (a+I,a+J)$ is an isomorphism.\end{lem}
\begin{pf} Standard algebra (see \cite{LangAlgebra}).\end{pf}

\begin{lem} Given a ring $R$ and an ideal $I$ in $R$ we have $\mathbb{Q}\tensor (R/I)\simeq (\mathbb{Q}\tensor
R)/(\mathbb{Q}\tensor I)$.\end{lem}
\begin{pf} The quotient map $R\twoheadrightarrow R/I$ induces a surjection $f:\mathbb{Q}\tensor R\to
\mathbb{Q}\tensor (R/I)$. Given $r\in \ker(f)$ there is an $N\in \mathbb{N}$ such that $Nr\in R$ and then
$f(Nr)=Nf(r)=0$ so that $Nr\in I$. Thus $r=\frac{1}{N}\tensor Nr\in \mathbb{Q}\tensor I$ and $\ker(f)\subseteq
\mathbb{Q}\tensor I$. The reverse inclusion is clear. Thus we have an isomorphism $(\mathbb{Q}\tensor
R)/(\mathbb{Q}\tensor I)\iso \mathbb{Q}\tensor (R/I)$, as required.\end{pf}

We will use the following corollary.

\begin{cor}\label{Q tensor CRT} Let $R$ be a ring and $I$ and $J$ ideals in $R$ with $k\in I+J$ for some $0\neq k\in\mathbb{Z}$. Then the induced map
$\mathbb{Q}\tensor (R/IJ)\to \mathbb{Q}\tensor (R/I)\times \mathbb{Q}\tensor (R/J)$ is an isomorphism.\end{cor}
\begin{pf} Note that $\mathbb{Q}\tensor(IJ)=(\mathbb{Q}\tensor I)(\mathbb{Q}\tensor J)$ and $\mathbb{Q}\tensor (I+J)=\mathbb{Q}\tensor I + \mathbb{Q}\tensor J$. Now, $1=\frac{1}{k}\tensor k\in \mathbb{Q}\tensor (I+J)=\mathbb{Q}\tensor I
+\mathbb{Q}\tensor J$ so that the Chinese Remainder Theorem applies and we get
$$
\xymatrix{ (\mathbb{Q}\tensor R)/(\mathbb{Q}\tensor (IJ)) \ar[r]^-\sim  \ar[d]_\wr & (\mathbb{Q}\tensor
R)/(\mathbb{Q}\tensor
I)\times (\mathbb{Q}\tensor R)/(\mathbb{Q}\tensor J) \ar[d]^\wr\\
\mathbb{Q}\tensor (R/IJ) \ar[r] & \mathbb{Q}\tensor (R/I)\times \mathbb{Q}\tensor (R/J)}$$ showing that the
bottom map is an isomorphism, as claimed.
\end{pf}

\subsection{Duality algebras}\label{sec:duality algebras}

Let $A$ be an algebra over a ring $R$. Then the group $\Hom_R(A,R)$ is an $A$-module via the action
$(a.\phi)(b)=\phi(ab)$ and we say that $A$ is a {\em duality algebra} if there is an $A$-module isomorphism
$\Theta:A\iso\Hom_R(A,R)$. Note that such an isomorphism will be determined by $\theta=\Theta(1)$. Thus $A$ is a
duality algebra if and only if there is a $R$-linear map $\theta:A\to R$ such that the map $A\to
\Hom_R(A,R),~a\mapsto a.\theta$ is an isomorphism of $R$-modules. Such a $\theta$ is known as a {\em Frobenius
form}.

\begin{lem}\label{ann I = ann_theta I} Let $A$ be a duality algebra over $R$ and $\theta$ a Frobenius form on
$A$. Let $I$ be any ideal in $A$. Then if $a\in A$ and $\theta(aI)=0$ we have $aI=0$. Hence $$\ann_A(I)=\{a\in
A\mid \theta(aI)=0\}.$$\end{lem}
\begin{pf} Let $s\in I$. Then $\theta(as)=0$. Now, for any $b\in A$ we have
$\theta(bas)=\theta(a(bs))\in\theta(aI)=0$. Thus $(as).\theta$ is the zero map and it follows that $as=0$. Hence
$a$ annihilates $I$, as claimed.\end{pf}

\begin{cor}\label{ann P is a summand} Let $A$ be a duality algebra over $R$ and suppose that $P$ is an $R$-module summand in $A$. Then
$\ann_A(P)$ is a summand in $A$.\end{cor}
\begin{pf} Write $A=P\oplus A/P$ and so $\Hom_R(A,R)=\Hom_R(P,R)\oplus \Hom_R(A/P,R)$. Put
$Q=\Theta^{-1}(\Hom_R(A/P,R))$. Then $Q$ is a summand in $A$ and, with the notation of the opening paragraph,
\begin{eqnarray*}
x\in Q & \iff & \Theta(x) \in \Hom_R(A/P,R)\\
& \iff & (x.\theta)(P)=0\\
& \iff & \theta(xP)=0\\
& \iff & x\in\ann_A(P).\end{eqnarray*} Thus $Q=\ann_A(P)$ is a summand in $A$.\end{pf}

Now suppose that $K$ is a field and $A$ is a local and finite-dimensional duality algebra over $K$. Recall from
Section \ref{sec:local rings} that the socle of $A$ is defined to be the annihilator of its maximal ideal. We
will shortly have a useful characterisation of such duality algebras in terms of the dimension of their socles,
but first we need a couple of lemmas.

\begin{lem}\label{(A,I) local iff I^N=0}
Let $K$ be field, $A$ a finite-dimensional $K$-algebra and $I$ an ideal in $A$. Then $A$ is a local ring with
maximal ideal $I$ if and only if $A/I$ is a field and $I^N=0$ for some $N$.
\end{lem}
\begin{pf}First suppose that $A/I$ is a field and $I^N=0$ for some $N$. Then $I$ is a maximal ideal and $I\subseteq
\text{nil}(A)$. Further, since $\text{nil}(A)$ is equal to the intersection of all the prime ideals of $A$, we
have $I\subseteq \text{nil}(A)\subseteq\text{\text{rad}}(A)\subseteq I$. It follows that $I$ is the unique
maximal ideal of $A$ and $A$ is local.

For the converse, since $A$ is a finite-dimensional vector space it is Artinian and hence all prime ideals are
maximal (see, for example, \cite[pg 30]{Matsumura}). Thus if $A$ is local with maximal ideal $I$ we have
$I=\text{rad}(A)=\text{nil}(A)$ so that every element of $I$ is nilpotent. Since the descending chain of vector
spaces $I\geqslant I^2\geqslant I^3\geqslant \ldots$ must be eventually constant it follows that we must have
$I^N=0$ for some $N$.\end{pf}

\begin{lem}\label{soc in ideals}
Let $K$ be field and $(A,\mathfrak{m})$ be a finite-dimensional local $K$-algebra and suppose that $soc(A)$ is
one-dimensional over $A/\mathfrak{m}$. Then if $I$ is any non-trivial ideal of $A$ we have $soc(A)\subseteq I$.
\end{lem}
\begin{pf} By Lemma \ref{(A,I) local iff I^N=0} we must have $\mathfrak{m}^N=0$ for some $N$. Thus we have a descending chain
$$I\geqslant \mathfrak{m}I\geqslant\mathfrak{m}^2 I\geqslant\ldots
\geqslant \mathfrak{m}^N I=0$$ and hence there exists $t\geq 0$ such that $\mathfrak{m}^{t+1}I=0$ but
$\mathfrak{m}^t I\neq 0$. Since $\mathfrak{m}.(\mathfrak{m}^t I)=\mathfrak{m}^{t+1} I=0$ we see that
$\mathfrak{m}^t I$ is a non-zero $A/\mathfrak{m}$-subspace of the one-dimensional vector space $\soc(A)$. Thus
we have $\soc(A)=\mathfrak{m}^t I\leqslant I$, as required.\end{pf}

We are now able to prove some useful results.

\begin{prop}\label{gorenstein} Let $K$ be field and $(A,\mathfrak{m})$ a finite-dimensional local $K$-algebra such that the
composition $K\to A\to A/\mathfrak{m}$ is an isomorphism. Then $A$ is a duality algebra if and only if $\soc(A)$
has dimension one over $K$.
\end{prop}
\begin{pf} Let $\theta$ be a Frobenius form on $A$ and let $\Theta:A\iso \Hom_K(A,K)$ be the associated isomorphism of
$A$-modules. The inclusion $\mathfrak{m}\hookrightarrow A$ gives us a splitting $A\simeq(A/\mathfrak{m})\oplus
\mathfrak{m}\simeq K\oplus \mathfrak{m}$ of $K$-vector spaces. If $a\in \soc(A)$ with $\theta(a)=0$ then, since
$a$ annihilates $\mathfrak{m}$, we have $Aa=(K\oplus \mathfrak{m})a=Ka$ and so
$\theta(Aa)=\theta(Ka)=K\theta(a)=0$. It follows that $\Theta(a):A\to K$ is the zero map and hence, since
$\Theta$ is an isomorphism, $a=0$. Thus $\theta:\soc(A)\to K$ is an injective map of vector spaces meaning
$\text{dim}_K(\text{soc(A)})\leq 1$. Letting $t$ be the largest non-zero power of $\mathfrak{m}$ (necessarily
finite) we see that $0<\mathfrak{m}^t\leqslant\soc(A)$ so that $\soc(A)$ is non-zero and therefore that
$\dim_K(\soc(A))=1$, as required.

Conversely, suppose that $\dim_K(\soc(A))=1$. Let $0\neq v\in \soc(A)$ so that $\soc(A)=Kv$. Extend $\{v\}$ to a
basis $\{v,v_1,\ldots,v_{d-1}\}$ for $A$ over $K$ and let $\phi:A\to K$ be the linear map sending $v\mapsto 1$
and $v_i\mapsto 0$ for $i=1,\ldots,d-1$. We show that $\phi$ is a Frobenius form.

Let $\Phi:A\to \Hom_K(A,K)$ be the $K$-linear map defined by $\Phi(a)=a.\phi$. Hence, if $a\in A$ and
$\Phi(a)=0$ then $\phi(aA)=0$. But $aA$ is an ideal in $A$ and every non-trivial ideal contains $\soc(A)$ by
Lemma \ref{soc in ideals}. Thus, since $\phi(\soc(A))=K\neq 0$, we must have $aA=0$ and hence $a=0$, so $\Phi$
is injective. Since $\text{dim}_K(\Hom_K(A,K))=\text{dim}_K(A)$ (and both are finite) it follows that $\Phi$
must be an isomorphism of vector spaces. Thus $\phi$ is a Frobenius form on $A$.\end{pf}

\begin{prop}\label{A^G has duality if A does}
Let $A$ and $K$ be as above. Suppose $\text{char}(K)=p$ and let $G$ be a group of $K$-algebra automorphisms of
$A$ such that $p\nmid |G|$. Then if $A$ admits a $G$-invariant Frobenius form its restriction is a Frobenius
form on $A^G$ and thus $A^G$ has duality over $K$.
\end{prop}
\begin{pf} As in the proof of Lemma \ref{gorenstein} we have $A\simeq K\oplus \mathfrak{m}$ as $K$-modules. Since $G$ acts
by $K$-algebra automorphisms we have $A^G=(K\oplus \mathfrak{m})^G=K\oplus\mathfrak{m}^G$ so that
$A^G/\mathfrak{m}^G\simeq K$. Further, since $\mathfrak{m}^G\subseteq \mathfrak{m}$ and $\mathfrak{m}^N=0$ for
some $N$ we have $(\mathfrak{m}^G)^N=0$. Hence, by Lemma \ref{(A,I) local iff I^N=0}, $A^G$ is local and
$\mathfrak{m}^G$ its maximal ideal.

Let $\theta:A\to K$ be the $G$-invariant Frobenius form on A, so that $\theta(g.a)=\theta(a)$ for all $a\in A$
and $g\in G$, and let $\Theta:A\iso \Hom_K(A,K)$ be the associated isomorphism of vector spaces. Let $\Phi$ be
the composite $A^G\rightarrowtail A \overset{\Theta}{\longrightarrow}
\Hom_K(A,K)\overset{\text{res}}{\longrightarrow} \Hom_K(A^G,K)$. Then $\Phi$ is an $A^G$-linear map of finite
dimensional $K$-vector spaces of the same dimension.

Now, define an $A^G$-linear map $r:A\to A^G$ by $r(b)=\frac{1}{|G|}\sum_{g\in G}g.b$. Then, for $a\in A^G$ we
have $ar(b)=a\frac{1}{|G|}\sum_{g\in G}g.b=\frac{1}{|G|}\sum_{g\in G}g.(ab)=r(ab)$ so that
$$(\Phi(a)\circ r)(b)=((a.\theta)\circ r)(b)=\theta(ar(b))=\theta(r(ab)).$$
But $\theta\circ r=\theta$ since $\theta$ is $G$-invariant. Hence we have $\Phi(a)\circ r= \Theta(a)$ as maps
from $A\to K$. Thus, if $\Phi(a)=0$ we have $\Theta(a)=\Phi(a)\circ r=0$ so that $a=0$. Hence $\Phi$ is
injective and therefore an isomorphism. Since $\Phi$ is $A^G$-linear it follows that $A^G$ has duality over $K$
with Frobenius form $\Phi(1)=\theta|_{A^G}$.\end{pf}

\subsection{The $p$-divisibility of $k^s-1$.}

For this section we will assume that $p$ is an odd prime and let $k$ be any integer. For reasons that should
become clear later we aim to get a good understanding of the $p$-divisibility of $k^s-1$ for varying $s\in
\mathbb{N}$. That is, in the notation of Section \ref{sec:p-adics}, we are looking to calculate $v_p(k^s-1)$.
Note that if $k$ is divisible by $p$ then $v_p(k^s-1)=0$ for all $s$. Hence we can assume that $k$ is coprime to
$p$ and start with the case where $k=1$ mod $p$.

\begin{lem}\label{v_p(k^s-1) for (s,p)=1} Suppose $v_p(k-1)=v>0$ and take $s\geq 1$ with $(s,p)=1$. Then $v_p(k^s-1)=v$.\end{lem}

\begin{pf} Write $k=1+ap^v$ with $(a,p)=1$. If $s=1$ the result is clear. Otherwise, $s>1$ and for all
$1<i\leq s$ we have $(p^v)^i=p^{iv}=p^{v+1}.p^{(i-1)v-1}$ which is divisible by $p^{v+1}$ since $(i-1)v-1\geq
0$. Then
\begin{eqnarray*}
k^s & = & (1+ap^v)^s\\
&=& 1+s.ap^v + \sum_{i=2}^{s}{s \choose i}(ap^v)^i\\
&=& 1 + s.ap^v + p^{v+1}.b\end{eqnarray*} for some $b$. Hence $k^s-1=p^v(sa+pb)$ whereby $v_p(k^s-1)=v$ since
$p\nmid sa$.\end{pf}

\begin{lem} For $0<i<p$ we have $v_p \left({p \choose i}\right)=1$.\end{lem}

\begin{pf} We have ${p \choose i}=\frac{p!}{i!(p-i)!}$ so that
$i!(p-i)!{p \choose i}=p!$. Since $v_p(p!)=1$ and $v_p(i!)=v_p((p-i)!)=0$ we see that $v_p{p \choose i}=1$, as
required.\end{pf}

\begin{cor}\label{v_p(k^p-1)} Let $v_p(k-1)=v>0$. Then $v_p(k^p-1)=v+1$.\end{cor}
\begin{pf} Writing $k=1+ap^v$ with $(a,p)=1$ we have
\begin{eqnarray*}
k^p = (1+ap^v)^p & = & 1 + p.ap^v + \left(\sum_{i=2}^{p-1} {p \choose i}(ap^v)^i\right) + (ap^v)^p\\
& = & 1 + ap^{v+1} + \left(\sum_{i=2}^{p-1} {p \choose i}a^ip^{iv}\right) + a^pp^{vp}.\end{eqnarray*} But for
$2\leq i\leq p-1$ we have
\begin{eqnarray*}
v_p\left({p \choose i}a^ip^{iv}\right) & = & v_p{p \choose i} + v_p(ap^{iv})\\
& = & 1 + iv\\
& \geq & v+2\end{eqnarray*} since $v\geq 1$ and $i\geq 2$. Similarly, $v_p(a^pp^{vp})=vp\geq v+2$ since $v\geq
1$ and $p>2$. Hence we have
$$k^p-1=ap^{v+1} + p^{v+2}b=p^{v+1}(a+pb)$$
for some $b$ and the result follows.\end{pf}

Assembling the above results we get the following.

\begin{lem}\label{v_p(k^s-1) for v_p(k-1)>0} Let $v_p(k-1)=v>0$. Then $v_p(k^s-1)=v+v_p(s)$.\end{lem}

\begin{pf} Write $s=ap^w$ with $(a,p)=1$. Noting that $v_p(k^s-1)>0$ for any $s$ (since $k-1$ divides $k^s-1$) we have
\begin{eqnarray*}
v_p(k^s-1) & = & v_p(k^{ap^w}-1)\\
&=& v_p((k^{p^w})^a-1)\\
&=& v_p(k^{p^w}-1)\quad\text{by Lemma \ref{v_p(k^s-1) for (s,p)=1}}\\
&=& v_p(k-1)+w\end{eqnarray*} by repeated use of Corollary \ref{v_p(k^p-1)}. Thus $v_p(k^s-1)=v + v_p(s)$, as
claimed.\end{pf}

We can now deal with the general case.

\begin{prop}\label{v_p(k^s-1) for general k} Let $a$ be the order of $k$ in $(\mathbb{Z}/p)^\times$.
Then
$$v_p(k^s-1)=\left\{\begin{array}{ll} 0 & \text{if $a\nmid s$}\\
v_p(k^a-1) + v_p(s) & \text{otherwise.}\end{array}\right.$$\end{prop}

\begin{pf} We have
\begin{eqnarray*}
v_p(k^s-1)>0 & \iff & p\mid k^s-1\\
& \iff & k^s=1 \text{ mod $p$}\\
& \iff & a\mid s.\end{eqnarray*} Thus $v_p(k^s-1)=0$ when $a\nmid s$. If $a\mid s$, write $k'=k^a$. Then
$v_p(k'-1)>0$ and Lemma \ref{v_p(k^s-1) for v_p(k-1)>0} gives us
$$v_p(k^s-1)=v_p((k')^{(s/a)}-1)=v_p(k'-1)+v_p(s/a)=v_p(k^a-1)+v_p(s)$$
where we have used the fact that $v_p(s/a)=v_p(s)$ since $a|p-1$ and so is coprime to $p$.\end{pf}


\chapter{The $\mathbf{p}$-local structure of the symmetric and general linear groups}\label{ch:groups}

There is a close connection between the Sylow $p$-subgroups of $\Sigma_d$ and those of $GL_d(K)$, where $K$ is a
finite field of characteristic different to $p$. We begin with an analysis of the former.

\section{The Sylow $p$-subgroups of the symmetric groups}

The work here is well-known; similar expositions can be found in \cite{AdemMilgram} and \cite{Hall}.

\begin{lem}\label{v_p(n!)} Take $d\in\mathbb{N}$ and write $d=\sum_{i=0}^r a_ip^i$ with $0\leq a_i<p$. Then
$$v_p(d!)=\frac{d-\sum_ia_i}{p-1}$$
where $v_p$ is the $p$-adic valuation of Section \ref{sec:p-adics}.\end{lem}

\begin{pf} Noting that the integer part of $\frac{d}{p^j}$ is just $\sum_{i=j}^{r} a_ip^{i-j}$ and that, by the
usual arguments\footnote{There are $\left\lfloor\frac{d}{p}\right\rfloor$ terms in the sequence $1,2,3\ldots,d$
divisible by $p$, $\left\lfloor\frac{d}{p^2}\right\rfloor$ terms divisible by $p^2$ and so on.},
$v_p(d!)=\sum_{j=1}^{\infty} \left\lfloor\frac{d}{p^j}\right\rfloor$ we get
$$v_p(d!)=\sum_{j=1}^{\infty} \left\lfloor\frac{d}{p^j}\right\rfloor=\sum_{j=1}^\infty \sum_{i=j}^{r}
a_ip^{i-j}=\sum_{k=1}^r \left(a_k\sum_{l=0}^{k-1}p^l\right)=\sum_{k=1}^r a_k\left(\frac{p^k-1}{p-1}\right).$$
But
$$\sum_{k=1}^r a_k\left(\frac{p^k-1}{p-1}\right)=\frac{\sum_{k=1}^r a_k p^k-\sum_{k=1}^r
a_k}{p-1}=\frac{(d-a_0)-\sum_{k=1}^r a_k}{p-1}=\frac{d-\sum_i a_i}{p-1},$$ and we have the claimed
result.\end{pf}

\begin{cor}\label{v_p(p^k!)} For any $k>0$ we have $\ds v_p(p^{k}!)=(p^k-1)/(p-1)$.\end{cor}
\begin{pf} Writing $p^k=1.p^k$ the result is immediate from the preceding lemma.\end{pf}

\begin{prop}\label{Syl_p(Sigma_p^k)} Let $C_p=\langle\gamma_p\rangle\leqslant \Sigma_p$ be the cyclic
group of order $p$ generated by the standard $p$-cycle. Then, for any $k\geq 1$, the $k$-fold wreath product
$C_p\wr\ldots\wr C_p$ is a Sylow $p$-subgroup of $\Sigma_{p^k}$.\end{prop}
\begin{pf} We prove this by induction on $k$. The result is clear for $k=1$ since
$v_p(|\Sigma_p|)=v_p(p!)=1=v_p(C_p)$. Next suppose that the $k$-fold wreath product $P_k=C_p\wr\ldots\wr C_p$ is
a Sylow $p$-subgroup of $\Sigma_{p^k}$. Then an application of Lemma \ref{A wr B wr G} shows that
$P_{k+1}=C_p\wr P_{k}$ is a subgroup of $\Sigma_{p.p^k}=\Sigma_{p^{k+1}}$. Noting that $|P_{k+1}|=p|P_k|^p$,
using Corollary \ref{v_p(p^k!)} we have
\begin{eqnarray*}
v_p(|P_{k+1}|) = 1 + p.v_p(|P_k|) = 1 + p.v_p(p^k!) &=& 1 + p.\frac{p^k-1}{p-1}\\
&=& \frac{p-1 + p^{k+1} - p}{p-1}\\
&=& v_p(p^{k+1}!)\end{eqnarray*} so that $P_{k+1}$ is a Sylow $p$-subgroup of $\Sigma_{p^{k+1}}$.\end{pf}

\begin{prop}\label{Syl_p(Sigma_d)} Let $d\in \mathbb{N}$ and write $d=\sum_{i=0}^r a_ip^i$. By partitioning $\{1,\ldots,d\}$ appropriately, there is an embedding
$\prod_i (\Sigma_{p^i})^{a_i}\rightarrowtail\Sigma_d$ and any Sylow $p$-subgroup of
$\prod_i(\Sigma_{p^i})^{a_i}$ is a Sylow $p$-subgroup of $\Sigma_d$. In particular, $Syl_p(\Sigma_d)$ is a
product of iterated wreath products of $C_p$.\end{prop}

\begin{pf} We partition $d$ as
$$d=\underbrace{1+\ldots+1}_{a_0\text{ times}}+\underbrace{p+\ldots+p}_{a_1\text{ times}}+\ldots+\underbrace{p^r+\ldots+p^r}_{a_r\text{
times}}.$$ This induces the required embedding of $\prod_i (\Sigma_{p^i})^{a_i}$ in $\Sigma_d$. Now, using Lemma
\ref{v_p(n!)} and Corollary \ref{v_p(p^k!)}, we get
\begin{eqnarray*}
v_p\left(|\prod_i (\Sigma_{p^i})^{a_i}|\right) = v_p\left(\prod_i \left|\Sigma_{p^i}\right|^{a_i}\right) &=& \sum_i a_i v_p(p^i!)\\
&=& \sum_i a_i (p^i-1)/(p-1)\\
&=& \frac{\sum_i a_ip^i - \sum_i a_i}{p-1}\\
&=& \frac{d-\sum_i a_i}{p-1}\\
&=& v_p(d!)\\
&=& v_p(|\Sigma_d|).\end{eqnarray*} Thus $\Syl_p(\Sigma_d)\simeq \Syl_p(\prod_i (\Sigma_{p^i})^{a_i})\simeq
\prod_i \Syl_p(\Sigma_{p^i})^{a_i}$ which is of the form claimed using Proposition \ref{Syl_p(Sigma_p^k)}.
\end{pf}

\subsection{The normalizer of $\Syl_p(\Sigma_p)$}\label{sec:normalizer of C_p in Sigma_p}

Here we consider the Sylow $p$-subgroup $C_p=\langle \gamma_p\rangle$ of $\Sigma_p$.

\begin{lem}\label{Aut(C_p) in Sigma_p} There is an embedding $\phi:\Aut(C_p)\rightarrowtail \Sigma_p$ such that for all
$\alpha\in\Aut(C_p)$ we have $\alpha.\gamma_p=\phi(\alpha)\gamma_p\phi(\alpha)^{-1}$.\end{lem}

\begin{pf}  We begin by noting that $\Aut(C_p)\simeq(\mathbb{Z}/p)^\times$ where $s.\gamma_p=\gamma_p^s$. Let $1\neq s\in
(\mathbb{Z}/p)^\times$. Then $s+1,s^2+1,\ldots,s^{p-1}+1$ are all distinct modulo $p$ and (writing $p$ for $0\in
\mathbb{Z}/p$) we have a $(p-1)$-cycle $(s+1~~s^2+1~ \ldots~s^{p-1}+1)$. We define a map $\phi:\Aut(C_p)\to
\Sigma_p$ by $\phi(s)=(s+1~~s^2+1~\ldots~s^{p-1}+1)$. Then it is straightforward to check that this gives an
embedding of $\Aut(C_p)$ in $\Sigma_p$. Further, the action of $\Aut(C_p)$ on $C_p$ is now given by conjugation,
that is $\phi(s)\gamma_p\phi(s)^{-1}=\gamma_p^s=s.\gamma_p$.\end{pf}

\begin{cor} With the embedding of Lemma \ref{Aut(C_p) in Sigma_p} we have $N_{\Sigma_p}(C_p)=\Aut(C_p)\ltimes C_p$.\end{cor}

\begin{pf} By Lemma \ref{Aut(C_p) in Sigma_p} we can view $\Aut(C_p)$ as a subgroup of $\Sigma_p$ and,
further, $\Aut(C_p)\leqslant N_{\Sigma_p}(C_p)$. Since $\Aut(C_p)\cap C_p=1$ we have
$\Aut(C_p).C_p=\Aut(C_p)\ltimes C_p$ by Proposition \ref{GH as semidirect product}. Thus, as $C_p$ is normal in
$\Aut(C_p)\ltimes C_p$, it remains to show that $N_{\Sigma_p}(C_p)\subseteq\Aut(C_p)\ltimes C_p$.

Take $\sigma\in N_{\Sigma_p}(C_p)$. Then $\sigma\gamma_p\sigma^{-1}=\gamma_p^s$ for some $s$ and therefore there
is $\tau\in \Aut(C_p)$ with $\sigma\gamma_p\sigma^{-1}=\tau\gamma_p\tau^{-1}$ so that
$(\tau^{-1}\sigma)\gamma_p(\tau^{-1}\sigma)^{-1}=\gamma_p$. But, remembering that $\gamma_p=(1\ldots p)$, this
identity means that $(1\ldots p)=((\tau^{-1}\sigma)(1)\ldots (\tau^{-1}\sigma)(p))$ whereby we must have
$(\tau^{-1}\sigma)(a)=a+k$ for some $k$; that is $\tau^{-1}\sigma=\gamma_p^k$ so that
$\sigma=\tau.\gamma_p^k\in\Aut(C_p).C_p=\Aut(C_p)\ltimes C_p$.\end{pf}

\section{The Sylow $p$-subgroups of the finite general linear groups}\label{sec:Syl_p(GL_d(K))}

Let $K$ be a finite field. Let $T_d=(K^\times)^d$ be the maximal torus of $GL_d(K)$ and recall from Section
\ref{sec:Symmetric and General Linear Groups} the embedding of $\Sigma_d$ as a subgroup of $GL_d(K)$ given by
$\sigma\mapsto \left(\sigma_{ij}\right)$, where
$$\sigma_{ij}=\left\{
\begin{array}{ll}
1 & \mbox{if $\sigma(j)=i$}\\
0 & \mbox{otherwise.}\end{array}\right.
$$
We are interested in the structure of $N_d=\Sigma_dT_d$, that is the set $$\{g\in GL_d(K)\mid
g=\sigma(b_1,\ldots,b_d)\text{ for some $\sigma\in\Sigma_d$ and some $(b_1,\ldots,b_d)\in T_d$}\}.$$ We will
show that this is a subgroup of $GL_d(K)$ and that, whenever $v_p(|K^\times|)>0$, it contains one of $GL_d(K)$'s
Sylow $p$-subgroups.

\begin{lem}\label{N_d products}
Let $(b_1,\ldots,b_d)\in T_d$ and $\sigma\in\Sigma_d$. Then
$$\sigma(b_1,\ldots,b_d)\sigma^{-1}=(b_{\sigma^{-1}(1)},\ldots,b_{\sigma^{-1}(d)})$$ and hence $\Sigma_d\leqslant
N_{GL_d(K)}(T_d)$.\end{lem}

\begin{pf} It is a straight forward calculation to check that
\[(\sigma(b_1,\ldots,b_d))_{ij} = \left\{
  \begin{array}{ll}
    b_j & \mbox{if $\sigma(j)=i$}\\
    0 & \mbox{otherwise}
  \end{array}\right. = ((b_{\sigma^{-1}(1)},\ldots,b_{\sigma^{-1}(d)})\sigma)_{ij}.\qedhere\]
\end{pf}

\begin{cor}\label{N_d as a semi direct product} $N_d$ is a subgroup of $GL_d(K)$ isomorphic to $\Sigma_d\ltimes T_d$.\end{cor}

\begin{pf} Since $\Sigma_d\cap T_d=1$, this is a straight application of Proposition \ref{GH as semidirect product}.\end{pf}

\begin{cor}\label{N_d as wreath product} The map $\Sigma_d\wr K^\times\to N_d$
given by $(\sigma;b_1,\ldots,b_d)\mapsto \sigma(b_1,\ldots, b_d)$ is an isomorphism of groups.\end{cor}

\begin{pf} Follows immediately from \ref{N_d products} and \ref{N_d as a semi direct product}.
\end{pf}

We have the following alternative characterisation of $N_d$.

\begin{prop}\label{N(T_d) and W(T_d)} $N_d$ is the normalizer of $T_d$ in $GL_d(K)$ with associated Weyl group $\Sigma_d$.\end{prop}

\begin{pf} Let $g\in N_{GL_d(K)}(T_d)$ and choose $a\in K^\times$ with $a\neq 1$. Take any $1\leq s\leq d$ and
define $e_s=(a,\ldots,a,1,a,\ldots,a)$ with $1$ in the $s^\text{th}$ place. Write $ge_sg^{-1}=(b_1,\ldots,b_d)$.
Then $ge_s = (b_1,\ldots,b_d)g$. By consideration of the $(i,k)^\text{th}$ entry we get the equations $g_{is} =
b_i g_{is}$ and $g_{ik}a = b_i g_{ik}$ for $k\neq s$. Now, since $g$ is invertible we can find $i$ with
$g_{is}\neq 0$. Then $g_{is}=b_i g_{is}$ whereby $b_i=1$. Hence, for $k\neq s$, we have $g_{ik}=g_{ik}a$ so
that, since $a\neq 1$, we get $g_{ik}=0$.

Summarising, for any $s$ there is $i=i(s)$ such that $g_{is}\neq 0$ and $g_{ik}=0$ for all $k\neq s$. Since $g$
is invertible each $i(s)$ must be distinct, that is $i$ is a permutation of $\{1,\ldots,d\}$. It follows that
$g\in \Sigma_dT_d=N_d$. Thus $N_{GL_d(K)}(T_d)\leqslant N_d\leqslant N_{GL_d(K)}(T_d)$ so we have equality.

For the Weyl group, it is not hard to show that the centralizer of $T_d$ in $GL_d(K)$ is just $T_d$ by a similar
calculation to the above so that
\[W_{GL_d(K)}(T_d)=N_{GL_d(K)}(T_d)/C_{GL_d(K)}(T_d)=N_d/T_d=\Sigma_d.\qedhere\]\end{pf}

\begin{lem}\label{v_p(GL_d(Fl))} Let $l$ be a prime different to $p$ and let $q=l^r$ for some $r$. Let $a$ be the order of $q$ in $(\mathbb{Z}/p)^\times$ and put $m=\left\lfloor\frac{d}{a}\right\rfloor$. Then
$$v_p(|GL_d(\mathbb{F}_q)|)=m v_p(q^a-1) + v_p(m!).$$\end{lem}
\begin{pf} By the fact that $GL_d(\mathbb{F}_q)$ consists of all $d\times d$
matrices of maximal rank, there are $q^d-1$ choices for the first column, $q^d-q$ choices for the second and so
on. Hence we get \begin{eqnarray*} |GL_d(\mathbb{F}_q)| & = & (q^d-1)(q^d-q)\ldots(q^d-q^{d-1})\\ & = &
q^{1+\ldots+(d-1)}(q^d-1)(q^{d-1}-1)\ldots(q-1). \end{eqnarray*} Thus, since $q$ is coprime to $p$, using
Proposition \ref{v_p(k^s-1) for general k} we have
\begin{eqnarray*}
v_p(|GL_d(\mathbb{F}_q)|) & = & v_p((q^d-1)(q^{d-1}-1)\ldots(q-1))\\
& = & \sum_{s=1}^d v_p(q^s-1)\\
& = & \sum_{k=1}^m (v_p(q^a-1)+v_p(ka))\\
& = & mv_p(q^a-1)+v_p(m!),
\end{eqnarray*}
as claimed, where we have used the fact that $v_p(ka)=v_p(k)$ since $a$ is coprime to $p$.\end{pf}

\begin{prop}\label{Syl_p(G)} Let $K$ be a finite field such that $v_p(|K^\times|)>0$. Let
$P_0$ be a Sylow $p$-subgroup of $\Sigma_d$ and $P_1=\{a\in K^\times\mid a^{p^k}=1\text{ for some $k$}\}$ be the
$p$-part of $K^\times$. Let $P$ be the image of $P_0\wr P_1$ in $N_d$. Then $P$ is a Sylow $p$-subgroup of
$GL_d(K)$.\end{prop}
\begin{pf} Choose an isomorphism $K\simeq \mathbb{F}_{q}$, where $q=l^r$ for some prime $l$ necessarily
different from $p$. Applying Lemma \ref{v_p(GL_d(Fl))} then gives $v_p(|GL_d(K)|) = dv_p(q-1) + v_p(d!)$. On the
other hand, since $P_0$ and $P_1$ are Sylow $p$-subgroups of $\Sigma_d$ and $K^\times$ respectively, we have
$v_p(|P_0|)=v_p(d!)$ and $v_p(|P_1|)=v_p(q-1)$ so that
$$v_p(|P|)=v_p(|P_0\wr P_1|)=v_p(|P_0|)+v_p(|P_1^d|)=v_p(d!)+dv_p(q-1)$$
showing that $P$ is indeed a Sylow $p$-subgroup of $GL_d(K)$.\end{pf}

\begin{rem}\label{Syl_p(Gfl)} Suppose that $v_p(|K^\times|)>0$. Note that for $d<p$ we have $v_p(d!)=0$ so that $\Syl_p(\Sigma_d)=1$ and hence
$\Syl_p(GL_d(K))=\Syl_p(T_d)$, which is abelian. For $d\geq p$ the Sylow $p$-subgroup $P_0$ of $\Sigma_d$ is non
trivial and the corresponding Sylow $p$-subgroup of $GL_d(K)$ is no longer abelian.\end{rem}

The Sylow $p$-subgroups of $GL_d(K)$ for $v_p(|K^\times|)=0$ are harder to get a handle on, although we do have
the following result, valid when $d<p$.

\begin{prop}\label{Sylow p-subgroup for d<p and v_p(K^x)=0} Let $d<p$ and let $K=\mathbb{F}_q$, where $q=l^r$ for some prime $l$ different to $p$ and some $r$. Let $a$ be the order
of $q$ in $(\mathbb{Z}/p)^\times$ and put $m=\left\lfloor\frac{d}{a}\right\rfloor$. Choose a basis for
$\mathbb{F}_{q^a}$ over $\mathbb{F}_{q}$ to get an embedding $\mathbb{F}_{q^a}\rightarrowtail
GL_a(\mathbb{F}_{q})$. Then, using this embedding, we can view $(\mathbb{F}_{q^a}^\times)^m$ as a subgroup of
$GL_d(\mathbb{F}_q)$ and, writing $P_2=\Syl_p(\mathbb{F}_{q^a}^\times)$, we find that $P_2^m$ is a Sylow
$p$-subgroup of $GL_d(\mathbb{F}_q)= GL_d(K)$.\end{prop}
\begin{pf} We have $v_p(|GL_d(K)|)=mv_p(q^a-1) + v_p(m!) = v_p((q^a-1)^m) = v_p(|(\mathbb{F}_{q^a}^\times)^m|)$.\end{pf}

\section{The Abelian $p$-subgroups of $GL_p(K)$ for $v_p(|K^\times|)>0$.}\label{sec:abelian subgroups of
GL_p(K)}

Here we specialise to the general linear groups of dimension $p$. To begin with we look at the abelian
$p$-subgroups of $N_p=\Sigma_p.T_p\leqslant GL_p(K)$ which, by earlier work, contains a Sylow $p$-subgroup of
$GL_p(K)$. Let $\pi:N_p\twoheadrightarrow \Sigma_p$ denote the projection $\sigma(b_0,\ldots,b_{p-1})\mapsto
\sigma$. For the remainder of this chapter we will omit the subscripts and write $N$ for $N_p$ and $T$ for
$T_p$. We will also write $v=v_p(|K^\times|)$, which will be positive by assumption.

\begin{lem}\label{A in B^p or epic} Let $A$ be a $p$-subgroup of $N$. Then either
$A\leqslant T$ or $\pi(A)\leq \Sigma_p$ is cyclic of order $p$, generated by a $p$-cycle.\end{lem}
\begin{pf} By Proposition \ref{N_d as wreath product} there is an exact sequence of groups $T \hookrightarrow
N \overset{\pi}{\twoheadrightarrow} \Sigma_p$ and this shows that either $A\leqslant \ker(\pi)=T$ or $\pi(A)$ is
a non-trivial $p$-subgroup of $\Sigma_p$. Since $v_p(|\Sigma_p|)=v_p(p!)=1$ the latter case means $\pi(A)$ is
cyclic of order $p$. To see that $\pi(A)$ is generated by a $p$-cycle, take a generator of $\pi(A)$; the cycle
decomposition of this generator contains only $p$-cycles (by consideration of its order). As the cycles are
disjoint there can only be one.\end{pf}

\begin{lem}\label{conj_a permutes B^p} Let $a\in N$ with $\pi(a)\neq 1$ in $\Sigma_p$. Then if $b\in
T$ we have $\conj_a(b)=\conj_{\pi(a)}(b)$. Hence $\conj_a$ permutes the coordinates of $T$ by a non-trivial
cyclic permutation.\end{lem}
\begin{pf} For the first statement, using Lemma \ref{A in B^p or epic}, write $\pi(a)=\sigma$ for
some non-trivial $p$-cycle $\sigma$ so that $a=\sigma a'$ for some $a'\in \ker(\pi)=T$. Then for $b\in T$ we
have
\begin{eqnarray*}
\conj_{a}(b) & = & \sigma a' b(a')^{-1} \sigma^{-1}\\
& = & \sigma b\sigma^{-1}\\
& = & \conj_{\pi(a)}(b)
\end{eqnarray*}
where the second equality uses the fact that $T$ is abelian. By Lemma \ref{N_d products} we see that $\conj_a$
permutes the coordinates of $T$ by a non-trivial cyclic permutation, as required.\end{pf}

\begin{defn} Let $\Delta\simeq K^\times$ denote the diagonal subgroup of $T$, and let $\Delta_p\leqslant\Delta$
denote the $p$-elements of $\Delta$. Note that $\Delta_p$ is cyclic of order $p^v$.\end{defn}

\begin{lem}\label{A' in Delta} Let $A$ be an abelian $p$-subgroup of $N$ with $A\to \Sigma_p$ non-trivial. Then we have $A\cap T\leqslant\Delta_p$.\end{lem}
\begin{pf} Let $a\in A$ with $\pi(a)\neq 1$. Then, since $A$ is abelian, we get $\conj_a(a')=a'$
for all $a'\in A$. If $a'\in A\cap T$, then since $\pi(a)\neq 1$ we can use Lemma \ref{conj_a permutes B^p} to
see that all coordinates of $a'$ must be equal. That is, we must have $a'\in\Delta$. Since $A$ is a $p$-group we
get $A\cap T\leqslant \Delta_p$, as required.\end{pf}

\begin{cor}\label{a^p in Delta} Let $A$ be an abelian $p$-subgroup of $N$ with $A\to \Sigma_p$
non-trivial and let $a\in A$ map to a generator of $\pi(A)$. Then $a^p\in \Delta_p$.\end{cor}
\begin{pf} We have $\pi(a^p)=\pi(a)^p=1$ so that
$a^p\in\ker(\pi)=T$. Thus $a^p\in T\cap A\leqslant \Delta_p$ by Lemma \ref{A' in Delta}.\end{pf}

\begin{cor}\label{A leq <A>.Delta} Let $A$ be an abelian $p$-subgroup of $N$ with $A\to \Sigma_p$
non-trivial and let $a\in A$ map to a generator of $\pi(A)$. Then $A\leq \langle a\rangle.\Delta_p$.\end{cor}
\begin{pf} By Corollary \ref{a^p in Delta} we have $a^p\in \Delta_p$. Now, let $a'\in A$ and write $\pi(a')=\pi(a)^k$ for some $0\leq k<p$. Then
$a^{-k}a'\in\ker(\pi)\cap A\leqslant\Delta_p$ so $a'\in\langle a\rangle.\Delta_p$. Thus we have $A\leqslant
\langle a\rangle.\Delta_p$. Note that $\langle a\rangle.\Delta_p$ is a subgroup of $N$ since $\Delta_p$ is
contained in the centre of $N$.\end{pf}

We are now ready to give a coarse classification of the abelian $p$-subgroups of $N$.

\begin{prop}\label{Classified Subgroups} Let $A$ be an abelian $p$-subgroup of $N$. Then either
\begin{enumerate}
\item $A\leqslant T$, \item $\pi(A)$ is non-trivial and $A$ is cyclic of
order $p^{v+1}$, or \item $\pi(A)$ is non-trivial and $A$ is $N$-conjugate to a subgroup of $\langle
\gamma\rangle.\Delta$, where $\gamma$ denotes the standard $p$-cycle $(1 \ldots p)\in\Sigma_p$.\end{enumerate}
Further, all those of type 2 are $N$-conjugate.\end{prop}
\begin{pf} Suppose $A\nleqslant T$. Then we know from Lemma \ref{A in B^p or epic} that
$\pi(A)$ is cyclic of order $p$, generated by a $p$-cycle, $\sigma$ say. Take $a\in A$ mapping to $\sigma$. Then
$a^p\in \Delta_p$ by Corollary \ref{a^p in Delta}.

If $a^p$ is a generator of $\Delta_p$ then $a^{p^{v+1}}=1$ and, from Corollary \ref{A leq <A>.Delta},
$$A\leqslant\langle a \rangle.\Delta_p=\langle a \rangle.\langle a^p\rangle=\langle a\rangle\leqslant
A$$ so that $A=\langle a\rangle$ is cyclic of order $p^{v+1}$.

Otherwise, $a^p=\delta^p$ for some $\delta=(\delta,\ldots,\delta)\in\Delta_p$. Since $\sigma$ is a $p$-cycle, by
basic combinatorics there is a permutation $\tau\in\Sigma_p$ such that $\tau\sigma\tau^{-1}=\gamma$. Put
$A'=\tau A \tau^{-1}$ and $a'=\tau a \tau^{-1}$. Then we have $\pi(a')=\tau\sigma\tau^{-1}=\gamma$ and
$\pi(A')=\langle\gamma\rangle$. Further, $(a')^p=\tau a^p \tau^{-1}=\tau \delta^p\tau^{-1} = \delta^p$.

Write $a'=\gamma.(b_1,\ldots,b_p)$ for some $(b_1,\ldots,b_p)\in T$. Then an application of Lemma \ref{N_d
products} gives
$$(a')^p=\gamma.(b_1,\ldots,b_p)\ldots\gamma.(b_1,\ldots,b_p)=\gamma^p.(b_1\ldots
b_p,\ldots,b_1\ldots b_p)=(b_1\ldots b_p,\ldots,b_1\ldots b_p)$$ so that $b_1\ldots b_p=\delta^p$. Now, putting
$u=(1,b_1\delta^{-1},b_1 b_2\delta^{-2},\ldots,b_1 \ldots b_{p-1}\delta^{-(p-1)})$ we see that
\begin{eqnarray*}
u\gamma u^{-1} & = & (1,b_1\delta^{-1},\ldots,b_1 \ldots
b_{p-1}\delta^{-(p-1)})\gamma(1,b_1\delta^{-1},\ldots,b_1
\ldots b_{p-1}\delta^{-(p-1)})^{-1}\\
& = & \gamma.(b_1\delta^{-1},\ldots,b_1\ldots b_{p-1}\delta^{-(p-1)},1)(1,b_1^{-1}\delta,\ldots,b_1^{-1} \ldots
b_{p-1}^{-1}\delta^{p-1})\\
& = & \gamma.(b_1\delta^{-1},\ldots,b_p\delta^{-1})\\
& = & a'\delta^{-1}.
\end{eqnarray*}
It follows that $u^{-1}a' u=\gamma\delta$ so that
$$(u^{-1}\tau).A.(u^{-1}\tau)^{-1}=u^{-1}A' u\leq u^{-1}(\langle a'\rangle.\Delta) u=u^{-1}\langle a'\rangle
u.\Delta=\langle \gamma\delta\rangle.\Delta=\langle\gamma\rangle.\Delta$$ using Corollary \ref{A leq <A>.Delta}.
Hence $A$ is conjugate to a subgroup of $\langle\gamma\rangle.\Delta$.

For the final statement, take $A=\langle a\rangle$ of type 2, that is cyclic of order $p^{v+1}$. Then we know
that there is a generator $\delta\in\Delta_p$ with $a^p=\delta$.

Now, $\pi(a)$ is a $p$-cycle and so we can choose $\tau\in\Sigma_p$ with $\tau\pi(a)\tau^{-1}=\gamma=(1\ldots
p)$. Note that $(\tau a\tau^{-1})^p=\tau\delta\tau^{-1}=\delta$ since $\delta$ is in the centre of $N$. Hence
$A$ is conjugate to the group $\tau A\tau^{-1}$ which is cyclic of order $p^{v+1}$ generated by an element of
the form $\gamma(b_1,\ldots,b_p)$ with the property that $(b_1\ldots b_p,\ldots,b_1\ldots b_p)=\delta$.

By the above working, taking two subgroups of type 2, $A$ and $A'$ say, we can assume, without loss of
generality, that they are generated by elements $a=\gamma(b_1,\ldots,b_p)$ and $a'=\gamma(b_1',\ldots,b_p')$
with $b_1\ldots b_p=b_1'\ldots b_p'$. Now putting $u=(b_1 (b_1')^{-1},\ldots,b_1\ldots b_{p-1} (b_1' \ldots
b_{p-1})^{-1},1)$ it is a straight forward calculation to check that
$$u.a.u^{-1}=u.\gamma(b_1\,\ldots,b_p).u^{-1}=\gamma(b_1',\ldots,b_p')=a'$$
showing that $A$ and $A'$ are $N$-conjugate, as required.
\end{pf}

We are now able to give a stronger statement about the abelian $p$-subgroups of $GL_p(K)$. Let $a_0$ denote a
generator of the $p$-part of $K^\times\simeq C_{p^v}$. As usual, we let $\gamma=\gamma_p=(1\ldots p)\in
\Sigma_p$ denote the standard $p$-cycle. Put $a=\gamma(a_0,1,\ldots,1)\in \Sigma_p\wr K^\times\subseteq GL_p(K)$
and let $A=\langle a\rangle$. Note that $a^p=(a_0,\ldots,a_0)$ so that $a^{p^{v+1}}=1$ and $A$ is a cyclic
subgroup of $GL_p(K)$ of order $p^{v+1}$.

\begin{prop}\label{Abelian p-subgroups of GL_p(K)} Let $H$ be an abelian $p$-subgroup of $GL_p(K)$. Then $H$ is $GL_p(K)$-conjugate to either a
subgroup of $T$ or to $A$.\end{prop}

\begin{pf} By Sylow's theorems we know that $H$ is $GL_p(K)$-conjugate to a subgroup of $P\leqslant N$.
Thus, by Proposition \ref{Classified Subgroups}, it is conjugate to either a subgroup of $T$, $A$ or a subgroup
of $\langle \gamma\rangle.\Delta_p$. We will assume the latter case and show that $H$ is actually
$GL_p(K)$-conjugate to a subgroup of $T$.

We can assume, without loss of generality, that $H$ is itself a subgroup of $\langle\gamma\rangle.\Delta$. Let
$\gamma(b,\ldots,b)\in H$ where $b\in K^\times$. Then, for any $g\in GL_p(K)$, we have
$g.\gamma(b,\ldots,b).g^{-1}=g\gamma g^{-1}(b,\ldots,b)$ since $(b,\ldots,b)\in Z(GL_p(K))$. Thus it remains to
show that $\gamma$ is diagonalisable.

Let $u$ be a generator of $K^\times\simeq C_{p^v}$. For $k=0,\ldots,p-1$ put
$$v_k=\left(\begin{array}{c} 1\\ u^{kp^{v-1}}\\
\vdots \\ u^{k(p-1)p^{v-1}}\end{array}\right)\in (K^\times)^p.$$ Then
$$\gamma.v_k=\left(\begin{array}{c} u^{kp^{v-1}}\\
\vdots \\ u^{k(p-1)p^{v-1}}\\ 1\end{array}\right)=u^{kp^{v-1}}\left(\begin{array}{c} 1\\ u^{kp^{v-1}}\\
\vdots \\ u^{k(p-1)p^{v-1}}\end{array}\right)=u^{kp^{v-1}}v_k.$$ Thus $v_k$ is an eigenvector of $\gamma$ with
eigenvalue $u^{kp^{v-1}}$. Hence $\gamma$ has distinct eigenvalues $1,u^{p{v-1}},\ldots,u^{(p-1)p^{v-1}}$ and
so, putting $g=(v_0 | \ldots | v_{p-1})$, we find that $g$ is invertible and \[g^{-1}.\gamma(b,\ldots,b).g = g
\gamma g^{-1} (b,\ldots,b)= (1,u^{p^{v-1}},\ldots,u^{(p-1)p^{v-1}})(b,\ldots,b)\in T.\qedhere\]\end{pf}

Later we will need some understanding of the action of $N_{GL_p(K)}(A)$ on $A$. We can understand this as
follows.

\begin{lem}\label{Normalizer of A} Let $g\in N_{GL_p(K)}(A)$. Then, writing $a$ for the usual generator of $A$, we have $gag^{-1}=a^{1+kp^v}$ for some $k$.\end{lem}
\begin{pf} If $g\in N_{GL_p(K)}(A)$ then $gag^{-1}=a^s$
for some $s$ not divisible by $p$. But $a^p\in \Delta$ and hence $a^p=ga^pg^{-1}=(gag^{-1})^p=a^{sp}=(a^p)^s$.
It follows that $p=ps$ mod $p^{v+1}$ so that $p(s-1)=0$ mod $p^{v+1}$ and therefore $s=1$ mod $p^v$.\end{pf}


\chapter{Formal group laws and the Morava $\mathbf{E}$-theories}\label{ch:FGLs}

\section{Formal group laws}

We outline the basic theory of formal group laws, covering only the material needed for the development of this
thesis. For more comprehensive accounts of the area see \cite{Frohlich}, \cite{Hazewinkel} or \cite{RavenelCC}.
As before, all rings and algebras are commutative and unital.

\subsection{Basic definitions and results}

\begin{defn}\label{Axioms of a FGL} Let $R$ be a commutative ring. A \emph{formal group law} over $R$ is a power series $F(x,y)\in R\lpow x,y\rpow$ such that
\begin{enumerate}
\item $F(x,0)=x,$
\item $F(x,y)=F(y,x)$,
\item $F(F(x,y),z)=F(x,F(y,z))$ in $R\lpow x,y,z\rpow$.
\end{enumerate}
We sometimes refer to axioms 1-3 above as identity, commutativity and associativity for $F$ respectively.
\end{defn}

\begin{examples} The two easiest examples of a formal group law (defined over any ring) are the \emph{additive formal group law},
$F_a(x,y)=x+y$, and the \emph{multiplicative formal group law}, $F_m(x,y)=x+y+xy$. These examples are atypical,
however: most formal group laws are genuine power series as opposed to polynomials.\end{examples}

It is perhaps not too surprising that a lot can be said about the form of such power series. We start with the
following lemma. From here on $F$ will denote a formal group law over a commutative ring $R$.

\begin{lem}\label{F(x,y)=x+y mod xy} $F(x,y)=x+y$ mod $(xy)$.\end{lem}

\begin{pf} Write $F(x,y)=\sum_{i,j}
a_{ij}x^iy^j$ with $a_{ij}\in R$. Since $F(x,0)=x$ we get $a_{10}=1$ and $a_{i0}=0$ for $i\neq 1$. Similarly
$a_{01}=1$ and $a_{0j}=0$ for $j\neq 1$. But modulo $(xy)$ we have $\sum_{i,j} a_{ij}x^iy^j = a_{00} +
\sum_{i>0} a_{i0}x^i + \sum_{j>0} a_{0j}x^i = x+y$, as required.\end{pf}

\begin{lem}\label{Formal Inverses} (Formal inverse) For any formal group law $F$ there exists a unique power series $\iota(x)\in R\lpow x\rpow$ such that
$F(x,\iota(x))=0$.\end{lem}

\begin{pf} We define $\iota(x)$ inductively. Put $\iota_1(x)=-x$. Then we have $F(x,\iota_1(x))=x+(-x)=0$ mod $(x^2)$ by
the previous lemma. Suppose now that we have a power series $\iota_k(x)$ such that $F(x,\iota_k(x))=0$ mod
$(x^{k+1})$ and $\iota_k(x)=0$ mod $(x)$. Write $F(x,\iota_k(x))=ax^{k+1}$ mod $(x^{k+2})$ and put
$\iota_{k+1}(x)=\iota_{k}(x)-ax^{k+1}$. Then for any $j>0$, working modulo $(x^{k+2})$ we have
$\iota_{k+1}(x)^j=(\iota_k(x)-ax^{k+1})^j=\iota_k(x)^j$ (since $x|\iota_k(x)$) and similarly
$x^j\iota_{k+1}(x)=x^j(\iota_k(x)-ax^{k+1})=x^j\iota_{k}(x)$. It follows that
$F(x,\iota_{k+1}(x))=F(x,\iota_k(x))-ax^{k+1}=0$ mod $(x^{k+2})$.

Put $\iota(x)=\lim_{k\to\infty} \iota_k(x)$ which, since $\iota_k(x)=\iota_{k+1}(x)$ mod $(x^{k+1})$, is a well
defined power series with $F(x,\iota(x))=0$ mod $(x^k)$ for all $k$; that is $F(x,\iota(x))=0$. It is not hard
to see that any other $f(x)$ with this property must have $f(x)=\iota_k(x)=\iota(x)$ mod $(x^{k+1})$ for each
$k$ so that $f(x)=\iota(x)$, proving uniqueness.\end{pf}

\begin{cor} With the notation of Lemma \ref{Formal Inverses} we have $\iota(\iota(x))=x$.\end{cor}
\begin{pf} This is an immediate consequence of uniqueness and commutativity.\end{pf}

\begin{defn} We usually write $x+_F y$ for $F(x,y)$ and refer to this as the \emph{formal sum} of $x$ and $y$. The axioms
of Definition \ref{Axioms of a FGL} then translate as
\begin{enumerate}\item $x+_F 0= x,$
\item $x+_F y=y+_F x,$
\item $(x+_F y)+_F z = x+_F (y+_F z)$.
\end{enumerate}
Note that we may now use expressions of the form $x+_F y+_F z$ with no ambiguity.\end{defn}

\begin{lem}\label{iota(x+y)=iota(x)+iota(y)} $\iota(x+_F y)=\iota(x)+_F\iota(y)$.\end{lem}

\begin{pf} We have $ (x+_F y)+_F(\iota(x)+_F\iota(y))= x+_F y+_F\iota(y)+_F\iota(x) = 0.$
Thus, by uniqueness, $\iota(x+_F y)= \iota(x)+_F\iota(y)$, as required.\end{pf}

\begin{defn} We sometimes write $-_Fx$ for $\iota(x)$ and define $x-_F y$ to be $F(x,\iota(y))$. For any $m\in \mathbb{N}$ we define $[m]_F(x)=x+_F\ldots +_F x$ ($m$
times) and $[-m]_F(x)=[m](\iota(x))$. We call $[m]_F(x)$ the \emph{$m$-series on $x$}. When there is no
ambiguity we may simply write $[m](x)$ for $[m]_F(x)$.\end{defn}

\begin{lem} For any $m\in \mathbb{Z}$ we have $[m](x+_F y)=[m](x)+_F [m](y)$.\end{lem}
\begin{pf} This follows straight from symmetry, associativity and Lemma \ref{iota(x+y)=iota(x)+iota(y)} since we can reorder the terms in the formal sum however we like.\end{pf}

\begin{lem} For any $m,n\in\mathbb{Z}$ we have $[m+n](x)=[m](x)+_F[n](x)$.\end{lem}

\begin{pf} If one of $m$ or $n$ is zero the result follows immediately since $x+_F 0 = x$. The cases $m,n>0$ and $m,n<0$ are exercises in counting.
Hence we can assume, without loss of generality, that $m>0>n$. Note that $[-n](x)=[n](\iota(x))=\iota([n](x))$
by Lemma \ref{iota(x+y)=iota(x)+iota(y)} so that we can take $m,n>0$ and prove $[m-n](x)=[m](x)-_F[n](x)$.

Suppose first that $m-n\geq0$. Then $[m-n](x)+_F[n](x)=[m](x)$ and hence
$[m-n](x)=[m](x)+_F\iota([n](x))=[m](x)-_F[n](x)$, as required. If, on the other hand, $m-n<0$ we have
$[n-m](x)=[n](x)-_F[m](x)$ by the previous workings and then
\[[m-n](x)=[n-m](\iota(x))=\iota([n-m](x))=\iota([n](x)-_F[m](x))=[m](x)-_F[n](x).\qedhere\]\end{pf}

\begin{lem} $[m](x)=mx$ mod $(x^2)$\end{lem}
\begin{pf} This is a simple induction argument. We know $[1](x)=x$. If $[k](x)=kx$ mod $(x^2)$ for some $k$
then $[k+1](x)=[k](x)+_F x=[k](x) + x$ mod $(x[k](x))$. It follows that, modulo $(x^2)$, we have
$[k+1](x)=[k](x)+x=kx+x=(k+1)x$. Hence $[m](x)=mx$ mod $(x^2)$ for $m\geq 0$. For $m<0$ we have
$[m](x)=-_F[-m](x)$ and the result follows.\end{pf}

\begin{cor}\label{[m](x) twiddles x} $[m](x)$ is a unit multiple of $x$ in $R\lpow x\rpow$ if and only if $m\in R^\times$.\end{cor}
\begin{pf} By the previous result we have $[m](x)=x.f(x)$ for some power series $f(x)$ with constant term
$m$. Any such series is a unit in $R\lpow x\rpow$ if and only if the constant term is invertible.\end{pf}

\begin{defn} We write $\langle m\rangle(x)$ for the {\em divided $m$-series on $x$} which is defined to be
$[m](x)/x$. Note that, by the above result, this is a unit in $R\lpow x\rpow$ if and only if $m\in
R^\times$.\end{defn}

\begin{lem}\label{x-_F y twiddles x-y} $x-_F y$ is a unit multiple of $x-y$ in $R\lpow x,y\rpow$.\end{lem}

\begin{pf} Using Lemma \ref{F(x,y)=x+y mod xy} we have $x=x-_F y +_F y=x-_F y + y + (x-_F y)y f(x,y)$ for some $f(x,y)\in R\lpow
x,y\rpow$. Thus we get $(x-_F y)(1+yf(x,y))=x-y$. Since $1+yf(x,y)$ has invertible constant term it is a unit in
$R\lpow x,y\rpow$ and we are done.\end{pf}

\subsection{Formal logarithms and $p$-typical formal group laws}

We fix a prime $p$ and make the following definitions.

\begin{defn} Given formal groups $F$ and $G$ over a ring $R$, we define a homomorphism from $F$ to $G$ to be a power series $f(x)\in R\lpow x\rpow$ with
zero constant term such that $f(F(x,y))=G(f(x),f(y))$. Such a power series is an isomorphism if and only if
$f(x)$ is invertible under composition, that is if there is $g(x)$ with $f(g(x))=x=g(f(x))$. Note that this
occurs precisely when the coefficient of $x$ is invertible in $R$. We call an isomorphism of formal groups
\emph{strict} if $f(x)=x$ mod $x^2$.\end{defn}

This construction allows us, should we so desire, to form a category of formal groups laws $FGL(R)$ over any
ring $R$.

\begin{prop}\label{FGLs over Q-algebras are iso to additive} Let $F(x,y)$ be a formal group law over a $\mathbb{Q}$-algebra $R$. Then there is a unique power series $l_F(x)\in R\lpow x\rpow$
such that $l_F(0)=0$, ${l_F}'(0)=1$, and $l_F (F(x,y))=l_F(x)+l_F(y)$. That is, $F$ is canonically isomorphic to
the additive formal group law $F_a$.\end{prop}

\begin{pf} This is proved in \cite[Theorem A2.1.6]{RavenelCC}; if we write $F_2(x,y)=\partial F/\partial y$ the series
$l_F(x)$ which does the job is given by \[l_F(x)=\int_0^x \frac{dt}{F_2(t,0)}.\qedhere\]\end{pf}

\begin{defn} A strict isomorphism between a formal group law $F$ and the additive formal
group law $x+y$ is known as a \emph{formal logarithm} for $F$ and written $\log_F(x)$. By Proposition \ref{FGLs
over Q-algebras are iso to additive}, such a thing exists uniquely if $R$ is a $\mathbb{Q}$-algebra. Of course
they can also occur for other rings.\end{defn}

\begin{defn} Let $F(x,y)$ be a formal group law over a torsion-free $\mathbb{Z}_{(p)}$-algebra $R$. Then we call $F$
\emph{$p$-typical} if it has a formal logarithm of the form $\log_F(x)=x+\sum_{i>0} l_ix^{p^i}$ over
$\mathbb{Q}\tensor R$.\end{defn}

Our next result concerns one such $p$-typical formal group law.

\begin{prop}\label{Universal deformation FGL} Let $n$ be a positive integer. Then there is a $p$-typical formal group law $F$ over $\mathbb{Z}_p\lpow u_1,\ldots,u_{n-1}\rpow$ such that
$$[p](x)=\exp_F(px)+_F u_1x^p+_F\ldots+_F u_{n-1}x^{p^{n-1}}+_F x^{p^n},$$
where $\exp_F(x)$ is the inverse to $\log_F(x)$. In particular,
$$[p](x)=u_ix^{p^i}\text{ mod }(p,u_1,\ldots,u_{i-1},x^{p^{i+1}}).$$\end{prop}
\begin{pf} We give the logarithm for $F$ explicitly, following \cite[pp.204-205]{FSFG} which in turn follows
\cite[Section 4.3]{RavenelCC}. Let $\mathcal{I}$ be the set of non-empty sequences of the form
$I=(i_1,\ldots,i_m)$ with $0<i_k\leq n$ for each $k$. We define $|I|=m$ and $\|I\|=i_1+\ldots+i_m$. Letting
$j_k=\sum_{1\leq l<k} i_l$ we put $u_I=\prod_{k=1}^m u_{i_k}^{p^{j_k}}=u_{i_1}^{p^0}.u_{i_2}^{p^{i_1}}\ldots
u_{i_m}^{p^{i_1+\ldots+i_{m-1}}}$ (where we use the convention $u_n=1$). We then let
$$l(x)=x+\sum_{I\in\mathcal{I}} \frac{u_I}{p^{|I|}}x^{p^{\|I\|}}\in \mathbb{Q}_p\lpow u_1,\ldots,u_{n-1}\rpow\lpow x\rpow$$
and put $F(x,y)=l^{-1}(l(x)+l(y))$. It can be shown that $l(x)$ satisfies a functional equation of a suitable
form so that the functional equation lemma can be applied (see \cite{Hazewinkel}) and $F$ is in fact a formal
group law over $\mathbb{Z}_p\lpow u_1,\ldots,u_{n-1}\rpow$.

For the statement concerning the $p$-series, it suffices to show that
$$l([p](x))=l(\exp_F(px)+_F u_1x^p+_F\ldots+_F u_{n-1}x^{p^{n-1}}+_F x^{p^n})$$
as the result would follow on applying $l^{-1}$. Now, using the notation $u_n=1$ for simplicity, we have
\begin{eqnarray*}
l(\exp_F(px)+_F u_1x^p+_F\ldots+_F u_{n-1}x^{p^{n-1}}+_F x^{p^n}) & = & px + l(u_1x^p)+ \ldots+
l(x^{p^n})\\
&=& px + \sum_{j=1}^n l(u_j x^{p^j}).
\end{eqnarray*}
Given $I\in \mathcal{I}$, write $I(j)=(i_1,\ldots,i_m,j)$. Then $|I(j)|=|I|+1$, $\|I(j)\|=\|I\|+j$ and
$u_{I(j)}=u_I.u_j^{p^{\|I\|}}$. Thus
\begin{eqnarray*}
l(u_j x^{p^j}) & = & u_j x^{p^j} + \sum_{I\in\mathcal{I}} \frac{u_I}{p^{|I|}}(u_j x^{p^j})^{p^{\|I\|}}\\
& = & p.\frac{u_j}{p} x^{p^j} + p\sum_{I\in\mathcal{I}} \frac{u_I}{p^{|I|+1}}u_j^{p^{\|I\|}} x^{p^{\|I\|+j}}\\
& = & p\left(\frac{u_{(j)}}{p^{|(j)|}} x^{p^{\|(j)\|}} + \sum_{I\in\mathcal{I}}
\frac{u_{I(j)}}{p^{|I(j)|}}x^{p^{\|I(j)\|}}\right).\end{eqnarray*} Hence
\begin{eqnarray*}
px + \sum_{j=1}^n l(u_j x^{p^j}) & = & p\left(x + \sum_{j=1}^n\left(\frac{u_{(j)}}{p^{|(j)|}} x^{p^{\|(j)\|}} +
\sum_{I\in\mathcal{I}} \frac{u_{I(j)}}{p^{|I(j)|}}x^{p^{\|I(j)\|}}\right)\right)\\
&=& p\left(x+\sum_{I\in\mathcal{I}} \frac{u_I}{p^{|I|}}x^{p^{\|I\|}}\right)\\
&=& p(l(x))\\
&=& l([p](x))\end{eqnarray*} and we are done.\end{pf}

\begin{cor}\label{p-typical FGL over Z} Let $n$ be a positive integer and let $l(x)=\sum_{i\geq 0} x^{p^{ni}}/p^i\in \mathbb{Q}_p\lpow x\rpow$. Then there is
a formal group law $F_1(x,y)\in \mathbb{Z}_p\lpow x,y\rpow$ with $\log_{F_1}(x)=l(x)$ over $\mathbb{Q}_p$.
Further, $[p](x)=l^{-1}(px)+_{F_1} x^{p^n}$.\end{cor}

\begin{pf} These claims all follow easily by reducing the results of Proposition \ref{Universal deformation FGL}
modulo $(u_1,\ldots,u_{n-1})$. The proof of the final statement is, perhaps, worth including as a simplified
version of the proof of the corresponding result in \ref{Universal deformation FGL}. We have
\begin{eqnarray*}
l(l^{-1}(px)+_{F_1}x^{p^n}) & = & l(l^{-1}(px))+l(x^{p^n})\\
&=& px + \sum_{i\geq 0} (x^{p^n})^{p^{ni}}/p^i\\
&=& px + p\sum_{i\geq 0} x^{p^{(n+1)i}}/p^{i+1}\end{eqnarray*} which is easily seen to be $pl(x)=l([p](x))$.
Applying $l^{-1}$ we get the result.\end{pf}

\begin{rem} If we define $F_0(x,y)$ to be $F_1(x,y)$ reduced mod $p$ (a formal group law over $\mathbb{F}_p$) then it can be shown that
$F$ is the universal deformation of $F_0$ (see \cite{FSFG} and \cite{LubinTate}). We will later use $F$ to
define our cohomology theory.\end{rem}

\subsection{Formal group laws over fields of characteristic $p$}

\begin{lem}\label{height of FGL} Let $F$ be a formal group law over a field $K$ of characteristic $p$. Then either $[p](x)=0$ or there exists an integer $n> 0$ such that
$[p](x)=ax^{p^n}$ mod $(x^{p^n + 1})$ for some $a\in K^\times$.\end{lem}
\begin{pf} This is covered in \cite{RavenelCC}. In fact this is a special case of a more general result, namely
that any endomorphism $f$ of $F$ is either trivial or is such that $f(x)=g(x^{p^n})$ for some $g(x)\in R\lpow
x\rpow$ with $g'(0)\neq 0$ and some $n\geq 0$. Since $[p](x)$ is an endomorphism of $F$ and $[p](x)=px=0$ modulo
$x^2$ it follows that either $[p](x)=0$ or $[p](x)$ has leading term $ax^{p^n}$ for $a\in K^\times$ and some
$n>0$.\end{pf}

Given a formal group law $F$ over $K$ of characteristic $p$ we define the {\em height} of $F$ to be the integer
$n$ occurring in Lemma \ref{height of FGL} or $\infty$ if $[p](x)=0$. This is an isomorphism invariant (that is,
isomorphic formal group laws have the same height). Further, for any field $K$ of characteristic $p$ there
exists a formal group law of height $n$ for each $n>0$ (see \cite{RavenelCC}).

Given a complete local $\mathbb{Z}_p$-algebra $(R,\mathfrak{m})$ (that is, a complete local ring with
$p\in\mathfrak{m}$) and a formal group law $F$ over $R$ we can reduce the coefficients of $F$ modulo
$\mathfrak{m}$ to get a formal group law $F_0$ over $R/\mathfrak{m}$ which is a field of characteristic $p$. We
then define the height of $F$ to be the height of its mod-$\mathfrak{m}$ reduction $F_0$.

\subsection{Lazard's ring and the universal formal group law}

\begin{defn} Given a ring homomorphism $\phi:R\to S$ and a formal group law $F$ over $R$ we obtain a new formal
group law $\phi F$ over $S$ by applying $\phi$ to the coefficients of $F$. Note that if $\log_F$ exists then, by
uniqueness, $\log_{\phi F}$ exists and is equal to $\phi \log_F$.\end{defn}

The following is a result of Lazard.

\begin{prop} There is a ring $L$ and a formal group law $F_{\text{univ}}$ over $L$ such that for any ring $R$ and any formal
group law $F$ over $R$ there is a unique homomorphism $\phi:L\to R$ such that $\phi
F_{\text{univ}}=F$.\end{prop}
\begin{pf} Let $S$ be the polynomial ring over $\mathbb{Z}$ generated by symbols $\{a_{i,j}\mid i,j\geq 0\}$. Let $G$ be the power series $G(x,y)=\sum_{i,j} a_{i,j}x^iy^j\in S\lpow
x,y\rpow$. Then, letting $I$ be the ideal in $S$ generated by the relations that would force $G$ to be a formal
group law, on passing to the quotient ring $L=S/I$ we get a formal group law $F_{\text{univ}}$ over $S/I$. Given
any formal group law $F$ over $R$ there is a unique map $S\to R$ sending $a_{i,j}$ to the coefficient of
$x^iy^j$ in $F(x,y)$. It is clear this map factors through a map $\phi:S/I\to R$ which has the properties
claimed.\end{pf}

The ring $L$ is often referred to as {\em Lazard's ring} and $F_{\text{univ}}$ the {\em universal formal group
law}, for obvious reasons. We will see later that $L$ has a fundamental role in the development of a certain
class of cohomology theories, of which the Morava $E$-theories are examples.

\subsection{The Weierstrass preparation theorem}

Let $(R,\mathfrak{m})$ be a complete local ring. Then $f(x)=\sum_ia_ix^i\in R\lpow x\rpow$ is a {\em Weierstrass
series of degree $d$} if $a_0,\ldots,a_{d-1}\in\mathfrak{m}$ and $a_d\in R^\times$. We call $f$ a {\em
Weierstrass polynomial (of degree $d$)} if, in addition, $a_i=0$ for all $i>d$, that is if $f$ is in fact a
polynomial of degree $d$. We have the following theorem.

\begin{lem} (Weierstrass preparation theorem) Let $(R,\mathfrak{m})$ be a (graded) complete local ring. If $f(x)\in R\lpow x\rpow$ is a Weierstrass series of degree $d$ there is a unique factorisation $f(x)=u(x)g(x)$ where $g(x)$
is a Weierstrass polynomial of degree $d$ and $u(x)$ is a unit in $R\lpow x\rpow$.\end{lem}
\begin{pf} This is proved in \cite{Lang}.\end{pf}

\begin{cor} If $f(x)\in R\lpow x\rpow$ is a Weierstrass polynomial of degree $d$ then $R\lpow
x\rpow/(f(x))$ is a free $R$-module of rank $d$ with basis $\{1,x,\ldots,x^{d-1}\}$.\end{cor}
\begin{pf} By the Weierstrass preparation theorem $f(x)$ is a unit multiple of a monic polynomial of degree $d$
and the result follows.\end{pf}

We see the relevance of this diversion in the following result.

\begin{prop} Let $F$ be a formal group law of height $n$ over a (graded) complete local $\mathbb{Z}_p$-algebra
$(R,\mathfrak{m})$. Then, for each $r$, $[p^r](x)$ is a Weierstrass series of degree $p^{nr}$.\end{prop}
\begin{pf} Since $F$ has height $n$ we know that $[p](x)=ax^{p^n}$ modulo $\mathfrak{m},x^{p^n+1}$ for some
$a\in(R/\mathfrak{m})^\times$. Thus, writing $[p](x)=\sum_i a_ix^i$, we have
$a_0,\ldots,a_{p^n-1}\in\mathfrak{m}$ and, since $a_{p^n}$ is invertible modulo $\mathfrak{m}$, there is $b\in
R$ with $a_{p^n}b=1-m$ for some $m\in\mathfrak{m}$. Then $a_{p^n}b(1+m+m^2+\ldots)=(1-m)(1+m+m^2+\ldots)=1$ so
that $a_{p^n}\in R^\times$ and $[p](x)$ is a Weierstrass series of degree $p^n$. Using the fact that
$[p^{i+1}](x)=[p]([p^i](x))$ it then follows easily that $[p^r](x)$ is a Weierstrass series of degree
$p^{nr}$.\end{pf}

\begin{cor}\label{Basis for R[[x]]/[p](x)} Let $F$ and $R$ be as above. Then $R\lpow x\rpow/([p^r](x))$ is free over $R$ of rank $p^{nr}$ with
basis $\{1,x,\ldots,x^{p^{nr}-1}\}$.\end{cor}

\begin{defn} Let $F$ be the formal group law over $\mathbb{Z}_p\lpow u_1,\ldots,u_{n-1}\rpow$ of Proposition
\ref{Universal deformation FGL}. Then, using the Weierstrass preparation theorem, we define $g_r(t)$ to be the
Weierstrass polynomial of degree $p^{nr}$ which is a unit multiple of $[p^r]_F(t)$ in $\mathbb{Z}_p\lpow
u_1,\ldots,u_n\rpow[t]$.\end{defn}

\subsection{Formal group laws over complete local $\mathbb{Z}_p$-algebras}

\begin{lem}\label{[p^r](x) in m^r} Let $(R,\mathfrak{m})$ be a complete local $\mathbb{Z}_p$-algebra and let $F$ be a formal group law over $R$.
Then, writing $\mathfrak{m}_{R\lpow x\rpow}$ for the maximal ideal of $R\lpow x\rpow$, we have
$[p^r](x)\in(\mathfrak{m}_{R\lpow x\rpow})^r$.\end{lem}
\begin{pf} Working modulo $\mathfrak{m}$ we find that either $[p](x)=0$ or $[p](x)=ax^{p^n}$ mod $(x^{p^n+1})$ for some $n$; in either
case we conclude that $[p](x)\in \mathfrak{m}_{R\lpow x\rpow}$. Noting that the $[p](0)=0$ a simple induction
argument shows that $[p^{r+1}](x)=[p]([p^r](x))\in (\mathfrak{m}_{R\lpow x\rpow})^{r+1}$, as required.\end{pf}

This gives us the following useful result.

\begin{lem}\label{p-adic formal series} Let $R$ be a $\mathbb{Z}_p$-algebra and $F$ a formal group law over $R$. Then given any $a\in \Zp$ there is a well-defined power series $[a](x)\in R\lpow
x\rpow$ such that
\begin{enumerate}
\item if $a\in\mathbb{Z}$ then $[a](x)$ coincides with the
standard $a$-series on $x$,
\item $[a]([b](x))=[a.b](x)$, and
\item if $(a_i)$ is a sequence in $\Zp$ converging to $a$ then
$[a_i](x)$ converges to $[a](x)$ in $R\lpow x\rpow$.
\end{enumerate}
\end{lem}
\begin{pf} Take $a\in\mathbb{Z}_p$ and write $a=\sum_{i=0}^\infty a_ip^i$ and put $\alpha_k=\sum_{i=0}^{k} a_ip^i$.
Then, using Lemma \ref{x-_F y twiddles x-y}, we have $$\textstyle
\left[\alpha_{k+1}\right](x)-\left[\alpha_k\right](x)\sim
\left[\alpha_{k+1}\right](x)-_F\left[\alpha_k\right](x)=\left[a_{k+1}p^{k+1}\right](x)=[a_{k+1}]([p^{k+1}](x)),$$
where $a\sim b$ denotes that $a$ is a unit multiple of $b$. Hence
$\left[\alpha_{k+1}\right](x)-\left[\alpha_k\right](x)\in (\mathfrak{m}_{R\lpow x\rpow})^{k+1}$ using Lemma
\ref{[p^r](x) in m^r} and the limit $$\textstyle [a](x)=\left[\sum_{i=0}^{\infty}
a_ip^i\right](x)=\displaystyle{\lim_{k\to \infty}} \left[\alpha_k\right](x)$$ is well defined. It is
straightforward to check that this definition of $[a](x)$ satisfies the properties listed.\end{pf}

\begin{lem} Let $R$ be a torsion-free $\mathbb{Z}_p$-algebra and $F$ a formal group law over $R$. Then, for any
$a\in\mathbb{Z}_p$, we have $\log_F([a](x))=a\log_F(x)$.\end{lem}
\begin{pf} First note that, for any $x$ and $y$, $$\log_F(x)=(\log_F(x)-\log_F(y)) +
\log_F(y)=\log_F(\log_F^{-1}(\log_F(x)-\log_F(y))+_F y).$$ Hence $x=\log_F^{-1}(\log_F(x)-\log_F(y))+_F y$ so
that $\log_F(x-_F y)=\log_F(x)-\log_F(y)$. Then for $a\in\mathbb{Z}_p$, using the notation of the proof of Lemma
\ref{p-adic formal series}, we have
$$\textstyle\log_F([a](x))-\log_F([\alpha_k](x))=\log_F([a](x)-_F[\alpha_k](x))=\log_F([\sum_{i=k+1}^\infty a_ip^i](x))$$ which lies in $(\mathfrak{m}_{R\lpow
x\rpow})^{k+1}$, so that \[\log_F([a](x))=\lim_{k\to \infty} \log_F([\alpha_k](x))=\lim_{k\to \infty} \alpha_k
\log_F(x)=a\log_F(x).\qedhere\]\end{pf}

\begin{lem}\label{Teichmuller Series} Let $R$ be a torsion-free $\mathbb{Z}_p$-algebra and $F$ a $p$-typical
formal group law over $R$. Then, for all $k\in(\mathbb{Z}/p)^\times$, we have $[\hat{k}](x)=\hat{k}x$, where
$\hat{k}$ denotes the Teichm\"uller lift of $k$ of section \ref{sec:Theichmuller lift map}.\end{lem}
\begin{pf} Recall that, by definition, $F$ has a logarithm over $\mathbb{Q}\tensor R$ of the form $\log_F(x)=x+\sum_{i>0}
l_ix^{p^i}$. Let $k\in(\mathbb{Z}/p)^\times$. Then, since $\hat{k}^{p-1}=1$ we have $\hat{k}^{p^i}=\hat{k}$ for
all $i>0$. Hence
\begin{eqnarray*}
\log_F([\hat{k}](x)) &=& \hat{k}\log_F(x)\\
&=& \hat{k}x+\sum_{i>0}l_i \hat{k}x^{p^i}\\
&=& \hat{k}x + \sum_{i>0}l_i (\hat{k}x)^{p^i}\\
&=& \log_F(\hat{k}x)\end{eqnarray*} The result follows on applying $\log_F^{-1}$ to both sides.\end{pf}

\begin{cor}\label{<p>(kx)=<p>(x) for k a teichmuller lift} Let $R$ be a torsion-free $\mathbb{Z}_p$-algebra and $F$ a $p$-typical
formal group law over $R$. Then $\langle p\rangle([\hat{k}](x))=\langle p \rangle(x)$ for all
$k\in(\mathbb{Z}/p)^\times$.\end{cor}
\begin{pf} By Lemma \ref{Teichmuller Series} we have
$[p]([\hat{k}](x))=[p\hat{k}](x)=[\hat{k}]([p](x))=\hat{k}[p](x)=\hat{k}x\langle p\rangle(x)$. But
$[p]([\hat{k}](x))=[\hat{k}](x)\langle p\rangle([\hat{k}](x))=\hat{k}x\langle p\rangle([\hat{k}](x))$ and so,
since $\hat{k}$ is a unit and $x$ is not a zero divisor in $R\lpow x\rpow$, we get $\langle
p\rangle([\hat{k}](x))=\langle p \rangle(x)$.\end{pf}

\section{The Morava $E$-theories}

We outline the development of the cohomology theories that we will be using. There is some variation in the
literature, but this definition is consistent with relevant earlier work of Strickland and others. Full accounts
of this material are not easy to come by, but good starting points are \cite{RavenelNil} and \cite{RavenelCC}. A
thorough treatment of related theory is found in \cite{HoveyStrickland}.

\subsection{Complex oriented cohomology theories}

A  {\em (multiplicative) cohomology theory} is a contravariant functor from topological spaces to graded rings
satisfying the first three of the Eilenberg-Steenrod axioms. More precisely, we make the following definitions.

\begin{defn} We define the category of {\em CW pairs} to be the category with objects $(X,A)$, where
$A$ is a subcomplex of the CW complex $X$, and morphisms $(X,A)\to (Y,B)$ given by continuous cellular maps
$X\to Y$ which restrict to a map $A\to B$. We sometimes write $X$ for the object $(X,\emptyset)$.

A {\em generalised cohomology theory}, $h$, is a contravariant functor from CW pairs to $\mathbb{Z}$-graded
abelian groups satisfying the following conditions.
\begin{itemize}
\item If $f,g:(X,A)\to (Y,B)$ are homotopic then $f^*=g^*:h^*(Y,B)\to h^*(X,A)$.
\item Writing $i:A\hookrightarrow X$ and $j:X\hookrightarrow (X,A)$ there are {\em connecting homomorphisms}
$\partial^q:h^q(A)\to h^{q+1}(X,A)$ for each $q$ such that there is a natural long exact sequence
$$\ldots\overset{\partial^{q-1}}{\longrightarrow} h^q(X,A) \overset{h^q(j)}{\longrightarrow}h^q(X)
\overset{h^q(i)}{\longrightarrow}
h^q(A)\overset{\partial^q}{\longrightarrow}h^{q+1}(X,A)\overset{h^{q+1}(j)}{\longrightarrow}\ldots$$
\item If $U$ is an open subset of $X$ with $\overline{U}$ contained in the interior of $A$ and such that $(X\setminus U,A\setminus U)$ can be given a CW structure, then the map
$j:(X\setminus U,A\setminus U)\hookrightarrow (X,A)$ induces an isomorphism
$$h^*(j):h^*(X,A)\iso h^*(X\setminus U,A\setminus U).$$
\end{itemize}

An immediate consequence of the definition is that if $X$ is homotopy equivalent to $Y$ then $h^*(X)\simeq
h^*(Y)$. We define the {\em coefficients} of the cohomology theory to be the graded abelian group $h^*=h^*(pt)$,
where $pt$ is the one-point space. Note that for any space $X$ there is a unique map $X\to pt$ giving a map
$h^*(pt)\to h^*(X)$ which makes $h^*(X)$ a module over $h^*$.

We define a functor $\tilde{h}$ from topological spaces to graded abelian groups known as the {\em reduced
theory} by $\tilde{h}^*(X)=\coker(h^*(pt)\to h^*(X))$. Note that, by choosing a map $pt\to X$, we get a
splitting $h^*(X)\simeq \tilde{h}^*(X)\oplus h^*(pt)$.

Often a cohomology theory will have additional structure making the groups $h^*(X,A)$ into graded rings,
commutative in the graded sense so that if $a\in h^i(X)$ and $b\in h^j(X)$ then $ab=(-1)^{ij}ba$. In such a
situation we say that $h$ is {\em multiplicative}.
\end{defn}

We call a cohomology theory $h$ \textit{complex oriented} if there is a class $x\in h^2(\mathbb{C}P^\infty)$
such that its restriction to $h^2(\mathbb{C}P^1)$ generates $\tilde{h}^2(\mathbb{C}P^1)$ as an $h^0$-module (see
\cite{Adams}). The class $x$ is known as a {\em complex orientation} for $h$.

Any complex oriented cohomology theory $h$ with complex orientation $x$ satisfies
$h^*(\mathbb{C}P^\infty)=h^*\lpow x\rpow$ and $h^*(\mathbb{C}P^\infty\times\mathbb{C}P^\infty)=h^*\lpow
\pi_1^*(x),\pi_2^*(x)\rpow$ where $\pi_1$ and $\pi_2$ are the two projection maps
$\mathbb{C}P^\infty\times\mathbb{C}P^\infty\to \mathbb{C}P^\infty$ (again, see \cite{Adams}). Since
$\mathbb{C}P^\infty=BS^1$, the commutative multiplication map $S^1\times S^1\to S^1$ gives a product
$\mu:\mathbb{C}P^\infty\times\mathbb{C}P^\infty\to \mathbb{C}P^\infty$ making $\mathbb{C}P^\infty$ into an
$H$-space. The induced map $\mu^*:h^*(\mathbb{C}P^\infty)\to h^*(\mathbb{C}P^\infty\times\mathbb{C}P^\infty)$
sends the complex orientation $x$ to a power series $F(\pi_1^*(x),\pi_2^*(x))=F(x_1,x_2)$.

\begin{lem} The power series $F(x_1,x_2)=\mu^*(x)$ is a formal group law over $h^*$.\end{lem}
\begin{pf} We check that the axioms for a formal group law hold. Firstly, write $j:S^1\to S^1\times S^1$ for the map $z\mapsto (z,1)$. Then the commutative diagram
$$
\begin{array}{ccc}
\xymatrix{ S^1 \ar[d]_j \ar[dr]^{\id_{S^1}}\\
S^1\times S^1 \ar[r]_-\mu & S^1} & \begin{array}{c} ~\\~\\\text{induces}\end{array} & \xymatrix{ h^*(\mathbb{C}P^\infty)\\
h^*(\mathbb{C}P^\infty\times\mathbb{C}P^\infty) \ar[u]^{j^*} & h^*(\mathbb{C}P^\infty) \ar[l]^-{\mu^*}
\ar@{=}[ul]}\end{array}
$$
and one can check that $j^*(x_1)=x$ and $j^*(x_2)=0$ so that $F(x,0)=j^*(F(x_1,x_2))=j^*(\mu^*(x))=x$. It is
easy to see that, writing $\tau:S^1\times S^1\to S^1\times S^1$ for the twist map, we have $\mu\circ\tau=\mu$
and that this, on passing to cohomology, gives $F(x_1,x_2)=F(x_2,x_1)$. The final axiom is a consequence of the
associativity diagram
\[
\xymatrix{ S^1\times S^1\times S^1 \ar[r]^-{1\times\mu} \ar[d]_{\mu\times 1} & S^1\times S^1 \ar[d]^\mu\\
S^1\times S^1 \ar[r]_\mu & S^1.}\]
\end{pf}

\subsection{Defining the Morava $E$-theories}\label{Choice of coord on E^0(CP^infty)}

We aim to define our cohomology theory $E$ and, in doing so, fix a complex orientation with favourable
properties.

\begin{lem}\label{h is complex oriented if in even degrees} Let $h$ be a cohomology theory such that $h^*$ is
concentrated in even degrees. Then there exists $y\in h^2(\mathbb{C}P^\infty)$ such that, for each $n>0$,
$h^*(\mathbb{C}P^n)=h^*\lpow y\rpow/y^n$ and $h^*(\mathbb{C}P^\infty)=h^*\lpow y\rpow$. In particular $h$ is
complex oriented.\end{lem}
\begin{pf} This is an application of the Atiyah-Hirzebruch spectral sequence
$$H^*(\mathbb{C}P^n;h^*)\Rightarrow
h^*(\mathbb{C}P^n)$$ where $H$ denotes ordinary (singular) cohomology (see \cite{Adams} for further details). By
consideration of the cellular structure of $\mathbb{C}P^n$ we have
$$H^k(\mathbb{C}P^n;A)=\left\{\begin{array}{ll} A & \text{for $k=0,2,\ldots,2n$}\\ 0 & \text{otherwise}\end{array}\right.$$ which, in particular, lies in even degrees. Since $h^*$
is also concentrated in even degrees it follows that all terms with at least one degree odd in the $E_2$-page
are zero and the spectral sequence collapses. As usual, writing $J_k=\ker(h^*(\mathbb{C}P^n)\to
h^*(\skel^{2k}(\mathbb{C}P^n)))$ we have a canonical isomorphism $J_k/J_{k+1}\simeq H^{2k}(\mathbb{C}P^n;h^*)$
so that, in particular, $J_1/J_2\simeq h^*.x$, where $x$ is the chern class of the tautological line bundle over
$\mathbb{C}P^n$. Lifting $x$ under this map gives a homogeneous element $y_n\in J_1\subset h^*(\mathbb{C}P^n)$
of degree 2 such that $h^*(\mathbb{C}P^n)=h^*\lpow y_n\rpow/y_n^n$. By naturality, we can make sure the elements
$y_i\in h^2(\mathbb{C}P^i)$ are compatible for each $i$ and hence we get $y\in \limit h^2(\mathbb{C}P^n)$. Since
the maps $h^*(\mathbb{C}P^n)\to h^*(\mathbb{C}P^{n-1})$ are all surjective, an application of the
Milnor-sequence (again, see \cite{Adams}) gives $h^*(\mathbb{C}P^\infty)=\limit h^*(\mathbb{C}P^n)=h^*\lpow
y\rpow$ and $y$ is a complex orientation for $h$.\end{pf}

We turn our attention to complex cobordism and have the following well known result.

\begin{lem}\label{MU has canonical coordinate} The complex cobordism spectrum $MU$ is complex oriented and there is a canonical orientation
$x_{MU}\in MU^2(\mathbb{C}P^\infty)$.\end{lem} \begin{pf} For more details see, for example, \cite[Chapter
4]{RavenelCC}. It is known that $MU^*=\mathbb{Z}[a_1,a_2,\ldots]$ with $|a_i|=-2i$, so that $MU^*$ is
concentrated in even degrees and Lemma \ref{h is complex oriented if in even degrees} applies. A canonical
orientation is the class corresponding to the map
\[\mathbb{C}P^\infty=BU(1)\simeq MU(1)\longrightarrow \Sigma^2 MU.\qedhere\]\end{pf}

Now, using the coordinate $x_{MU}$ of Lemma \ref{MU has canonical coordinate} we get an identification
$MU^*(\mathbb{C}P^\infty)=MU^*\lpow x_{MU}\rpow$ and, as usual, we use the multiplication map
$\mu:\mathbb{C}P^\infty\times \mathbb{C}P^\infty\to \mathbb{C}P^\infty$ to get formal group law
$$F_{MU}(x_1,x_2)=\mu^*(x_{MU})\in MU^*(\mathbb{C}P^\infty\times \mathbb{C}P^\infty)=MU^*\lpow x_1,x_2\rpow.$$ This
is classified by a map $L\to MU^*$ where $L$ is the Lazard ring and we have the following famous
theorem of Quillen.

\begin{prop}[Quillen's theorem]\label{Quillen's theorem}The map $L\to MU^*$ classifying $F_{MU}$ is an
isomorphism.\end{prop}
\begin{pf} This is the main result of \cite{Quillen} and is covered in \cite{RavenelNil}.\end{pf}

For each prime $p$ and $k\geq 0$, let $v_{p,k}\in MU^*$ be the coefficient of $x^{p^k}$ in the $p$-series for
$F_{MU}$ and let $I_{p,k}=(v_{p,0},v_{p,1},\ldots,v_{p,k-1})\vartriangleleft MU^*$. Note that $v_{p,0}=p$ and
$I_{p,0}$ is defined to be $0$. We use the following result of Landweber.

\begin{prop}[Exact functor theorem]\label{Exact functor theorem} Let $M$ be an $MU^*$-module. Then the functor $X\mapsto M\tensor_{MU^*} MU_*(X)$ defines a homology
theory if and only if for each prime $p$ and each $k\geq 0$ multiplication by $v_{p,k}$ in $M/I_{p,k}M$ is
injective. In particular, there is a spectrum $E$ with $E_*(X)=M\tensor_{MU^*} MU_*(X)$.\end{prop}
\begin{pf} See \cite{Landweber} and \cite{HoveyStrickland}.\end{pf}

Now, fixing a prime $p$ and an integer $n>0$, let $R=\mathbb{Z}_p\lpow u_1,\ldots,u_{n-1}\rpow[u,u^{-1}]$ and
let $F$ be the $p$-typical formal group law over $\mathbb{Z}_p\lpow u_1,\ldots,u_{n-1}\rpow$ of Proposition
\ref{Universal deformation FGL}. Define a $\mathbb{Z}_p$-algebra map $\phi:\mathbb{Z}_p\lpow
u_1,\ldots,u_{n-1}\rpow\to R$ by $u_i\mapsto u^{p^i-1}.u_i$ and let $\phi F$ be the formal group law obtained by
applying $\phi$ to the coefficients of $F$. We give $R$ a grading by letting each $u_i$ lie in degree 0 and $u$
lie in degree $-2$. Then, using Quillen's theorem, $\phi F$ is classified by a map $MU^*\to R$ which respects
the grading. We show that, equipped with this map, the $MU^*$-module $R$ satisfies the the exact functor
theorem. Recall that we have
$$[p]_F(x)=\exp_F(px)+_F u_1 x^p+_F\ldots+_F u_{n-1}x^{p^{n-1}}+_F x^{p^n}\in \mathbb{Z}_p\lpow u_1,\ldots,u_{n-1}\rpow \lpow x\rpow$$ so that
$$[p]_{\phi F}(y)=\exp_{\phi F}(py)+_{\phi F} u^{p-1}u_1 y^p+_{\phi F}\ldots+_{\phi F} u^{p^{n-1}-1}u_{n-1}y^{p^{n-1}}+_{\phi F} y^{p^n}\in R\lpow y\rpow.$$ We use the following lemma.

\begin{lem}\label{R/I_p,k R} For any $1\leq k< n$ we have $R/I_{p,k}R=\mathbb{F}_p\lpow u_k,\ldots,u_{n-1}\rpow[u,u^{-1}]$ and $v_{p,k}$
acts as multiplication by $u^{p^k-1}u_k$. Further $R/I_{p,n}R=\mathbb{F}_p[u,u^{-1}]$ and $v_{p,n}$ acts as the
identity map.\end{lem}
\begin{pf} We proceed by induction on $k$. For $k=1$ we have $$R/I_{p,k}R=R/pR=\mathbb{F}_p\lpow
u_1,\ldots,u_{n-1}\rpow[u,u^{-1}]$$ and $[p]_{\phi F}(y)=u^{p-1}u_1 y^p+_{\phi F}\ldots+_{\phi F} y^{p^n}$. It
follows that the coefficient of $y^p$ in $[p]_{\phi F}(y)$ is $u^{p-1}u_1$, so that $v_{p,1}\mapsto u^{p-1}u_1$
in $R/I_{p,1}R$. The induction step is similar, noting that $v_{p,k}$ acts as a unit multiple of $u_k$ in
$R/I_{p,k}R$ we have $R/I_{p,k+1}R=(R/I_{p,k}R)/v_{p,k}(R/I_{p,k}R)=\mathbb{F}_p\lpow
u_{k+1},\ldots,u_{n-1}\rpow[u,u^{-1}]$.\end{pf}

Thus we get the following corollary.

\begin{prop} The $MU^*$-module $R$ satisfies the conditions of the exact functor theorem and hence there is a
spectrum $E$ with $E_*(X)=R\tensor_{MU^*} MU_*(X)$.\end{prop}
\begin{pf} The cases with $k=0$ are immediate since $R$ is torsion-free. The cases at the prime $p$ with $1\leq k\leq n$ are covered by Lemma \ref{R/I_p,k R} since
multiplication by $u_k$ is injective. For $k> n$ we have $R/I_{p,k}R=0$ since $v_{p,n}\in I_{p,k}$ and hence
$1\in I_{p,k}R$. If $q$ is a prime different to $p$ then $q$ is invertible in $R$ and $q\in I_{q,k}$ for all
$k\geq 1$, so that $R/I_{q,k}R=0$ and, again, there is nothing to check. Hence the conditions of the Exact
Functor Theorem hold.\end{pf}

As usual, we can now use the spectrum $E$ to define a cohomology theory. This has the following properties.

\begin{prop} The cohomology theory $E$ outlined above is multiplicative and complex oriented and there is a canonical map $\theta_X:MU^*(X)\to
E^*(X)$ for each $X$. In particular, the map $MU^*(\mathbb{C}P^\infty)\to E^*(\mathbb{C}P^\infty)$ sends the
complex orientation $x_{MU}$ to an orientation $x=\theta(x_{MU})$ which gives rise to the formal group law $\phi
F$.\end{prop}
\begin{pf} Since $E^*=R$ is concentrated in even degrees we see that $E$ is complex oriented by
Lemma \ref{h is complex oriented if in even degrees}. That $E$ is multiplicative is covered in \cite[Proposition
2.21]{HoveyStrickland} and using \cite[Proposition 2.20]{HoveyStrickland} we get a map of spectra $\theta:MU\to
E$ which induces the map $\theta_X:MU^*(X)\to E^*(X)$ for each $X$. It remains to show that $\theta(x_{MU})\in
E^2(\mathbb{C}P^\infty)$ is a complex orientation for $E$.

By naturality of the Atiyah-Hirzebruch spectral sequence, there is a commutative diagram
$$
\xymatrix{ H^*(\mathbb{C}P^\infty;MU^*) \ar@{=>}[r] \ar[d] & MU^*(\mathbb{C}P^\infty) \ar[d]\\
H^*(\mathbb{C}P^\infty;E^*) \ar@{=>}[r] & E^*(\mathbb{C}P^\infty)}.
$$
By the same arguments as in the proof of Lemma \ref{h is complex oriented if in even degrees} the map of the
$E_\infty$-pages just corresponds to the map $MU^*\lpow x\rpow \to E^*\lpow x\rpow, \sum_i a_i x^i\mapsto \sum_i
\theta(a_i)x^i$. Since $x_{MU}$ is a lift of the class $x$ to $MU^2(\mathbb{C}P^\infty)$ it follows that
$\theta(x_{MU})$ is a lift of $x$ to $E^2(\mathbb{C}P^\infty)$ and hence is an orientation for $E$. By
naturality we then have a commutative square
$$
\xymatrix{ MU^*(\mathbb{C}P^\infty) \ar[r]^-{\mu^*} \ar[d]_\theta & MU^*(\mathbb{C}P^\infty \times \mathbb{C}P^\infty) \ar[d]^\theta\\
E^*(\mathbb{C}P^\infty) \ar[r]_-{\mu^*} & E^*(\mathbb{C}P^\infty\times\mathbb{C}P^\infty)}
$$
and it follows that $\mu^*(\theta(x_{MU}))=\theta(\mu^*(x_{MU}))=\theta(F(x_1,x_2))=(\phi
F)(\theta(x_1),\theta(x_2))$, showing that the formal group law associated to $\theta(x_{MU})$ is $\phi F$, as
required.\end{pf}

\begin{cor} Let $y=\theta(x_{MU})$ be the complex orientation for $E^*$ defined above and put $x=u.y\in
E^0(\mathbb{C}P^\infty)$. Then $E^0(\mathbb{C}P^\infty)=E^0\lpow x\rpow$ and
$$\mu^*(x)=F(x_1,x_2)\in
E^0(\mathbb{C}P^\infty\times \mathbb{C}P^\infty)=E^0\lpow x_1,x_2\rpow,$$ where $F$ is the standard $p$-typical
formal group law of Proposition \ref{Universal deformation FGL}.\end{cor}

\begin{pf} For the first statement, recall that $E^*(\mathbb{C}P^\infty)=E^*\lpow y\rpow=\mathbb{Z}_p\lpow
u_1,\ldots,u_{n-1}\rpow [u,u^{-1}]\lpow y\rpow$ with $|u|=-2$ and $|y|=2$. First note that $E^0\lpow x\rpow$ is
clearly contained in $E^0(\mathbb{C}P^\infty)$ since $x$ has degree zero. Now, take $a\in
E^0(\mathbb{C}P^\infty)$ and write $a=\sum_i a_i y^i$, where $a_i\in E^*$ for each $i$. Then $|a_i|=-2i$ so that
we have $a_i=u^i a_i'$ for some $a_i'\in E^0$. Hence $a=\sum_i a_i'(uy)^i=\sum_i a_i'x^i\in E^0\lpow x\rpow$.

For the second statement, we have $$\mu^*(x)=\mu^*(uy)=u\mu^*(y)=u(\phi F)(y_1,y_2)=u(\phi
F)(u^{-1}x_1,u^{-1}x_2).$$ Note that, by uniqueness of the logarithm, we have $\log_{\phi F}(t)=(\phi
\log_F)(t)$. Further
$$\log_{\phi F}(u^{-1}t)=(\phi \log_F)(u^{-1}t)=u^{-1}t + \sum_I \frac{u_I}{p^{|I|}} (u^{-1}t)^{p^{\|I\|}}.u^N$$
where $u^N=(u^{p^{i_1}-1}).(u^{p^{i_2}-1})^{p^{i_1}}\ldots
(u^{p^{i_m}-1})^{p^{i_1+\ldots+i_{m-1}}}=u^{p^{\|I\|}-1}$ is the factor coming from the application of $\phi$ to
the coefficients. Hence $\log_{\phi F}(u^{-1}t)=u^{-1}\log_F(t)$. Further, $u^{-1}t=\log_{\phi
F}^{-1}(u^{-1}\log_F(t))$ so that $u^{-1}\log_F^{-1}(s)=\log_{\phi F}^{-1}(u^{-1}s)$. Hence
\begin{eqnarray*}
u(\phi F)(u^{-1}x_1,u^{-1}x_2) & = & u\log_{\phi F}^{-1}(\log_{\phi F}(u^{-1}x_1)+\log_{\phi F} (u^{-1}x_2))\\
& = & u\log_{\phi F}^{-1}(u^{-1}\log_F(x_1)+u^{-1}\log_F(x_2))\\
& = & uu^{-1}\log_F^{-1}(\log_F(x_1)+\log_F(x_2))\\
& = & F(x_1,x_2)
\end{eqnarray*}
so that $\mu^*(x)=F(x_1,x_2)$, as claimed.\end{pf}

\begin{defn} We refer to the theory $E$ developed above as the {\em Morava $E$-theory of height $n$ at the prime $p$}.
Clearly there is one such theory for each choice of prime $p$ and each integer $n>0$. Note that the coefficient
ring $E^*=\mathbb{Z}_p\lpow u_1,\ldots,u_{n-1}\rpow[u,u^{-1}]$ is concentrated in even degrees and there is an
invertible element, $u$ in degree $-2$. It follows that multiplication by $u$ gives rise to an isomorphism
$E^{k+2}(X)\iso E^{k}(X)$ for all $X$ and all $k$. We refer to the class $x=\theta(x_{MU})\in
E^2(\mathbb{C}P^\infty)$ as the {\em standard complex orientation for $E$} and the class $u.x\in
E^0(\mathbb{C}P^\infty)$ as the {\em standard complex coordinate for $E$}. Often, when working in degree 0, we
will abuse notation slightly and write the latter simply as $x$.\end{defn}

\begin{rem} Using a modified exact functor theorem due to Yagita (\cite{Yagita}) one can define, for each prime $p$ and each $n>0$, a related
cohomology theory $K$ with $K^*=\mathbb{F}_p[u,u^{-1}]$ where $u\in K^{-2}$. We refer to this as the {\em Morava
$K$-theory of height $n$ at $p$}. The convention here is slightly non-standard: in the literature, the term
Morava $K$-theory is often used with reference to a theory $K(n)$ with $K(n)^*=\mathbb{F}_p[v_n,v_n^{-1}]$ where
$v_n\in K(n)^{-2(p^n-1)}$. In fact $K$ is just a modified version of $K(n)$ obtained by setting
$K^*(X)=\mathbb{F}_p[u,u^{-1}]\tensor_{K(n)^*} K(n)^*(X)$, where $\mathbb{F}_p[u,u^{-1}]$ is made into a
$K(n)^*$-module by letting $v_n$ act as $u^{p^n-1}$. See \cite{RavenelNil} for further details on these
theories.\end{rem}

\subsection{The cohomology of finite abelian groups}

\begin{lem}\label{E^*(C_m)=E^*[[x]]/[m](x)} Let $C_m$ be the cyclic subgroup of $S^1$ of order $m$. Then, writing $x$
for the restriction of the complex orientation $x\in E^*(\mathbb{C}P^\infty)=E^*(BS^1)$ to $E^*(BC_m)$, we have
$E^*(BC_m)=E^*\lpow x\rpow/([m](x))$.\end{lem}
\begin{pf} This is Lemma 5.7 in \cite{HKR}.\end{pf}

\begin{cor}\label{Cyclic group iso to p-part} Let $m=ap^r$ where $a$ is coprime to $p$. Then $C_{p^r}$ is a subgroup of $C_m$ and the restriction map $E^*(BC_m)\to
E^*(BC_{p^r})$ is an isomorphism.\end{cor}
\begin{pf} Since $[m](x)=[ap^r](x)=[a]([p^r](x))$ and $[a](x)$ is a unit multiple of $x$ by Corollary \ref{[m](x) twiddles x} we see that $[m](x)$ is a unit multiple of $[p^r](x)$. Hence
\[E^*(BC_m)=E^*\lpow x\rpow/([m](x))=E^*\lpow x\rpow/([p^r](x))\iso E^*(BC_{p^r}).\qedhere\]\end{pf}

\begin{prop} [K\"unneth isomorphism] Let $X$ be any space and $Y$ be a space with $E^*(Y)$ free and finitely
generated over $E^*$. Then the map $E^*(X)\tensor_{E^*} E^*(Y)\to E^*(X\times Y)$ is an isomorphism.
\end{prop}
\begin{pf} This is Lemma 5.9 in \cite{HKR}.\end{pf}

\begin{cor}\label{E^*(B(GxC_m)) as tensor product} If $G$ is any group then, for any $m>0$, $$E^*(B(G\times C_m))\simeq E^*(BG)\tensor_{E^*}
E^*(BC_m).$$\end{cor}
\begin{pf} As in \ref{Cyclic group iso to p-part} we have $E^*(BC_m)\simeq E^*(BC_{p^r})=E^0\lpow x\rpow/([p^r](x))$ for some
$r$ and the latter is finitely generated and free over $E^*$ by the Weierstrass preparation theorem. Hence the
K\"unneth isomorphism holds.\end{pf}

We are now able to compute the Morava $E$-theory of any finite abelian group $A$ by writing $A$ as a product of
cyclic groups and applying Corollary \ref{E^*(B(GxC_m)) as tensor product} repeatedly. That is, we have the
following.

\begin{prop}\label{E^0(B prod C_m_i)} Let $A$ be a finite abelian group with $A\simeq \prod_i C_{m_i}$. Then there is an isomorphism
$$E^*(BA) \simeq E^*\lpow x_1,\ldots,x_r\rpow/([m_1](x_1),\ldots,[m_r](x_r))$$
where, writing $\alpha_i$ for the map $A\twoheadrightarrow C_{m_i} \rightarrowtail S^1$, we have
$x_i=\alpha_i^*(x)$.\end{prop}

\begin{cor}\label{E^0(BA) to E^0(BA_p) is iso} Let $A$ be a finite abelian group and let $A_{(p)}=\{a\in A\mid
a^{p^s}=1\text{ for some $s$}\}$ be the $p$-part of $A$. Then the restriction map $E^*(BA)\to E^*(BA_{(p)})$ is
an isomorphism.\end{cor}
\begin{pf} We write $A$ as a product of cyclic groups, say $A\simeq\prod_i C_{m_i}$ where $m_i=a_ip^{r_i}$. Then, using the K\"unneth isomorphism, we have
$$
\xymatrix{ E^*(BA) \ar[r]^-\sim \ar[d] & \bigotimes_{E^*} E^*(BC_{m_i}) \ar[d]\\
E^*(BA_{(p)}) \ar[r]^-\sim & \bigotimes_{E^*} E^*(BC_{p^{r_i}}).}
$$

By Corollary \ref{Cyclic group iso to p-part} the right hand map is an isomorphism and hence so is the left-hand
one.\end{pf}

\section{The cohomology of classifying spaces}

We outline some general theory that will be used in proving our results.

\begin{prop}\label{E^*(BG) finitely generated} If $G$ is a finite group then $E^*(BG)$ is finitely generated as an $E^*$-module.\end{prop}
\begin{pf} This is Corollary 4.4 in \cite{GreenleesStrickland}, although the related proof that $K(n)^*(BG)$ is finitely generated goes back to Ravenel \cite{RavenelK}. Note that for a $G$-space $Z$ they define $E_G^*(Z)=E^*(EG\times_G Z)$ and letting $Z$ be a single point gives $E_G^*(Z)=E^*(BG)$.\end{pf}

\begin{prop}\label{E^*(X) in even degrees} Suppose $X$ is a space with $E^*(X)$ finitely generated over $E^*$ and with $K(n)^*(X)$ concentrated in
even degrees. Then $E^*(X)$ is free over $E^*$ and concentrated in even degrees.\end{prop}

\begin{pf} This is Proposition 3.5 from \cite{StricklandSymmetricGroups}.\end{pf}

Recall that $K^*=\mathbb{F}_p [u,u^{-1}]$. This is a module over $E^*=\mathbb{Z}_p\lpow
u_1,\ldots,u_{n-1}\rpow[u,u^{-1}]$ under the map sending $u_i\mapsto 0$ for $i=0,\ldots,n-1$ (with $u_0$
understood to be $p$). We find that we can often recover $K^*(BG)$ from $E^*(BG)$.

\begin{prop}\label{K^*(BG)=K^* tensor E^*(BG)} If $E^*(BG)$ is free over $E^*$ then $K^*(BG)=K^*\tensor_{E^*} E^*(BG)$.\end{prop}

\begin{pf} This is Corollary 3.8 in \cite{StricklandSymmetricGroups}.\end{pf}

We assemble the above results to arrive at the following.

\begin{prop}\label{E^0(BG) only interesting group} Let $G$ be a finite group with $K(n)^*(BG)$ concentrated in even degrees. Then $E^*(BG)$ is free over $E^*$ and
concentrated in even degrees. Further
$$
E^i(BG) \simeq \left\{\begin{array}{ll} E^0(BG) & \text{if $i$ is even}\\
                                 0 & \text{otherwise}\end{array}\right.
\text{ and}\quad
K^i(BG)  \simeq \left\{\begin{array}{ll} K^0\tensor_{E^0} E^0(BG) & \text{if $i$ is even}\\
                                 0 & \text{otherwise.}\end{array}\right.
$$\end{prop}
\begin{pf} The first statement follows straight from Propositions \ref{E^*(X) in even degrees} and \ref{E^*(BG) finitely
generated}. Since $E^*$ contains the unit $u\in E^{-2}$, multiplication by $u^{-i}$ provides an isomorphism
$E^0(BG)\iso E^{2i}(BG)$, proving the statements about $E^i(BG)$. The final statement follows from an
application of Proposition \ref{K^*(BG)=K^* tensor E^*(BG)}.\end{pf}

\begin{lem}\label{I_i.I_j contained in I_i+j} Let $X$ be a connected CW-complex and $X_k$ denote its $k$-skeleton. Suppose that $X_0$ is a single
point and let $I_k=\ker(E^0(X)\overset{\res}{\longrightarrow} E^0(X_{k-1}))$. Then, for any $i$ and $j$ we have
$I_iI_j\subseteq I_{i+j}$.\end{lem}
\begin{pf} Let $\Delta:X\to X\times X$ denote the diagonal map. Then, by standard topological arguments, $\Delta$ is
homotopic to a skeleton-preserving map $\Delta'$, that is a map $\Delta':X\to X\times X$ such that
$\Delta'(X_k)\subseteq (X\times X)_k=\bigcup_{i+j=k} X_i\times X_j$. In fact, if $i$ and $j$ are any integers
with $i+j=k$ we have $\Delta'(X_{k-1})\subseteq (X_{i-1}\times X)\cup (X\times X_{j-1})$. Notice also that
$$(X\times X)/((X_{i-1}\times X)\cup (X\times X_{j-1}))=(X/X_{i-1})\wedge (X/X_{j-1}).$$ Thus we get an induced
map $\Delta':X/X_{k-1} \to (X/X_{i-1})\wedge (X/X_{j-1})$ fitting into the commutative diagram
$$
\xymatrix{ X_{k-1}  \ar[rr]^-{\Delta'} \ar[d] & & (X_{i-1}\times X)\cup (X\times X_{j-1}) \ar[d]\\
X \ar[rr]^-{\Delta'} \ar[d]  & & X\times X \ar[d]\\
X/X_{k-1} \ar[rr]^-{\Delta'} & &  (X/X_{i-1})\wedge (X/X_{j-1}).}
$$
On looking at the lower square and passing to cohomology we get
$$
\xymatrix{E^0(X)\\
\tilde{E}^0(X/X_{k-1}) \ar[u] & \tilde{E}^0(X/X_{i-1})\tensor_{E^0} \tilde{E}^0(X/X_{j-1}).\ar[l] \ar[ul]}
$$
Now, by consideration of the cofibre sequence $X_{k-1}\to X\to X/X_{k-1}$ and the associated long exact sequence
in $E$-theory it follows that $I_k=\im(\tilde{E}^0(X/X_{k-1})\to E^0(X))$. Hence the above diagram gives the
result.\end{pf}

\begin{lem}\label{E^0(X) local} Let $X$ be a connected CW-complex and $X_k$ denote its $k$-skeleton. Suppose that $X_0$ is a single
point and let $\mathfrak{m}=\ker(E^0(X)\overset{\epsilon}{\to} E^0\overset{\pi}{\to} \mathbb{F}_p)$. Then
$\mathfrak{m}$ is the unique maximal ideal of $E^0(X)$ and $E^0(X)$ is local.\end{lem}
\begin{pf} Since $E^0(X)/\mathfrak{m}\simeq \mathbb{F}_p$ we know that $\mathfrak{m}$ is maximal. Take $x\in
E^0(X)\setminus\mathfrak{m}$. We will show that $x$ is invertible. Note that $x$ is non-zero mod $\mathfrak{m}$
and so $\epsilon(x)$ lies in $E^0\setminus \mathfrak{m}_{E^0}$. Since $E^0$ is local it follows that
$\epsilon(x)$ is invertible in $E^0$. Let $y=1-\epsilon(x)^{-1}x$. Then $y\in\ker(\epsilon)=I_1$ (where $I_k$
($k> 0$) are the ideals of Lemma \ref{I_i.I_j contained in I_i+j}). It follows that $y^k\in I_k$ so that
$z=\sum_{k=0}^\infty y^k$ converges in $E^0(X)$ with respect to the skeletal topology. But $1+yz=z$ so that
$(1-y)z=1$ and hence $1-y=\epsilon(x)^{-1}x$ is invertible. It follows that $x$ is invertible and $E^0(X)$ is
local, as claimed.\end{pf}

\begin{lem}\label{Noetherian algebra is noetherian} Let $R$ be a Noetherian ring and $A$ an algebra over $R$, finitely generated as an $R$-module. Then
$A$ is Noetherian.\end{lem}
\begin{pf} See, for example, \cite{Sharp}. Every finitely generated $R$-module is Noetherian and every ideal of $A$ is
an $R$-submodule of $A$, so it follows that every ascending chain of ideals of $A$ is necessarily eventually
constant.\end{pf}

\begin{prop} Let $G$ be a finite group. Then $E^0(BG)$ is a complete local Noetherian ring.\end{prop}
\begin{pf} By Lemma \ref{Noetherian algebra is noetherian}, $E^0(BG)$ is Noetherian as it is a finitely generated
module over the Noetherian ring $E^0$. By Lemma \ref{E^0(X) local} $E^0(BG)$ is local with maximal ideal
$\mathfrak{m}=\ker(E^0(BG)\overset{\epsilon}{\to} E^0\overset{\pi}{\to} \mathbb{F}_p)$. It remains to show that
$E^0(BG)$ is complete with respect to the $\mathfrak{m}$-adic topology.

Now, $E^0(BG)$ inherits a topology from $E^0$, namely the $\mathfrak{m}_{E^0}$-adic topology generated by open
balls of the form $a+ \mathfrak{m}_{E^0}^rE^0(BG)$ and an application of the Artin-Rees lemma (specifically,
Theorem 8.7 in \cite{Matsumura}) shows that $E^0(BG)$ is complete with respect to this topology. The discussion
in \cite[p55]{Matsumura} shows that the $\mathfrak{m}_{E^0}$-adic topology coincides with the
$\mathfrak{m}$-adic one if and only if for each $N\in\mathbb{N}$ there exist $r$ and $s$ such that
$\mathfrak{m}^r\subseteq \mathfrak{m}_{E^0}^N E^0(BG)$ and $\mathfrak{m}_{E^0}^s E^0(BG)\subseteq
\mathfrak{m}^N$. Since $\mathfrak{m}_{E^0}E^0(BG) \subseteq \mathfrak{m}$ the latter of these conditions is
easily satisfied and it remains to show that the former holds.

Let $I=I_1=\ker(E^0(BG)\to E^0)$. Take the descending chain of ideals $$E^0(BG)\vartriangleright I\geqslant
I^2\geqslant I^3\geqslant\ldots$$ Then, writing $J_k=I^k/\mathfrak{m}_{E^0}I^k$, we get a descending chain of
finite dimensional $\mathbb{F}_p$-vector spaces
$$J_1\geqslant J_2\geqslant J_3\geqslant\ldots$$ which must therefore be eventually constant. In particular, there exists $r\in \mathbb{N}$ with
$J_{r+1}=J_r$ so that $I^{r+1}/\mathfrak{m}_{E^0} I^{r+1}=I^{r}/\mathfrak{m}_{E^0} I^{r}$. An application of
Proposition \ref{M/IM=N/IN implies M=N} then gives $I^{r+1}=I^r$ whereby, since $I\subseteq\mathfrak{m}$,
Nakayama's lemma leaves us with $I^r=0$. But $E^0(BG)/I_1=E^0$ and $E^0/\mathfrak{m}_{E^0}=\mathbb{F}_p$ so it
follows that $\mathfrak{m}=I_1+\mathfrak{m}_{E^0}E^0(BG)$ and $\mathfrak{m}^r\subseteq
I_1^r+\mathfrak{m}_{E^0}E^0(BG)=\mathfrak{m}_{E^0}E^0(BG)$. Hence we see that
$\mathfrak{m}^{rN}\subseteq\mathfrak{m}_{E^0}^N E^0(BG)$ and we are done.\end{pf}

\subsection{Transfers and the double coset formula}

Given a finite group $G$ and a subgroup $H$ of $G$ there is a map of $E^*$-modules $\transfer_H^G:E^*(BH)\to
E^*(BG)$ known as the \emph{transfer map} generalising an analogous map for ordinary cohomology (see
\cite{BensonRepAndCohy1}). The map is characterised by the following key properties.

\begin{lem}\label{Transfers} Let $G$ and $G'$ be finite groups. Then the
following hold.
\begin{enumerate}
\item If $K\leqslant H\leqslant G$ then $\transfer_K^G=\transfer_H^G\circ\transfer_K^H$.
\item If $H\leqslant G$ and $H'\leqslant G'$ then $\transfer_{H\times H'}^{G\times G'}=\transfer_H^G\tensor
\transfer_{H'}^{G'}$ as maps $E^*(BH)\tensor E^*(BH')\to E^*(BG)\tensor E^*(BG')$.
\item (Frobenius Reciprocity) If $H\leqslant G$ and then $\transfer_H^G(\res_H^G(a).b)=a.\transfer_H^G(b)$ for all $a\in E^*(BG),~b\in
E^*(BH)$. That is, viewing $E^*(BH)$ as an $E^*(BG)$-module, $\transfer_H^G$ is an $E^*(BG)$-module map.
\item If $N$ is normal in $G$ and $\pi$ denotes the projection $G\to G/N$ then
$\transfer_N^G(1)=\pi^*\transfer_1^{G/N}(1)$.
\item (Double Coset Formula) Suppose $H$ and $K$ are
subgroups of $G$. Then, considered as maps $E^*(BK)\to E^*(BH)$, we have the identity
$$\res_H^G\tr_K^G=\sum_{g\in H\setminus G/K} \tr_{H\cap (g K g^{-1})}^H \res_{H\cap (g K g^{-1})}^{g K
g^{-1}}\conj_g^*.$$ where $H\setminus G/K$ denotes the set of double cosets $\{HgK\mid g\in G\}$.
\end{enumerate}\end{lem}
\begin{pf} See \cite{Benson} and \cite{BensonRepAndCohy1}.\end{pf}

There are many useful applications of this map, as we will see in subsequent sections.

\begin{rem} Note that if $H$ and $K$ are subgroups of $G$ with $K$ normal in $G$ then the double coset formula
gives $\res_H^G\tr_K^G=\sum_{g\in H\setminus G/K} \tr_{H\cap K}^H \res_{H\cap K}^K\conj_g^*.$ If, in addition,
$H=K$ this simplifies further to leave us with $\res_H^G\tr_H^G=\sum_{g\in G/H}\conj_g^*.$\end{rem}

\subsection{Further results in $E$-theory}

\begin{prop}\label{Conjugation induces identity} Let $g\in G$ and write $\conj_g$ for the conjugation map $G\to G$. Then $\conj_g^*:E^*(BG)\to
E^*(BG)$ is the identity map.\end{prop}

\begin{pf} This is a simple corollary to Proposition \ref{conj_g:BG to BG is homotopic to identity}.\end{pf}

\begin{prop}\label{Restriction to Sylow injective} Let $G$ be a finite group and let $H$ be a subgroup of $G$ with $[G:H]$ coprime to $p$. Then the restriction map $E^*(BG)\to E^*(BH)$ is
injective.\end{prop}

\begin{pf} Take $a\in E^*(BG)$. Then, by Frobenius reciprocity, we have
$\transfer_H^G(\res_H^G(a))=a.\transfer_H^G(1)$. Now, by an application of the double coset formula,
$\res_1^G\transfer_H^G(1)=[G:H]$ which is coprime to $p$ and hence a unit in $E^0$. Thus, reducing modulo the
maximal ideal of $E^0(BG)$ we find that $\transfer_H^G(1)$ maps to a unit in $\mathbb{F}_p^\times$ and so, by
the general theory of local rings, it follows that it is a unit in $E^0(BG)$. Thus if $\res_H^G(a)=0$ then
$a=0$, that is $\res_H^G$ is injective.\end{pf}

\begin{lem}\label{E^*(BG) lands in invariants} Let $H$ and $K$ be subgroups of $G$ with $K\subseteq N_G(H)$. Then $K$ acts on $E^*(BH)$ and the map $E^*(BG)\to E^*(BH)$
lands in the $K$-invariants.\end{lem}

\begin{pf} Since $K\subseteq N_G(H)$, taking $k\in K$ we have a commutative diagram
$$
\xymatrix{ H \ar[d]_{\conj_k} \ar@{^{(}->}[r] & G \ar[d]^{\conj_k}\\
H \ar@{^(->}[r] & G}
$$
and thus, since $\conj_k^*:E^*(BG)\to E^*(BG)$ is just the identity by Proposition \ref{Conjugation induces
identity}, we get
$$
\xymatrix{ E^*(BH) & E^*(BG) \ar[l]\\
E^*(BH) \ar[u]^{\conj_k^*} & E^*(BG) \ar[l] \ar@{=}[u]}
$$
showing that $E^*(BG)\to E^*(BH)$ lands in $E^*(BH)^K$.\end{pf}

\begin{prop}\label{E^0(BG)=E^0(BN)^G/N if N is normal in G} Let $N$ be a normal subgroup of $G$ with $[G:N]$ coprime to $p$. Then the restriction map induces an isomorphism $E^*(BG)\iso
E^*(BN)^{G/N}$.\end{prop}

\begin{pf} Since $[G:N]$ is coprime to $p$ it follows from Proposition \ref{Restriction to Sylow injective} that the
restriction map $E^*(BG)\to E^*(BN)$ is injective. Further, by Lemma \ref{E^*(BG) lands in invariants}, the map
lands in the $G$-invariants of $E^*(BN)$. But, since $N$ acts trivially on $E^*(BN)$, the $G$-action on
$E^*(BN)$ factors through a $G/N$ action and thus we have an injective map $E^*(BG)\to E^*(BN)^{G/N}$.

For surjectivity, the double coset formula gives
$$\res_N^G\transfer_N^G=\sum_{g\in N\backslash G/N} \transfer_N^N \res_N^N\conj_g^*=\sum_{g\in G/N} \conj_g^*$$
(where we have used the fact that $N$ is normal) so that if $a\in E^*(BN)^{G/N}$ we get
$$\res_N^G\transfer_N^G(a)=\sum_{g\in G/N} \conj_g^*(a)=|G/N|a.$$
Thus, since $|G/N|$ is coprime to $p$ and hence a unit in $E^*$, we have
$$\res_N^G\left(\frac{1}{|G/N|}\transfer_N^G(a)\right)=\frac{1}{|G/N|}.|G/N|a=a$$
so that $\res_N^G:E^*(BG)\to E^*(BN)^{G/N}$ is surjective, as required.
\end{pf}

\subsection{Understanding $E^*(BG)$ as a categorical limit}

For a finite group $G$ let $\mathcal{A}(G)$ be the category with objects the abelian subgroups of $G$ and
morphisms from $B$ to $A$ being the maps of $G$-sets $G/B\to G/A$. These can be understood as shown below.

\begin{prop} Let $G$ be a finite group and $A$ and $B$ be abelian subgroups of $G$. If $f:G/B\to G/A$ is a map
of $G$-sets then $f$ is determined by $f(B)$ and, writing $f(B)=gA$, we have $g^{-1}Bg\subseteq A$.

Conversely, if there is $g\in G$ with $g^{-1}Bg\subseteq A$ then there is a map of $G$-sets $G/B\to G/A$ given
by $f(B)=gA$.\end{prop}

\begin{pf} Let $f:G/B\to G/A$ be a map of $G$-sets. Then, for all $g\in G$, we have $f(gB)=gf(B)$ so that $f$
is determined by $f(B)$. Write $f(B)=gA$. Then for all $b\in B$ we have
$$gA=f(B)=f(bB)=bf(B)=bgA$$
so that $g^{-1}bg\in A$. Hence $g^{-1}Bg\subseteq A$.

For the converse, if $g\in G$ with $g^{-1}Bg\subseteq A$ the $G$-map $f:G/B\to G/A$ given by $f(B)=gA$ is well
defined since if $hB=kB$ we have $k=hb$ for some $b\in B$ whereby
\[f(kB)=kf(B)=hbf(B)=hbgA=hg(g^{-1}bg)A=hgA=hf(B)=f(hB).\qedhere\]\end{pf}

The following result is due to Hopkins, Kuhn and Ravenel (\cite[Theorem A]{HKR}).

\begin{thm} Let $E$ be a complex oriented cohomology theory and
$G$ be a finite group. Then the map
$$\frac{1}{|G|}E^*(BG) \to \lim_{A\in\mathcal{A}(G)}
\frac{1}{|G|}E^*(BA)$$ is an isomorphism.\end{thm}

\begin{rem} Note that $\lim_{A\in\mathcal{A}(G)} \frac{1}{|G|}
E^*(BA)$ is a subring of $\prod_{A\in\mathcal{A}(G)} \frac{1}{|G|}E^*(BA)$. Note also that $E^*$ is $p$-local
and $E^*(BG)$ is trivial if $p\nmid |G|$. It follows that inverting the order of $G$ can be replaced by
inverting $p$ or, equivalently, by tensoring with $\mathbb{Q}$.\end{rem}

In fact we can simplify the right-hand limit somewhat. We will use the following modified version of the
theorem.

\begin{prop}\label{E^*(BG) as limit over p-groups} Let $\mathcal{A}(G)_{(p)}$ denote the full subcategory of $\mathcal{A}(G)$ consisting of the abelian $p$-subgroups of
$G$. Then the map
$$\mathbb{Q}\tensor E^*(BG) \to \lim_{A\in\mathcal{A}(G)_{(p)}}
\mathbb{Q}\tensor E^*(BA)$$ is an isomorphism.\end{prop}

\begin{pf} This was remarked on in \cite[Remark 3.5]{HKR}. By abstract category theory there is a unique map $\lim_{A\in\mathcal{A}(G)}
\mathbb{Q}\tensor E^*(BA)\to \lim_{A\in\mathcal{A}(G)_{(p)}} \mathbb{Q}\tensor E^*(BA)$ commuting with the
arrows. We will show this is an isomorphism.

Given any $C\in \mathcal{A}(G)$ we have $C_{(p)}\in \mathcal{A}(G)_{(p)}$ and hence, using the fact that
$\res_{C_{(p)}}^C$ is an isomorphism (Proposition \ref{E^0(BA) to E^0(BA_p) is iso}), a composite map
$$\lim_{A\in\mathcal{A}(G)_{(p)}} \mathbb{Q}\tensor E^*(BA)\to \mathbb{Q}\tensor
E^*(BC_{(p)})\iso \mathbb{Q}\tensor E^*(BC).$$ Now, take any map $C\to D$ in $\mathcal{A}(G)$ corresponding to
an element $g\in G$. Then $g$ induces a morphism $C_{(p)}\to D_{(p)}$ and we have the following commutative
diagram.
$$
\xymatrix{\\ & \mathbb{Q}\tensor E^*(BD_{(p)})  \ar[dd]_{\conj_g^*} \ar[r]^-\sim& \mathbb{Q}\tensor E^*(BD) \ar[dd]^{\conj_g^*}\\
\ds \lim_{A\in\mathcal{A}(G)_{(p)}} \mathbb{Q}\tensor E^*(BA) \ar[ur] \ar@/^4pc/@{-->}[urr] \ar[dr] \ar@/_4pc/@{-->}[drr]\\
& \mathbb{Q}\tensor E^*(BC_{(p)})  \ar[r]_-\sim& \mathbb{Q}\tensor E^*(BC).\\
&&}
$$
Thus by category theory we have a map $\lim_{A\in\mathcal{A}(G)_{(p)}} \mathbb{Q}\tensor E^*(BA)\to
\lim_{A\in\mathcal{A}(G)} \mathbb{Q}\tensor E^*(BA)$ which must be the inverse to our original map.\end{pf}

Before a corollary, we make the following definition.

\begin{defn} Let $f_i:R\to S_i$ ($i\in I$) be a family of maps. Then we say that the maps $(f_i)_{i\in I}$ are {\em
jointly injective} if the map $\prod_i f_i:R\to \prod_i S_i$ is injective.\end{defn}

\begin{cor}\label{Maximal abelian p-subgroups give jointly injective maps} Let $G$ be a finite group with $E^*(BG)$ free over $E^*$. Let $A_1,\ldots,A_s$ be abelian subgroups of $G$ such
that for each abelian $p$-subgroup $A$ of $G$ there is $g\in G$ such that $gAg^{-1}\subseteq A_i$ for some $i$.
Then the restriction maps $E^*(BG)\to E^*(BA_i)$ are jointly injective.\end{cor}
\begin{pf} Take $a\in E^*(BG)$ and suppose that $a$ maps to $0$
in $\prod_{i=1}^s E^*(BA_i)$, that is $a$ maps to $0$ in each $E^*(BA_i)$. Take any abelian $p$-subgroup $A$ of
$G$. Then, by the hypothesis, we have $g\in G$ with $gAg^{-1}\subseteq A_i$ for some $i$. Hence we get
$$
\xymatrix{ E^*(BG) \ar@{=}[d]_{\conj_g^*} \ar[r]^-{\res} & E^*(BA_i) \ar[r]^-{\res} & E^*(B(gAg^{-1}))
\ar[d]_\wr^{\conj_g^*}\\
E^*(BG) \ar[rr]^{\res} & & E^*(BA)}
$$
showing that $a$ maps to $0$ in $E^*(BA)$. Thus $a$ maps to $0$ in $\lim_{A\in\mathcal{A}(G)_{(p)}}
\mathbb{Q}\tensor E^*(BA)$ and hence in $\mathbb{Q}\tensor E^*(BG)$ by Proposition \ref{E^*(BG) as limit over
p-groups}. But $E^*(BG)$ is free over $E^*$ so that $E^*(BG)\to\mathbb{Q}\tensor E^*(BG)$ is injective. Thus
$a=0$, as required.\end{pf}

\subsection{Hopkins, Kuhn and Ravenel's good groups.}

We use the following concept from \cite{HKR}.

\begin{defn}\label{def:good groups} Let $G$ be a finite group. Then we define $G$ to be \emph{good} if $K(n)^*(BG)$ is generated over $K(n)^*$ by
transfers of Euler classes of complex representations of subgroups of $G$. Note that if $G$ is good then
$K(n)^*(BG)$ is concentrated in even degrees and Proposition \ref{E^0(BG) only interesting group}
applies.\end{defn}

Hence for good groups the only interesting cohomology is $E^0(BG)$ and we can recover $K^0(BG)$ as
$K^0\tensor_{E^0} E^0(BG)$. The following is Theorem E from \cite{HKR}.

\begin{prop}\label{Theorem E, HKR} Using the terminology of Definition \ref{def:good groups},\begin{enumerate}
\item every finite abelian group is good;
\item if $G_1$ and $G_2$ are good then so is $G_1\times G_2$;
\item if $\Syl_p(G)$ is good then so is $G$;
\item if $G$ is good then so is the wreath product $C_p\wr G$.\end{enumerate}
\end{prop}

We apply the results to our finite general linear groups to get the following.

\begin{prop} Let $K$ be a finite field with $v_p(|K^\times|)>0$. Then, for any $d>0$, the finite group $GL_d(K)$ is good.\end{prop}

\begin{pf} By Proposition \ref{Theorem E, HKR} part 3 it is sufficient to show that the Sylow $p$-subgroup of $GL_d(K)$ is good. By Proposition \ref{Syl_p(G)} we know that $\Syl_p(GL_d(K))=P_0\wr P_1$ where
$P_0=\Syl_p(\Sigma_d)$ and $P_1$ is the $p$-part of $K^\times$. Clearly $P_1$ is good by Proposition
\ref{Theorem E, HKR} part 1. Hence if $d<p$ then $\Syl_p(GL_d(K))=P_1$, which is good.

If $d=p^k$ for some $k>0$ then, by Proposition \ref{Syl_p(Sigma_p^k)}, $P_0$ is an iterated wreath product of
copies of $C_p$ so that $P_0\wr P_1=C_p\wr\ldots\wr C_p\wr P_0$ which is good by repeated use of Proposition
\ref{Theorem E, HKR} part 4.

For arbitrary $d$ we write $d=\sum_i a_ip^i$ and then, by Proposition \ref{Syl_p(Sigma_d)}, $P_0=\prod_i
\Syl_p(\Sigma_{p^i})^{a_i}$ with each $\Syl_p(\Sigma_{p^i})^{a_i}$ an iterated wreath product of copies of
$C_p$. Thus, using Lemma \ref{Wreath products distribute over times}, we have
$$P_0\wr P_1=\left(\prod_i \Syl_p(\Sigma_{p^i})^{a_i}\right)\wr P_1 = \prod_i
\left(\Syl_p(\Sigma_{p^i})\wr P_1\right)^{a_i}$$ with each term in the product good. Thus $P_0\wr P_1$ is
good.\end{pf}

\subsection{The cohomology of general linear groups over algebraically closed fields}\label{sec:cohomology of general linear groups}

Our main starting point for the cohomology of the finite general linear groups comes from the well understood
theory of general linear groups over the relevant algebraically closed fields. For any prime $l$ different to
$p$ we briefly outline the relevant results.

\begin{prop} Let $T\simeq S^1\times\ldots\times S^1$ denote the maximal torus of $GL_d(\mathbb{C})$. Then restriction
induces an isomorphism $H^*(BGL_d(\mathbb{C});\mathbb{Z}_{(p)})\iso
H^*(BT;\mathbb{Z}_{(p)})^{\Sigma_d}$.\end{prop}
\begin{pf} This is a classical theorem of Borel (\cite{Borel}).\end{pf}

\begin{cor} With the above notation, the restriction map gives an isomorphism $$E^*(BGL_d(\mathbb{C}))\iso
E^*(BT)^{\Sigma_d}.$$\end{cor}
\begin{pf} The proof is analogous to that of \cite[Corollary 2.10]{Tanabe}.\end{pf}

\begin{prop}\label{Friedlander and Mislin} Let $h$ be any multiplicative cohomology theory for which $p=0$ in
$h^*$ and let $l$ be a prime different to $p$. Then, writing $\overline{T}_d$ for the maximal torus of
$\Gflbar$, there are compatible maps $B\overline{T}_d\to BT$ and $BGL_d(\Flbar)\to BGL_d(\mathbb{C})$ which
induce isomorphisms $h^*(BT)\to h^*(B\overline{T}_d)$ and $h^*(BGL_d(\mathbb{C}))\iso
h^*(BGL_d(\Flbar))$.\end{prop}
\begin{pf} The main result of \cite{FriedlanderMislin} gives maps which induce an isomorphism on mod-$p$
homology and hence also on mod-$p$ cohomology. The proof then relies on a comparison of the relevant
Atiyah-Hirzebruch spectral sequences $H^*(X;h^*)\implies h^*(X)$.\end{pf}

\begin{lem}\label{K(n)^*(X)=0 implies E^*(X)=0} Let $X$ be a CW-complex with $X_0$ a single point. Then if $K(n)^*(X)=0$ we have $E^*(X)=0$.\end{lem}
\begin{pf} If $K(n)^*(X)=0$ then $K(n)^*(X)$ is (trivially) finitely generated over $K(n)^*$. Proposition 2.4 of
\cite{HoveyStrickland} then tells us that $E^*(X)$ is finitely generated over $E^*$. Further, $K(n)^*(X)$ is
(trivially) concentrated in even degrees so Proposition 2.5 of \cite{HoveyStrickland} tells us that $E^*(X)$ is
concentrated in even degrees and $0=K(n)^0(X)=E^0(X)/\mathfrak{m}_{E^0}E^0(X)$. But $E^0(X)$ is local by
Proposition \ref{E^0(X) local} and an application of Nakayama's lemma gives $E^0(X)=0$. It follows that
$E^*(X)=0$.\end{pf}

\begin{cor}\label{K(n)^*(X) to K(n)^*(Y) iso implies E-theory iso} Let $X$ and $Y$ be spaces as in Lemma \ref{K(n)^*(X)=0 implies E^*(X)=0}. Let $f:X\to Y$ be such that the induced map $f^*:K(n)^*(Y)\to K(n)^*(X)$ is an isomorphism. Then
$f^*:E^*(Y)\to E^*(X)$ is also an isomorphism.\end{cor}
\begin{pf} The cofibre sequence $X\to Y\to Y/X$ gives, on passing to cohomology, $K(n)^*(X/Y)=0$ since
$K(n)^*(Y) \to K(n)^*(X)$ is an isomorphism. Hence, by Lemma \ref{K(n)^*(X)=0 implies E^*(X)=0}, we find that
$E^*(X/Y)=0$ and, reversing the previous argument, an isomorphism $E^*(Y)\to E^*(X)$.\end{pf}

\begin{prop}\label{E^0(BGL_d(Flbar)) iso to E^0(BT_d)^Sigma_d} The maps of Proposition \ref{Friedlander and Mislin} give rise to compatible isomorphisms $E^*(BT)\simeq E^*(B\overline{T}_d)$ and
$E^*(BGL_d(\mathbb{C}))\simeq E^*(BGL_d(\Flbar)).$ Hence restriction induces an isomorphism
$E^*(BGL_d(\Flbar))\simeq E^*(B\overline{T}_d)^{\Sigma_d}$.\end{prop}
\begin{pf} By Proposition \ref{Friedlander and Mislin} we know that we get such maps in $K(n)$-theory. But an
application of Corollary \ref{K(n)^*(X) to K(n)^*(Y) iso implies E-theory iso} shows that we get the same result
in $E$-theory.\end{pf}

Now, recall from Section \ref{sec:finite fields} that we have a chosen embedding $\Flbar^\times\to S^1$. Let
$x\in E^2(B\Flbar^\times)$ be the restriction of the complex orientation for $E$ under the induced map
$E^*(\mathbb{C}P^\infty)\to E^*(B\Flbar^\times)$. Write $\pi_i$ for the $i^\text{th}$ projection
$\overline{T}_d\simeq (\Flbar^\times)^d\to \Flbar^\times$. Proposition \ref{E^0(BGL_d(Flbar)) iso to
E^0(BT_d)^Sigma_d} then gives us following corollary.

\begin{cor}\label{E^0(BGL_d(Flbar)=E^0[[c_1...c_d]])} Restriction induces an isomorphism
$$E^*(BGL_d(\Flbar))\iso E^*(B\overline{T}_d)^{\Sigma_d}=E^*\lpow x_1,\ldots,x_d\rpow^{\Sigma_d}= E^*\lpow \sigma_1,\ldots,\sigma_d\rpow$$
where $\sigma_i$ is the $i^\text{th}$ elementary symmetric function in
$x_1=\pi_1^*(x),\ldots,x_d=\pi_d^*(x)$.\end{cor}

\subsection{The cohomology of the symmetric group $\Sigma_p$}\label{sec:Sigma_p}

Using the analysis of Section \ref{sec:normalizer of C_p in Sigma_p} we can get an understanding of the
cohomology of $\Sigma_p$.

\begin{prop}\label{Cohomology of Sigma_p} Let $C_p=\langle \gamma_p\rangle$ denote standard cyclic subgroup of $\Sigma_p$ of order $p$. Then the inclusion $C_p\rightarrowtail\Sigma_p$ induces an isomorphism
$\ds E^0(B\Sigma_p)\simeq E^0(BC_p)^{\Aut(C_p)}.$\end{prop}
\begin{pf} Let $M=\Aut(C_p)\ltimes C_p\leqslant
\Sigma_p$. Then a transfer argument gives an isomorphism $H^*(B\Sigma_p;\mathbb{F}_p)\iso H^*(BM;\mathbb{F}_p)$
(see, for example, \cite[p74]{BensonRepAndCohy1}). Hence, an application of the Atiyah-Hirzebruch spectral
sequence shows that restriction gives a $K(n)^*$-isomorphism and, by Corollary \ref{K(n)^*(X) to K(n)^*(Y) iso
implies E-theory iso}, we get an isomorphism $E^0(B\Sigma_p)\iso E^0(BM)$. But $C_p$ is normal in $M$ and
$|M/C_p|=|\Aut(C_p)|=p-1$ which is coprime to $p$ so that we can use Proposition \ref{E^0(BG)=E^0(BN)^G/N if N
is normal in G} to get $E^0(BM)\iso E^0(BC_p)^{\Aut(C_p)}$. Hence $E^0(B\Sigma_p)\simeq
E^0(BC_p)^{\Aut(C_p)}$.\end{pf}

\begin{lem}\label{hat(k)(x)=k(x) mod p(x)} In $E^0\lpow x\rpow/[p](x)$ we have $[k](x)=[\hat{k}](x)$ for all $k\in (\mathbb{Z}/p)^\times$.\end{lem}
\begin{pf} We know that $\hat{k}=k$ mod $p$, so that $k=\hat{k}+ap$ for some $a \in \mathbb{Z}_p$. Then, modulo $[p](x)$, we have
$[k](x)=[\hat{k}+ap](x)=[\hat{k}](x)+_F[ap](x)=[\hat{k}](x)+_F[a]([p](x))=[\hat{k}](x),$ as claimed.\end{pf}

\begin{lem} With the obvious embedding $C_p\rightarrowtail S^1$, let $x$ denote the corresponding generator of $E^0(BC_p)$.
Then $x^{p-1}\in E^0(BC_p)^{\Aut(C_p)}$.\end{lem}
\begin{pf} First note that $\Aut(C_p)\simeq (\mathbb{Z}/p)^\times$ acts on $E^0(BC_p)$ by $k.x=[k](x)$. It
follows that $\prod_{k=1}^{p-1} [k](x)\in E^0(BC_p)^{\Aut(C_p)}$. But, by Lemmas \ref{hat(k)(x)=k(x) mod p(x)}
and \ref{Wilson's theorem}, we have $\prod_{k=1}^{p-1} [k](x)=\prod_{k=1}^{p-1}[\hat{k}](x)=\prod_{k=1}^{p-1}
\hat{k}x=-x^{p-1}$ so that $x^{p-1}=-\prod_{k=1}^{p-1} [k](x)\in E^0(BC_p)^{\Aut(C_p)}$.\end{pf}

\begin{prop}\label{E^0(BC)^Aut(C) basis} Put $d=-x^{p-1}\in E^0(BC_p)^{\Aut(C_p)}$. Then $E^0(BC_p)^{\Aut(C_p)}$ is free over $E^0$ with basis
$\{1,d,\ldots,d^{(p^n-1)/(p-1)}\}$.\end{prop}
\begin{pf} We have a basis $\{1,x,\ldots,x^{p^n-1}\}$ for $E^0(BC_p)$ over $E^0$. Thus, for
$k\in\Aut(C_p)\simeq(\mathbb{Z}/p)^\times$, we have $k.x^i=[k](x)^i=[\hat{k}](x)^i=\hat{k}^ix^i$. Taking any
$\sum_i a_i x^i\in E^0(BC_p)$, for all $k\in \Aut(C_p)$ we have
\begin{eqnarray*}
\textstyle \sum_i a_i x^i\in E^0(BC_p)^{\Aut(C_p)} & \iff & \textstyle \sum_i a_i x^i=k.\sum_i a_i x^i\\
& \iff & \textstyle \sum_i a_i x^i=\sum_i a_i \hat{k}^i x^i\\
& \iff & \textstyle a_i\hat{k}^i=a_i\quad\text{for all $i$}\\
& \iff & \textstyle a_i(\hat{k}^i-1)=0\quad\text{for all $i$.}
\end{eqnarray*} But $\hat{k}^i=1$ if and only if $i=0$ mod $p-1$. Hence $\sum_i a_i x^i\in E^0(BC_p)^{\Aut(C_p)}$ if and only if $a_i=0$ for $i\neq 0$
mod $p-1$. Thus \[E^0(BC_p)^{\Aut(C_p)}=E^0\{x^{j(p-1)}\mid 0\leq j \leq \text{$\textstyle
\frac{p^n-1}{p-1}$}\}=E^0\{1,d,\ldots,d^{(p^n-1)/(p-1)}\}.\qedhere\]\end{pf}

\begin{lem} With $x$ and $d$ as above, $\langle p\rangle(x)\in E^0(BC_p)^{\Aut(C_p)}$. Hence there is a polynomial
$f(t)\in E^0[t]$ such that $\langle p\rangle(x)=f(d)$. Further, $f(0)=p$.\end{lem}
\begin{pf} Taking $k\in \Aut(C_p)\simeq(\mathbb{Z}/p)^\times$ we have $k.\langle p\rangle(x)=\langle p\rangle([k](x))=\langle p\rangle([\hat{k}](x))=\langle p
\rangle(x)$ by Lemma \ref{<p>(kx)=<p>(x) for k a teichmuller lift}. Thus, using Proposition \ref{E^0(BC)^Aut(C)
basis} we can write $\langle p\rangle(x)=\sum_i a_i d^i=f(d)$ for some $a_i\in E^0$, as claimed. Putting $x=0$
then gives $f(0)=\langle p\rangle(0)=p$.\end{pf}

\begin{prop}\label{E^0(BSigma_p) in terms of d} Let $x,~d$ and $f$ be as above. Then $E^0(BC_p)^{\Aut(C_p)}\simeq E^0\lpow
d\rpow/df(d).$\end{prop}
\begin{pf} Firstly, note that we have $df(d)=-x^{p-1}\langle p\rangle(x)=0$ in $E^0(BC_p)=E^0\lpow x\rpow/x\langle p\rangle(x)$. Take any $g(t)\in E^0\lpow t\rpow$ with $g(d)=0$ in
$E^0(BC_p)$. Then $g(d)=h(x)[p](x)$ in $E^0\lpow x\rpow$ for some $h$. Putting $x=0$ we see that
$g(0)=h(0)[p](0)=0$ and so $d|g(d)=h(x)[p](x)=xh(x)\langle p\rangle(x)=xh(x)f(d)$. Since $f(0)=p\neq 0$ we
conclude that $d|xh(x)$. Hence $g(d)\in df(d)E^0(BC_p)$; that is, $g(d)=df(d)k_0(x)$ for some $k_0(x)$. But then
$k(x)=\frac{1}{p-1}\sum_{\alpha\in\Aut(C_p)} \alpha.k_0(x)$ is $\Aut(C_p)$-invariant and $g(d)=df(d)k(x)\in
df(d)E^0(BC_p)^{\Aut(C_p)}$, as required.\end{pf}

We will need the following standard results later.

\begin{lem}\label{Transfer_1^C_p(1)} Let $x,~d$ and $f$ be as above. Then
$\transfer_1^{C_p}(1)=\langle p\rangle(x)$ and $\transfer_1^{\Sigma_p}(1)=(p-1)!f(d)$.\end{lem}
\begin{pf} Write $\transfer_1^{C_p}(1)=g(x)$ mod $[p](x)$ for some $g(x)\in E^0\lpow x\rpow$. Then, using Frobenius reciprocity (Lemma \ref{Transfers}) we have
$x.\transfer_1^{C_p}(1)=\transfer_1^{C_p}(\res_1^{C_p}(x).1)=0$ so that we must have $xg(x)=0$ mod $[p](x)$.
Thus $xg(x)=x\langle p\rangle(x)h(x)$ for some $h$ and hence $g(x)=\langle p\rangle(x)h(x)$. Thus, mod $[p](x)$,
we find that $g(x)=\langle p\rangle(x)h(0)$. But $g(0)=\res_1^{C_p}\transfer_1^{C_p}(1)=|C_p|=p$ by an
application of the double-coset formula. Hence $p=g(0)=\langle p\rangle(0)h(0)=ph(0)$ so that $h(0)=1$. Thus
$\transfer_1^{C_p}(1)=\langle p\rangle(x)$, as claimed. For the second statement, we have
\begin{eqnarray*}
\res_{C_p}^{\Sigma_p}\transfer_1^{\Sigma_p}(1) &=& \sum_{\sigma\in {C_p}\backslash
\Sigma_p/1}\transfer_{{C_p}\cap 1}^{~C_p}
\res_{{C_p}\cap 1}^{~1}(\conj_\sigma^*(1))\\
&=& \sum_{\sigma\in {C_p}\backslash \Sigma_p} \transfer_1^{C_p}(1)\\
&=& (p-1)!\langle p\rangle(x).\end{eqnarray*} But $\res_{C_p}^{\Sigma_p}$ is injective so we find that
$\transfer_1^{\Sigma_p}(1)=(p-1)!\langle p\rangle(x)=(p-1)!f(d)$.\end{pf}

\subsection{$l$-Chern classes}

The results of Section \ref{sec:cohomology of general linear groups} allow us to construct convenient cohomology
classes which are analogous to the construction of ordinary Chern classes.

\begin{defn} Let $G$ be a group. Recall that $E^0(BGL_d(\Flbar))\simeq E^0\lpow \sigma_1,\ldots,\sigma_d\rpow$. Given a group homomorphism $\alpha:G\to GL_d(\Flbar)$ (equivalently, a
finite-dimensional $\Flbar$-representation of $G$) and a natural number $k$ we define the \emph{$k^\text{th}$
$l$-Chern class of $\alpha$} by
$$\hat{c}_k(\alpha)=\left\{\begin{array}{ll}\alpha^*(\sigma_k) & \text{for $0\leq k\leq d$}\\
0 & \text{otherwise,}\end{array}\right.\\$$ (where $\sigma_0$ is understood to be 1).\end{defn}

\begin{prop}\label{l-Chern Properties} The $l$-Chern classes satisfy the following properties.
\begin{enumerate}
\item For any group homomorphism $\alpha:G\to GL_d(\Flbar)$ we have $\hat{c}_0(\alpha)=1$.
\item (Functoriality) Given group homomorphisms $\alpha:G\to GL_d(\Flbar)$ and $f:H\to G$ we get
$\hat{c}_k(\alpha\circ f)=f^*(\hat{c}_k(\alpha))$ for all $k$.
\item Given group homomorphisms $\alpha:G\to GL_{d_1}(\Flbar)$ and $\beta:G\to GL_{d_2}(\Flbar)$ we get
$$\hat{c}_k(\alpha\oplus\beta)=\sum_{i+j=k}\hat{c}_i(\alpha)\hat{c}_j(\beta)$$
where $\alpha\oplus\beta:G\to GL_{d_1+d_2}(\Flbar)$ is given by $g\mapsto \alpha(g)\oplus\beta(g)$.
\item Let $\text{id}_l:\Flbar^\times\to\Flbar^\times$ be the identity. Then $\hat{c}_1(\text{id}_l)=x$, the
restriction of our complex orientation.\end{enumerate}
\end{prop}

\begin{proof} Properties 1 and 4 follow straight from the definition. For property 2, if $k\leq 0 $ or $k>d$
the result is clear since $f^*(0)=0$ and $f^*(1)=1$. Otherwise $1\leq k\leq d$ and we have
$$\hat{c}_k(\alpha\circ f)=(\alpha\circ f)^*(\sigma_k)=(f^*\circ
\alpha^*)(\sigma_k)=f^*(\alpha^*(\sigma_k))=f^*(\hat{c}_k(\alpha)),$$ as required. It remains to prove property
3.

We note first that the diagram
$$\xymatrix{
(\Flbar^\times)^{d_1}\times(\Flbar^\times)^{d_2} \ar[rr]^-\sim \ar[d] & & (\Flbar^\times)^{d_1+d_2} \ar[d]\\
GL_{d_1}(\Flbar)\times GL_{d_2}(\Flbar) \ar[rr]^-\mu & & GL_{d_1+d_2}(\Flbar)}
$$
induces
$$\xymatrix{
E^0(B(\Flbar^\times)^{d_1})\widehat{\tensor}_{E^0} E^0(B(\Flbar^\times)^{d_2})  & & E^0(B(\Flbar^\times)^{d_1+d_2}) \ar[ll]_-\sim\\
E^0(BGL_{d_1}(\Flbar))\widehat{\tensor}_{E^0} E^0(BGL_{d_2}(\Flbar)) \ar[u] & & E^0(BGL_{d_1+d_2}(\Flbar))
\ar[u]
\ar[ll]_-{\mu^*}.}\\
$$
In the usual way, we can write $E^0(B(\Flbar^\times)^{d_1+d_2})\simeq E^0\lpow
x_1,\ldots,x_{d_1},x_{d_1+1},\ldots,x_{d_1+d_2}\rpow$ to get $E^0(B(\Flbar^\times)^{d_1})\simeq E^0\lpow
x_1,\ldots,x_{d_1}\rpow$ and $E^0(B(\Flbar^\times)^{d_2})\simeq E^0\lpow x_{d_1+1},\ldots,x_{d_1+d_2}\rpow$
where $x_i=\hat{c}_1((\Flbar^\times)^{d_1+d_2}\overset{\pi_i}{\longrightarrow} \Flbar^\times)$. We use the
following lemma.

\begin{lem}\label{Elementary Symmetric Split} Let $d_1,d_2\in\mathbb{N}$ and $d=d_1+d_2$. For $1\leq k\leq d$ let $\sigma_k$ be the $k^\text{th}$
elementary symmetric function in $x_1,\ldots,x_d$. Let $\sigma_{d_1,i}$ (respectively $\sigma_{d_2,j}$) be the
$i^\text{th}$ (respectively $j^\text{th}$) elementary symmetric function in $x_1,\ldots,x_d$ (respectively
$x_{d_1+1},\ldots,x_d$). Then
$$\sigma_k=\sum_{i+j=k}\sigma_{d_1,i}\sigma_{d_2,j}.$$
\end{lem}
\begin{pf} Straightforward combinatorics.\end{pf}
Using the notation of Lemma \ref{Elementary Symmetric Split} we then have
$E^0(BGL_{d_1}(\Flbar))=E^0\lpow\sigma_{d_1,1}\ldots,\sigma_{d_1,d_1}\rpow$ and
$E^0(BGL_{d_2}(\Flbar))=E^0\lpow\sigma_{d_2,1}\ldots,\sigma_{d_2,d_2}\rpow$. Thus, chasing the diagram and using
injectivity of the vertical arrows, we see that
$\mu^*(\sigma_k)=\sum_{i+j=k}\sigma_{d_1,i}\tensor\sigma_{d_2,j}$.

Now given homomorphisms $\alpha:G\to GL_{d_1}(\Flbar)$ and $\beta:G\to GL_{d_2}(\Flbar)$, the map
$\alpha\oplus\beta$ is given by the composition
$G\overset{\alpha\times\beta}{\longrightarrow}GL_{d_1}(\Flbar)\times
GL_{d_2}(\Flbar)\overset{\mu}{\to}GL_d(\Flbar)$, whereby
\begin{eqnarray*}
\hat{c}_k(\alpha\oplus\beta) & = & (\alpha\oplus\beta)^*\mu^*(\sigma_k)\\
& = & (\alpha\oplus\beta)^*\left(\sum_{i+j=k} \sigma_{d_1,i}\tensor \sigma_{d_2,j}\right)\\
& = & \sum_{i+j=k} \alpha^*(\sigma_{d_1,i})\tensor \beta^*(\sigma_{d_2,j})\\
& = & \sum_{i+j=k} \hat{c}_i(\alpha)\hat{c}_j(\beta).
\end{eqnarray*}\end{proof}

\begin{rem} The properties of Proposition \ref{l-Chern Properties} are enough to completely determine the
$l$-Chern classes. That is, they can be viewed as axioms for the $l$-Chern classes.\end{rem}

\begin{prop} Given one-dimensional $\Flbar$-representations $\alpha,\beta:G\to \Flbar^\times$ we have $$\hat{c}_1 (\alpha\tensor\beta)=\hat{c}_1 (\alpha)+_F
\hat{c}_1(\beta).$$\end{prop}
\begin{pf} We have a commutative diagram
$$
\xymatrix{ G \ar[rr]^-{(\alpha,\beta)} \ar[rrd]_{\alpha\tensor\beta} & &  \Flbar^\times\times\Flbar^\times \ar[d]^\mu\\
& & \Flbar^\times}
$$
which, on passing to cohomology, gives the result.\end{pf}

\begin{defn} Given a group homomorphism $\alpha:G\to GL_d(\Flbar)$ we define the \textit{$l$-Euler class of
$\alpha$} to be the top $l$-Chern class, that is $\text{euler}_l(\alpha)=\hat{c}_d(\alpha)$.\end{defn}


\chapter{Generalised character theory}\label{ch:HKR}

\section{Generalised characters}

In this section we outline the generalised character theory developed by Hopkins, Kuhn and Ravenel in
\cite{HKR}. We start by recalling some results on Pontryagin duality.

\subsection{Locally compact groups}

A topological group $G$ is called \emph{locally compact} if there is a neighbourhood of the identity which is
contained in a compact set. Given a locally compact abelian group $G$ we define its \emph{Pontryagin dual} or
\emph{character group}, $G^*$, by $G^*=\Hom_{\text{cts}}(G,S^1)$. We give $G^*$ the weakest topology such that
the maps $G^*\to S^1, \chi\mapsto \chi(g)$ are continuous for each $g\in G$. If $G$ is locally compact then
$G^*$ is also locally compact and the assignment $G\mapsto G^*$ is a contravariant endofunctor on locally
compact abelian groups. Further, for each $G$ there is a canonical (continuous) isomorphism $G\to (G^*)^*$ given
by $g\mapsto (G^*\to S^1, \chi\mapsto \chi(g))$. From here on $G$ may be identified with $(G^*)^*$ without
comment.

Every discrete (and hence every finite) group is locally compact. Other examples of locally compact groups
include $\mathbb{R},S^1$ and $\mathbb{Z}_p$. If $G$ and $H$ are locally compact then $G\oplus H$ is locally
compact. Using the result $\Hom(\colim A_i, B)=\limit \Hom (A_i, B)$ from general category theory, we see that
$(\colim A_i)^*=\limit (A_i)^*$ wherever the limits exist. Of particular interest to us will be the result
$(\mathbb{Z}/p^\infty)^*=(\colim \mathbb{Z}/p^r)^*=\limit (\mathbb{Z}/p^r)^*=\mathbb{Z}_p$.

For each $m>1$ there is a canonical isomorphism $\mathbb{Z}/m \iso (\mathbb{Z}/m)^*$ given by $1\mapsto
(\mathbb{Z}/m \to S^1, 1\mapsto e^{2\pi i/m})$. Further, if $G$ and $H$ are locally compact abelian groups then
$$(G\oplus
H)^*=\Hom_{\text{cts}}(G\oplus H,S^1)\simeq \Hom_{\text{cts}}(G,S^1)\oplus \Hom_{\text{cts}}(H,S^1)= G^*\oplus
H^*.$$ Hence, if $G$ is any finite abelian group then there is an isomorphism $G\simeq G^*$ but in general this
will be non-canonical.

We will mainly be looking at a group $\Theta\simeq (\mathbb{Z}/p^\infty)^n$. From the remarks earlier it is
clear that $\Theta^*\simeq \mathbb{Z}_p^n$.

\subsection{Generalised characters}

As usual, let $x\in E^0(\mathbb{C}P^\infty)$ be our standard complex coordinate and $F$ the associated standard
$p$-typical formal group law. Recall that for each $m\geq 0$ we let $g_m(t)$ denote the Weierstrass polynomial
of degree $p^{nm}$ which is a unit multiple of $[p^m]_F(t)$ in $E^0\lpow t\rpow$.

Let $Q(E^0)$ denote the field of fractions of $E^0$ and fix an algebraic closure $\overline{Q(E^0)}$. For each
$m\geq 0$ let $\Theta_m=\{a\in \overline{Q(E^0)}\mid g_m(a)=0\}$. Note that $\Theta_m\subseteq \Theta_{m+1}$ and
define $\Theta=\bigcup_m \Theta_m$.

\begin{prop} There are abelian group structures on $\Theta_m$ and $\Theta$ given by $+_F$ and (non-canonical) isomorphisms $\Theta_m \simeq (\mathbb{Z}/p^m)^n$ and $\Theta \simeq (\mathbb{Z}/p^\infty)^n$.
\end{prop}

\begin{proof}
If $a,b\in\Theta_m$ then $[p^m](a+_F b)=[p^m](a)+_F[p^m](b)=0$, so $\Theta_m$ is closed under $+_F$. The
identity element is $0$. Given $a\in\Theta_m$ we have $[p^m](-_F a)=-_F[p^m](a)=0$ which gives inverses.
Finally, associativity and commutativity follow easily from properties of $+_F$. Thus $\Theta_m$ is an abelian
group for all $m$ and hence so is $\Theta=\bigcup_m \Theta_m$.

The roots of $g_m$ are distinct, so that $|\Theta_m|=p^{mn}$ (see, for example, \cite[Proposition 27]{FSFG}).
Hence $\Theta_m$ is a finite abelian $p$-group and we can write $\Theta_m\simeq
\mathbb{Z}/p^{r_1}\oplus\ldots\oplus\mathbb{Z}/p^{r_s}$ for some $r_1,\ldots,r_s$ and some $s$ with
$r_1+\ldots+r_s=mn$. Since $\Theta_m(p)=\{a\in \overline{Q(E^0)}\mid [p](a)=0\}=\Theta_1$ we get
$|\Theta_m(p)|=p^n$ and it follows that $s=n$. The fact that $[p^m](a)=0$ for all $a\in\Theta_m$ gives $r_i\leq
m$ for each $i$. Thus we conclude that $r_i=m$ for all $i$ and $\Theta_m\simeq (\mathbb{Z}/p^m)^n$. Choosing
compatible isomorphisms for each $m$ then gives an isomorphism $\Theta\simeq
(\mathbb{Z}/p^\infty)^n$.\end{proof}

We now outline the development of the generalised character map for $E^0(BG)$ for finite groups $G$.

\begin{defn} Let $A\subseteq \overline{Q(E^0)}$ be a finite abelian
group under $+_F$. Then we define $L_A\subseteq\overline{Q(E^0)}$ to be the ring generated by $\mathbb{Q}\tensor
E^0$ and $A$. Note that given a map of two such groups $f:A\to B$ there is an induced $\mathbb{Q}\tensor
E^0$-algebra map $f_*:L_A\to L_B$ sending each $a\in A\subseteq L_A$ to $f(a)\in B\subseteq L_B$. We define
$L_m=L_{\Theta_m}$.\end{defn}

\begin{prop} Let $A\subseteq \overline{Q(E^0)}$ be a finite abelian
$p$-group under $+_F$. Then there is a unique $E^0$-algebra map $\psi_A:E^0(BA^*)\to L_A$ such that
$\psi_A(a^*(x))=a$ for all $a\in A=(A^*)^*=\Hom_\text{cts}(A^*,S^1)$. Further these maps are natural for
homomorphisms of such groups.\end{prop}

\begin{pf} Choose an isomorphism $A\iso \prod_{i=1}^m
C_{p^{r_i}}$ to get generators $a_1,\ldots,a_m$ for $A$. Then, by Lemma \ref{E^0(B prod C_m_i)}, we get
$E^0(BA^*)\simeq E^0\lpow x_1,\ldots,x_m \rpow /([p^{r_1}](x_1),\ldots,[p^{r_m}](x_m))$ where, viewing $a_i$ as
an element of $(A^*)^*$, we have $x_i={a_i}^*(x)$. Hence, since $[p^{r_i}](a_i)=0$ in $L_A$, there is an
$E^0$-algebra map $\psi_A:E^0(BA^*)\to L_A$ given by ${a_i}^*(x)=x_i\mapsto a_i$.

Now, for any $a,b\in A=(A^*)^*$, the commutative diagram
$$
\begin{array}{ccc}
\xymatrix{A^* \ar^{a\times b}[r] \ar_{a+b}[rd]& S^1\times S^1 \ar^\mu[d]\\
& S^1} & \begin{array}{c}\phantom{ \text{induces}}\\\phantom{ \text{induces}}\\
\text{induces}\end{array} &
\xymatrix{E^0(BA^*) & E^0\lpow x_{(1)},x_{(2)} \rpow \ar_{(a\times b)^*}[l]\\
& E^0\lpow x \rpow \ar^{(a+b)^*}[lu] \ar_{\mu^*}[u]}
\end{array}
$$
which translates to $(a+b)^*(x)=F((a\times b)^*(x_{(1)}),(a\times b)^*(x_{(2)}))=a^*(x)+_Fb^*(x)$. Hence, given
an arbitrary $a\in A$ we can write $a=\sum_i{a_{j_i}}$ to get
$$\textstyle \psi_A(a^*(x))=\psi_A((\sum_i{a_{j_i}})^*(x))=\psi_A(\sum_F x_{j_i})=\sum_F a_{j_i}=a.$$
Thus the map $\psi_A$ has the required property and is unique by construction. For naturality, let $B\subseteq
\overline{Q(E^0)}$ be another finite abelian group under $+_F$ and $f:A\to B$ a group homomorphism. Then we have
maps $f_*:E^0(BA^*)\to E^0(BB^*)$ and $f_*:L_A\to L_B$ and
$$\psi_B(f_*(a^*(x)))=\psi_B(f(a)_*(x))=f(a)=f_*(a)=f_*(\psi_A(a^*(x)))$$
which is the claimed naturality condition.
\end{pf}

This gives us the following immediate corollary.

\begin{cor} The inclusions $\Theta_m\hookrightarrow\Theta_{m+1}$
induce commutative squares
$$\xymatrix{ E^0(B(\Theta_m)^*) \ar[d] \ar^-{\psi_m}[rr] & & L_m \ar@{^(->}[d]\\
E^0(B(\Theta_{m+1})^*) \ar_-{\psi_{m+1}}[rr] &  & L_{m+1}.}$$ Hence, writing $L=\bigcup_m L_m$ we get an induced
map $\psi:\colim E^0(B(\Theta_m)^*)\to L$.
\end{cor}

\begin{defn} Given topological groups $G$ and $H$ we have an action of $G$ on
$\Hom_\text{cts}(H,G)$ given by $(g.\alpha)(h)=g\alpha(h)g^{-1}$ for $g\in G, \alpha\in\Hom_\text{cts}(H,G)$ and
$h\in H$. We define the set $\Rep(H,G)$ by $\Rep(H,G)=\Hom_\text{cts}(H,G)/G$.\end{defn}

\begin{rem} Notice that, by Proposition \ref{Conjugation induces identity}, if $\alpha\in \Hom_\text{cts}(H, G)$ then
the induced map $\alpha^*:E^0(BG)\to E^0(BH)$ depends only on the class of $\alpha$ in
$\Hom_\text{cts}(H,G)/G=\Rep(H,G)$. Also note that if $G$ is a finite group, then any homomorphism $f:\thstar\to
G$ must factor through $\thstar/\ker f$ which is a finite quotient of $\thstar$ and hence discrete. It follows
that $f$ is automatically continuous; that is, $\Hom_{\text{cts}}(\thstar,G)=\Hom(\thstar,G)$.\end{rem}

\begin{defn} Let $G$ be a finite group and let $\alpha\in\Hom(\thstar,G)$. Then $\alpha$ induces
maps $\alpha_m^*:E^0(BG)\to E^0(B\Theta_m^*)$ for each $m$ which fit into the diagram
$$
\xymatrix{ E^0(BG) \ar^{\alpha_m^*}[r] \ar_{\alpha_{m+1}^*}[dr] & E^0(B\Theta_m^*) \ar[d]\\
& E^0(B\Theta_{m+1}^*)}
$$
and hence a map $\alpha^*:E^0(BG)\to \colim E^0(B\Theta_m^*)$. We define $\chi_\alpha$ to be the composite
$$E^0(BG)\overset{\alpha^*}{\longrightarrow} \colim E^0(B\Theta_m^*)\overset{\psi}{\longrightarrow} L.$$

Finally, writing $\Map(S,T)$ for the set of functions (that is, set-maps) $S\to T$, we define $\chi:E^0(BG)\to
\Map(\Rep(\thstar,G),L)$ by $(\chi(a))(\alpha)=\chi_\alpha(a)$ for $a\in E^0(BG)$ and $\alpha\in\Hom(\thstar,
G)$ and refer to it as the {\em generalised character map for $G$}. By the remarks above this is a well defined
map of $E^0$-algebras.\end{defn}

We are now able to state the result of Hopkins, Kuhn and Ravenel.

\begin{prop}\label{HKR} (Hopkins, Kuhn, Ravenel) Let $G$ be a finite group. The map $\chi$ defined above
induces an isomorphism of $L$-algebras
$$L\tensor_{E^0} E^0(BG) \iso \Map(\Rep(\thstar,G),L).$$\end{prop}
\begin{pf} This is part of Theorem C from \cite{HKR}.\end{pf}

\section{Generalised characters and the finite general linear groups} \label{sec:HKRandReps}

In this section we apply the generalised character theory to the group $GL_d(\Fq)$, where $q=l^r$ is a power of
some prime different to $p$ and $d$ is a positive integer. Our main result is the following.

\begin{theorem}\label{Rep(thstar,G)} Let $\Phi=(\mathbb{Z}/p^\infty)^n$ and let $\Lambda$ be the subgroup of $\mathbb{Z}_p^\times$ generated by $q$. Then there
is a bijection $\Rep(\Theta^*,GL_d(\Fq))\iso (\Phi^d/\Sigma_d)^\Lambda$. Further, there there is a natural
addition structure on each of $\coprod_k\Rep(\Theta^*,GL_k(\Fq))$ and $\coprod_k (\Phi^k/\Sigma_k)^\Lambda$ and
the bijection $$\coprod_k\Rep(\Theta^*,GL_k(\Fq))\iso\coprod_k (\Phi^k/\Sigma_k)^\Lambda$$ is an isomorphism of
abelian semigroups.\end{theorem}

We will study two particular cases, firstly where $d<p$ and secondly where $d=p$ and find that
$(\Phi^d/\Sigma_d)^\Lambda$ is easy to understand in both cases. Further, this will give us a good understanding
of the ring $L\tensor_{E^0} E^0(BGL_p(\Fq))$.

\subsection{Representation theory}

As above, $q=l^r$ is a power of a prime different to $p$ and $d$ is a positive integer. We may write
$\overline{\mathbb{F}}_q$ for the algebraic closure $\Flbar$ and $\mathbb{F}_{q^d}$ for $\fl{rd}$.

\begin{defn} We write $\Phi=\Rep(\Theta^*,GL_1(\Flbar))=\Hom_\text{cts}(\Theta^*,\Flbar^\times)$.\end{defn}

As in Section \ref{sec:finite fields}, let $\Gamma=\langle \Frob_q\rangle$ be the subgroup of
$\Gal(\Flbar/\mathbb{F}_q)$ generated by the $r^\text{th}$ power of the Frobenius map. Since $\Gamma$ acts on
$\Flbar$ we have an induced action of $\Gamma$ on $\Phi$ given by $(\gamma.\alpha)(a)=\gamma.\alpha(a)$ (for
$\gamma\in\Gamma$, $\alpha:\Theta^*\to GL_1(\Flbar)$ and $a\in\thstar$). Note also that we have a componentwise
action of $\Gamma$ on $\Gflbar$ which induces an action of $\Gamma$ on $\Rep(\thstar,\Gflbar)$ in the obvious
way.

\begin{prop}\label{Phi is abelian p-group} Point-wise multiplication makes $\Phi$ into an
abelian $p$-group. Further, there is an isomorphism $\Phi\simeq (\mathbb{Z}/p^\infty)^n$.\end{prop}
\begin{pf} Let $\phi\in \Phi$. Then $\phi:\thstar\to \Flbar^\times$ is continuous. In particular, since
$\Flbar^\times$ is discrete, $\ker(\phi)=\phi^{-1}(1)$ is open. But the open neighbourhoods of $0$ in
$\mathbb{Z}_p$ are just the subgroups $p^m\mathbb{Z}_p$ for $m\in\mathbb{N}$. Since $\thstar\simeq
\mathbb{Z}_p^n$ it follows that $\ker(\phi)\supseteq p^N\thstar$ for some $N$. Thus
$\phi^{p^N}(a)=\phi(a)^{p^N}=\phi({p^N}a)=0$ and $\phi$ has order a power of $p$.

For the final statement we notice that
$$\Hom_\text{cts}(\mathbb{Z}_p,\Flbar^\times)=\Hom_{\text{cts}}(\mathbb{Z}_p,\Syl_p(\Flbar^\times))\simeq\Hom_{\text{cts}}(\mathbb{Z}_p,\mathbb{Z}/p^\infty).$$
Since any continuous homomorphism $f:\mathbb{Z}_p\to \mathbb{Z}/p^\infty$ will have finite image it will be
determined by $f(1)$. Thus $\Hom_{\text{cts}}(\mathbb{Z}_p,\mathbb{Z}/p^\infty)=\mathbb{Z}/p^\infty$ and the
result follows.\end{pf}

Let $K$ be a field. Then, remembering that any two representations $\rho_1,\rho_2:\thstar\to GL_d(K)$ are
isomorphic if and only if there is some $g\in GL_d(K)$ with $\rho_1=g\rho_2g^{-1}$, we see that there is an
obvious correspondence of $\Rep(\thstar,GL_d(K))$ with isomorphism classes of continuous $d$-dimensional
$K$-representations of $\thstar$. We will denote elements in the former by $[\alpha]$ where $\alpha$ is a
homomorphism from $\thstar$ to $GL_d(K)$ and elements of the latter by pairs $(V,\alpha)$ where $V$ is a
$d$-dimensional $K$-vector space and $\alpha:\thstar\to GL(V)$.

\begin{defn} Using the correspondence outlined above we define
$\Irr(\thstar,GL_d(K))$ to be the subset of $\Rep(\thstar,GL_d(K))$ corresponding to the irreducible
representations.\end{defn}

Our next step is to try to understand the $\Flbar$-representation theory of $\thstar$. For this we need the
following results. Recall that Schur's Lemma (see \cite{Serre}) tells us that any map $f:V\to W$ of irreducible
$K$-representations of a group $G$ is either zero or an isomorphism.

\begin{lem}\label{Irreps} Let $G$ be an abelian group and $K$ an algebraically closed field. Then any irreducible finite dimensional $K$-representation of $G$ is one-dimensional.\end{lem}
\begin{pf} Let $(V,\rho)$ be any irreducible finite-dimensional
$K$-representation of $G$. If $g\in G$ then, since $K$ is algebraically closed, $\rho(g)$ has an eigenvalue,
$\lambda$ say. Thus $\ker(\rho(g)-\lambda.\text{id}_V)\neq\emptyset$ since it contains the eigenvectors
corresponding to $\lambda$. Hence, by Schur's lemma, we must have $\rho(g)-\lambda.\text{id}_V=0$ so that
$\rho(g)=\lambda.\text{id}_V$. Given any subspace $W$ of $V$ we then get $\rho(g)(W)\subseteq W$ so that $W$ is
a subrepresentation of $V$. But $V$ is irreducible, so that either $W=0$ or $W=V$. Hence $V$ has no non-trivial
proper subspaces and so is one-dimensional, as required.\end{pf}

\begin{lem}\label{Maschke} Let $K$ be a (discrete) field with $p$ invertible in $K$. Then every continuous finite-dimensional $K$-representation of $\thstar$ is a sum of irreducible representations.\end{lem}

\begin{pf} This is essentially Maschke's theorem. Let $(V,\rho)$ be a continuous finite-dimensional
$K$-representation of $\thstar$. Suppose $\dim V=d>1$. Note that, as in the proof of Proposition \ref{Phi is
abelian p-group}, $\ker(\rho)\supseteq p^N\thstar$ for some $N\in\mathbb{N}$ and $\rho$ factors through the
finite abelian $p$-group $G=\thstar/p^N\thstar$. If $V$ is irreducible, then we are done. Otherwise there is a
proper subrepresentation $W$ of $V$ of dimension less than $d$. Let $W'$ be any vector space complement of $W$
in $V$. Writing $\pi$ for the projection $V=W'\oplus W\twoheadrightarrow W$ we define a map $r:V\to V$ by
$$r(v) = \frac{1}{|G|}\sum_{g\in G} \rho(g)\pi (\rho(g)^{-1}(v)).$$
Then the image of $\rho$ is contained in $W$ and for any $w\in W$ we have $r(w)=w$; that is, $r$ is a projection
of $V$ onto $W$. Let $U=\ker r$ so that $V=U\oplus W$. For any $h\in G$ it is easy to check that we have
$r(\rho(h)v)=\rho(h)r(v)$ so that if $u\in U$ we have $r(\rho(h)u)= \rho(h)r(u)=0$ showing that $\rho(h)u\in U$.
Thus $U$ is also a subrepresentation of $V$ and $V=U\oplus W$ is a sum of representations of $\thstar$ of
dimension less than $d$. The result then follows by induction, noting that all 1-dimensional representations are
(trivially) irreducible.\end{pf}

\begin{cor}\label{Flbar reps of theta* split into 1-dims} Every continuous $\Flbar$-representation of $\thstar$ is a sum of 1-dimensional representations.
Moreover, this decomposition is unique up to reordering.\end{cor}
\begin{pf} The first statement is immediate from Lemmas \ref{Irreps} and \ref{Maschke}. The second is a standard
application of Schur's Lemma: any irreducible representation $W$ appears in the decomposition for $V$ with
multiplicity $\dim \Hom(V,W)$.\end{pf}

\begin{prop}\label{Rep Thstar} Let $\Sigma_d$ act on $\Phi^d$ by permuting the factors. Then the map
$\Phi^d\to \Rep(\thstar,GL_d(\Flbar))$ given by $(\phi_1,\ldots,\phi_d)\mapsto [\phi_1\oplus\ldots\oplus
\phi_d]$ induces a $\Gamma$-equivariant bijection $\Phi^d/\Sigma_d\iso \Rep(\thstar,GL_d(\Flbar))$.\end{prop}

\begin{pf} Apply Corollary \ref{Flbar reps of theta* split into 1-dims} to see that the map is a bijection. For
$\gamma\in\Gamma$ we have $$\gamma.[\phi_1,\ldots,\phi_d]=[\gamma.\phi_1,\ldots,\gamma.\phi_d]\mapsto
[\gamma.\phi_1\oplus\ldots\oplus\gamma.\phi_d]=\gamma.[\phi_1\oplus\ldots\oplus\phi_d]$$ so the map is
$\Gamma$-equivariant.\end{pf}

\begin{lem}\label{understanding irreducibles} Let $V$ be an $\Fq$-representation of
$\thstar$ of dimension $d$. Then $V$ is irreducible if and only if there is a 1-dimensional $\fq{d}$-vector
space structure on $V$ and a representation $\phi:\thstar\to GL_1(\fq{d})$ with $\min \{s\in \mathbb{N}\mid
\phi^{q^s}=\phi\}=d$ such that the diagram
$$
\xymatrix{~\phantom{\text{commutes}} & \thstar \ar[r] \ar[d]_\phi & \Fq[\thstar] \ar[r] \ar[d] & \End_{\Fq}(V)\\
& \fq{d}^\times \ar[r] & \fq{d} \ar[r] & \End_{\mathbb{F}_{q^d}}(V) \ar[u]& \text{commutes.}}
$$\end{lem}
\begin{pf} Suppose $V$ is irreducible. Viewing $V$ as an $\Fq[\thstar]$-module, let $\alpha:\Fq[\thstar]\to \End_{\Fq}(V)$ be the map corresponding to
the action of $\Fq[\thstar]$ on $V$ and put $K=\im(\alpha)$. We will show that $K\simeq\fq{d}$ and that $V$ is a
1-dimensional $K$-vector space. Let $0\neq \psi\in K$ and write $\psi=\alpha(a)$ for $a\in \Fq[\thstar]$. Then,
since $\thstar$ and hence $\Fq[\thstar]$ is commutative, taking any $b\in\Fq[\thstar]$ we have
$$\psi(b.v)=(\alpha(a)\circ\alpha(b))(v)=(\alpha(ab))(v)=(\alpha(ba))(v)=\alpha(b)(\psi(v))=b.\psi(v).$$  Thus $\psi$ is an $\Fq[\thstar]$-endomorphism of $V$ and
hence, by Schur's lemma, is an automorphism. Further, by the Cayley-Hamilton theorem, $\psi$ satisfies an
equation $\psi^r + c_1 \psi^{r-1} + \ldots + c_r=0$ for some $c_1,\ldots,c_r\in \Fq$ with $c_r=\det(\psi)\in
\Fq^\times$. Then $\psi^{-1}=-\alpha((c_{r-1} + \ldots + c_1 a^{r-2} + a^{r-1})c_r^{-1})$ lies in $K$ and $K$ is
a field.

Notice that, by definition of $K$, $V$ is a $K$-vector space. Suppose $\dim_K(V)>1$ and let $U$ be a proper
$K$-subspace of $V$. Then, for $a\in \Fq[\thstar]$ we have $\alpha(a)\in K$ so $\alpha(a)(u)\in U$ for all $u\in
U$. Thus $U$ is an $\Fq[\thstar]$-module; that is, $U$ is a proper subrepresentation of $V$ contradicting the
fact that $V$ is irreducible. Hence we must have $\dim_K(V)=1$ so that $V\simeq K$. Since $\dim_{\Fq}(V)=d$ it
follows that $V\simeq K \simeq \fq{d}$, as required.

For the final statement, let $\phi$ be the composition $\thstar \to
\Fq[\thstar]\overset{\alpha}{\longrightarrow} \End_{\Fq}(V)$. Then, since $\im(\phi)\subseteq \fq{d}$, it
follows that $\phi^{q^d}=\phi$. Suppose that $\phi^{q^s}=\phi$ for some $s<d$. Then $\im(\phi)\subseteq
\fq{s}\subset \fq{d}$. But then $\fq{s}$ is a proper $\Fq$-subrepresentation of $V$, contradicting the fact that
$V$ is irreducible over $\Fq$.

Conversely, suppose we have an $\fq{d}$-representation $\phi:\thstar\to GL_1(\fq{d})=\fq{d}^\times$ satisfying
the relevant properties. Note that $\phi$ factors through a finite $p$-group $\thstar/p^N\thstar$ for some $N$,
so that $\im(\phi)\subseteq\Syl_p(\fq{d})\simeq C_{p^v}$, where $v=v_p(q^d-1)$. Further, the conditions on
$\phi$ ensure that $\phi$ has order precisely $p^v$ and that, for any $s<d$, writing $v_s=v_p(q^s-1)$ we have
$v_s<v$. Suppose that $(\fq{d},\phi)$ is not irreducible and write $\fq{d}\simeq U_1\oplus \ldots\oplus U_m$ as
a sum of irreducible $\Fq$-representations $U_i$ of dimensions $d_i$. Applying the first part of the result we
find that $U_i\simeq \fq{d_i}$ with action map $\phi_i:\thstar\to (\fq{d_i})^{\times}$. Then we find that
$\im(\phi_i)\simeq C_{p^{v_{d_i}}}$. In particular, since $v_i<v$, we find $(\phi_i)^{p^{v-1}}=1$ so that
$(\phi_1\times\ldots\times\phi_m)^{p^{v-1}}=1$. Hence we see that $\phi^{p^{v-1}}=1$, contradicting the
assumption on its order. Hence $(\fq{d},\phi)$ is irreducible.\end{pf}

We move towards an understanding of $\Rep(\thstar,GL_d(\Fq))$. By post-composing with the inclusion
$GL_d(\Fq)\rightarrowtail GL_d(\Flbar)$ we get a map $\Rep(\thstar,GL_d(\Fq))\to \Rep(\thstar,GL_d(\Flbar))$.
Since $\Fq=\Flbar^\Gamma$ we see that this map lands in the $\Gamma$-invariants. Thus, using
$\Gamma$-equivariance of the bijection $\Phi^d/\Sigma_d\iso \Rep(\thstar,GL_d(\Flbar))$, we find that we have a
map $\Rep(\thstar,GL_d(\Fq))\to (\Phi^d/\Sigma_d)^\Gamma$. We will shortly show this is bijective, but first we
need a result from field theory.

\begin{lem}\label{Galois extensions} Let $K$ be a field and $L$ be a Galois extension of
$K$ with Galois group $G$. Then there is a $\overline{K}$-vector space isomorphism $\overline{K}\tensor_K L\iso
\Map(G,\overline{K})$ given by $a\tensor b\mapsto (g\mapsto a.g(b))$. \end{lem}
\begin{pf} This is a well known result. An application of \cite[Theorem 14.1]{Adamson} shows that the map
$\overline{K}[G]\to \Hom_{K}(L,\overline{K})$ is an isomorphism of $\overline{K}$-vector spaces and the result
follows on applying $\Hom_{\overline{K}}(-,\overline{K})$.\end{pf}

\begin{prop}\label{Irreducibles split over Flbar} Let $V$ be an irreducible $\Fq$-representation of $\thstar$ of dimension $d$. Then there
is a continuous $\Flbar$-representation $(\Flbar,\phi)$ of $\thstar$ such that $\min\{s\in\mathbb{N}\mid
\phi^{l^{rs}}=\phi\}=d$ and
$$\Flbar\tensor_{\Fq} V\simeq (\Flbar,\phi)\oplus
(\Flbar,\phi^{l^r})\oplus\ldots\oplus(\Flbar,\phi^{l^{r(d-1)}}).$$\end{prop}
\begin{pf} Using Lemma \ref{understanding irreducibles} we can assume that $V=\fl{rd}$ and we get an
$\Flbar$-representation $\phi:\thstar\to GL_1(\fl{rd})\rightarrowtail GL_1(\Flbar)$ with the required properties
on its order. Next we apply Lemma \ref{Galois extensions} to the extension $\fl{rd}/\Fq$. Note that
$\Gal(\fl{rd}/\Fq)=\langle \Frob^r\rangle$ is cyclic of order $d$ generated by the $r^\text{th}$ power of the
Frobenius map so that $\Map(\Gal(\fl{rd}/\Fq),\Flbar)$ is just $\Flbar^d$. Thus we have an isomorphism of
$\Flbar$-vector spaces $\Flbar\tensor_{\Fq} \fl{rd} \iso \Flbar^d$ given by $a\tensor b\mapsto
(ab,ab^{l^r},\ldots,ab^{l^{r(d-1)}}).$ It is clear that, giving the right-hand side a $\thstar$-action by
$x.(a_1,\ldots,a_d)=(a_1\phi(x),a_2\phi(x)^{l^r},\ldots,a_d\phi(x)^{l^{r(d-1)}})$, the isomorphism is
$\thstar$-equivariant and thus is the claimed isomorphism of $\Flbar$-representations.\end{pf}

We will call an element of $(\Phi^d/\Sigma_d)^\Gamma$ \emph{irreducible} if it is of the form
$[\phi,\phi^{l^r},\ldots,\phi^{l^{r(d-1)}}]$ for some $\phi\in\Phi$ with $\min\{s\in\mathbb{N}\mid
\phi^{l^{rs}}=\phi\}=d$. We write $\Irr_d(\Phi)$ for the set of irreducible elements of
$(\Phi^d/\Sigma_d)^\Gamma$.

\begin{prop} The binary operation $+:(\Phi^s/\Sigma_s)^\Gamma\times(\Phi^t/\Sigma_t)^\Gamma\to (\Phi^{s+t}/\Sigma_{s+t})^\Gamma$
given by $[\phi_1,\ldots,\phi_s]+[\phi_1',\ldots,\phi_t']=[\phi_1,\ldots,\phi_s,\phi_1',\ldots,\phi_t']$ makes
$\coprod_{d>0}(\Phi^d/\Sigma_d)^\Gamma$ into an abelian semigroup freely generated by the irreducible
elements.\end{prop}
\begin{pf} If $[\phi_1,\ldots,\phi_d]\in\Phi^d/\Sigma_d$ is $\Gamma$-invariant then
$[\phi_1,\ldots,\phi_d]=[\Frob^r.\phi_1,\ldots,\Frob^r.\phi_d]$ so that for each $i$ there is a $j$ with
$\phi_j=\phi_i^{l^r}$. Thus $\phi_1,\phi_1^{l^r},\phi_1^{l^{2r}},\ldots$ all appear in the expression
$[\phi_1,\ldots,\phi_d]$ and so, by finiteness, we must at some point have $\phi_1^{l^{rs}}=\phi_1$. Thus
$[\phi_1,\ldots,\phi_1^{l^{r(s-1)}}]$ is an irreducible element of $(\Phi^{s}/\Sigma_s)^\Gamma$ which appears as
a summand in $[\phi_1,\ldots,\phi_d]$. Continuing in this way it is easy to see that $[\phi_1,\ldots,\phi_d]$
decomposes as a sum of irreducibles in a unique way.\end{pf}

\begin{prop}\label{Rep thstar like phi} The map
$\alpha:\Rep(\Theta^*,GL_d(\Fq))\to (\Phi^d/\Sigma_d)^\Gamma$ is bijective.\end{prop}

\begin{pf} By Proposition \ref{Irreducibles split over Flbar}, $\alpha$ restricts to a map
$\Irr(\thstar,GL_d(\Fq))\to\Irr_d(\Phi)$; we will show that this is a bijection. It then follows,
using Lemma \ref{Maschke}, that $\coprod_{d>0}\Rep(\thstar,GL_d(\Fq))$ bijects with
$\coprod_{d>0}(\Phi^d/\Sigma_d)^\Gamma$ and from that we conclude that $\alpha$ itself is bijective.

Take any irreducibles $V,W\in \Irr(\thstar,GL_d(\Fq))$ with $\alpha(V)=\alpha(W)$. Then, as before, we can
assume that $V=(\fl{rd},\phi)$ and $W=(\fl{rd},\phi')$ for some $\phi,\phi':\thstar\to \fl{rd}^\times$. Viewing
$\phi$ and $\phi'$ as maps $\thstar\to \fl{rd}^\times\to\Flbar^\times$ we have
$\alpha(V)=[\phi,\phi^{l^r},\ldots,\phi^{l^{r(d-1)}}]$ and
$\alpha(W)=[\phi',(\phi')^{l^r},\ldots,(\phi')^{l^{r(d-1)}}]$. Thus we must have $\phi'=\phi^{l^{rs}}$ for some
$0\leq s<d$ and it follows that the $\Fq$-linear isomorphism $V\iso W$, $a\mapsto a^{l^{rs}}$ gives an
isomorphism of $\Fq$-representations $V\simeq W$; that is $V=W$ in $\Irr(\thstar,GL_d(\Fq))$.

For surjectivity, given $[\phi,\phi^{l^r},\ldots,\phi^{l^{r(d-1)}}]\in \text{Irr}_d(\Phi)$ we have
$\phi^{l^{rd}}=\phi$ so that the image of $\phi$ is contained in $\fl{rd}\subseteq\Flbar$. Using Proposition
\ref{understanding irreducibles} the irreducible $d$-dimensional $\Fq$-representation $V=(\fl{rd},\phi)$
satisfies $\alpha(V)=[\phi,\phi^{l^r},\ldots,\phi^{l^{r(d-1)}}]$.\end{pf}

Hence, to get a good understanding of $\Rep(\thstar,GL_d(\Fq))$ it suffices to study the set
$(\Phi^d/\Sigma_d)^\Gamma$.

\begin{lem} The set $\Irr_d(\Phi)$ bijects with $\Gamma$-orbits of size $d$ on $\Phi$.\end{lem}
\begin{pf} Each of these consists of all unordered $d$-tuples $(\phi,\phi^{q},\ldots,\phi^{q^{d-1}})$ with
$\phi^{q^d}=\phi$.\end{pf}

To take our analysis further, we will make the simplifying assumption that $v_p(q-1)=v>0$.

\begin{prop}\label{Irr_d(Phi) must be a power of p} Suppose that $v_p(q-1)=v>0$. Then
$$\Irr_d(\Phi)=\left\{\begin{array}{ll} \phantom{(}\Phi(p^v) & \text{if $d=1$}\\(\Phi(p^{v+k})\setminus\Phi(p^{v+k-1}))/\Gamma & \text{if $d=p^k$ for some $k>0$}\\
\phantom{(}0 & \text{otherwise}.\end{array}\right.$$\end{prop}
\begin{pf} The case $d=1$ is clear. For $d>1$, using the fact that $\Phi$ is a $p$-group we see that $\phi\in\Phi$ generates a $\Gamma$-orbit of
size $d$ if and only if it has order $p^{v_p(q^d-1)}$ and $v_p(q^s-1)<v_p(q^d-1)$ for all $s<d$. By Proposition
\ref{v_p(k^s-1) for general k}, this latter requirement is equivalent to $v_p(s)<v_p(d)$ for all $s<d$ which is
satisfied precisely when $d$ is a power of $p$.\end{pf}

We look at two easy to understand cases.

\begin{prop} Suppose that $v_p(q-1)=v>0$ and $d<p$. Then $$(\Phi^d/\Sigma_d)^\Gamma=
\Irr_1(\Phi)^d/\Sigma_d=\Phi(p^v)^d/\Sigma_d.$$\end{prop}
\begin{pf} Since every element of the left-hand side is a sum or irreducibles and $d<p$ it follows from Proposition \ref{Irr_d(Phi) must be a power of p} that all
these irreducibles must be in $\Irr_1(\Phi)$.\end{pf}

\begin{rem}\label{Representations for d<p} Suppose that $v_p(q-1)=v>0$ and $d<p$. As $\Phi(p^v)$ consists of all
$\phi\in\Phi$ for which $\phi^{p^v}=1$, any such $\phi$ therefore satisfies $\phi^q=\phi$ and thus has image
contained in $\Fq$. Hence $\Rep(\thstar,GL_1(\Fq))\to \Phi(p^v)$ is a bijection and the diagram
$$
\xymatrix{ \Rep(\thstar,GL_d(\Fq)) \ar[r]^-\sim & (\Phi^d/\Sigma_d)^\Gamma\\
\Rep(\thstar,GL_1(\Fq))^d/\Sigma_d \ar[u]  \ar[r]^-\sim & \Phi(p^v)^d/\Sigma_d \ar[u]_-\wr}
$$
shows that $\Rep(\thstar,GL_d(\Fq))=\Rep(\thstar,GL_1(\Fq))^d/\Sigma_d$; that is, any representation
$\rho:\thstar\to GL_d(\Fq)$ is diagonalisable.\end{rem}

\begin{prop}\label{(Phi^p/Sigma_p)^Gamma)} Suppose that $v_p(q-1)=v>0$. Then
$$\textstyle (\Phi^p/\Sigma_p)^\Gamma=
(\Irr_1(\Phi)^p/\Sigma_p)\coprod \Irr_p(\Phi)=(\Phi(p^v)^p/\Sigma_p)\coprod
(\Phi(p^{v+1})\setminus\Phi(p^v))/\Gamma.$$\end{prop}
\begin{pf} Writing any element in the left-hand side as a sum of irreducibles, it is clear that either it is
already irreducible or it is a sum of $1$-dimensional irreducibles.\end{pf}

For $d>p$ things get more complicated.

\subsection{Applying the character theory}\label{sec:character theory for GL_d(Fq)}

Now that we have a good understanding of the set $\Rep(\thstar,GL_d(\Fq))$ we can apply the generalised
character theory to the group $GL_d(\Fq)$ to better understand the structure of its cohomology. We first need
the following lemma.

\begin{lem} Let $G$ act on a set $X$. Then, if $Y$ is any set, $G$ acts on $\Map(X,Y)$ by $(g.f)(x)=f(g^{-1}.x)$
and the obvious map $\Map(X/G,Y)\to \Map(X,Y)$ gives a bijection $$\Map(X/G,Y)\simeq \Map(X,Y)^G.$$\end{lem}
\begin{pf} Given $f:X/G\to Y$ it is clear that the map $\tilde{f}:X\to Y$, $x\mapsto f(\bar{x})$ is
$G$-invariant and it is easy to check that this construction gives a bijection of sets.\end{pf}

Recall that $T_d$ denotes the maximal torus of $GL_d(\Fq)$.

\begin{prop} Suppose that $v_p(q-1)=v>0$ and $d<p$. Then the restriction map induces an isomorphism
$L\tensor_{E^0} E^0(BGL_d(\Fq))\iso L\tensor_{E^0} E^0(BT_d)^{\Sigma_d}$.\end{prop}
\begin{pf} Note first that
$\Rep(\thstar,T_d)=\Rep(\thstar,(\Fq^\times)^d)=\Rep(\thstar,\Fq^\times)^d=\Rep(\thstar,GL_1(\Fq))^d$. Then,
using the remarks in \ref{Representations for d<p}, applying the generalised character theory we get a diagram
$$
\xymatrix{ L\tensor_{E^0} E^0(BGL_d(\Fq)) \ar[r]^-\sim \ar[d] & \Map(\Rep(\thstar,GL_d(\Fq)),L) \ar[d]_-\wr\\
L\tensor_{E^0} E^0(BT_d)^{\Sigma_d} \ar[r]^-\sim \ar[d] & \Map(\Rep(\thstar,GL_1(\Fq))^d,L)^{\Sigma_d} \ar[d]\\
L\tensor_{E^0} E^0(BT_d) \ar[r]^-\sim & \Map(\Rep(\thstar,GL_1(\Fq))^d,L)}
$$
which gives the result.\end{pf}

In fact it is not hard to show that the map $E^0(BGL_d(\Fq))\to E^0(BT_d)^{\Sigma_d}$ is itself an isomorphism,
which we do in Chapter \ref{ch:E^0(BGL_d(Fq))}.

We next look at the case where $d=p$ and again take the simplifying assumption that $v_p(q-1)=v>0$. Fix a basis
for $\fq{p}$ as a vector space over $\Fq$. Then for any $a\in \fq{p}^\times$ we have an $\Fq$-linear
automorphism $\mu_a:\fq{p}\to \fq{p}$ given by multiplication by $a$. Let $\mu:\fq{p}^\times\to GL_p(\Fq),
a\mapsto \mu_a$ be the corresponding group homomorphism. This gives us a map $\mu^*:E^0(BGL_p(\Fq))\to
E^0(B\fq{p}^\times)$. We need the following lemma.

\begin{lem}\label{CRT applied rationally to p^s+1 series} For any $s>0$, the quotient maps $E^0\lpow x\rpow/[p^{s+1}](x)\to E^0\lpow x\rpow/[p^s](x)$ and $E^0\lpow x\rpow/[p^{s+1}](x)\to E^0\lpow x\rpow/\langle p\rangle([p^s](x))$
induce an isomorphism
$$\mathbb{Q}\tensor \frac{E^0\lpow x\rpow}{[p^{s+1}](x)}\overset{\sim}{\longrightarrow} \left(\mathbb{Q}\tensor
\frac{E^0\lpow x\rpow}{[p^{s}](x)}\right)\times \left(\mathbb{Q}\tensor \frac{E^0\lpow x\rpow}{\langle
p\rangle([p^{s}](x))}\right)$$ where $\langle p\rangle (t)=[p](t)/t\in E^0\lpow t\rpow$ is the divided
$p$-series.\end{lem}
\begin{pf} Recall that $[p^{s+1}](x)$ and $[p^s](x)$ are unit multiples of the Weierstrass polynomials $g_{s+1}(x)$ and $g_s(x)$
respectively. Further $g_{s}(x)$ divides $g_{s+1}(x)$. Put $g(x)=g_{s+1}(x)/g_s(x)$. Then, since $\langle
p\rangle(x)=p$ mod $x$ it follows that $\langle p\rangle([p^s](x))=p$ mod $[p^s](x)$ so that $p$ is in the ideal
of $E^0\lpow x\rpow/[p^{s+1}](x)$ generated by $g(x)$ and $g_s(x)$. Hence, using the Chinese remainder theorem
(in particular, Corollary \ref{Q tensor CRT}) the result follows.\end{pf}

Note that $v_p(|\fq{p}^\times|)=v_p(q^p-1)=v+1$ and our complex orientation gives us a presentation
$E^0(B\fq{p}^\times)\simeq E^0\lpow x\rpow/[p^{v+1}](x)$. Further, $\Gamma$ acts on $E^0(B\fq{p}^\times)$ and
hence also on the quotient ring $E^0\lpow x\rpow/\langle p\rangle([p^v](x))$. Write $q:E^0(B\fq{p}^\times)\to
E^0\lpow x\rpow/\langle p\rangle([p^v](x))$ for the quotient map and let $\alpha=q\circ\mu^*:E^0(BGL_p(\Fq))\to
E^0\lpow x\rpow/\langle p\rangle ([p^v](x))$.

\begin{prop}\label{L tensor E^0(BFq^p) splits} Suppose that $v_p(q-1)=v>0$ and let $\alpha$ be as above. Then $\alpha$ lands in the $\Gamma$-invariants of $E^0\lpow x\rpow/\langle
p\rangle([p^v](x))$ and the map
$$L\tensor_{E^0} E^0(BGL_p(\Fq)) \longrightarrow \left(L\tensor_{E^0} E^0(BT_p)^{\Sigma_p}\right)\times \left(L\tensor_{E^0} (E^0\lpow
x\rpow/\langle p\rangle ([p^v](x)))^\Gamma\right)$$ is an isomorphism.\end{prop}
\begin{pf} We will defer an explicit proof that $\alpha$ lands in the $\Gamma$-invariants until Chapter \ref{ch:E^0(BGL_d(Fq))} although it is implicit in the workings below. Using the generalised character isomorphism we have a diagram
$$
\xymatrix{ L\tensor_{E^0} E^0(B(\fq{p})^\times) \ar[r]^-\sim \ar[d] & \Map(\Rep(\thstar,(\fq{p})^\times),L)
\ar@{=}[r] \ar[d]
& \Map(\Phi(p^{v+1}),L) \ar[d]\\
L\tensor_{E^0} E^0(B(\Fq)^\times) \ar[r]^-\sim & \Map(\Rep(\thstar,(\Fq)^\times),L) \ar@{=}[r] &
\Map(\Phi(p^{v}),L)}
$$
But
\begin{eqnarray*}
\Map(\Phi(p^{v+1}),L)& = & \textstyle  \Map(\Phi(p^v)\coprod\Phi(p^{v+1})\setminus \Phi(p^v) ,L)\\
&= &\Map(\Phi(p^v),L)\times\Map(\Phi(p^{v+1})\setminus \Phi(p^v),L).
\end{eqnarray*}
Thus, using the isomorphism of Lemma \ref{CRT applied rationally to p^s+1 series} we have
\begin{eqnarray*}
L\tensor_{E^0} \frac{E^0\lpow x\rpow}{\langle p\rangle([p^{v}](x))} & = & \ker(L\tensor_{E^0}
E^0(B(\fq{p})^\times)\to L\tensor_{E^0} E^0(B(\Fq)^\times))\\
&\simeq& \ker(\Map(\Phi(p^{v+1}),L)\to \Map(\Phi(p^{v}),L))\\
&=& \Map(\Phi(p^{v+1})\setminus \Phi(p^v),L).
\end{eqnarray*}
Finally, Proposition \ref{(Phi^p/Sigma_p)^Gamma)} gives
$$
\xymatrix{ L\tensor_{E^0} E^0(BGL_p(\Fq)) \ar[r] \ar[d]_-\wr &  \left(L\tensor_{E^0}
E^0(BT_p)^{\Sigma_p}\right)\times
\left(L\tensor_{E^0} (E^0\lpow x\rpow/\langle p\rangle ([p^v](x)))^\Gamma\right) \ar[d]^-\wr\\
\Map((\Phi^p/\Sigma_p)^\Gamma,L)  \ar[r]^-\sim & \Map((\Phi(p^v)^p/\Sigma_p),L)\times
\Map((\Phi(p^{v+1})\setminus\Phi(p^v))/\Gamma,L)}
$$
and the result follows.\end{pf}

In Chapter \ref{ch:E^0(BGL_d(Fq))} we will look at the maps of Proposition \ref{L tensor E^0(BFq^p) splits}
again and see that there is a close relationship between them even before applying the functor $L\tensor_{E^0}
-$. However, we will no longer obtain a splitting; the relationship is more subtle.


\chapter{The ring $\mathbf{E^0(BGL_d(K))}$}\label{ch:E^0(BGL_d(Fq))}

In this chapter we examine the structure of the ring $E^0(BGL_d(K))$ for finite fields $K$ of characteristic
different from $p$. We first consider the low dimensional case where $d<p$ before moving on to study the more
complex situations that arise when $d=p$. Our main results will assume that $v_p(|K|^\times)>0$ so that we have
a good understanding of the Sylow $p$-subgroups of $GL_d(K)$ from Section \ref{sec:Syl_p(GL_d(K))}.

\section{Tanabe's calculations}

Let $l$ be a prime different to $p$ and $q=l^r$ for some $r$. Let $\Gamma=\Gamma_q$ be the subgroup of
$\Gal(\Flbar/\mathbb{F}_q)$ generated by $\Frob_q=\Frob^r$, where $\Frob$ is the Frobenius homomorphism
$a\mapsto a^l$ of Section \ref{sec:finite fields}. Then $\Gamma$ acts on $GL_d(\Flbar)$ component-wise and
$GL_d(\mathbb{F}_q)=GL_d(\Flbar)^{\Gamma}$. We use the following lemma.

\begin{lem} Let $h$ be any cohomology theory. Let $K$ act on a group $G$ and $H$ be a subgroup of $G^K$. Then the restriction map $h^*(BG)\to
h^*(BH)$ factors through $h^*(BG)_K$, where $h^*(BG)_K$ denotes the coinvariants of the induced action, that is
the quotient of $h^*(BG)$ by the ideal $\{a-k^*a\mid a\in h^*(BG), k\in K\}$.\end{lem}
\begin{pf} Let $k\in K$. Then the commuting diagram
$$\begin{array}{ccc}
\xymatrix{ G  &  \ar[l] \ar[ld] H\\
G \ar[u]^{k}} & \begin{array}{c}\phantom{induces}\\\phantom{induces}\\\text{ induces }\end{array} &
\xymatrix{ h^*(BG) \ar[r] \ar[d]_{k^*} &  h^*(BH)\\
h^*(BG) \ar[ur]}
\end{array}
$$
showing that $a-k^*a$ is in the kernel of the restriction map $h^*(BG)\to h^*(BH)$.\end{pf}

 In his paper \cite{Tanabe},
Tanabe proved the following result.

\begin{prop} $K(n)^*(BGL_d(\mathbb{F}_q))$ is concentrated in even degrees and restriction induces an
isomorphism $K(n)^*(BGL_d(\Flbar))_{\Gamma}\simeq K(n)^*(BGL_d(\mathbb{F}_q))$.\end{prop}

Thus we have $K(n)^*(BGL_d(\mathbb{F}_q))\simeq
K(n)^*\lpow\sigma_1,\ldots,\sigma_d\rpow/(\sigma_1-(\Frob_q)^*\sigma_1,\ldots,\sigma_d-(\Frob_q)^*\sigma_d)$. In
fact, Tanabe's result lifts to $E$-theory; that is, we will prove the following.

\begin{prop}\label{E^*(BGL_d(Flbar))_Gamma_q=E^0(BGL_d(Fl))} $E^*(BGL_d(\mathbb{F}_q))$ is concentrated in even degrees and is free and finitely generated over
$E^*$. Further, the restriction map induces an isomorphism $$E^*(BGL_d(\Flbar))_{\Gamma}\simeq
E^*(BGL_d(\mathbb{F}_q)).$$\end{prop}

To prove this we need a number of intermediate results. Note first that, by Proposition \ref{E^*(X) in even
degrees} and the fact that $K(n)^*(BGL_d(\mathbb{F}_q))$ is concentrated in even degrees,
$E^*(BGL_d(\mathbb{F}_q))$ is free over $E^*$ and concentrated in even degrees and, further,
$K^*(BGL_d(\mathbb{F}_q))=K^*\tensor_{E^*} E^*(BGL_d(\mathbb{F}_q))$. Hence it suffices to prove the result in
degree $0$.

\begin{lem} $K^0\tensor_{E^0}
E^0(BGL_d(\Flbar))_{\Gamma}= K^0(BGL_d(\Flbar))_{\Gamma}$.\end{lem}
\begin{pf} Since $E^0(BGL_d(\Flbar))=E^0\lpow \sigma_1,\ldots,\sigma_d\rpow=\mathbb{Z}_p\lpow u_1,\ldots,u_{n-1},\sigma_1,\ldots,\sigma_d\rpow$ (by Corollary \ref{E^0(BGL_d(Flbar)=E^0[[c_1...c_d]])}) it
follows that $K^0(BGL_d(\Flbar))=K^0\tensor_{E^0} E^0(BGL_d(\Flbar))$ and we have a diagram
$$
\xymatrix{ E^0(BGL_d(\Flbar)) \ar@{->>}[d] \ar[r] & E^0(BGL_d(\Flbar))_{\Gamma} \ar@{->>}[d] \ar[r] & E^0(BGL_d(\mathbb{F}_q)) \ar@{->>}[d]\\
K^0(BGL_d(\Flbar)) \ar[r] \ar@{->>}[rd] & K^0\tensor_{E^0} E^0(BGL_d(\Flbar))_{\Gamma} \ar[r] \ar[d] & K^0(BGL_d(\mathbb{F}_q))\\
& K^0(BGL_d(\Flbar))_{\Gamma} \ar[ru]_-{\sim}}
$$
where the second row is reduction modulo $(p,u_1,\ldots,u_{n-1})$. Chasing the diagram we see that the map
$E^0(BGL_d(\Flbar))_{\Gamma}\to K^0(BGL_d(\Flbar))_{\Gamma}$ is surjective and, further, it is just reduction
modulo the ideal $(p,u_1,\ldots,u_{n-1})$; that is, $K^0\tensor_{E^0}
E^0(BGL_d(\Flbar))_{\Gamma}=K^0(BGL_d(\Flbar))_{\Gamma}$.
\end{pf}

\begin{lem} The sequence $p,u_1,\ldots,u_{n-1}$ is regular on $E^0(BGL_d(\Flbar))_{\Gamma}$.\end{lem}
\begin{pf} Write $\sigma_i^*$ for $(\Frob_q)^*\sigma_i$. Then, by \cite[Proposition 4.6]{Tanabe},
$p,~\sigma_1-\sigma_1^*,\ldots,{\sigma_d-\sigma_d^*}$ is a regular sequence on
$$\widehat{K(n})^0(BGL_d(\Flbar))=\mathbb{Z}_p\lpow \sigma_1,\ldots,\sigma_d\rpow=\frac{E^0\lpow
\sigma_1\ldots,\sigma_d\rpow}{(u_1,\ldots,u_{n-1})}=\frac{E^0(BGL_d(\Flbar))}{(u_1,\ldots,u_{n-1})}.$$ It
follows that $u_1,\ldots,u_{n-1},p,\sigma_1-\sigma_1^*,\ldots,\sigma_d-\sigma_d^*$ is regular on
$E^0(BGL_d(\Flbar))$. But, since $E^0(BGL_d(\Flbar))$ is a Noetherian ring, the corollary to Theorem 16.3 in
\cite{Matsumura} tells us that $\sigma_1-\sigma_1^*,\ldots,\sigma_d-\sigma_d^*, p,~ u_1,\ldots,u_{n-1}$ is also
regular. Hence $p,u_1,\ldots,u_{n-1}$ is regular on
$E^0(BGL_d(\Flbar))/(\sigma_1-\sigma_1^*,\ldots,\sigma_d-\sigma_d^*)=E^0(BGL_d(\Flbar))_{\Gamma}$.\end{pf}

\begin{lem} $E^0(BGL_d(\Flbar))_{\Gamma}$ is finitely generated over $E^0$.\end{lem}
\begin{pf} Since $E^0(BGL_d(\Flbar))=\mathbb{Z}_p\lpow u_1,\ldots,u_{n-1},\sigma_1,\ldots,\sigma_d\rpow$ has Krull dimension $d+n$ it follows that
$E^0(BGL_d(\Flbar))_{\Gamma}/(p,u_1,\ldots,u_{n-1})=K^0(BGL_d(\Flbar))_{\Gamma}$ has Krull dimension $0$. Thus,
using Lemma \ref{dimension zero local K-algebras are finite dimensional}, we see that
$K^0(BGL_d(\Flbar))_{\Gamma}$ is finite-dimensional over $\mathbb{F}_p$. Thus, an application of Lemma \ref{M/mM
finitely generated implies M finitely generated} shows that $E^0(BGL_d(\Flbar))_{\Gamma}$ is finitely generated
over $E^0$.\end{pf}

\begin{proof}[Proof of Proposition \ref{E^*(BGL_d(Flbar))_Gamma_q=E^0(BGL_d(Fl))}] Since $p,u_1,\ldots,u_{n-1}$
is regular on the finitely generated $E^0$-module $E^0(BGL_d(\Flbar))_{\Gamma}$ an application of Lemma \ref{If
m is regular on M then M is free over R} shows that $E^0(BGL_d(\Flbar))_{\Gamma}$ is free over $E^0$. Thus the
map $E^0(BGL_d(\Flbar))_{\Gamma}\to E^0(BGL_d(\Fl))$ is a map of finitely generated free $E^0$-modules which
becomes an isomorphism modulo $(p,u_1,\ldots,u_{n-1})$. Hence, from Proposition \ref{M/IM=N/IN implies M=N}, the
map is surjective. But a surjective map of free modules of the same rank must be an isomorphism, and we are
done.\end{proof}

\section{The restriction map $E^0(BGL_d(K))\to E^0(BT_d)$}\label{sec:restriction map}

Recall that $T_d\simeq (K^\times)^d$ denotes the maximal torus of $GL_d(K)$ and that there is a restriction map
$E^0(BGL_d(K))\to E^0(BT_d)^{\Sigma_d}$. This map plays a large part in our later calculations and we show now
that, for all $d$ and $K$, it is surjective. Put $v=v_p(|K|^\times)$. Then using our complex orientation we have
an identification $$E^0(BT_d)\simeq E^0\lpow x_1,\ldots,x_d\rpow/([p^v](x_1),\ldots,[p^v](x_d)).$$ Note that
$T_d$ naturally sits inside $\overline{T}_d\simeq(\overline{K}^\times)^d\subseteq GL_d(\overline{K})$. We need a
couple of definitions.

\begin{defn}\label{def:Permutation Module} Let $M$ be a finitely generated free $R$-module and $G$ a (finite) group acting on $M$. Then we call $M$ a \emph{permutation module for $G$} if there is a basis for $M$ over $R$ such that the action of $G$ permutes the basis.\end{defn}

\begin{lem} Let $M$ and $G$ be as above and let $S$ be a basis for $M$ closed under the action of $G$. For
$s\in S$ write $$s^G=\sum_{s'\in \orb_G(s)} s'.$$ Then the set $\{s^G\mid s\in S\}$ is a basis for
$M^G$.\end{lem}

\begin{pf} Let $m=\sum_{s\in S} m_s s\in M^G$. Then, for $g\in G$, we get
$g.\sum_{s\in S} m_s s = \sum_{s\in S} m_s s$ so that $\sum_{s\in S} m_s (g.s) = \sum_{s\in S} m_s s$. Thus,
using the fact that $G$ permutes $S$, we have $m_{g.s}=m_s$ for all $s$. Hence $m_{s'}=m_s$ for all $s'\in
\orb_G(s)$. It follows that $m$ is an $R$-linear sum of the $s^G$'s. That these are linearly independent follows
easily from the fact that $S$ was a basis.\end{pf}

The basis introduced above for $M^G$ is known as the {\em basis of orbit sums}, for obvious reasons.

\begin{lem} Let $\Sigma_d$ act on $T_d$ and $\overline{T}_d$ by permuting the coordinates. Then restriction induces a surjective map $E^0(B\overline{T}_d)^{\Sigma_d}\to E^0(BT_d)^{\Sigma_d}$.\end{lem}
\begin{pf} Firstly, since the restriction map $E^0(B\overline{T}_d)\to E^0(BT_d)$ is $\Sigma_d$-equivariant
it induces a map of the $\Sigma_d$-invariants. We have identifications $E^0(B\overline{T}_d)\simeq E^0\lpow
x_1,\ldots,x_d\rpow$ and $E^0(BT_d)\simeq E^0\lpow x_1,\ldots,x_d\rpow/([p^v](x_i))$ with restriction being just
the obvious quotient map. Using Corollary \ref{Basis for R[[x]]/[p](x)}, $E^0\lpow
x_1,\ldots,x_d\rpow/([p^v](x_i))$ has basis $S=\{x_1^{\alpha_1}\ldots x_d^{\alpha_d}\mid 0\leq
\alpha_k<p^{nv}\}$ over $E^0$ and is a permutation module for $\Sigma_d$. Thus we can take the basis of orbit
sums for the ${\Sigma_d}$-invariants. It is easy to see that any such basis element can be lifted to a
$\Sigma_d$-invariant element of $E^0(B\overline{T}_d)$ under our map and we are done.\end{pf}

\begin{prop}\label{E^0(BGL_d(K))to E^0(BT_d)^Sigma_d is surjective} The restriction map $E^0(BGL_d(K))\to E^0(BT_d)^{\Sigma_d}$ is surjective.\end{prop}
\begin{pf} Using Corollary \ref{E^0(BGL_d(Flbar)=E^0[[c_1...c_d]])}, a system of inclusions induces the diagram
$$
\xymatrix{E^0(B\overline{T}_d)^{\Sigma_d} \ar@{->>}[d] & E^0(BGL_d(\overline{K})) \ar_{\sim}[l]
\ar[d]\\
E^0(T_d)^{\Sigma_d} & E^0(BGL_d(K)) \ar[l]}
$$
showing that the composition $E^0(BGL_d(\overline{K}))\to E^0(BT_d)^{\Sigma_d}$ is surjective and hence that
$E^0(BGL_d(K))\to E^0(BT_d)^{\Sigma_d}$ is also surjective.
\end{pf}

\section{Low dimensions}

Here we deal with the case where $d<p$. As usual, we let $q=l^r$ be a power of a prime different to $p$ and let
$T_d$ denote the maximal torus of $GL_d(\Fq)$. Our main theorem in this case is as follows.

\begin{theorem}\label{d<p} If $d<p$ and $v_p(q-1)=v>0$ then the restriction map $E^0(BGL_d(\Fq))\to E^0(BT_d)$ induces an isomorphism
$E^0(BGL_d(\Fq))\simeq E^0(BT_d)^{\Sigma_d}$.\end{theorem}
\begin{pf} By Proposition \ref{Syl_p(G)}, $T_d$ is a Sylow $p$-subgroup of $GL_d(\Fq)$ and it follows that the
map $E^0(BGL_d(\Fq))\to~E^0(BT_d)^{\Sigma_d}$ is injective. Further, the image is the whole of
$E^0(BT_d)^{\Sigma_d}$ by Proposition \ref{E^0(BGL_d(K))to E^0(BT_d)^Sigma_d is surjective}.\end{pf}

\begin{cor} With the hypotheses of Theorem \ref{d<p} we have a presentation
$$E^0(BGL_d(\Fq))\simeq
\left(\frac{E^0\lpow x_1,\ldots,x_d\rpow}{([p^v](x_1),\ldots,[p^v](x_d))}\right)^{\Sigma_d}$$ where
$\res_{T_d}^{GL_d(K)}(x_i)=\euler_l(T_d \overset{\pi_i}{\longrightarrow} \Fq^\times\rightarrowtail\Flbar)$ and,
as $E^0$-modules,
$$E^0(BGL_d(\Fq))\simeq E^0\{\sigma_1^{\alpha_1}\ldots \sigma_d^{\alpha_d}\mid 0\leq\alpha_1+\ldots+\alpha_d<p^{nv}\}$$ where
$\sigma_i$ is the $i^\text{th}$ elementary symmetric function in $x_1,\ldots,x_n$.\end{cor}
\begin{pf} The first statement follows from the corresponding presentation of $E^0(BT_d)^{\Sigma_d}$. The $E^0$-basis is an application of Proposition \ref{Basis for Sigma_d-invariants}.\end{pf}

Before we move on to the higher dimensional cases we note that, by the work of Strickland in
\cite{StricklandK(n)duality}, both of $E^0(BGL_d(\Fq))$ and $K^0(BGL_d(\Fq))$ have duality over their respective
coefficient rings in the sense of Section \ref{sec:duality algebras}. In fact we can verify the latter claim
directly.

\begin{lem} Let $N\in\mathbb{N}$ and write $A=\mathbb{F}_p[[x_1,\ldots,x_d]]/(x_1^N,\ldots,x_d^N)$. Then, taking
the standard basis for $A$, the map $A\to
\mathbb{F}_p$ given by $\sum_\mathbf{\alpha}k_\mathbf{\alpha} \mathbf{x}^\mathbf{\alpha}\longmapsto
k_{N-1,\ldots,N-1}$ is a Frobenius form on $A$.\end{lem}
\begin{pf} Note first that $$\soc(A)=\ann_A(\mathfrak{m})=\ann_A(x_1,\ldots,x_d)=((x_1\ldots x_d)^{N-1})=\mathbb{F}_p.(x_1\ldots x_d)^{N-1}.$$ As in the proof of Proposition
\ref{gorenstein}, it follows that $A$ has duality over $\mathbb{F}_p$ and that the given map is indeed a
Frobenius form on $A$.\end{pf}

\begin{cor} For any $N\in\mathbb{N}$, the $\mathbb{F}_p$-algebra
$\left(\mathbb{F}_p[[x_1,\ldots,x_d]]/(x_1^N,\ldots,x_d^N)\right)^{\Sigma_d}$ has duality over
$\mathbb{F}_p$.\end{cor}
\begin{pf} Apply Proposition \ref{A^G has duality if A does} with $A$ being the $\mathbb{F}_p$-algebra $\mathbb{F}_p[[x_1,\ldots,x_d]]/(x_1^N,\ldots,x_d^N)$ and $G=\Sigma_d$, noting that $A$ has $\mathbb{F}_p$-basis
$\{x_1^{\alpha_1}\ldots x_d^{\alpha_d}\mid 0\leq\alpha_1,\ldots,\alpha_d<N\}$ and that the map $A\to
\mathbb{F}_p$ given by $\sum_\mathbf{\alpha}k_\mathbf{\alpha} \mathbf{x}^\mathbf{\alpha}\longmapsto
k_{N-1,\ldots,N-1}$ is a Frobenius form on $A$.
\end{pf}
\begin{cor} With the hypotheses of Theorem \ref{d<p} the algebra $K^0(BGL_d(K))$ has duality over
$K^0=\mathbb{F}_p$.\end{cor}
\begin{pf} By Proposition \ref{K^*(BG)=K^* tensor E^*(BG)} we have $K^0(BGL_d(K))=K^0\tensor_{E^0} E^0(BGL_d(K))$. Since $[p^v](x)$ is a unit multiple of $x^{p^{nv}}$ modulo $(p,u_1,\ldots,u_n)$ we
find that
$$K^0(BGL_d(K))=\mathbb{F}_p\lpow x_1,\ldots,x_d\rpow/(x_1^{p^{nv}},\ldots,x_d^{p^{nv}}).$$ Applying the
preceding proposition then gives the result.\end{pf}

\begin{rem} The case where $v_p(q-1)=0$ seems to be quite complicated. However, by Proposition \ref{Sylow p-subgroup for d<p and
v_p(K^x)=0}, we have identified a Sylow $p$-subgroup $P$ of $GL_d(K)$ and it follows that the restriction map
$E^0(BGL_d(K))\to E^0(BP)$ is injective. It remains to identify the image of the map, although this will be
easier said than done.\end{rem}

\section{Dimension $p$}

As usual, let $q=l^r$ be a power of a prime $l$ different to $p$ and suppose that $v_p(q-1)=v>0$. We aim to get
a handle on $E^0(BGL_p(\Fq))$ by consideration of two comparison maps to more easily understandable rings. As
for the low-dimensional case we have the inclusion of the maximal torus $T=T_p$ which induces a map to one of
these rings, although unlike for the lower dimensional case this will no longer be an isomorphism. For the
other, choosing a basis of $\mathbb{F}_{q^p}^{\times}$ over $\Fq$ leads to an embedding
$\mu:\mathbb{F}_{q^p}^{\times}\to GL_p(\Fq)$ and we use this to define our second map. It is the interplay
between these two maps that will give us the structure theorem in this case.

For the remainder of this section, we will let $\beta:E^0(BGL_p(\Fq))\to E^0(BT)^{\Sigma_p}$ denote the
surjective restriction map of Section \ref{sec:restriction map}. As in Section \ref{sec:character theory for
GL_d(Fq)}, writing $x=\euler_l(\fq{p}^\times\rightarrowtail\Flbar^\times)$ we get $E^0(B\fq{p}^\times)\simeq
E^0\lpow x\rpow/[p^{v+1}](x)$ and hence a quotient map $q:E^0(B\fq{p}^\times)\to E^0\lpow x\rpow/\langle
p\rangle([p^v](x))$ and we denote the target ring here by $D$. We then let $\alpha=q\circ
\mu^*:E^0(BGL_p(\Fq))\to D$. Note that $\Gamma=\Gal(\fq{p}/\Fq)$ acts on $E^0(B\fq{p}^\times)$ and hence also on
$D$.

\begin{lem}\label{Maps land in Gamma-invariants} The map $E^0(BGL_p(\Fq))\to E^0(B\fq{p}^\times)$ lands in the $\Gamma$-invariants. Hence the image of $\alpha$ is contained in $D^\Gamma$.\end{lem}
\begin{pf} Let $k\in \fq{p}^\times$. Then the $\Fq$-linear isomorphism $\Frob_q$ fits into the diagram
$$
\xymatrix{ \fq{p} \ar[r]^{\times k} \ar[d]_{\Frob_q}^\wr & \fq{p} \ar[d]^{\Frob_q}_\wr\\
\fq{p} \ar[r]_{\times k^q} & \fq{p}.}
$$
Hence, with our chosen basis for $\fq{p}$ over $\Fq$, there is an $F_q\in GL_p(\Fq)$ corresponding to the
Frobenius map and the diagram
$$
\xymatrix{ \fq{p}^\times~ \ar@{>->}[r] \ar[d]_{\Frob_q} & GL_p(\Fq) \ar[d]^{\conj_{F_q}}\\
\fq{p}^\times~ \ar@{>->}[r] & GL_p(\Fq)}
$$
commutes. On passing to cohomology, $\conj_{F_q}^*$ is just the identity map so that the map $E^0(BGL_p(\Fq))\to
E^0(B\fq{p}^\times)$ lands in the $\Gamma$-invariants. The result follows.\end{pf}

We now state our main theorem in this case.

\begin{theorem}\label{Theorem, d=p} Let $q=l^r$ be a power of a prime $l$ different to $p$ and suppose that $v_p(q-1)=v>0$. Then there are jointly injective $E^0$-algebra epimorphisms
$$
\xymatrix{& E^0(BGL_p(\Fq)) \ar@{->>}[dl]_{\beta} \ar@{->>}[dr]^{\alpha}\\
E^0(BT)^{\Sigma_p} & & D^\Gamma }
$$
which induce a rational isomorphism
$$\mathbb{Q}\tensor E^0(BGL_p(\Fq)) \overset{\sim}{\longrightarrow}
(\mathbb{Q}\tensor E^0(BT)^{\Sigma_p})\times (\mathbb{Q}\tensor D^\Gamma).$$ Further, both $\ker(\alpha)$ and
$\ker(\beta)$ are $E^0$-module summands in $E^0(BGL_p(\Fq))$ and the latter is principal. Finally, we have
$\ker(\alpha)=\ann(\ker(\beta))$ and $\ker(\beta)=\ann(\ker(\alpha))$.\end{theorem}

The rest of this chapter is devoted to proving this result.

\subsection{A cyclic $p$-subgroup of maximal order}\label{sec:definition of A}

Recall that we defined $\alpha:E^0(BGL_p(\Fq))\to D$ as the composite of the restriction under some embedding
$\fq{p}^\times\to GL_p(\Fq)$ and a quotient map. For convenience, we choose the embedding a little more
carefully to allow us easier analysis of the structure.

Recall, from Section \ref{sec:finite fields}, that we have a fixed embedding $\Flbar^\times\to S^1$ and this
gives us an isomorphism of groups $\mathbb{Z}/p^\infty\iso \{a\in\Flbar^\times|a^{p^s}=1 \mbox{ for some
$s$}\}$. Hence, for each $s\geq 1$, we have canonical generators, $a_s$ say, for the cyclic subgroups $C_{p^s}$
compatible in the sense that $a_{s+1}^p=a_s$.

As in Section \ref{sec:abelian subgroups of GL_p(K)}, we let $\gamma=\gamma_p$ denote the standard $p$-cycle
$(1\ldots p)\in \Sigma_p$ and put $$a=\gamma(a_v,1,\ldots,1)\in \Sigma_p\wr \Fq^\times\subseteq GL_p(\Fq).$$ We
then let $A=\langle a\rangle$ be the subgroup of $GL_p(\Fq)$ generated by $a$. Note that $a^p=(a_v,\ldots,a_v)$
so that $a^{p^{v+1}}=1$ and $A$ is cyclic of order $p^{v+1}$.

Now, from Proposition \ref{Abelian p-subgroups of GL_p(K)}, $A$ is a maximal abelian $p$-subgroup of $GL_p(\Fq)$
and any other cyclic subgroup of order $p^{v+1}$ is $GL_p(\Fq)$-conjugate to $A$. In particular, $A$ is
conjugate to the $p$-part of $\fq{p}^\times\simeq C_{p^{v+1}}$. That is, there exists $g\in GL_p(\Fq)$ such that
$gAg^{-1}\subseteq \fq{p}^\times$. But then, by post composing our chosen embedding $\fq{p}^\times\to GL_p(\Fq)$
by $\conj_{g^{-1}}$, we get a new embedding such that $A$ is precisely the $p$-part of $\fq{p}^\times$. Thus we
can assume that our original embedding was such that $A\subseteq \fq{p}^\times$ and hence we can identify
$E^0(B\fq{p}^\times)$ with $E^0(BA)$. Thus the map $\alpha$ can be viewed as the composition $GL_p(\Fq)\to
E^0(BA)\simeq E^0\lpow x\rpow/[p^{v+1}](x)\to D$, where $x=\euler_l(A\rightarrowtail
\fq{p}^\times\rightarrowtail (\Flbar)^\times)$.

\subsection{The ring $D^\Gamma$}

We defined $D=E^0\lpow x\rpow/\langle p\rangle([p^v](x))$ as a quotient ring of $E^0(B\fq{p}^\times)$ in the
previous section and noted that it inherited an action of $\Gamma=\langle \Frob_q\rangle=\Gal(\fq{p}/\Fq)\simeq
C_{p}$. Our first step is to understand how this action works.

\begin{lem} The action of $\Gamma$ on $D$ is given by $\Frob_q.x=[q](x)$.\end{lem}
\begin{pf} Using the embedding $\Flbar^\times\rightarrowtail S^1$ we get a diagram
$$
\xymatrix{ \fq{p}^\times \ar[r] \ar[d]^\wr_{\Frob_q} & S^1 \ar[d] & z \ar@{|->}[d]\\
\fq{p}^\times \ar[r] & S^1 & z^q.}
$$
On passing to cohomology the result follows.\end{pf}

To progress, we record a couple of results concerning norm maps. Let $R\hookrightarrow S$ be an extension of
rings and suppose that $S$ is finitely generated and free over $R$ of rank $n$. The \emph{norm map},
$N_{S/R}:S\to R$, is defined by $N_{S/R}(s)=\det(\mu_s)$ where $\mu_s:S\to S$ is multiplication by $s$.

\begin{lem}\label{Norm map}
Let $R\hookrightarrow S$ be an extension of rings and suppose that $S$ is finitely generated and free over $R$
of rank $n$. Then $N_{S/R}(s)$ is divisible by $s$ for all $s\in S$.
\end{lem}
\begin{pf} After choosing a basis for $S$ over $R$ let
$\chi_s(t)=\text{det}(\mu_s-tI)$ so that $\chi_s(t)=t^n+a_{n-1} t^{n-1}+\ldots + a_0$ with $a_i\in R$ for each
$i$. By the Cayley-Hamilton theorem we have $\mu_s^n+a_{n-1} {\mu_s}^{n-1}+\ldots + a_0=0$ and hence $s^n +
a_{n-1} s^{n-1} + \ldots + a_0=0$. Since $a_0=\chi_s(0)=\text{det}(\mu_s)=N_{S/R}(s)$ we see that
$s|N_{S/R}(s)$, as required.
\end{pf}

Recall that any integral domain $R$ has a field of fractions which we denote by $Q(R)$.

\begin{lem}\label{lin alg over int domains}
Let $R\hookrightarrow S$ be an extension of integral domains and suppose that $S$ is finitely generated and free
over $R$ of rank $n$. Let $0\neq s\in S$ and suppose that $N_{S/R}(s)=0$. Then $s=0$.
\end{lem}
\begin{pf} Choose a basis for $S$ over $R$ and let $\mu_s:S\to
S$ be multiplication by $s$, so that we can view $\mu_s$ as an $n\times n$ matrix over $R$. Then
$\text{det}(\mu_s)=N_{S/R}(s)=0$ so, by linear algebra over $Q(R)$, there is a non-trivial
$u=(u_1,\ldots,u_n)\in Q(R)^n$ with $\mu_s.u=0$. Write $u_i=v_i/w_i$ for some $v_i,w_i\in R$ with $w_i\neq 0$.
Put $w=\prod_i w_i\neq 0$ and $\tilde{u}=wu$. Then $0\neq\tilde{u}\in R^n$ and $\mu_s.\tilde{u}=w\mu_s.u=0$.
That is, there is a non-zero element $\tilde{u}\in S$ such that $s\tilde{u}=\mu_s(\tilde{u})=0$. Hence, since
$S$ is an integral domain, $s=0$.\end{pf}

\begin{prop}\label{field of fractions}
Let $R\hookrightarrow S$ be an extension of integral domains and suppose that $S$ is finitely generated and free
over $R$. Then $Q(S)=Q(R)\tensor_{R} S$.
\end{prop}

\begin{pf} Let $\phi:Q(R) \tensor_{R} S\to Q(S)$ be the ring map $\phi(k\tensor s)=ks$. We show that $\phi$ is an
isomorphism. For injectivity, take $a=\sum_{i}(b_i\tensor c_i/d_i)\in \text{ker}(\phi)$ with $b_i\in S$ and
$c_i,d_i\in R$ with $d_i\neq 0$. Put $d=\prod_i d_i$ (necessarily non-zero) and $\bar{d_i}=\prod_{j\neq i} d_j$.
Then we have $a=\sum_i(b_i\tensor c_i\bar{d_i}/d)=\sum_i(b_i c_i \bar{d_i} \tensor 1/d)$ and so, since
$\phi(a)=0$, we get $(\sum_i b_i c_i \bar{d_i})/d=0$ in $Q(S)$. Thus, we see $\sum_i b_i c_i \bar{d_i}=0$ and
$a=(\sum_i b_i c_i \bar{d_i})\tensor 1/d = 0$, as required.

To show that $\phi$ is surjective, take $a/b\in Q(S)$ and let $c=N_{S/R}(b)\in R$. Since $b\neq 0$ we see from
Lemma \ref{lin alg over int domains} that $c\neq 0$ so $1/c$ exists in $Q(R)$. By Lemma \ref{Norm map} we have
$c=b.\bar{b}$ for some $\bar{b}\in S$ and then $a/b=a\bar{b}/c=\phi(1/c\tensor a\bar{b})$. Hence $\phi$ is an
isomorphism, as claimed.
\end{pf}

We now have the tools we need to study the rings $D$ and $D^\Gamma$ and, in particular, their module structures
over $E^0$.

\begin{defn}\label{definition of y} We will let $N=(p^{n(v+1)}-p^{nv})/p$ and define $y =\prod_{k=0}^{p-1}[1+kp^v](x)\in D$. We also let
$g(x)=g_{v+1}(x)/g_v(x)$ be the Weierstrass polynomial of degree $Np$ which is a unit multiple of $\langle p
\rangle ([p^v](x))$ in $E^0\lpow x\rpow$. Note that $D=E^0\lpow x\rpow/g(x)$.\end{defn}

\begin{lem}\label{g is irreducible} With the notation above, $g(x)$ is monic and irreducible over $E^0$.\end{lem}
\begin{pf} By definition $g(x)$ is monic. Note that $g(x)\sim [p^{v+1}](x)/[p^{v}](x)=\langle
p\rangle([p^v](x))$ so that $g(0)\sim p$, say $g(0)=ap$ for some unit $a\in (E^0)^\times$. Then, since
$g(x)=x^{Np}$ modulo $(p,u_1,\ldots,u_{n-1})$, an application of Eisenstein's criterion (\cite[p228]{Matsumura})
shows that $g(x)$ is irreducible over $E^0$.\end{pf}

\begin{cor} $D$ is an integral domain and is free of rank $Np$ over $E^0$.\end{cor}
\begin{pf} The first result is immediate since $D=E^0\lpow x\rpow/g(x)$ and the second is an application of the
Weierstrass preparation theorem.\end{pf}

\begin{lem}\label{D basis} $D$ is free over $E^0$ with basis $S=\{x^i y^j\mid 0\leq i<p,~0\leq j<N\}$.\end{lem}

\begin{pf}
We first show that $S$ generates $D$ over $E^0$. Working modulo $(p,u_1,\ldots,u_{n-1})$ we have
$g_{v+1}(x)=x^{n(v+1)}$ and $g_v(x)=x^{nv}$ so that $g(x)=x^{n(v+1)-nv}=x^{Np}$. Hence
$D/(p,u_1,\ldots,u_{n-1})=\mathbb{F}_p\lpow x \rpow / x^{Np}=\mathbb{F}_p\{1,x,\ldots,x^{Np-1}\}$ as an
$\mathbb{F}_p$-vector space. Now, since $y=\prod_{k=0}^{p-1}[1+kp^v](x)=\prod_{k=0}^{p-1}\left((1+kp^v)x +
O(2)\right)$ we find that, mod $(p,u_1,\ldots,u_{n-1})$, we have $y=x^p + O(p+1)$. It follows easily that $\{x^i
y^j\mid 0\leq i<p, 0\leq j<N\}$ is also a basis for $D/(p,u_1,\ldots,u_{n-1})$. Applying Lemma \ref{M/mM
finitely generated implies M finitely generated} gives us a generating set $S$ for $D$ over $E^0$. But $D$ is
free over $E^0$ of rank $Np=|S|$ so that $S$ is a basis for $D$.\end{pf}

\begin{lem} Let $y'=\prod_{\gamma\in\Gamma} \gamma.x\in D^\Gamma$. Then $y=y'$ and, in particular, $y$ is $\Gamma$-invariant.\end{lem}
\begin{pf} Since $\Gamma=\langle \Frob_q\rangle$ is cyclic of order $p$ we find that $y'=\prod_{k=0}^{p-1}
[q^k](x)$. But, by assumption, $q=1+ap^v$ for some $a$ not divisible by $p$ and it follows that $q$ generates
the subgroup $1+p^v\mathbb{Z}/p^{v+1}\subseteq (\mathbb{Z}/p^{v+1})^\times$. Thus
$y'=\prod_{k=0}^{p-1}[1+kp^v](x)=y$.\end{pf}

We now aim to understand $D^\Gamma$. We already know that $y\in D^\Gamma$.

\begin{lem} $\Gamma$ acts on $Q(D)$ by $\ds \gamma.\frac{a}{b}=\frac{\gamma.a}{\gamma.b}$ and $Q(D)^\Gamma=Q(D^\Gamma)$.\end{lem}
\begin{pf} It is clear that we have an inclusion $Q(D^\Gamma)\hookrightarrow Q(D)$ which lands in the $\Gamma$-invariants. It
remains to show that the map is surjective. Using Proposition \ref{field of fractions} we have
$Q(D)=Q(D^\Gamma)\tensor_{D^\Gamma} D$. Take any $\frac{ac}{b}=\frac{a}{b}\tensor c\in
Q(D)^\Gamma\tensor_{D^\Gamma} D$, where $a,b\in D^\Gamma$ and $c\in D$. Then, for all $\gamma\in\Gamma$ we have
$\gamma.\frac{ac}{b}=\frac{ac}{b}$ so that $\frac{a\gamma.c}{b}=\frac{ac}{b}$ whereby $\gamma.c=c$. Thus $c\in
D^\Gamma$ and $\frac{ac}{b}\in Q(D^\Gamma)$.\end{pf}

\begin{lem}\label{Q(D)=Q(D^Gamma)^p} $Q(D)$ has dimension $p$ over $Q(D^\Gamma)$.\end{lem}
\begin{pf} We can embed $\Gamma\rightarrowtail\Gal(Q(D)/Q(E^0))$ and so, by Galois theory, $Q(D)$ has dimension $|\Gamma|=p$ over $Q(D)^\Gamma=Q(D^\Gamma)$.\end{pf}

\begin{prop} $D$ is free over $D^\Gamma$ with basis $\{1,x,\ldots,x^{p-1}\}$.\end{prop}
\begin{pf} By Lemma \ref{D basis} we know that $S=\{x^iy^j\mid 0\leq i<p,~0\leq j<N\}$
generates $D$ over $E^0$. Since $y\in D^\Gamma$ we find that the map $D^\Gamma\{1,x,\ldots,x^{p-1}\}\to D$ is
surjective. Applying the functor $Q(D^\Gamma)\tensor_{D^\Gamma} -$ we get a map
$Q(D^\Gamma)\{1,x,\ldots,x^{p-1}\}\to Q(D)$ which, by right exactness, is also surjective. But, by Lemma
\ref{Q(D)=Q(D^Gamma)^p}, both source and target are $Q(D^\Gamma)$-vector spaces of dimension $p$ and therefore
the map is an isomorphism. Consider the diagram
$$
\xymatrix{ D^\Gamma\{1,x,\ldots,x^{p-1}\} \ar@{->>}[rr] \ar[d] & & D \ar[d]\\
Q(D^\Gamma)\{1,x,\ldots,x^{p-1}\} \ar[rr]^-{\sim} & & Q(D).}
$$
Since $D^\Gamma$ is an integral domain, the left-hand map is injective. Thus we see that the map
$D^\Gamma\{1,x,\ldots,x^{p-1}\}\to D$ is also injective and hence is an isomorphism.\end{pf}

\begin{prop}\label{structure of D^Lambda} $D^\Gamma$ is free over $E^0$ with basis $\{1,y,\ldots,y^{N-1}\}$ and there is a polynomial
$h(t)\in E^0[t]$ of degree $N$ such that $D^\Gamma=E^0\lpow y\rpow/h(y)$.\end{prop}
\begin{pf} For the first claim, note that we have a map $E^0\{1,y,\ldots,y^{N-1}\}\to D^\Gamma$ and the diagram
$$
\xymatrix{ (E^0\{1,y,\ldots,y^{N-1}\})\{1,x,\ldots,x^{p-1}\} \ar[rr] \ar[d]_-\wr & &
D^\Gamma\{1,x,\ldots,x^{p-1}\} \ar[d]^-\wr\\
E^0\{S\} \ar[rr]^-\sim & & D}
$$
shows that it must be an isomorphism. For the second claim, write $y^N=\sum_i a_iy^i$ for unique $a_i\in E^0$;
then the polynomial $h(y)=y^N-\sum_i a_iy^i$ does the job.\end{pf}

\begin{lem}\label{D^Gamma/I one dim over F_p} $D^\Gamma/(y,u_1,\ldots,u_{n-1})$ is a one dimensional vector space over $\mathbb{F}_p$.\end{lem}
\begin{pf} The isomorphism of
$D^\Gamma$-modules $D^\Gamma\{1,x,\ldots,x^{p-1}\}\iso D$ induces a module isomorphism
$D^\Gamma/(y,u_1,\ldots,u_{n-1})\{1,x,\ldots,x^{p-1}\}\iso D/(y,u_1,\ldots,u_{n-1}).$ But $q$ is coprime to $p$
and so $y=\prod_{k=0}^{p-1} [q^k](x)$ is a unit multiple of $\prod_{k=0}^{p-1} x=x^p$ in $D$. Hence
$D/(y,u_1,\ldots,u_{n-1})=D/(x^p,u_1,\ldots,u_{n-1})$. It then follows that
$$x^{p-1} D/(y,u_1,\ldots,u_{n-1})=x^{p-1} D/(x^p,u_1,\ldots,u_{n-1})\simeq D^\Gamma/(y,u_1,\ldots,u_{n-1})$$
so that $D^\Gamma/(y,u_1,\ldots,u_{n-1})\{1,x,\ldots,x^{p-2}\}\simeq D/(x^{p-1},u_1,\ldots,u_{n-1})$. We can
continue in this way to see that $D^\Gamma\simeq D/(x,u_1,\ldots,u_{n-1})$. But $g(x)=p$ mod $x$ whereby
\[D^\Gamma/(y,u_1,\ldots,u_{n-1})\simeq D/(x,u_1,\ldots,u_{n-1})=E^0\lpow x\rpow/(x,p,u_1,\ldots,u_{n-1})\simeq \mathbb{F}_p.\qedhere\]\end{pf}

\begin{prop} $D^\Gamma$ is a regular local ring and
$y,u_1,\ldots,u_{n-1}$ a system of parameters.\end{prop}
\begin{pf} As $D^\Gamma$ is finitely generated over $E^0$ it follows that the Krull dimension of $D^\Gamma$ is
equal to the Krull dimension of $E^0$, namely $n$. Thus, since the maximal ideal of $D^\Gamma$ is generated by
$y,u_1,\ldots,u_{n-1}$ it follows that $\text{embdim}(D^\Gamma)\leq n$ and hence that $D^\Gamma$ is a regular
local ring.\end{pf}

\begin{prop}\label{alpha is surjective} The map $\alpha:E^0(BGL_p(\Fq))\to D^\Gamma$ sends $\sigma_i$ to the $i^\text{th}$ elementary symmetric function in $x,[q](x),\ldots,[q^{p-1}](x)$. Further, $\alpha$ is surjective.\end{prop}
\begin{pf} Recall that an application of Lemma \ref{Galois extensions} gives us the isomorphism of $\Flbar$-vector spaces $\psi:\Flbar\tensor_{\Fq} \fq{p}
\iso \Flbar^p$, $a\tensor b\mapsto (ab,ab^q,\ldots,ab^{q^{p-1}})$. Thus there is $g=g_{\psi}\in GL_p(\Flbar)$
such that
$$
\xymatrix{\fq{p}^\times \ar[d] \ar[r] & GL_p(\Flbar) \ar[d]^{\conj_g}\\
(\Flbar^\times)^p \ar[r] & GL_p(\Flbar)}
$$
commutes, where the left hand map is $a \mapsto (a,a^q,\ldots,a^{q^{p-1}})$. Passing to cohomology gives a
commutative diagram
$$
\xymatrix{ E^0(BGL_p(\Flbar)) \ar[r] \ar[rd] & E^0(B\fq{p}^\times) \ar[r] & D^\Gamma\\
& E^0(B(\Flbar^\times)^p) \ar[u] & E^0\lpow x_1,\ldots,x_p\rpow \ar[l]_-\sim}
$$
where the map $E^0(B(\Flbar^\times)^p)\to E^0(B\fq{p}^\times)$ sends $x_i$ to $[q^{i-1}](x)$. Remembering that
$$E^0(BGL_p(\Flbar))\simeq E^0(BT)^{\Sigma_p}=E^0\lpow \sigma_1,\ldots,\sigma_p\rpow$$ we see that
$\alpha(\sigma_i)$ is the $i^\text{th}$ elementary symmetric function in $x,[q](x),\ldots,[q^{p-1}](x)$; in
particular, $\alpha(\sigma_p)=y$. Since $y$ generates $D^\Gamma$, we are done.\end{pf}

\subsection{An important subgroup}

Recall that, from Section \ref{sec:Syl_p(GL_d(K))}, there is a subgroup $N=N_p=\Sigma_p\wr \Fq^\times$ of
$GL_p(\Fq)$. The cohomology of wreath products is fairly well understood (see \cite{Nakaoka}); we will use
methods similar to those of \cite{Hunton} to calculate $E^0(BN)$. The structure of such rings is usually
expressed in terms of standard euler classes, but we get analogous results with our $l$-euler classes. We begin
with a couple of standard results from group cohomology.

\begin{lem}[Shapiro's lemma]\label{Shapiro's lemma} Let $R$ be a ring and $H$ be a subgroup of $G$. Then
$H^*(G;R[G/H])=H^*(H;R)$, where $R$ is a trivial $H$-module.\end{lem}
\begin{pf} This is \cite[6.3.2]{Weibel} with $A=R$.\end{pf}

\begin{lem} Let $G$ be a finite group and $M$ a $G$-module. Then $|G|.H^i(G;M)=0$ for all $i>0$. In particular, if multiplication by $|G|$ is an isomorphism $M\to M$ then $H^i(G;M)=0$.\end{lem}
\begin{pf} This is \cite[Theorem 6.5.8]{Weibel}.\end{pf}

\begin{lem}\label{Sigma_p orbits on S^p} Let $S$ be a set and let $\Sigma_p$ act on $S^p$ in the usual way. If $s\in S^p$ then either
$s\in (S^p)^{\Sigma_p}=\Delta(S)$ or $p$ divides $|\orb_{\Sigma_p}(s)|$.\end{lem}
\begin{pf} Let $H$ be a subgroup of $\Sigma_p$ and suppose that $s$ is fixed by $H$. If $p$ divides the order of $H$ then $H$
contains a cyclic subgroup of order $p$ necessarily generated by a $p$-cycle, $\sigma$ say. But $\sigma$ acts
transitively on $S^p$ so that we must have $s=(s_1,\ldots,s_1)\in \Delta(S)$. Thus, for $s\in S^p$, either
$s\in\Delta(S)$ or $\stab_{\Sigma_p}(s)=H$ for some subgroup $H\subseteq \Sigma_p$ with order not divisible by
$p$. The result follows.\end{pf}

\begin{lem} We have $H^*(\Sigma_p;\mathbb{Z}_p)\simeq \mathbb{Z}_p\lpow z\rpow/pz$ for a class $z$ in degree $2p-2$.\end{lem}
\begin{pf} As in Proposition \ref{Cohomology of Sigma_p}, $H^*(B\Sigma_p;\mathbb{Z}_p)\simeq
H^*(BC_p;\mathbb{Z}_p)^{\Aut(C_p)}$, where $H^*(BC_p;\mathbb{Z}_p)\simeq \mathbb{Z}_p\lpow x\rpow/px$ with $x$
in degree 2 and $\Aut(C_p)\simeq(\mathbb{Z}/p)^\times$ acting by $k.x=kx$. Thus we have
$H^*(B\Sigma_p;\mathbb{Z}_p)=(\mathbb{Z}_p\lpow x\rpow/px)^{\Aut(C_p)}=\mathbb{Z}_p\lpow z\rpow/pz$ where
$z=\prod_{k\in(\mathbb{Z}/p)^\times} (kx)=-x^{p-1}$. Using general theory (see, for example, \cite{Weibel}) we
identify $H^*(B\Sigma_p;\mathbb{Z}_p)$ with $H^*(\Sigma_p;\mathbb{Z}_p)$ and the result follows.\end{pf}

\begin{lem}\label{H^*(BSigma_p;E^*(Fq times ^p))} There are sets $B'$ and $T$ such that $$H^*(B\Sigma_p;E^0(B(\Fq^\times)^p))  \simeq  E^0\{B'\}^{\Sigma_p}\oplus (E^0\lpow z
\rpow/pz)\{T\},$$ where $z$ is in degree $2p-2$.\end{lem}
\begin{pf}As before, we can identify the ring
$H^*(B\Sigma_p;E^0(B(\Fq^\times)^p))$ with the group cohomology $H^*(\Sigma_p;E^0(B(\Fq^\times)^p))$. We let
$B=\{x_1^{\alpha_1}\ldots x_p^{\alpha_p}\mid 0\leq \alpha_1,\ldots,\alpha_p<p^{nv}\}\subseteq
E^0(B(\Fq^\times)^p)$ and note that $E^0(B(\Fq^\times)^p)=E^0\{B\}$ so that we can apply Lemma \ref{Sigma_p
orbits on S^p} to get $B=T\cup B'$, where $T=B^{\Sigma_p}=\{x_1^{\alpha}\ldots x_p^{\alpha}\mid 0\leq
\alpha<p^{nv}\}$ and $B'$ is a disjoint union of orbits of size divisible by $p$.

Now, each orbit in $B'$ is of the form $\Sigma_p/H$ for some $H$ with order not divisible by $p$, and
$H^*(\Sigma_p;E^0[\Sigma_p/H])\simeq H^*(H;E^0)$ by Lemma \ref{Shapiro's lemma}. But, since $|H|$ is invertible
in $E^0$, we find that $H^i(\Sigma_p;E^0[\Sigma_p/H])=0$ for all $i>0$. Hence $H^i(\Sigma_p;E^0\{B'\})=0$ for
all $i>0$. Further, $H^0(\Sigma_p;E^0\{B'\})=E^0\{B'\}^{\Sigma_p}$.

We now see that $H^*(\Sigma_p;E^0\{B\})=H^*(\Sigma_p;E^0\{B'\})\oplus H^*(\Sigma_p;E^0\{T\})$, where the latter
summand is just $H^*(\Sigma_p;\mathbb{Z}_p)\tensor_{\mathbb{Z}_p} E^0\{T\}$; that is,
$$H^*(B\Sigma_p;E^0(B(\Fq^\times)^p)  \simeq  E^0\{B'\}^{\Sigma_p}\oplus (E^0\lpow z \rpow/pz)\{T\},$$
where $z$ is in degree $2p-2$.\end{pf}

We are now in a position to establish the cohomology of $N=\Sigma_p\wr \Fq^\times$. Note first that the
projection $N\to \Sigma_p$ makes $E^*(BN)$ into a $E^*(B\Sigma_p)$-module. Recall, from Section
\ref{sec:Sigma_p}, that the embedding $C_p\rightarrowtail S^1$ gives a class $w\in E^0(BC_p)$ such that
$E^0(BC_p)=E^0\lpow w\rpow/[p](w)$ and $E^0(B\Sigma_p)\simeq E^0(BC_p)^{\Aut(C_p)}\simeq E^0\lpow d\rpow/df(d)$
where $d=-w^{p-1}$ and $f(d)=\langle p\rangle (w)$.

\begin{lem} In $E^0(B\Sigma_p)=E^0\lpow d\rpow/df(d)$ we have $pd\in (d^2)$.\end{lem}
\begin{pf} Since $df(d)=0$ and $f(0)=p$, we find $pd=f(0)d = f(d)d-f(0)d=(f(d)-f(0))d$ which is divisible by $d^2$, as
claimed.\end{pf}

\begin{prop}\label{E^0(BSigma_p wr GL_1) as a module} Let $J=\left\{\alpha\in \mathbb{N}^p\mid
0\leq\alpha_1\leq\ldots\leq\alpha_p<p^{nv}\text{ and }\alpha_1<\alpha_p\right\}.$ Then there is an isomorphism
of $E^0(B\Sigma_p)$-modules $$E^0(BN)\simeq E^0(B\Sigma_p)\{c_p^i\mid 0\leq i<p^{nv}\}\oplus (E^0(B\Sigma_p)/d)
\{b_\alpha\mid \alpha\in J\}$$ where $c_p = \euler_l(N\hookrightarrow GL_p(\Fq)\hookrightarrow GL_p(\Flbar))$
and $\ds b_{\bf{\alpha}} = \tr_{(\Fq^\times)^p}^{~N~}(x_1^{\alpha_1}\ldots x_p^{\alpha_p})$.
\end{prop}
\begin{pf} First note that in $E^0(BN)$,
$$d.b_{\alpha}=d.\transfer_{(\Fq^\times)^p}^N(x_1^{\alpha_1}\ldots
x_p^{\alpha_p})=\transfer_{(\Fq^\times)^p}^N(\res_{(\Fq^\times)^p}^N(d)x_1^{\alpha_1}\ldots x_p^{\alpha_p})=0$$
since the composite $(\Fq^\times)^p\to N\to \Sigma_p$ is zero. Thus, writing $$R=E^0(B\Sigma_p)\{c_p^i\mid 0\leq
i<p^{nv}\}\oplus (E^0(B\Sigma_p)/d) \{b_\alpha\mid \alpha\in J\}$$ there is an evident (well defined) map of
$E^0(B\Sigma_p)$-modules $\phi:R\to E^0(BN)$. We introduce a filtration on $R$ by
\begin{eqnarray*}
F^0 R & = & R,\\
F^1 R = \ldots = F^{2p-2}R & = & Rd,\\
F^{2p-1} R = \ldots = F^{4p-4}R & = & Rd^2,~\ldots\quad .
\end{eqnarray*}
Then
\begin{eqnarray*}
\frac{F^{k(2p-2)} R}{F^{(k+1)(2p-2)} R} & \simeq & (Rd^k/Rd^{k+1}) = \left\{\begin{array}{ll}
          E^0\{c_p^i\mid 0\leq i<p^{nv}\}\oplus
          E^0\{b_{\alpha}\mid \alpha\in J\} & \mbox{for $k=0$}\\
          (E^0/p)\{c_p^i\mid 0\leq i<p^{nv}\}d^k & \mbox{for $k>0$}.
          \end{array}\right.
\end{eqnarray*}

We use the spectral sequence $H^*(B\Sigma_p;E^*(B(\Fq^\times)^p))\Rightarrow E^*(BN)$ associated to the
fibration $B(\Fq^\times)^p\to BN\to B\Sigma_p$. By Lemma \ref{H^*(BSigma_p;E^*(Fq times ^p))},
$H^*(\Sigma_p;E^*(B(\Fq^\times)^p))$ is in even degrees and the spectral sequence collapses. Thus, we have a
filtration $E^0(BN)= F_0 \geqslant F_{2p-2} \geqslant F_{4p-4} \geqslant \ldots$ with
$F_{k(2p-2)}/F_{(k+1)(2p-2)} = E^{k(2p-2),0}_{\infty}$; that is,
$$F_{k(2p-2)}/F_{(k+1)(2p-2)} = \left\{\begin{array}{ll} E^0\{T\} \oplus E^0\{B'\}^{\Sigma_p} & \text{for $k=0$}\\
(E^0/p)\{T\} z^k & \text{for $k>0$.}\end{array}\right.$$ It remains to show that $\phi$ induces an isomorphism
$$\frac{F^{k(2p-2)} R}{F^{(k+1)(2p-2)} R} \simeq \frac{F_{k(2p-2)}}{F_{(k+1)(2p-2)}}$$
for all $k$ since, if so, an application of the five-lemma (\cite[p129]{Hatcher}) gives an isomorphism
$$\frac{F^{0} R}{F^{k(2p-2)} R} \simeq \frac{F_{0}}{F_{k(2p-2)}}$$ and, on taking limits, an isomorphism $R=F^0
R\simeq F_0 = E^0(BN)$.

Firstly note that the map $E^0(BN)\to F_0/F_{2p-2}= E^0\{T\} \oplus E^0\{B'\}^{\Sigma_p} \subseteq
E^0(B(\Fq^\times)^p)$ is just the restriction map (see, for example, \cite{McCleary}). An application of the
double coset formula shows that $\res_{(\Fq^\times)^p}^{~N}(b_\alpha)=\sum_{\sigma\in\Sigma_p}
\sigma.x_1^{\alpha_1}\ldots x_p^{\alpha_p}$ so that the images of $b_\alpha$ ($\alpha\in J$) are precisely the
basis elements of $E^0\{B'\}^{\Sigma_p}$, where the latter is given the basis of orbit sums. Further, the
restriction of the $l$-euler class $c_p$ to $E^0(B(\Fq^\times)^p)$ is just $x_1\ldots x_p$, so that the classes
$c_p^i$ ($0\leq i< p^{nv}$) give precisely the set $T$. Similarly, the class $d^kc_p^j\in F_{k(2p-2)}$ lifts
$(x_1\ldots x_p)^j z^k\in F_{k(2p-2)}/F_{(k+1)(2p-2)}$. The result follows.\end{pf}

\subsection{Summary of notation}

We summarise the notation for the generators that will be used to study $E^0(BGL_p(\Fq))$.
\begin{itemize}
\item We have $v=v_p(q-1)$ which is assumed to be positive.
\item We let $T$ denote the maximal torus of $GL_p(\Fq)$ and have $$E^0(BT)=E^0\lpow
x_1\ldots,x_p\rpow/([p^v](x_1),\ldots,[p^v](x_p)).$$ We write $\beta$ for the surjective restriction map
$E^0(BGL_p(\Fq))\to E^0(BT)^{\Sigma_p}$.
\item We let $\Delta$ denote the diagonal subgroup of $T$ and write $E^0(B\Delta)=E^0\lpow
x\rpow/[p^v](x)$, where $x=\euler_l(\Delta\simeq \Fq^\times\rightarrowtail (\Flbar)^\times)$. We also let
$\Delta_p$ denote the $p$-part of $\Delta$.
\item We let $A$ be the maximal cyclic $p$-subgroup of $GL_p(\Fq)$ of Section \ref{sec:definition of A} and, remembering that we can view $A$ as a subgroup of $\Fq^\times$, we have $E^0(BA)\simeq E^0\lpow
x\rpow/[p^{v+1}](x)$ where $x=\euler_l(A\rightarrowtail \fq{p}^\times\rightarrowtail (\Flbar)^\times)$. Note
that this notation is consistent with that of $E^0(B\Delta)=E^0(B\Delta_p)$ since $\Delta_p$ sits inside $A$ in
a compatible way. We write $D$ for the quotient ring $E^0(BA)/\langle p\rangle([p^v](x))$ and $\alpha$ for the
surjective map $E^0(BGL_p(\Fq))\to D^\Gamma$.
\item We let $C_p$ denote the cyclic subgroup of order $p$. Using the embedding $(\Flbar)^\times\rightarrowtail
S^1$ we get an embedding $C_p\rightarrowtail(\Flbar)^\times$ and we write $E^0(BC_p)=E^0\lpow w\rpow/[p](w)$,
where $w=\euler_l(C_p\rightarrowtail (\Flbar)^\times)$.
\item We write $E^0(B\Sigma_p)=E^0\lpow d\rpow/df(d)$ as in Proposition \ref{E^0(BSigma_p) in terms of d}, where
the restriction of $d$ to $E^0(BC_p)$ is $-w^{p-1}$ and $f(d)$ restricts to $\langle p\rangle (w)$.
\item As in the previous section, we write $N=\Sigma_p\wr \Fq^\times$ and have
$$E^0(BN)\simeq E^0(B\Sigma_p)\{c_p^i\mid 0\leq i<p^{nv}\}\oplus E^0 \{b_\alpha\mid \alpha\in J\}$$
where $c_p = \euler_l(N\hookrightarrow GL_p(\Fq)\hookrightarrow GL_p(\Flbar))$, $\ds b_{\bf{\alpha}} =
\tr_{(\Fq^\times)^p}^{~N~}(x_1^{\alpha_1}\ldots x_p^{\alpha_p})$ and $$J=\left\{\alpha\in \mathbb{N}^p\mid
0\leq\alpha_1\leq\ldots\leq\alpha_p<p^{nv}\text{ and }\alpha_1<\alpha_p\right\}.$$
\end{itemize}

Recall the results of Section \ref{sec:abelian subgroups of GL_p(K)}, namely that every $p$-subgroup of $N$ is
subconjugate to one of $A$, $\Sigma_p\times \Delta$ or $T$. We begin with the diagram of inclusions below.
$$
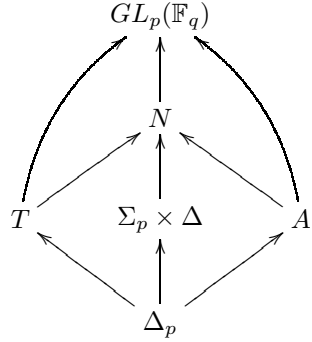
\xymatrix{ & GL_p(\Fq)\\
 & N \ar[u]\\
T \ar[ur] \ar@/^1pc/[uur] & \Sigma_p\times \Delta \ar[u] & A \ar[ul] \ar@/_1pc/[uul]\\
& \Delta_p \ar[ul] \ar[u] \ar[ur]}
$$

Applying the functor $E^0(B-)$ we get the following diagram.
$$
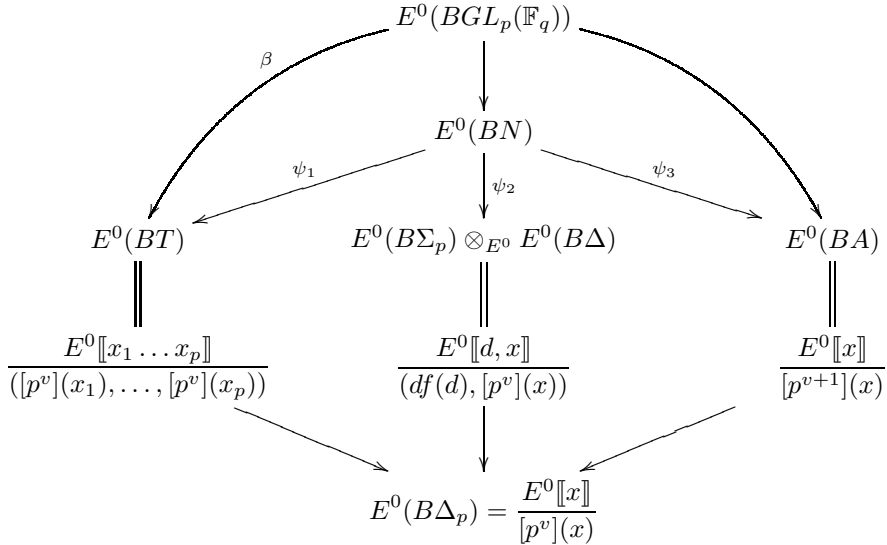
\xymatrix{~\phantom{\frac{E^{x^y}}{z^{p^q}}} & E^0(BGL_p(\Fq)) \ar[d] \ar@/^2pc/[ddr] \ar@/_2pc/[ddl]_{\beta}\\
& E^0(BN) \ar[dr]^{\psi_3} \ar[d]^{\psi_2} \ar[dl]_{\psi_1}\\
E^0(BT) \ar@{=}[d] & E^0(B\Sigma_p)\tensor_{E^0} E^0(B\Delta) \ar@{=}[d] & \qquad \quad E^0(BA)\qquad \quad \ar@{=}[d]\\
\ds \frac{E^0\lpow x_1\ldots x_p\rpow}{([p^v](x_1),\ldots,[p^v](x_p))} \ar[dr] & \ds \frac{E^0\lpow
d,x\rpow}{(df(d),[p^v](x))} \ar[d] & \ds \qquad \quad\frac{E^0\lpow x\rpow}{[p^{v+1}](x)}\qquad \quad \ar[dl]\\
& \ds E^0(B\Delta_p)=\frac{E^0\lpow x\rpow}{[p^v](x)}}
$$

\begin{prop}\label{psi 1,2 and 3 are jointly injective} The maps $\psi_1$, $\psi_2$ and $\psi_3$ shown above are jointly injective.\end{prop}
\begin{pf} This follows from Section \ref{sec:abelian subgroups of GL_p(K)} and Corollary \ref{Maximal abelian p-subgroups give jointly injective maps}: any abelian $p$-subgroup of $N$ is subconjugate to one of $T$, $A$ and $\Sigma_p\times \Delta$.\end{pf}

Hence we should be able to get a good understanding of the multiplicative structure of $E^0(BN)$ by studying the
maps $\psi_1,~\psi_2$ and $\psi_3$. As well as looking at the usual generators of $E^0(BN)$, we will be
particularly interested in a class $t$ defined below.

\begin{prop}\label{definition of t} There is a unique class $t\in E^0(BGL_p(\Flbar))$ which restricts to $\prod_i[p^v](x_i)$ in
$E^0(B(\Flbar^\times)^p)\simeq E^0\lpow x_1,\ldots,x_p\rpow$. Further, writing $t$ for the restriction of this
class to $E^0(BGL_p(\Fq))$, we find that $t\in\ker(\beta)$.\end{prop}
\begin{pf} By the results of Tanabe we have $E^0(BGL_p(\Flbar))=E^0(B(\Flbar^\times)^p)^{\Sigma_p}$, where
$\Sigma_p$ acts by permuting the $x_i$, and it is clear that $\prod_i[p^v](x_i)$ is $\Sigma_p$-invariant. For
the second claim, the commutative diagram
$$
\xymatrix{ E^0(BGL_p(\Flbar)) \ar[r] \ar[d] & E^0(B(\Flbar^\times)^p) \ar[d] \ar@{=}[r] & E^0\lpow x_1,\ldots,x_p\rpow \ar@{->>}[d]\\
E^0(BGL_p(\Fq)) \ar[r] & E^0(BT) \ar@{=}[r] & \ds \frac{E^0\lpow
x_1,\ldots,x_p\rpow}{([p^v](x_1),\ldots,[p^v](x_p))}}
$$ shows that $\beta(t)=\prod_i[p^v](x_i)=0$ in $E^0(BT)$.\end{pf}

For the remainder of this chapter we will write $I$ for the ideal of $E^0(BGL_p(\Fq))$ generated by $t$. Then,
by the previous result, we have $I\subseteq \ker(\beta)$. Later we will find that $I=\ker(\beta)$; that is,
$\ker(\beta)$ is a principal ideal generated by $t$.

The following proposition shows the images of the key elements of $E^0(BN)$ under each of the maps $\psi_1$,
$\psi_2$ and $\psi_3$ and will be proved in the subsequent section.

\begin{prop}\label{Generators of E^0(BN)} With the notation above, the following table shows the images of the classes $d$, $c_p$, $b_\alpha$ ($\alpha\in J$) and $t$ of $E^0(BN)$ under the maps $\psi_1,~\psi_2$ and
$\psi_3$.
$$
{\footnotesize
\begin{array}{|c||c|c|c|c|}
\hline ~ &&&&\\
\text{Map $/$ target} & d & c_p & b_\alpha~(\alpha\in J) & t\\
~ &&&&\\
\hline\hline ~&&&&\\
~\psi_1~/~E^0(BT)~ & 0 & x_1\ldots x_p & \ds \sum_{\sigma\in\Sigma_p}^{\phantom{\sigma}}\sigma.(x_1^{\alpha_1}\ldots x_p^{\alpha_p}) & 0\\
~&&&&\\
\hline ~ &&&&\\
~\psi_2~/~E^0(B(\Sigma_p\times\Delta))~ & d & ~ \ds \prod_{k=0}^{p-1} (x+_F[k](w))~ & \ds (p-1)!f(d)x^{\sum \alpha_i}& 0\\
~ &&&&\\
\hline ~ &&&&\\
\psi_3~/~E^0(BA) & ~ -[p^v](x)^{p-1}~ &\ds  \prod_{k=0}^{p-1}[1+kp^v](x) & \ds ~(p-1)!x^{\sum \alpha_i}\langle
p\rangle
([p^v](x))~& ~ \ds [p^v](x)^p ~\\
~ &&&&\\
\hline \end{array}}
$$\end{prop}

\subsection{Tracking the key classes in $E^0(BN)$}

\begin{prop} In $E^0(BT)=E^0\lpow x_1,\ldots,x_p\rpow/([p^v](x_i))$ we have
$$\psi_1(d)=0,~ \psi_1(c_p)=x_1\ldots x_p\text{ and }
\psi_1(b_\alpha)=\sum_{\sigma\in\Sigma_p}\sigma.(x_1^{\alpha_1}\ldots x_p^{\alpha_p}).$$\end{prop}
\begin{pf} For the first statement, the composition $T\rightarrowtail N \twoheadrightarrow
\Sigma_p$ is trivial and hence, since $d$ is the restriction of a class in $E^0(B\Sigma_p)$, it follows that
$\psi_1(d)=0$. For the second statement we have $\psi_1(c_p)=\psi_1(\euler_l(N \hookrightarrow
GL_p(\Flbar))=\euler_l(T\hookrightarrow GL_p(\Flbar))=x_1\ldots x_p$. For the final statement, using the
double-coset formula (Lemma \ref{Transfers}) we get
\begin{eqnarray*}
\psi_1(b_\alpha) &=& \res_T^N\transfer_T^N(x_1^{\alpha_1}\ldots x_p^{\alpha_p})\\
&=& \sum_{\sigma\in T \backslash N / T} \transfer_{T\cap T}^T \res_{T\cap T}^T (\conj_\sigma^*(x_1^{\alpha_1}\ldots x_p^{\alpha_p}))\\
&=& \sum_{\sigma\in\Sigma_p} \sigma.(x_1^{\alpha_1}\ldots x_p^{\alpha_p}),
\end{eqnarray*}
as claimed.\end{pf}

\begin{prop}\label{psi_2 of generators} In $E^0(B(\Sigma_p\times\Delta))=E^0\lpow d,x\rpow/(df(d),[p^v](x))$ we have
$$\psi_2(d)=d,~\psi_2(c_p)=\prod_{k=0}^{p-1} (x+_F[k](w))\text{ and }\psi_2(b_\alpha)=(p-1)!f(d)x^{\sum \alpha_i}.$$
\end{prop}

\begin{pf} The first claim is clear. For the second, let $V=\Flbar^p$ correspond to the representation
$C_p\times\Delta\hookrightarrow GL_p(\Flbar)$ and $L=\Flbar$ to $C_p\rightarrowtail (\Flbar)^\times$. Then,
since $C_p$ acts on $V$ by permuting the coordinates, it follows that $V$ is the regular representation of
$C_p$. Thus, standard representation theory (see, for example, \cite[2.4]{Serre}) gives
$V\simeq\bigoplus_{k=0}^{p-1} L^k$ as $\Flbar$-representations of $C_p$. Now, let $M=\Flbar$ be the standard
one-dimensional representation of $\Delta\simeq \Fq^\times\subseteq (\Flbar)^\times$. Then, since $\Delta$ acts
diagonally on $V$, we have $V\simeq \bigoplus_{k=0}^{p-1} M\tensor_{\Flbar} L^k$ as $\Flbar$-representations of
$C_p\times\Delta$. Hence we have
\begin{eqnarray*}
\euler_l(C_p\times\Delta\hookrightarrow GL_p(\Flbar)) & = & \euler_l\left(\bigoplus_{k=0}^{p-1} M\tensor_{\Flbar} L^k\right)\\
&=& \prod_{k=0}^{p-1} \euler_l(M\tensor_{\Flbar} L^k))\\
&=& \prod_{k=0}^{p-1} (\euler_l(M)+_F \euler_l(L^k))\\
 &=& \prod_{k=0}^{p-1} (x+_F [k](w)).
\end{eqnarray*}
Thus, $\psi_2(c_p)=\euler_l(\Sigma_p\times\Delta\hookrightarrow GL_p(\Flbar))$ is just the pullback of this
class under the injective map $E^0(B\Sigma_p\times \Delta)\hookrightarrow E^0(BC_p\times\Delta)$, as required.

For the final statement, note first that $(\Sigma_p\times\Delta) \backslash N / T=1$ and that
$(\Sigma_p\times\Delta) \cap T=\Delta$. Writing $S$ for $\Sigma_p\times\Delta$ we can apply the properties from
Lemma \ref{Transfers} to get
\begin{eqnarray*}
\psi_2(b_\alpha) &=& \res_{S}^N\transfer_T^N(x_1^{\alpha_1}\ldots x_p^{\alpha_p})\\
&=& \sum_{\sigma\in S\backslash N / T} \transfer_{S\cap T}^{~S} \res_{S\cap T}^{~T} (\conj_\sigma^*(x_1^{\alpha_1}\ldots x_p^{\alpha_p}))\\
&=& \transfer_{\Delta}^{S}\res_\Delta^T(x_1^{\alpha_1}\ldots x_p^{\alpha_p})\\
&=& \transfer_{\Delta}^{S}(x^{\sum \alpha_i})\\
&=&\transfer_{1\times\Delta}^{\Sigma_p\times\Delta}(1\tensor x^{\sum \alpha_i})\\
&=& \transfer_1^{\Sigma_p}(1)\tensor \transfer_\Delta^\Delta (x^{\sum \alpha_i})\\
&=& \transfer_1^{\Sigma_p}(1)\tensor x^{\sum \alpha_i}.
\end{eqnarray*}
But, from Lemma \ref{Transfer_1^C_p(1)}, $\transfer_1^{\Sigma_p}(1)=(p-1)!f(d)$ and the result follows.\end{pf}

\begin{prop}\label{psi_3(d)} In $E^0(BA)=E^0\lpow x\rpow/[p^{v+1}](x)$ we have $\psi_3(d)=-[p^v](x)^{p-1}.$\end{prop}

\begin{pf} Write $\chi$ and $\xi$ for the embeddings $A\rightarrowtail\Flbar^\times$ and $C_p\rightarrowtail
\Flbar^\times$ respectively. Note that under the map $\pi:N\to \Sigma_p$ we have $\pi(A)=C_p$ and we get a
commutative diagram
$$
\xymatrix{ & \Flbar^\times\\
 A \ar@{->>}[rr]^\pi \ar[d] \ar[ur]^{\xi\circ\pi} & & C_p \ar[ul]_\xi \ar[d]\\
N \ar@{->>}[rr]^\pi & & \Sigma_p}
$$
Now, writing $a=\gamma(a_v,1,\ldots,1)$ for the usual generator of $A$ we have
$(\xi\circ\pi)(a)=\xi(\gamma)=a_1=a_{v+1}^{p^v}=\chi(a)^{p^v}$ so that $\xi\circ\pi=\chi^{p^v}$. Thus, applying
$E^0(B-)$ to the above diagram gives
$$
\xymatrix{ & E^0(B\Flbar^\times) \ar[dl]_{(\chi^{p^v})^*} \ar[dr]^{\xi^*}\\
E^0(BA) & & E^0(BC_p) \ar[ll]_{\pi^*}\\
E^0(BN) \ar[u]^{\psi_3} & & E^0(B\Sigma_p) \ar[u] \ar[ll]_{\pi^*}}
$$
Hence $\pi^*(w)=\pi^*(\euler_l(\xi))=\euler_l(\xi\circ\pi)=\euler_l(\chi^{p^v})=[p^v](\euler_l(\chi))=[p^v](x)$.
Thus, since $d$ is a class in $E^0(B\Sigma_p)$ and $d\mapsto -w^{p-1}$ in $E^0(BC_p)$ we find that
\[\psi_3(d)=\pi^*(-w^{p-1})=-\pi^*(w)^{p-1}=-[p^v](x)^{p-1}.\qedhere\]\end{pf}

\begin{prop} In $E^0(BA)=E^0\lpow x\rpow/[p^{v+1}](x)$ we have $\psi_3(c_p)=\prod_{k=0}^{p-1}[1+kp^v](x).$\end{prop}
\begin{pf} We have $\psi_3(c_p)=\euler_l(A\hookrightarrow GL_p(\Flbar))=\prod_{k=0}^{p-1}[q^k](x)=\prod_{k=0}^{p-1}[1+kp^v](x)$ using Proposition \ref{alpha is surjective}.\end{pf}

\begin{rem} We can give an explicit diagonalisation of the representation $A\hookrightarrow
GL_p(\Flbar)$. For $k=0,\ldots,p-1$ let
$\nu_k=(1,a_{v+1}^{1+kp^v},a_{v+1}^{2(1+kp^v)},\ldots,a_{v+1}^{(p-1)(1+kp^v)})^T$. Then it turns out that
$\nu_k$ is an eigenvector for $a=\gamma(a_v,1,\ldots,1)$ with eigenvalue $a_{v+1}^{1+kp^v}$ since
$$a.\nu_k = \gamma.\left(\begin{array}{c} a_v\\ a_{v+1}^{1+kp^v}\\
\vdots \\ a_{v+1}^{(p-1)(1+kp^v)}\end{array}\right)=\left(\begin{array}{c} a_{v+1}^{1+kp^v}\\
\vdots \\ a_{v+1}^{(p-1)(1+kp^v)} \\ a_{v+1}^{p}\end{array}\right)=a_{v+1}^{1+kp^v}\nu_k$$ where we use the fact
that $a_v=a_{v+1}^{p}=a_{v+1}^{p(1+kp^v)}$. Thus $\gamma(a_v,1,\ldots,1)$ has distinct eigenvalues
$a_{v+1}^{1+kp^v}$ for $0\leq k<p$ and hence is conjugate to the matrix
$$\left(\begin{array}{cccc}
1 & 0 & \ldots & 0\\
0 & a_{v+1}^{1+p^v} & \ldots & 0\\
\vdots & \vdots & \ddots & \vdots\\
0 & 0 & \ldots & a_{v+1}^{1+(p-1)p^v}\end{array}\right),$$ the form of which is to be expected by earlier
diagonalisation results.\end{rem}

\begin{prop} In $E^0(BA)=E^0\lpow x\rpow/[p^{v+1}](x)$ we get $\psi_3(b_\alpha)=(p-1)!x^{\sum \alpha_i}\langle p\rangle ([p^v](x)).$\end{prop}

\begin{pf} Using the double coset formula we have
$$\psi_3(b_\alpha) = \res_A^{N}\transfer_T^{N}(x_1^{\alpha_1}\ldots x_p^{\alpha_p}) = \sum_{A\backslash N / T}
\transfer_{A\cap T}^{~A} \res_{A\cap T}^{~T}(\conj_\sigma^*(x_1^{\alpha_1}\ldots x_p^{\alpha_p})).$$ But it is
not hard to see that $A\backslash N / T \simeq \Sigma_p/\langle \gamma\rangle$ and $A\cap T=\Delta_p$ and it
follows that $\res_{A\cap T}^{~T}(\conj_\sigma^*(x_1^{\alpha_1}\ldots
x_p^{\alpha_p}))=\res_{\Delta_p}^T(\conj_\sigma^*(x_1^{\alpha_1}\ldots x_p^{\alpha_p}))=x^{\sum \alpha_i}$ for
all $\sigma\in A\backslash N / T$. Thus we have
\begin{eqnarray*}
\psi_3(b_\alpha) &=& \sum_{\Sigma_p/\langle \gamma\rangle} \transfer_{\Delta_p}^A
\res_{\Delta_p}^T(x_1^{\alpha_1}\ldots
x_p^{\alpha_p})\\
&=& |\Sigma_p/\langle \gamma\rangle|.\transfer_{\Delta_p}^A(x^{\sum \alpha_i})\\
&=& (p-1)!\transfer_{\Delta_p}^A(x^{\sum \alpha_i}).
\end{eqnarray*}
Now, the map $\res_{\Delta_p}^A:E^0(BA)\to E^0(B\Delta_p)$ sends $x$ to $x$ and hence
$$\transfer_{\Delta_p}^A(x)=\transfer_{\Delta_p}^A(\res_{\Delta_p}^A(x))=x.\transfer_{\Delta_p}^A(1)$$ by Frobenius reciprocity. Writing
$q:A\to A/\Delta_p$ for the quotient map and using the properties of the transfer map (Lemma \ref{Transfers}) we
get $\transfer_{\Delta_p}^A(1)=q^*\transfer_1^{A/\Delta_p}(1)$. But $A/\Delta_p$ is naturally identified with
$C_p$ (since $A\overset{\pi}{\to} C_p$ has kernel $\Delta_p$) and thus we have
$$\transfer_{\Delta_p}^A(1)=\pi^*\transfer_1^{C_p}(1)=\pi^*(\langle p\rangle(w)).$$ But, as in the proof of Proposition
\ref{psi_3(d)}, $\pi^*(w)=[p^v](x)$. Hence we have \[\psi_3(b_\alpha)=(p-1)!\transfer_{\Delta_p}^A(x^{\sum
\alpha_i})=(p-1)!x^{\sum\alpha_i}\transfer_{\Delta_p}^A(1)=(p-1)!x^{\sum\alpha_i}\langle
p\rangle([p^v](x)).\qedhere\]\end{pf}

\begin{prop} Let $t\in E^0(BN)$ be as in Proposition \ref{definition of t}. Then $\psi_1(t)=0$, $\psi_2(t)=0$
and $\psi_3(t)=[p^v](x)^p$.\end{prop}
\begin{pf} Since $t\in\ker(\beta)$ it follows straight away that its image in $E^0(BT)$ is zero, so that
$\psi_1(t)=0$. As in Proposition \ref{psi_2 of generators}, the representation $C_p\times\Delta\rightarrowtail
GL_p(\Flbar)$ is isomorphic to a diagonal representation $C_p\times\Delta\to ((\Flbar)^\times)^p$ sending
$\gamma$ to $(1,a_1,\ldots,a_1^{p-1})$ and mapping $\Delta$ along the diagonal. Thus in cohomology we find that
the map $E^0(B((\Flbar)^\times)^p)=E^0\lpow x_1,\ldots,x_p\rpow\to E^0(B(C_p\times\Delta))=E^0\lpow
w,x\rpow/([p](w),[p^v](x))$ sends $x_k$ to $x+_F [k-1](w)$. We now see that $\psi_2(t)=\prod_{k=0}^{p-1}
[p^v](x+_F [k](w))=\prod_{k=0}^{p-1}[p^v](x)+_F [p^v]([k](w))=0$. By the methods of Proposition \ref{alpha is
surjective}, we have a diagram
$$
\xymatrix{ E^0(BGL_p(\Flbar)) \ar[r]^-{\res} \ar[rd] & E^0(B\fq{p}^\times)=E^0(BA)\\
& E^0(B(\Flbar^\times)^p) \ar[u]}
$$
where the vertical map sends $x_i$ to $[q^{i-1}](x)$. Hence we see that
$\psi_3(t)=\prod_{i=1}^{p}[p^v]([q^{i-1}](x))=\prod_{i=1}^{p}[p^v]([1+kp^v](x))=\prod_{i=1}^{p}[p^v](x)=[p^v](x)^p$
since $[p^{v+1}](x)=0$.
\end{pf}

\subsection{A system of maps}\label{sec:system of maps}

In this section we will look at the earlier system of maps,
$$
\xymatrix{~\phantom{\frac{E^{x^y}}{z^{p^q}}} & E^0(BGL_p(\Fq)) \ar[d] \ar@/^2pc/[ddr] \ar@/^4pc/[dddrr]^{\alpha} \ar@/_2pc/[ddl]_\beta\\
& E^0(BN) \ar[dr]^{\psi_3} \ar[d]^{\psi_2} \ar[dl]_{\psi_1}\\
E^0(BT) \ar[dr] & E^0(B(\Sigma_p\times\Delta)) \ar[d] & E^0(BA) \ar[dl]^{q_1} \ar[dr]_{q_2}\\
& \ds E^0(B\Delta_p) & & ~~D~~}
$$

\begin{prop}\label{Q tensor A = Q tensor Delta times Q tensor D} The map $\mathbb{Q}\tensor E^0(BA)\longrightarrow \mathbb{Q}\tensor E^0(B\Delta_p)\times
\mathbb{Q}\tensor D$ induced by $q_1$ and $q_2$ is an isomorphism.\end{prop}
\begin{pf} Note first that $q_1(x)=x$ so that $q_1$ and $q_2$ are really nothing more than reduction modulo $([p](x))$ and
$(\langle p\rangle([p](x)))$ respectively. Now, $\langle p\rangle([p](x))=p$ mod $([p](x))$ and it follows that
$p\in ([p](x))+(\langle p\rangle([p](x))$. Thus Corollary \ref{Q tensor CRT} applies and the result is
immediate.\end{pf}

\begin{cor}\label{E^0(BA) split} The maps $q_1$ and $q_2$ defined above are jointly injective.\end{cor}

\begin{pf} We have a commutative diagram
$$
\xymatrix{ E^0(BA) \ar[r]^-{(q_1,q_2)} \ar[d] &  \ds E^0(B\Delta_p) \times D \ar[d]\\
\mathbb{Q}\tensor E^0(BA) \ar[r]^-\sim & \ds \mathbb{Q}\tensor E^0(B\Delta_p) \times \mathbb{Q}\tensor D.}
$$
Since $E^0(BA)$ is free over $E^0$ it follows that the map $E^0(BA)\to \mathbb{Q}\tensor E^0(BA)$ is injective
and a diagram chase shows that the top map is also injective.\end{pf}

\begin{cor}\label{jointly injective maps 2} The maps $\psi_1$, $\psi_2$ and $q_2\circ \psi_3$ defined above are jointly
injective.\end{cor}

\begin{pf} Suppose $z\in E^0(BN)$ with $z$ maps to $0$ in each of $E^0(BT),~E^0(B(\Sigma_p\times\Delta))$ and
$D$. Then $z$ maps to $0$ in $E^0(B\Delta_p)$ under $\res_{\Delta_p}^T$ so that, by Corollary \ref{E^0(BA)
split}, $z$ maps to $0$ in $E^0(BA)$. Thus we have $\psi_1(z),\psi_2(z)$ and $\psi_3(z)$ all zero, whereby $z=0$
by Proposition \ref{psi 1,2 and 3 are jointly injective}.\end{pf}

Recall, from Proposition \ref{Maps land in Gamma-invariants}, that the map $E^0(BGL_p(\Fq))\to E^0(BA)$ lands in
the $\Gamma$-invariants, where $\Gamma\simeq \mathbb{Z}/p$ acts on $E^0(BA)$ by $k.x=[q^k](x)$. It follows that
$\psi_3$ lands in $E^0(BA)^\Gamma$ and $q_2\circ \psi_3$ lands in $D^\Gamma$.

\begin{rem} We can in fact see the above $\Gamma$-invariance explicitly by looking at the images of the generators of $E^0(BN)$ under the map $q_2\circ\psi_3$. Note first
that, in $E^0(BA)$, $k.[p](x)=[p]([q^k](x))=[pq^k](x)=[p](x)$ since $q^k=1$ mod $p^v$. Hence $[p](x)$ and also
$(q_2\circ\psi_3)(d)=-[p](x)^{p-1}$ are $\Gamma$-invariant. We have $\psi_3(b_\alpha)=(p-1)!x^{\sum
\alpha_i}\langle p\rangle ([p^v](x))$ which is zero mod $\langle p\rangle([p^v](x))$, and so clearly maps into
$D^\Gamma$. Finally, $\psi_3(c_p)=\prod_{j=0}^{p-1}[q^j](x)$ so that
$k.\psi_3(c_p)=\prod_{j=0}^{p-1}[q^{j+k}](x)=\psi_3(c_p)$ since $q^p=1$ mod $p^{v+1}$. Hence the generators of
$E^0(BN)$ all land in $D^\Gamma$, as expected.\end{rem}

Note next that there is an injective restriction map $E^0(B(\Sigma_p\times\Delta))\to E^0(B(C_p\times\Delta)),$
where the latter ring has a presentation $E^0\lpow w,x\rpow/([p](w),[p^v](x))$ with
$w=\euler_l(C_p\rightarrowtail\Flbar^\times)$ and $x=\euler_l(\Delta\rightarrowtail\Flbar^\times)$. We will
write $\psi_2'$ for the composite map
$$E^0(BN)\overset{\psi_2}{\longrightarrow}E^0(B(\Sigma_p\times\Delta))\to E^0(B(C_p\times\Delta))\twoheadrightarrow E^0(B(C_p\times\Delta))/\langle p\rangle(w).$$
This allows us the following proposition which will prove useful to us later.

\begin{prop}\label{Jointly injective maps 3} The maps
$$
\xymatrix{ & E^0(BN) \ar[dl]_{\psi_1} \ar[d]^{\psi_2'} \ar[dr]^{q_2\circ\psi_3}\\
E^0(BT) & \ds \frac{E^0(B(C_p\times\Delta))}{\langle p\rangle(w)} & \ds D^\Gamma}
$$
are jointly injective.\end{prop}
\begin{pf} Since $[p](w)=w\langle p\rangle(w)$, an application of the Chinese remainder theorem (in particular, Corollary \ref{Q tensor
CRT}) shows that $$\mathbb{Q}\tensor \frac{E^0\lpow w,x\rpow}{([p](w),[p^v](x))}\simeq \mathbb{Q}\tensor
\frac{E^0\lpow w,x\rpow}{(w,[p^v](x))}\times \mathbb{Q}\tensor\frac{E^0\lpow w,x\rpow}{(\langle
p\rangle(w),[p^v](x))},$$ whereby, since $E^0(B(C_p\times\Delta))$ is free over $E^0$, we have jointly injective
maps $E^0(B(C_p\times\Delta))\to E^0\lpow x\rpow/[p^v](x)=E^0(B\Delta)$ and $E^0(B(C_p\times\Delta))\to
E^0(B(C_p\times\Delta))/\langle p\rangle(w)$. By an argument like that of Corollary \ref{jointly injective maps
2} the result follows.\end{pf}

\subsection{Applying the theory of level structures}

We now investigate some properties of the ring $E^0(B(C_p\times\Delta))/\langle p\rangle (w)$ which we identify
with $E^0\lpow w,x\rpow/(\langle p\rangle(w),[p^v](x))$ as in the remarks preceding Proposition \ref{Jointly
injective maps 3}.

\begin{lem} The ring $R=E^0\lpow w\rpow /\langle p\rangle(w)$ is an integral domain.\end{lem}
\begin{pf} With the usual notation we have $\langle p\rangle(w)\sim g_1(w)/w$ which is monic and irreducible
over the integral domain $E^0$ as in Lemma \ref{g is irreducible}.\end{pf}

\begin{lem}  The element $y=\prod_{k=0}^{p-1}(x+_F [k](w))\in R\lpow x\rpow=E^0\lpow w,x\rpow/\langle
p\rangle(w)$ is a unit multiple of $\prod_{k=0}^{p-1}(x-[k](w))$.\end{lem}
\begin{pf} Note that $y=\prod_{k=0}^{p-1} (x+_F[k](w))=\prod_{k=0}^{p-1} (x-_F[k](w))$. Then an application of Lemma \ref{x-_F y twiddles x-y} shows that $y\sim \prod_{k=0}^{p-1}(x-
[k](w))$.\end{pf}

To proceed further we need to use the results of \cite{FSFG} which require some familiarity with the language of
formal schemes and, in particular, level structures. The aforementioned paper gives a good account of the basic
definitions and notations.

Write $X=\spf(E^0)$ so that $\mathbb{G}=\spf(E^0(\mathbb{C}P^\infty))=\spf(E^0\lpow x\rpow)$ is a formal group
over $X$. Let $X_R=\spf(R)$. Then $\mathbb{G}_R=\mathbb{G}\times_X X_R = \spf(R\lpow x\rpow)$ is a formal group
over $X_R$. Put $\ds \mathbb{D}=\spf\left(R\lpow x\rpow/y\right)$ and note that $\mathbb{D}$ is a formal scheme
over $X_R$.

\begin{prop} $\mathbb{D}$ is a subgroup scheme of degree $p$ of $\mathbb{G}_R(1)=\ker(\times
p:\mathbb{G}_R\longrightarrow \mathbb{G}_R)$ and $y$ is a coordinate on the quotient group
$\mathbb{G}_R/\mathbb{D}$.\end{prop}
\begin{pf} As in \cite{HKR}, let $pF(R)$ denote the elements $a\in R$ for which $[p](a)=0$, which is a group
under $+_F$. Define a group homomorphism $\phi:\mathbb{Z}/p\to pF(R)$ given by $k\mapsto [k](w)$. Then $\phi$ is
injective (since $[k](w)\sim w$ which is non-zero for $k\neq 0$ mod $p$) and hence, using the terminology of
\cite[Proposition 26]{FSFG}, is a level structure on $\mathbb{G}_R$. Noting that, as divisors on $\mathbb{G}_R$,
we have $\mathbb{D}=\spf(R\lpow x\rpow /\prod_{k=0}^{p-1}(x-[k](w)))=[\phi(\mathbb{Z}/p)]$ and so, by
Proposition 32 and Corollary 33 in \cite{FSFG} we find that $\mathbb{D}$ is a subgroup scheme of $\mathbb{G}_R$
contained in $\mathbb{G}_R(1)$ and that $\prod_{k=0}^{p-1}(x-_F[k](w))=y$ is a coordinate on
$\mathbb{G}_R/\mathbb{D}$.\end{pf}

\begin{rem} It may be helpful to have a translation of some consequences of this result into algebra. We now know that there is a well defined map $R\lpow x\rpow/y\to R\lpow x\rpow/y$ sending $x$ mod
$y$ to $[p](x)$ mod $y$; in other words $\prod_{k=0}^{p-1}([p](x)-[k](w))$ is divisible by
$\prod_{k=0}^{p-1}(x-[k](w))$ in $R\lpow x\rpow$. This should not be too difficult to believe, as the former is
divisible by $[p](x)$ and $[p]([k](w))=[k]([p](w))=0$ in $R$ so that $(x-[k](w))$ is certainly a factor; the
above result tells us that $\prod_{k=0}^{p-1}(x-[k](w))$ is also a factor. The remarks about the quotient group
$\mathbb{G}_R/\mathbb{D}$ give us a subring $\mathcal{O}_{\mathbb{G}_R/\mathbb{D}}=R\lpow y\rpow$ of
$\mathcal{O}_{\mathbb{G}_R}=R\lpow x\rpow$ with favourable properties, as we will see below.\end{rem}

\begin{prop}\label{Definition of h} Let $y=\prod_{k=0}^{p-1} (x+_F[k](w))\in R\lpow x\rpow$, as above. Then
there is a unique power-series $h\in R\lpow y\rpow$ such that $h=[p^v](x)$ in $R\lpow x\rpow$.\end{prop}
\begin{pf} Since $\mathbb{D}$ is contained in $\mathbb{G}_R(1)$, the map $\mathbb{G}_R\overset{\times p^v}{\longrightarrow}\mathbb{G}_R$ factors through $\mathbb{G}_R/\mathbb{D}$. Thus there
is a map $\psi:\mathbb{G}_R/\mathbb{D}\to \mathbb{G}_R$ making the diagram below commute.
$$
\xymatrix{ \mathbb{G}_R \ar[rr]^{\times p^v} \ar[rd]& & \mathbb{G}_R\\
& \mathbb{G}_R/\mathbb{D}. \ar[ru]_{\psi}}
$$
Put $h(y)=\psi^*(x)\in \mathcal{O}_{\mathbb{G}_R/\mathbb{D}}=R\lpow y\rpow$. Then $h(y)=[p^v](x)$ in
$\mathcal{O}_{\mathbb{G}_R}=R\lpow x\rpow$, as required. Uniqueness is immediate since
$\mathcal{O}_{\mathbb{G}_R/\mathbb{D}}$ is a subring of $\mathcal{O}_{\mathbb{G}_R}$.\end{pf}

We explore some of the properties of the power-series $h$ defined above, but first need a couple of lemmas.

\begin{lem}\label{A^G/aA^G=(A/aA)^G} Let $A$ be a ring, let $G$ be a finite group acting on $A$ and suppose that $|G|$ is invertible in $A$. Then for any
$a\in A^G$ there is an action of $G$ on $A/aA$ and there is a $G$-equivariant isomorphism $A^G/aA^G\simeq
(A/aA)^G$.\end{lem}
\begin{pf} It is clear that the composite $f:A^G\rightarrowtail A\to A/aA$ has image contained in
$(A/aA)^G$. Given any $r+aA\in (A/aA)^G$ we find that $f(\frac{1}{|G|}\sum_{g\in G} g.r)=\frac{1}{|G|}\sum_{g\in
G} (g.r + aA)=\frac{|G|}{|G|}(r+aA)=r+aA$ so that $f:A^G\to (A/aA)^G$ is surjective. It is clear that the kernel
of the map is just $aA^G$ and the result follows.\end{pf}

\begin{lem}\label{Z/p^times invariants of R}  There is an action of $(\mathbb{Z}/p)^\times$ on $R$ given by
$k.w=[k](w)$ and $\langle p\rangle(w)$ is $(\mathbb{Z}/p)^\times$-invariant. Further, the element
$d=\prod_{k=1}^{p-1}[k](w)$ generates $R^{(\mathbb{Z}/p)^\times}$ over $E^0$ and
$R^{(\mathbb{Z}/p)^\times}=E^0\lpow d\rpow/f(d)$, where $f(d)=\langle p\rangle(x)$ as an element of $R\lpow
x\rpow$.\end{lem}
\begin{pf} As is Proposition \ref{E^0(BSigma_p) in terms of d}, we know that $(E^0\lpow w\rpow/[p](w))^{(\mathbb{Z}/p)^\times}=E^0\lpow d\rpow/df(d)$. Then an application of Lemma \ref{A^G/aA^G=(A/aA)^G} shows that $(E^0\lpow w\rpow/\langle
p\rangle(w))^{(\mathbb{Z}/p)^\times}=E^0\lpow d\rpow/f(d)$.\end{pf}

\begin{cor} With $h$ as in Proposition \ref{Definition of h} and $(\mathbb{Z}/p)^\times$ acting on $R$ as above,
we have $h\in R\lpow y\rpow^{(\mathbb{Z}/p)^\times}=E^0\lpow d,y\rpow/f(d)$.\end{cor}
\begin{pf} We know that the inclusion $R\lpow y\rpow\rightarrowtail R\lpow x\rpow$ maps $h$ to $[p^v](x)$, which is clearly invariant under the action of
$(\mathbb{Z}/p)^\times$ on $R$. Hence $h\in R\lpow x\rpow^{(\mathbb{Z}/p)^\times}\cap R\lpow y\rpow=R\lpow
y\rpow^{(\mathbb{Z}/p)^\times}$.\end{pf}

We will write $h=h(d,y)\in E^0\lpow d,y\rpow/f(d)$ thought of as a power-series in $d$ and $y$. Note that there
is a well defined element $d{h}(d,y)\in E^0\lpow d,y\rpow/df(d)$ and hence a well defined element $dh(d,c_p)\in
E^0(BN)$.

\begin{prop} With $h$ as in Proposition \ref{Definition of h} we have $t+dh(d,c_p)=0$ in $E^0(BN)$.\end{prop}
\begin{pf} We check that $t+dh(d,c_p)$ maps to zero in each of $E^0(BT)$, $E^0(B(C_p\times
\Delta))/\langle p\rangle(w)$ and $D^\Gamma$ and conclude that it must be zero in $E^0(BN)$ by Corollary
\ref{Jointly injective maps 3}. We use the results of Proposition \ref{Generators of E^0(BN)}. In $E^0(BT)$ we
have both $t$ and $d$ mapping to $0$ so that $t+dh(d,a)\mapsto 0$. In $E^0(B(C_p\times \Delta))/\langle
p\rangle(w)=E^0\lpow w,x\rpow/(\langle p\rangle(w),[p^v](x))$ we have
$$\textstyle t+dh(d,c_p)\mapsto 0+ dh\left(d,\prod_{k=0}^{p-1} (x+_F [k](w))\right)=d[p^v](x)=0.$$
We are left to consider the image in $D^\Gamma$. There is a well-defined map $E^0\lpow w,x\rpow/\langle
p\rangle(w)\to D^\Gamma$ given by $w\mapsto [p^v](x)$ and $x\mapsto x$. Then the identity
$h\left(-w^{p-1},\prod_{k=0}^{p-1}(x+_F[k](w))\right)=[p^v](x)$ gives
$h\left(-[p^v](x)^{p-1},\prod_{k=0}^{p-1}[1+kp^v](x)\right)=[p^v](x)$ in $D^\Gamma$. Thus we have
$$\textstyle t+dh(d,c_p)\mapsto [p^v](x)^p + (-[p^v](x)^{p-1})h\left(-[p^v](x)^{p-1},\prod_{k=0}^{p-1}[1+kp^v](x)\right)=0$$
in $D^\Gamma$ and we are done.\end{pf}

\begin{cor}\label{c_p^{p^n}+d.tilde{h}(d,c_p)=0} In $K^0(BN)$ we have $t=c_p^{p^{nv}}$ and $c_p^{p^{nv}}+dh(d,c_p)=0$.\end{cor}
\begin{pf} We have $t=\prod_{i=1}^{p}[p^v](x_i)=\prod_{i=1}^{p}x_i^{p^{nv}}$ mod $(p,u_1,\ldots,u_{n-1})$.\end{pf}

We will be interested in decoding this relation a bit further, and the following will help.

\begin{lem}\label{h(0,s)} In $K^0\tensor_{E^0} E^0\lpow w,s\rpow/\langle p\rangle(w)$ we have $h(-w^{p-1},s)=s^{p^{nv-1}}$ mod $w$. Hence, $h(0,s)=s^{p^{nv-1}}\in \mathbb{F}_p\lpow s\rpow$.\end{lem}
\begin{pf} The identity $h(-w^{p-1},\prod_{k=0}^{p-1}x+_F[k](w))=[p^v](x)$ in $E^0\lpow w,x\rpow/\langle p\rangle(w)$ read modulo $(p,u_1,\ldots,u_{n-1})$ gives $h(-w^{p-1},\prod_{k=0}^{p-1}x+_F[k](w))=x^{p^{nv}}$. Then, modulo $w$, we get $h(0,x^p)=x^{p^{nv}}$ and hence $h(0,s)=s^{p^{nv-1}}$.
\end{pf}

In Section \ref{sec:studying K^0 tensor ker beta} we will need these results to get an idea of the structure of
the kernel of the map $\beta:E^0(BGL_p(\Fq))\to E^0(BT)^{\Sigma_p}$. In particular, they will be crucial in
proving that a certain class in $K^0(BN)$ is non-zero.

\subsection{The kernels of $\alpha$ and $\beta$}

In a similar fashion to the results of Section \ref{sec:system of maps}, we aim to prove joint injectivity of
the maps $\alpha$ and $\beta$.

\begin{lem}\label{invertible g contains permutation matrix} Let $g\in GL_d(K)$. Then there exists a permutation $\rho\in\Sigma_d$ such that $g_{i\rho(i)}\neq 0$
for all $i$.\end{lem}
\begin{pf} Label the rows of $g$ as $r_1,\ldots,r_d$. Then, by considering the expansion of the determinant
along $r_1$, there must be a non-zero entry $r_{1j}$ such that the resultant matrix formed by deleting row 1 and
column $j$ has non-zero determinant, and is therefore invertible. Put $\rho(1)=j$ and continue inductively to
get a well-defined permutation $\rho\in \Sigma_d$ with the required property.\end{pf}

\begin{lem}\label{simplifying the limit diagram} Let $\mathcal{A}'$ be the full subcategory of $\mathcal{A}(G)_{(p)}$ with objects $A,~T_{(p)}$ and
$\Delta_p$. Then $$\lim_{H\in\mathcal{A}(G)_{(p)}} \mathbb{Q}\tensor E^0(BH) = \lim_{H\in
\mathcal{A}'}\mathbb{Q}\tensor E^0(BH).$$\end{lem}
\begin{pf} There is a unique map $\lim_{H\in\mathcal{A}(G)_{(p)}} \mathbb{Q}\tensor E^0(BH) \to \lim_{H\in
\mathcal{A}'}\mathbb{Q}\tensor E^0(BH)$ commuting with the arrows, by abstract category theory. We construct an
inverse.

Recall from Proposition \ref{Abelian p-subgroups of GL_p(K)} that any abelian $p$-subgroup of $GL_p(\Fq)$ is
sub-conjugate to either $A$ or $T$ (or both). Write $L=\lim_{H\in \mathcal{A}'}\mathbb{Q}\tensor E^0(BH)$. Note
that the structure maps $L\to \mathbb{Q}\tensor E^0(BA)$ and $L\to \mathbb{Q}\tensor E^0(BT_{(p)})$ land in the
invariant subrings under the action of the relevant normaliser. We consider abelian $p$-subgroups $H$ of
$GL_p(\Fq)$.

If $H$ is cyclic of order $p^{v+1}$ then $H$ must be conjugate to $A$. Further, given two such isomorphisms
$g_1Hg_1^{-1}=A=g_2Hg_2^{-1}$ the diagram
$$
\xymatrix{ A \ar[d]_{\conj_{g_2^{~}g_1^{-1}}} & & H \ar[ll]_{\conj_{g_1^{~}}} \ar[dll]^{\conj_{g_2^{~}}}\\
A}
$$
commutes. Thus $g_2 g_1^{-1}\in N_{GL_p(\Fq)}(A)$ and both maps $L\to \mathbb{Q}\tensor E^0(BA)\to
\mathbb{Q}\tensor E^0(BH)$ are equal. Hence we have a uniquely defined map $L\to \mathbb{Q}\tensor E^0(BH)$.
Further, given any map $\conj_g:H_1\to H_2$ of such subgroups in $\mathcal{A}(G)_{(p)}$ it is clear from similar
reasoning that we have a commuting diagram
$$
\xymatrix{& & \mathbb{Q}\tensor E^0(BH_1) \ar[dd]^{\conj_g^*}\\
L \ar[r] \ar@/^2pc/[urr] \ar@/_2pc/[drr] & \mathbb{Q}\tensor E^0(BA) \ar[ur] \ar[dr]\\
 & & \mathbb{Q}\tensor E^0(BH_2).}
$$
Any other $H\leqslant GL_p(\Fq)$ which is sub-conjugate to $A$ must be sub-conjugate to $\Delta_p\subseteq A$.
In particular, since $\Delta$ is central in $GL_p(\Fq)$, we find that $H\subseteq\Delta_p$ and, for any $g\in
GL_p(\Fq)$, the conjugation map induced by $g$ is the identity on $H$. Hence we have a uniquely defined map
$L\to \mathbb{Q}\tensor E^0(BA)\to \mathbb{Q}\tensor E^0(B\Delta_p)\to \mathbb{Q}\tensor E^0(BH)$ which respects
any arrow in $\mathcal{A}(G)_{(p)}$.

We are left with the case that $H$ is subconjugate to $T$. First, suppose that $H\subseteq T$ and let
$h=(h_1,\ldots,h_p)\in H$. Then $ghg^{-1}=k$ for some $k=(k_1,\ldots,k_p)\in T$. Letting $g=(g_{ij})$ we get
equations $g_{ij}h_j=k_i g_{ij}$ for all $i,j$. Hence $g_{ij}(h_j-k_i)=0$. But, by Lemma \ref{invertible g
contains permutation matrix}, there is a permutation $\rho\in\Sigma_p$ with $g_{i\rho(i)}\in \Fq^\times$ for all
$i$. Thus, for each $i$ we have $h_{\rho(i)}=k_i$ and so $ghg^{-1}=(h_{\rho(1)},\ldots,h_{\rho(p)})$. It follows
that the map $\conj_g:H\to T$ corresponds to permutation by $\rho$ and hence extends to a map $T\to T$ induced
by an element of $N_{GL_p(\Fq)}(T)$.

Now, given any $H$ subconjugate to $T_{(p)}$ it follows that the map $L\to \mathbb{Q}\tensor E^0(BT_{(p)})\to
\mathbb{Q}\tensor E^0(BH)$ is independent of the choice of conjugating element and, further, that any arrow in
$\mathcal{A}(G)_{(p)}$ commutes with these maps.

Thus, we conclude that given any arrow $H\to K$ in $\mathcal{A}(G)_{(p)}$ we have maps $L\to \mathbb{Q}\tensor
E^0(BH)$ and $L\to \mathbb{Q}\tensor E^0(BK)$ which commute with the arrow. Hence we get a well defined map
$L\to \lim_{H\in\mathcal{A}(G)_{(p)}} \mathbb{Q}\tensor E^0(BH)$ which is necessarily inverse to the map at the
start of the proof by abstract category theory.\end{pf}

\begin{prop}\label{Rational isomorphism} The map $\mathbb{Q}\tensor E^0(BGL_p(\Fl))\longrightarrow \mathbb{Q}\tensor E^0(BT)^{\Sigma_p}\times \mathbb{Q}\tensor
D^{\Gamma}$ induced by $\alpha$ and $\beta$ is an isomorphism.\end{prop}
\begin{pf} Writing $G=GL_p(\Fq)$, by Proposition \ref{E^*(BG) as limit over p-groups} we have an isomorphism
$$\mathbb{Q}\tensor E^0(BGL_p(\Fq)) \iso \lim_{H\in\mathcal{A}(G)_{(p)}}
\mathbb{Q}\tensor E^0(BH).$$ But, by Lemma \ref{simplifying the limit diagram}, the right-hand side simplifies
to $\lim_{H\in\mathcal{A}'}\mathbb{Q}\tensor E^0(BH)$. Thus, using Proposition \ref{N(T_d) and W(T_d)} and Lemma
\ref{Normalizer of A} we are left with a pullback
$$
\xymatrix{ \mathbb{Q}\tensor E^0(BGL_p(\Fq)) \ar[r] \ar[d] & \mathbb{Q}\tensor E^0(BA)^\Gamma \ar[d]\\
\mathbb{Q}\tensor E^0(BT_{(p)})^{\Sigma_p} \ar[r] & \mathbb{Q}\tensor E^0(B\Delta_{p}).}$$ From Proposition
\ref{Q tensor A = Q tensor Delta times Q tensor D} we know that $\mathbb{Q}\tensor E^0(BA)\simeq
\mathbb{Q}\tensor E^0(B\Delta_p)\times \mathbb{Q}\tensor D$. But the action of $\Gamma$ is trivial on
$E^0(B\Delta_p)$ so we get $\mathbb{Q}\tensor E^0(BA)^\Gamma\simeq \mathbb{Q}\tensor E^0(B\Delta_p)\times
\mathbb{Q}\tensor D^\Gamma$. Since $E^0(BT)^{\Sigma_p}\to E^0(B\Delta_p)$ is surjective, the result
follows.\end{pf}

\begin{cor}\label{alpha, beta are jointly injective} The maps
$$
\xymatrix{ & E^0(BGL_p(\Fq)) \ar[dl]_{\beta} \ar[dr]^{\alpha}\\
E^0(BT)^{\Sigma_p} & & D^\Gamma}
$$
are jointly injective.\end{cor}
\begin{pf} Since $E^0(BGL_p(\Fq))$ is free over $E^0$ the result follows analogously to the proof of Corollary \ref{E^0(BA) split}.\end{pf}

\begin{cor}\label{rank E^0(BGL_p(Fl))} $\rank_{E^0}(E^0(BGL_p(\Fl)))=\rank_{E^0}(E^0(BT)^{\Sigma_p})+\rank_{E^0}(D^\Gamma)$.
Hence we have $\rank_{E^0}(E^0(BGL_p(\Fq)))=\rank_{E^0}(\ker(\alpha))+\rank_{E^0}(\ker(\beta))$.\end{cor}
\begin{pf} Each of the rings in question is free over $E^0$ so
\begin{eqnarray*}
\rank_{E^0}(E^0(BGL_p(\Fl))) & = & \rank_{\mathbb{Q}\tensor E^0}(\mathbb{Q}\tensor E^0(BGL_p(\Fl)))\\
& = & \rank_{\mathbb{Q}\tensor E^0}(\mathbb{Q}\tensor_{E^0}E^0(BT)^{\Sigma_p})+\rank_{\mathbb{Q}\tensor E^0}(\mathbb{Q}\tensor_{E^0}D^\Gamma)\\
& = & \rank_{E^0}(E^0(BT)^{\Sigma_p})+\rank_{E^0}(D^\Gamma).\end{eqnarray*} For the final remark, note that
$$\rank_{E^0}(E^0(BGL_p(\Fq)))=\rank_{E^0}(\ker(\beta)) + \rank_{E^0}(\im(\beta))$$ and
$\im(\beta)=E^0(BT)^{\Sigma_p}$. It follows that $\rank_{E^0}(\ker(\beta))=\rank_{E^0}(D^\Gamma)$ and similarly
that $\rank_{E^0}(\ker(\alpha))=\rank_{E^0}(E^0(BT)^{\Sigma_p})$.\end{pf}

Recall that $I=E^0(BGL_p(\Fq))t$ was the ideal of $E^0(BGL_p(\Fq))$ generated by $t$ so that, by Proposition
\ref{definition of t}, we have $I\subseteq \ker(\beta)$. We suspect that $I=\ker(\beta)$ and work towards
proving the reverse inclusion.

\begin{lem}\label{ker alpha.ker beta=0} We have $\ker(\alpha)\cap\ker(\beta)=0$ and hence $\ker(\alpha).\ker(\beta)=0$.\end{lem}
\begin{pf} By Corollary \ref{alpha, beta are jointly injective} the map $(\alpha,\beta):E^0(BLG_p(\Fq))\to D^\Gamma\times
E^0(BT)^{\Sigma_p}$ is injective. If $a\in\ker(\alpha)\cap\ker(\beta)$ then $\alpha(a)=\beta(a)=0$ whereby
$a\in\ker(\alpha,\beta)=0$. The second claim follows as $\ker(\alpha).\ker(\beta)\subseteq
\ker(\alpha)\cap\ker(\beta)$.\end{pf}

\begin{cor}\label{I free over D^Gamma} The identification $E^0(BGL_p(\Fq))/\ker(\alpha)\simeq D^\Gamma$ makes $I$ into a free rank one module over
$D^\Gamma$ and hence a free $E^0$-module of rank $N$.\end{cor}
\begin{pf} Since $t\in\ker(\beta)$ we have $\ker(\alpha)t=0$ and it follows that
\[I=E^0(BGL_p(\Fq))t=\left(\frac{E^0(BGL_p(\Fq))}{\ker(\alpha)}\right)t\simeq D^\Gamma t.\]
Now, as $I$ is generated by one element over $D^\Gamma$ it is sufficient to show that $I$ is torsion free. Take
$0\neq s\in D^\Gamma$. Then $\alpha(s.t)=s\alpha(t)\sim s[p^v](x)^p$. Since $[p^v](x)$ divides $\langle
p\rangle([p^v](x))-p$ we see that $[p^v](x)$ divides $p$ in $D$. Thus $s[p^v](x)^p$ divides $s p^p$, which is
non-zero as $D$ is free over $E^0$. Hence $\alpha(s.t)\neq 0$ and $s.t\neq 0$, as required. The final statement
is immediate from Proposition \ref{structure of D^Lambda}.\end{pf}

\begin{lem}\label{ker alpha and ker beta summands} The ideal $\ker(\beta)$ is an $E^0$-module summand in $E^0(BGL_p(\Fq))$ and is free of rank $N$.
Similarly, $\ker(\alpha)$ is a free $E^0$-summand in $E^0(BGL_p(\Fq))$ of rank
$\frac{(p^{nv}+p-1)!}{p!(p^{nv}-1)!}$.\end{lem}
\begin{pf} Since each of $E^0(BGL_p(\Fq))$ and $E^0(BT)^{\Sigma_p}$ is free over $E^0$ it follows that the short exact sequence
$0\to\ker(\beta)\to E^0(BGL_p(\Fq))\to E^0(BT)^{\Sigma_p}\to 0$ splits. Thus $\ker(\beta)$ is a summand in
$E^0(BGL_p(\Fq))$ and hence is projective. But all projective modules over local rings are free (see
\cite[Theorem 2]{Kaplansky}). Hence $\ker(\beta)$ is free. By Proposition \ref{rank E^0(BGL_p(Fl))} we have
$\rank_{E^0}(E^0(BGL_p(\Fq)))=\rank_{E^0}(E^0(BT)^{\Sigma_p})+\rank_{E^0}(D^\Gamma)$ and it follows that
$\rank_{E^0}(\ker(\beta))=\rank_{E^0}(D^\Gamma)=N$.

Since $D^\Gamma$ is also free over $E^0$ an exactly analogous argument will prove that $\ker(\alpha)$ is a free
$E^0$-summand. The rank of $\ker(\alpha)$ is the same as that of $\rank_{E^0}(E^0(BT)^{\Sigma_p})$ which, by
Proposition \ref{Basis for Sigma_d-invariants}, is equal to the size of the set $\{\beta\in\mathbb{N}^p\mid
0\leq \beta_1+\ldots +\beta_p < p^{nv}\}$. Standard combinatorics ($p$ markers in $p^{nv}+p-1$ slots) gives this
as ${p^{nv}+p-1 \choose p}=\frac{(p^{nv}+p-1)!}{p!(p^{nv}-1)!}$.\end{pf}

\begin{cor}\label{ann ker alpha = ker beta} We have $\ker(\alpha)=\ann(\ker(\beta))$ and $\ker(\beta)=\ann(\ker(\alpha))$.\end{cor}
\begin{pf} Firstly note that $\ker(\alpha)\subseteq \ann(\ker(\beta))$ and $\ker(\beta)\subseteq
\ann(\ker(\alpha))$ by Lemma \ref{ker alpha.ker beta=0}. By \cite{StricklandK(n)duality} we know that
$E^0(BGL_p(\Fq))$ has duality over $E^0$. Thus we can apply Corollary \ref{ann P is a summand} to see that both
of $\ann(\ker(\alpha))$ and $\ann(\ker(\beta))$ are summands in $E^0(BGL_p(\Fq))$ and hence free. But
$\rank_{E^0}(\ann(\ker(\beta))) = \rank_{E^0}(E^0(BGL_p(\Fq))) - \rank_{E^0}(\ker(\beta))$ and the latter is
just $\rank_{E^0}(\ker(\alpha))$ using Proposition \ref{rank E^0(BGL_p(Fl))}. Thus $\ker(\alpha)\subseteq
\ann(\ker(\beta))$ is an inclusion of free summands of the same rank and so is an equality. The same argument
shows that $\ker(\beta)= \ann(\ker(\alpha))$.\end{pf}

\subsection{Studying $\ker(\beta)$ more closely}\label{sec:studying K^0 tensor ker beta}

To proceed further we apply the functor $K^0\tensor_{E^0}-$ or, equivalently, work modulo the maximal ideal
$(p,u_1,\ldots,u_{n-1})$. We have the following commutative diagram.
$$
\xymatrix{ I \ar[r] \ar[d] & \ker(\beta) \ar[r] \ar[d] & E^0(BGL_p(\Fq)) \ar@{->>}[r]^-\alpha
\ar[d] & D^\Gamma \ar[d]\\
K^0\tensor_{E^0} I \ar[r] & K^0\tensor_{E^0} \ker(\beta) \ar[r] & K^0\tensor_{E^0}E^0(BGL_p(\Fq)) \ar@{->>}[r] &
K^0\tensor_{E^0} D^\Gamma}
$$
We know that $K^0\tensor_{E^0}E^0(BGL_p(\Fq)) = K^0(BGL_p(\Fq))$ and aim to understand the remainder of the
bottom row. Recall that we defined $N$ to be $(p^{n(v+1)}-p^{nv})/p$.

\begin{prop} With $y$ as in Proposition \ref{structure of D^Lambda} we have $K^0\tensor_{E^0} D^\Gamma\simeq \mathbb{F}_p\lpow y\rpow/ y^N$.\end{prop}
\begin{pf} Modulo $(p,u_1,\ldots,u_{n-1})$ we know that $[p^m](x)=x^{p^{nm}}$ for all $m$ and hence that $\langle p\rangle([p^v](x))=x^{p^{n(v+1)}-p^{nv}}$. Thus
$$K^0\tensor_{E^0} D = \frac{K^0\tensor_{E^0} E^0\lpow x\rpow}{K^0\tensor_{E^0} (\langle
p\rangle([p](x)))}= \mathbb{F}_p\lpow x\rpow/(x^{p^{n(v+1)}-p^{nv}}).$$ Further, $K^0\tensor_{E^0} D^\Gamma$ is
the subring of $K^0\tensor_{E^0} D$ generated by $y\sim x^p$, so that $K^0\tensor_{E^0}
D^\Gamma=\mathbb{F}_p\lpow y\rpow/y^N$, as claimed.\end{pf}

\begin{lem} $\ker(K^0\tensor\beta)$ is a module over $K^0\tensor_{E^0} D^\Gamma$.\end{lem}
\begin{pf} By Lemma \ref{Right exactness for red mod ideals}, the maps
$K^0\tensor_{E^0}\ker(\alpha)\to \ker(K^0\tensor\alpha)$ and $K^0\tensor_{E^0}\ker(\beta)\to
\ker(K^0\tensor\beta)$ are both surjective. Take $a\in \ker(K^0\tensor\alpha)$. We can lift $a$ first to
$K^0\tensor_{E^0}\ker(\alpha)$ and then to $\ker(\alpha)$. Choose such a lift, $\tilde{a}\in \ker(\alpha)$ say.
Similarly, given any $b\in\ker(K^0\tensor_{E^0}\beta)$ we can choose a lift $\tilde{b}\in\ker(\beta)$. Then
$\tilde{a}.\tilde{b}=0$ in $E^0(BGL_p(\Fq))$ by Lemma \ref{ker alpha.ker beta=0} so that $a.b=0$ in
$K^0(BGL_p(\Fq))$. It follows that $\ker(K^0\tensor_{E^0}\beta)$ is a module over
$K^0(BGL_p(\Fq))/\ker(K^0\tensor\alpha)=K^0\tensor_{E^0}D^\Gamma$.\end{pf}

\begin{cor} $K^0(BGL_p(\Fq))c_p^{p^{nv}}$ is a module over $K^0\tensor_{E^0}D^\Gamma$.\end{cor}
\begin{pf} Since $t\in\ker(\beta)$ and $t=c_p^{p^{nv}}$ modulo $(p,u_1,\ldots,u_{n-1})$ it follows that $c_p^{p^{nv}}$ maps to zero under $K^0\tensor\beta$. Thus
$K^0(BGL_p(\Fq))c_p^{p^n}\subseteq\ker(K^0\tensor\beta)$ and so is annihilated by $\ker(K^0\tensor\alpha)$. The
result follows.\end{pf}

To understand the structure of $K^0(BGL_p(\Fq))c_p^{p^{nv}}$ as a $K^0\tensor_{E^0}D^\Gamma$-module we will need
an understanding of the nilpotency of $c_p^{p^{nv}}$ in $K^0(BN)$.

\begin{lem}\label{a^{p^{n+i}} in terms of a,d} In $K^0(BN)$ we have $c_p^{p^{nv+i}}\sim c_p^{p^{nv-1}}d^{1+p+\ldots+p^i}$ modulo
$d^{1+p+\ldots+p^i+1}$ for each $i\geq 0$.\end{lem}
\begin{pf} We proceed by induction. For $i=0$ we have $c_p^{p^{nv}}\sim dh(d,c_p)$ by Corollary \ref{c_p^{p^n}+d.tilde{h}(d,c_p)=0}. From
Lemma \ref{h(0,s)} we get $h(d,c_p)= c_p^{p^{nv-1}}$ mod $d$ so that $dh(d,c_p)= c_p^{p^{nv-1}}d$ mod $d^2$, as
required.

Supposing $c_p^{p^{nv+k}}\sim c_p^{p^{nv-1}}d^{1+p+\ldots+p^k}$ mod $d^{1+p+\ldots+p^k+1}$ write
$$c_p^{p^{nv+k}}=u c_p^{p^{nv-1}}d^{1+p+\ldots+p^k}+d^{1+p+\ldots+p^k+1}s$$ for some unit $u$ and some $s$. Then, raising to the power $p$ (a mod-$p$ automorphism) we have
$$c_p^{p^{nv+k+1}}=u^p c_p^{p^{nv}}d^{p+p^2+\ldots+p^{k+1}}+d^{p+p^2+\ldots+p^{k+1}+p}s^p.$$
Thus, modulo $d^{1+p+\ldots+p^{k+1}+1}$, we have $c_p^{p^{nv+k+1}}\sim c_p^{p^{nv}}d^{p+p^2+\ldots+p^{k+1}}$
since $p\geq 2$. But $c_p^{p^{nv}}\sim c_p^{p^{nv-1}}d$ mod $d^2$ so that
$c_p^{p^{nv}}d^{p+p^2+\ldots+p^{k+1}}\sim c_p^{p^{nv-1}}d^{1+p+\ldots+p^{k+1}}$ mod
$d^{p+\ldots+d^{p^{k+1}}+2}$. Hence $c_p^{p^{nv+k+1}}\sim c_p^{p^{nv-1}}d^{1+p+\ldots+p^{k+1}}$ mod
$d^{1+p+\ldots+p^{k+1}+1}$, completing the inductive step.\end{pf}

\begin{lem}\label{a^{p^{2n-1}+p^n-p^{n-1}-1} in terms of a,d} In $K^0(BN)$ we have $c_p^{p^{n(v+1)-1}}\sim c_p^{p^{nv-1}}d^{(p^n-1)/(p-1)}$, which is non-zero.\end{lem}
\begin{pf} Put $i=n-1$ in Lemma \ref{a^{p^{n+i}} in terms of a,d} to get the result $c_p^{p^{n(v+1)-1}}\sim c_p^{p^{nv-1}}d^{1+p+\ldots+p^{n-1}}$ mod
$d^{1+p+\ldots + p^{n-1}+1}$. But we have $d^{1+p+\ldots + p^{n-1}+1}=d^{(p^n-1)/(p-1)+1}=0$ in $K^0(BN)$
(since, by tensoring with $K^0$, $K^0(B\Sigma_p)\simeq \mathbb{F}_p\lpow d \rpow/d^{(p^n-1)/(p-1)+1}$). Thus we
get
$$c_p^{p^{n(v+1)-1}}\sim c_p^{p^{nv-1}}d^{1+p+\ldots+p^{n-1}}=c_p^{p^{nv-1}}d^{(p^n-1)/(p-1)}$$ in $K^0(BN)$. The right-hand side is a basis element for $K^0(BN)$ over $K^0$ so is non-zero.\end{pf}

\begin{prop}\label{c_p^{N+p^n-1} neq 0} We have $c_p^{N+p^{nv}-1}\neq 0$ in $K^0(BGL_p(\Fq))$.\end{prop}

\begin{pf} By Lemma \ref{a^{p^{2n-1}+p^n-p^{n-1}-1} in terms of a,d} we have $c_p^{p^{n(v+1)-1}}\sim
c_p^{p^{nv-1}}d^{(p^n-1)/(p-1)}$. Multiplying both sides by $c_p^{p^{nv}-p^{nv-1}-1}$ then gives
$c_p^{N+p^{nv}-1}\sim c_p^{p^{nv}-1}d^{(p^n-1)/(p-1)}\neq 0$. Thus $c_p^{N+p^{nv}-1}$ is non-zero in
$K^0(BN)$.\end{pf}

\begin{prop} The ideal $K^0(BGL_p(\Fq))c_p^{p^{nv}}$ is free of rank 1 over $K^0\tensor_{E^0} D^\Gamma=\mathbb{F}_p\lpow y\rpow/y^N$ and hence has dimension $N$ as a vector space over $\mathbb{F}_p$.\end{prop}

\begin{pf} Since $M=K^0(BGL_p(\Fq))c_p^{p^{nv}}$ is generated over $\mathbb{F}_p\lpow y\rpow/y^N$ by one element, namely
$c_p^{p^{nv}}$, it follows that $M\simeq \mathbb{F}_p\lpow y\rpow/y^m$ for some $m\leq N$, that is
$y^m.c_p^{p^{nv}}=0$ for some $m$. Since $\alpha$ maps $c_p$ to $y$ we have $y^{N-1}.c_p^{p^{nv}}\neq 0$ if and
only if $c_p^{N-1}c_p^{p^{nv}}\neq 0$. By Corollary \ref{c_p^{N+p^n-1} neq 0}, the latter holds so that $M\simeq
\mathbb{F}_p\lpow y\rpow/y^N$ is free over $K^0\tensor_{E^0} D^\Gamma$.\end{pf}

\begin{lem} The induced map $K^0\tensor_{E^0} I\to K^0(BGL_p(\Fq))c_p^{p^{nv}}$ is an isomorphism.\end{lem}
\begin{proof} That the map is surjective is immediate from Corollary \ref{Mod I reduction of principal ideal surjects}. Using Corollary \ref{I free over D^Gamma} we
see that $K^0\tensor_{E^0} I$ is an $\mathbb{F}_p$-vector space of dimension
$N=\dim_{\mathbb{F}_p}(K^0(BGL_p(\Fq))c_p^{p^{nv}})$. Thus the map is an isomorphism.\end{proof}

\begin{lem} The induced map $K^0\tensor_{E^0} I\to K^0\tensor_{E^0} \ker(\beta)$ is an isomorphism.\end{lem}

\begin{proof} Consider the following diagram.
$$
\xymatrix{ K^0\tensor_{E^0} I \ar[r] \ar[dr]_-{\sim} & K^0\tensor_{E^0} \ker(\beta) \ar[r] & \ker(K^0\tensor\beta)\\
& K^0(BGL_p(\Fq)).c_p^{p^{nv}} \ar@{>->}[ur]}
$$
Since the composite $K^0\tensor_{E^0} I\iso K^0(BGL_p(\Fq))c_p^{p^{nv}}\rightarrowtail \ker(K^0\tensor\beta)$ is
injective we see that the map $K^0\tensor_{E^0} I \to K^0\tensor_{E^0} \ker(\beta)$ is also injective. Both the
source and target are $\mathbb{F}_p$-vector spaces of dimension $N$ and it follows that the map is an
isomorphism.\end{proof}

\begin{cor}\label{I=ker beta} $I=\ker(\beta)$. That is, $\ker(\beta)=E^0(BGL_p(\Fq))$ is principal, generated by $t$.\end{cor}
\begin{proof} Since $K^0\tensor_{E^0} I\to K^0\tensor_{E^0}\ker(\beta)$ is an isomorphism, the result follows
immediately by an application Proposition \ref{M/IM=N/IN implies M=N}.\end{proof}

\begin{proof}[Proof of Theorem \ref{Theorem, d=p}] It just remains to assemble the results of this chapter. That
$\alpha$ and $\beta$ are jointly injective is Corollary \ref{alpha, beta are jointly injective}. We have shown
that $\beta$ is surjective in Proposition \ref{E^0(BGL_d(K))to E^0(BT_d)^Sigma_d is surjective} and surjectivity
of $\alpha$ was proved in Proposition \ref{alpha is surjective}. The rational isomorphism was Proposition
\ref{Rational isomorphism}. The remaining results were covered in Lemma \ref{ker alpha and ker beta summands},
Corollary \ref{I=ker beta} and Proposition \ref{ann ker alpha = ker beta}.\end{proof}

As a corollary to Theorem \ref{Theorem, d=p} we can give an explicit basis for $E^0(BGL_p(\Fq))$. Indeed, by
earlier work (see Section 6.2) we have a basis $B$ for $E^0(BT)^{\Sigma_p}$ which lifts canonically to
$E^0(BGL_p(\Fq))$; write $\tilde{B}$ for this lift. We then have the following result.

\begin{cor} The set $S=\tilde{B}\cup \{tc_p^i\mid 0\leq i< N\}$ is a basis for $E^0(BGL_p(\Fq))$ over
$E^0$.\end{cor}
\begin{proof} Since $\ker(\beta)$ is a summand in $E^0(BGL_p(\Fq))$ we have a decomposition $$E^0(BGL_p(\Fq))\simeq
\ker(\beta)\oplus E^0(BT)^{\Sigma_p}.$$ But $\ker(\beta)=E^0(BGL_p(\Fq))t\simeq D^\Gamma t=E^0\{tc_p^i\mid 0\leq
i<N\}$ and $E^0(BT)^{\Sigma_p}=E^0\{B\}$. Thus, as $E^0$-modules, \[E^0(BGL_p(\Fq))=E^0\{tc_p^i\mid 0\leq
i<N\}\oplus E^0\{\tilde{B}\}.\qedhere\]\end{proof}



\appendix
\addappheadtotoc

\chapter{Glossary}
\label{ap:notation}

To ease readability of this thesis, a glossary of frequently used notation is included below.

\begin{itemize}
\item We use $A$ to denote a chosen cyclic subgroup of $GL_p(\Fq)$ of size $p^{v+1}$.

\item We let $\alpha$ denote the composition $E^0(BGL_p(\Fq))\to E^0(BA)\to D$.

\item For ${\bf{\alpha}}\in J$ we let $b_{\bf{\alpha}} =
\tr_{(\Fq^\times)^p}^{~N~}(x_1^{\alpha_1}\ldots x_p^{\alpha_p})\in E^0(BN)$.

\item We let $\beta$ denote the restriction map $E^0(BGL_d(\Fq))\to E^0(BT_d)$.

\item We write $c_p$ for the $l$-euler class of any subgroup of $GL_p(\Flbar)$; that is, $c_p$ is the
restriction of the generator $\sigma_p$ of $E^0(BGL_p(\Flbar))\simeq E^0\lpow \sigma_1,\ldots,\sigma_p\rpow$.

\item We use $D$ to denote the ring $E^0\lpow x\rpow/\langle p\rangle([p^v](x))$.

\item We use $d$ to denote the generator of $E^0(B\Sigma_p)\simeq E^0\lpow d\rpow/df(d)$.\footnote{Not to be confused with the chosen integer greater than or equal to 1 which determines the rank of the general linear group being
studied.}

\item We let $\Delta$ denote the diagonal subgroup of $T=T_p$ and $\Delta_p$ denote the $p$-part of $\Delta$.

\item We use $f$ to denote the polynomial over $E^0$ for which $f(-w^{p-1})=\langle p\rangle (w)$ in $E^0\lpow
w\rpow/[p](w)$.

\item We let $F$ denote the Frobenius automorphism of $\Flbar$ and also to denote the standard $p$-typical
formal group law.

\item We write $\Gamma$ for the Galois group $\Gal(\overline{\mathbb{F}}_q/\mathbb{F}_q)$. We also abuse
this notation slightly to write $\Gamma=\Gal(\fq{p}/\Fq)$ where the difference is unimportant.

\item In Chapter \ref{ch:E^0(BGL_d(Fq))} we write $I$ for the ideal of $E^0(BGL_p(\Fq))$ generated by $t$.

\item We let $J$ be the set $\left\{\alpha\in \mathbb{N}^p\mid
0\leq\alpha_1\leq\ldots\leq\alpha_p<p^{nv}\text{ and }\alpha_1<\alpha_p\right\}.$

\item We use $K$ to denote a finite field of characteristic coprime to
$p$.\footnote{Not to be confused with the cohomology theory $K$; context should ensure there is no ambiguity.}

\item We write $L$ for the extension of $\mathbb{Q}\tensor E^0$ formed by adjoining a complete set of roots of
$[p^m](x)$ for each $m>0$.

\item We use $l$ to denote a chosen prime number different to $p$.

\item We let $N_d$ (or $N$) denote the normalizer of $T_d$ in $GL_d(\Fq)$.

\item We use $n$ to denote the integer corresponding to the height of the Morava $E$-theory.

\item We use $p$ to denote the prime at which the Morava $E$-theory is localised.

\item We let $\Phi$ be the group $\Rep(\Theta^*,GL_1(\Flbar))=\Hom_\text{cts}(\Theta^*,\Flbar^\times)$.

\item We write $\psi_1,~\psi_2$ and $\psi_3$ to denote the restriction maps from $E^0(BN)$ to $E^0(BT)$,
$E^0(B(\Sigma_p\times\Delta))$ and $E^0(BA)$ respectively.

\item We let $q=l^r$ be a power of a chosen prime number different to $p$.

\item We write $t$ for the unique class in $E^0(BGL_p(\Flbar))$ which restricts to $\prod_i[p^v](x_i)$ in
$E^0(B(\Flbar^\times)^p)\simeq E^0\lpow x_1,\ldots,x_p\rpow$.

\item We let $T_d$ (or $T$) denote the maximal torus of $GL_d(\Fq)$. Similarly, we may write $\overline{T}_d$ for the
maximal torus of $GL_d(\overline{\mathbb{F}}_q)$.

\item We let $\Theta=(\mathbb{Z}/p^\infty)^n$ and $\thstar=\Hom_{\text{cts}}(\Theta,S^1)=\mathbb{Z}_p^n$.

\item $u$ denotes the invertible polynomial generator of $E^*$ lying in degree -2.

\item $u_1,\ldots,u_{n-1}$ denote the standard power-series generators of $E^0$ over $\mathbb{Z}_p$.

\item We let $v=v_p(q-1)$.

\item We let $w$ denote the standard generator of $E^0(BC_p)\simeq E^0\lpow w\rpow/[p](w)$.

\item We let $x$ denote the complex orientation or complex coordinate of $E$, or a restriction there-of.

\end{itemize}


\bibliographystyle{alpha}

\begin{thebibliography}{HKR00}

\bibitem[Ada74]{Adams}
J.~F. Adams.
\newblock {\em Stable homotopy and generalised homology}.
\newblock University of Chicago Press, Chicago, Ill., 1974.
\newblock Chicago Lectures in Mathematics.

\bibitem[Ada81]{Adamson}
Iain~T. Adamson.
\newblock {\em Introduction to field theory}.
\newblock Cambridge University Press, Cambridge, second edition, 1981.

\bibitem[AM04]{AdemMilgram}
Alejandro Adem and R.~James Milgram.
\newblock {\em Cohomology of finite groups}, volume 309 of {\em Grundlehren der
  Mathematischen Wissenschaften [Fundamental Principles of Mathematical
  Sciences]}.
\newblock Springer-Verlag, Berlin, second edition, 2004.

\bibitem[Ben91]{Benson}
D.~J. Benson.
\newblock {\em Representations and cohomology. {II}}, volume~31 of {\em
  Cambridge Studies in Advanced Mathematics}.
\newblock Cambridge University Press, Cambridge, 1991.
\newblock Cohomology of groups and modules.

\bibitem[Ben98]{BensonRepAndCohy1}
D.~J. Benson.
\newblock {\em Representations and cohomology. {I}}, volume~30 of {\em
  Cambridge Studies in Advanced Mathematics}.
\newblock Cambridge University Press, Cambridge, second edition, 1998.
\newblock Basic representation theory of finite groups and associative
  algebras.

\bibitem[Bor53]{Borel}
Armand Borel.
\newblock Sur la cohomologie des espaces fibr\'es principaux et des espaces
  homog\`enes de groupes de {L}ie compacts.
\newblock {\em Ann. of Math. (2)}, 57:115--207, 1953.

\bibitem[FM84]{FriedlanderMislin}
Eric~M. Friedlander and Guido Mislin.
\newblock Cohomology of classifying spaces of complex {L}ie groups and related
  discrete groups.
\newblock {\em Comment. Math. Helv.}, 59(3):347--361, 1984.

\bibitem[Fri82]{Friedlander}
Eric~M. Friedlander.
\newblock {\em \'{E}tale homotopy of simplicial schemes}, volume 104 of {\em
  Annals of Mathematics Studies}.
\newblock Princeton University Press, Princeton, N.J., 1982.

\bibitem[Fr{\"o}68]{Frohlich}
A.~Fr{\"o}hlich.
\newblock {\em Formal groups}.
\newblock Lecture Notes in Mathematics, No. 74. Springer-Verlag, Berlin, 1968.

\bibitem[GS99]{GreenleesStrickland}
J.~P.~C. Greenlees and N.~P. Strickland.
\newblock Varieties and local cohomology for chromatic group cohomology rings.
\newblock {\em Topology}, 38(5):1093--1139, 1999.

\bibitem[Hal76]{Hall}
Marshall Hall, Jr.
\newblock {\em The theory of groups}.
\newblock Chelsea Publishing Co., New York, 1976.
\newblock Reprinting of the 1968 edition.

\bibitem[Hat02]{Hatcher}
Allen Hatcher.
\newblock {\em Algebraic topology}.
\newblock Cambridge University Press, Cambridge, 2002.

\bibitem[Haz78]{Hazewinkel}
Michiel Hazewinkel.
\newblock {\em Formal groups and applications}, volume~78 of {\em Pure and
  Applied Mathematics}.
\newblock Academic Press Inc. [Harcourt Brace Jovanovich Publishers], New York,
  1978.

\bibitem[HKR00]{HKR}
Michael~J. Hopkins, Nicholas~J. Kuhn, and Douglas~C. Ravenel.
\newblock Generalized group characters and complex oriented cohomology
  theories.
\newblock {\em J. Amer. Math. Soc.}, 13(3):553--594 (electronic), 2000.

\bibitem[HS99]{HoveyStrickland}
Mark Hovey and Neil~P. Strickland.
\newblock Morava {$K$}-theories and localisation.
\newblock {\em Mem. Amer. Math. Soc.}, 139(666):viii+100, 1999.

\bibitem[Hun90]{Hunton}
John Hunton.
\newblock The {M}orava {$K$}-theories of wreath products.
\newblock {\em Math. Proc. Cambridge Philos. Soc.}, 107(2):309--318, 1990.

\bibitem[Kap58]{Kaplansky}
Irving Kaplansky.
\newblock Projective modules.
\newblock {\em Ann. of Math (2)}, 68:372--377, 1958.

\bibitem[Lan76]{Landweber}
Peter~S. Landweber.
\newblock Homological properties of comodules over {$M{\rm U}\sb\ast (M{\rm
  U})$}\ and {BP{$\sb\ast $}}({BP}).
\newblock {\em Amer. J. Math.}, 98(3):591--610, 1976.

\bibitem[Lan78]{Lang}
Serge Lang.
\newblock {\em Cyclotomic fields}.
\newblock Springer-Verlag, New York, 1978.
\newblock Graduate Texts in Mathematics, Vol. 59.

\bibitem[Lan02]{LangAlgebra}
Serge Lang.
\newblock {\em Algebra}, volume 211 of {\em Graduate Texts in Mathematics}.
\newblock Springer-Verlag, New York, third edition, 2002.

\bibitem[LT66]{LubinTate}
Jonathan Lubin and John Tate.
\newblock Formal moduli for one-parameter formal {L}ie groups.
\newblock {\em Bull. Soc. Math. France}, 94:49--59, 1966.

\bibitem[Mat89]{Matsumura}
Hideyuki Matsumura.
\newblock {\em Commutative ring theory}, volume~8 of {\em Cambridge Studies in
  Advanced Mathematics}.
\newblock Cambridge University Press, Cambridge, second edition, 1989.
\newblock Translated from the Japanese by M. Reid.

\bibitem[McC01]{McCleary}
John McCleary.
\newblock {\em A user's guide to spectral sequences}, volume~58 of {\em
  Cambridge Studies in Advanced Mathematics}.
\newblock Cambridge University Press, Cambridge, second edition, 2001.

\bibitem[Nak61]{Nakaoka}
Minoru Nakaoka.
\newblock Homology of the infinite symmetric group.
\newblock {\em Ann. of Math. (2)}, 73:229--257, 1961.

\bibitem[Qui69]{Quillen}
Daniel Quillen.
\newblock On the formal group laws of unoriented and complex cobordism theory.
\newblock {\em Bull. Amer. Math. Soc.}, 75:1293--1298, 1969.

\bibitem[Rav82]{RavenelK}
Douglas~C. Ravenel.
\newblock Morava {$K$}-theories and finite groups.
\newblock In {\em Symposium on Algebraic Topology in honor of Jos\'e Adem
  (Oaxtepec, 1981)}, volume~12 of {\em Contemp. Math.}, pages 289--292. Amer.
  Math. Soc., Providence, R.I., 1982.

\bibitem[Rav86]{RavenelCC}
Douglas~C. Ravenel.
\newblock {\em Complex cobordism and stable homotopy groups of spheres}, volume
  121 of {\em Pure and Applied Mathematics}.
\newblock Academic Press Inc., Orlando, FL, 1986.

\bibitem[Rav92]{RavenelNil}
Douglas~C. Ravenel.
\newblock {\em Nilpotence and periodicity in stable homotopy theory}, volume
  128 of {\em Annals of Mathematics Studies}.
\newblock Princeton University Press, Princeton, NJ, 1992.
\newblock Appendix C by Jeff Smith.

\bibitem[Seg68]{Segal}
Graeme Segal.
\newblock Classifying spaces and spectral sequences.
\newblock {\em Inst. Hautes \'Etudes Sci. Publ. Math.}, (34):105--112, 1968.

\bibitem[Ser77]{Serre}
Jean-Pierre Serre.
\newblock {\em Linear representations of finite groups}.
\newblock Springer-Verlag, New York, 1977.
\newblock Translated from the second French edition by Leonard L. Scott,
  Graduate Texts in Mathematics, Vol. 42.

\bibitem[Sha00]{Sharp}
R.~Y. Sharp.
\newblock {\em Steps in commutative algebra}, volume~51 of {\em London
  Mathematical Society Student Texts}.
\newblock Cambridge University Press, Cambridge, second edition, 2000.

\bibitem[Str97]{FSFG}
Neil~P. Strickland.
\newblock Finite subgroups of formal groups.
\newblock {\em J. Pure Appl. Algebra}, 121(2):161--208, 1997.

\bibitem[Str98]{StricklandSymmetricGroups}
N.~P. Strickland.
\newblock Morava {$E$}-theory of symmetric groups.
\newblock {\em Topology}, 37(4):757--779, 1998.

\bibitem[Str00]{StricklandK(n)duality}
Neil~P. Strickland.
\newblock {$K(n)$}-local duality for finite groups and groupoids.
\newblock {\em Topology}, 39(4):733--772, 2000.

\bibitem[Tan95]{Tanabe}
Michimasa Tanabe.
\newblock On {M}orava {$K$}-theories of {C}hevalley groups.
\newblock {\em Amer. J. Math.}, 117(1):263--278, 1995.

\bibitem[Wei94]{Weibel}
Charles~A. Weibel.
\newblock {\em An introduction to homological algebra}, volume~38 of {\em
  Cambridge Studies in Advanced Mathematics}.
\newblock Cambridge University Press, Cambridge, 1994.

\bibitem[Yag76]{Yagita}
Nobuaki Yagita.
\newblock The exact functor theorem for {${\rm BP}\sb\ast /I\sb{n}$}-theory.
\newblock {\em Proc. Japan Acad.}, 52(1):1--3, 1976.

\end{thebibliography}


\end{document}